\documentclass[a4paper]{article}
\usepackage{stackengine}

\usepackage[all]{xy}\usepackage[latin1]{inputenc}        %accents
\usepackage[dvips]{graphics,graphicx}
\usepackage{amsfonts,amssymb,amsmath,xcolor,mathrsfs, amstext}
\usepackage{amsbsy, amsopn, amscd, amsxtra, amsthm, authblk, enumerate}
\usepackage{upref}
\usepackage{geometry}
\usepackage[displaymath]{lineno}
\usepackage{float}
\usepackage{bbm}
\usepackage[shortlabels]{enumitem}
\usepackage{multicol}

\usepackage[colorlinks,
            linkcolor=blue,
            anchorcolor=green,
            citecolor=blue
            ]{hyperref}

\numberwithin{equation}{section}

\def\epsilon{\varepsilon}

\newcommand{\wt}{\widetilde}
\def\alb#1\ale{\begin{align*}#1\end{align*}}
\newcommand{\eqb}{\begin{equation}}
\newcommand{\eqe}{\end{equation}}

\newcommand{\BB}{\mathbbm}

\newcommand{\op}{\operatorname}

\newcommand{\ep}{\varepsilon}

\newcommand{\wh}{\widehat}

\newcommand{\bdy}{\partial}

\newtheorem{theorem}{Theorem}[section]
\newtheorem{lemma}[theorem]{Lemma}

\newtheorem{proposition}[theorem]{Proposition}
\newtheorem*{proposition*}{Proposition}
\newtheorem{corollary}[theorem]{Corollary}
\newtheorem{remark}[theorem]{Remark}
\newtheorem{definition}[theorem]{Definition}
\newtheorem*{definitions*}{Definitions}

\newtheorem*{example*}{\bf Example}
\theoremstyle{remark}

\newtheorem{prob}[theorem]{\bf Problem}

\numberwithin{equation}{section}

\title{Tightness of exponential metrics for log-correlated Gaussian fields in arbitrary dimension}
\author{Jian Ding\thanks{Peking University}\qquad Ewain Gwynne\thanks{University of Chicago}\qquad Zijie Zhuang\thanks{University of Pennsylvania}}

\begin{document}

%Macros for hyperlinks

\newcommand{\KK}{{\hyperref[K-condition1]{\mathfrak K}}}
\newcommand{\rr}{{\hyperref[K-condition2]{\mathfrak r_0}}}

\newcommand{\mathscrL}{{\hyperref[eq:def-rescaled-lattice]{\mathscr{L}}}}
 
\newcommand{\mathscrE}{{\hyperref[def:e-m-n]{\mathscr{E}}}}
\newcommand{\mathscrG}{{\hyperref[def:e-n-n]{\mathscr{G}}}}
\newcommand{\mathscrF}{{\hyperref[eq:def-f-m-n]{\mathscr{F}}}}

\newcommand{\alphaC}{{\hyperref[eq:def-lambda]{\alpha}}}
\newcommand{\lambdaC}{{\hyperref[eq:def-lambda]{\lambda}}}
\newcommand{\RC}{{\hyperref[eq:def-lambda]{R}}}
\newcommand{\mm}{{\hyperref[eq:def:mm]{m}}}
\newcommand{\mathscrY}{{\hyperref[eq:def-mathscrY]{\mathscr{Y}}}}
\newcommand{\qq}{{\hyperref[eq:def-qj]{q}}}
\newcommand{\mathcalJ}{{\hyperref[eq:def-j12]{\mathcal{J}}}}
\newcommand{\mathcalX}{{\hyperref[prop:cover]{\mathcal{X}}}}
\newcommand{\mathcalU}{{\hyperref[eq:def-mathcal-U]{\mathcal{U}}}}
\newcommand{\mathcalXp}{{\hyperref[eq:def-mathcalXp]{\mathcal{X}'}}}
\newcommand{\mathcalW}{{\hyperref[eq:lem5.9-1]{\mathcal{W}}}}

\maketitle

\begin{abstract}
We prove the tightness of a natural approximation scheme for an analog of the Liouville quantum gravity metric on $\mathbb R^d$ for arbitrary $d\geq 2$. More precisely, let $\{h_n\}_{n\geq 1}$ be a suitable sequence of Gaussian random functions which approximates a log-correlated Gaussian field on $\mathbb R^d$. Consider the family of random metrics on $\mathbb R^d$ obtained by weighting the lengths of paths by $e^{\xi h_n}$, where $\xi > 0$ is a parameter. We prove that if $\xi$ belongs to the subcritical phase (which is defined by the condition that the distance exponent $Q(\xi)$ is greater than $\sqrt{2d}$), then after appropriate re-scaling, these metrics are tight and that every subsequential limit is a metric on $\mathbb R^d$ which induces the Euclidean topology. We include a substantial list of open problems.
\end{abstract}

\tableofcontents

\bigskip
\noindent
\textbf{Acknowledgments.} We thank Karim Adiprasito, Timothy Budd, Hugo Falconet, Josh Pfeffer, Scott Sheffield, and Xin Sun for helpful discussions. We also thank the anonymous referees for their careful and insightful comments, which helped improve the presentation of this work. J.D.\ is partially supported by NSFC Key Program Project No.\ 12231002. E.G.\ was partially supported by a Clay research fellowship and by NSF grant DMS-2245832. Z.Z.\ is partially supported by NSF grant DMS-1953848.
\bigskip

\section{Introduction}

There has been an enormous amount of research in the past several decades concerning random geometry in two dimensions. Some of the major topics in this subject include Schramm-Loewner evolution, conformal field theory, statistical mechanics models on planar lattices, random planar maps, Liouville quantum gravity, and random geometries related to the KPZ universality class. We will not attempt to survey this vast literature here, but see, e.g.,~\cite{sheffield-icm,gwynne-ams-survey,bp-lqg-notes,bn-sle-notes,pw-gff-notes,ghs-mating-survey,vargas-dozz-notes,legall-sphere-survey,ganguly-dl-survey} for some recent expository articles.  
However, most of the results in this area have not been extended to higher dimensions. One reason for this is that conformal invariance (or covariance) plays a central role in many of the results in two dimensions, and there are no non-trivial conformal maps in higher dimensions. Another reason is that many of the arguments in the two-dimensional case rely on topological properties which are not true in higher dimensions, e.g., the Jordan curve theorem.

In this paper, we consider the problem of constructing an analog of the \textit{Liouville quantum gravity} (LQG) metric on $\mathbb R^d$, for arbitrary $d\geq 2$. Heuristically speaking, LQG is the random geometry described by the random Riemannian metric tensor
\begin{equation} \label{eqn-lqg-tensor}
e^{\gamma h} (dx^2 + dy^2) 
\end{equation} 
where $\gamma \in (0,2]$ is a parameter, $dx^2 + dy^2$ is the Euclidean metric tensor, and $h$ is a variant of the \textit{Gaussian free field} (GFF) on $\mathbb R^2$ (or more generally on a Riemann surface). See, e.g.,~\cite{shef-gff,bp-lqg-notes,pw-gff-notes} for an introduction to the GFF. The definition~\eqref{eqn-lqg-tensor} does not make literal sense since $h$ is a generalized function (distribution) instead of a true function, so its exponential cannot be defined pointwise. Nevertheless, one can define various objects associated with~\eqref{eqn-lqg-tensor} by replacing $h$ with a sequence of continuous functions that approximates $h$, and then taking an appropriate limit. 

Perhaps the easiest object to construct in this way is the LQG area measure, which is a limit of regularized versions of $e^{\gamma h} \,dx\,dy$ (where $dx\,dy$ denotes Lebesgue measure). The construction of this measure is a special case of the theory of \textit{Gaussian multiplicative chaos} (GMC), which allows one to make sense of random measures of the form $e^{\alpha h(x)} \,d\sigma(x)$ for $\alpha > 0$, whenever $h$ is a log-correlated Gaussian field on a domain $U\subset \mathbb R^d$ (for arbitrary $d\geq 1$) and $\sigma$ is an appropriate deterministic base measure on $U$. See~\cite{shef-kpz,rhodes-vargas-review,bp-lqg-notes} for more on Gaussian multiplicative chaos and the LQG area measure.

Recent works have also constructed the Riemannian distance function associated with~\eqref{eqn-lqg-tensor}, i.e., the \textit{LQG metric}. This is a random metric $D_h$ on $\mathbb R^2$ constructed as follows. For $\ep > 0$, let $h_\ep$ be the convolution of the Gaussian free field with the heat kernel $p_{\ep^2/2}(z) = \frac{1}{ \pi \ep^2} e^{-|z|^2/\ep^2}$. Also let $\xi = \xi(\gamma) = \gamma/d_\gamma$, where $d_\gamma$ is the so-called LQG dimension exponent~\cite{dg-lqg-dim}. Then, let
\begin{equation} \label{eqn-lqg-metric-approx}
D_h^\ep(z,w) := \inf_{P : z\to w} \int_0^1 e^{\xi h_\ep(P(t))} |P'(t)| \,dt ,\quad\forall z,w\in \mathbb R^2,
\end{equation}
where the infimum is over all piecewise continuously differentiable paths $P:[0,1] \to\mathbb{R}^2$ from $z$ to $w$. 
The papers~\cite{dddf-lfpp,gm-uniqueness} prove that there exist normalizing constants $\{\mathfrak a_\ep\}_{\ep > 0}$ such that $\mathfrak a_\ep^{-1} D_h^\ep$ converges in probability to a limiting metric with respect to the topology of uniform convergence on compact subsets of $\mathbb R^2\times\mathbb R^2$ (the convergence in probability was recently improved to a.s.\ convergence in~\cite{devlin-lfpp-as}). In particular, it was shown in~\cite{dddf-lfpp} that the approximating metrics are tight, and in~\cite{gm-uniqueness} (building on~\cite{local-metrics,gm-confluence,lqg-metric-estimates}) that the subsequential limit is unique. The proofs in these papers are much more difficult than the proofs in the construction of the LQG area measure. Intuitively, this is because the minimizing path in~\eqref{eqn-lqg-metric-approx} depends on $\ep$. See~\cite{ddg-metric-survey} for a survey of known results about the LQG metric. 

In light of the theory of Gaussian multiplicative chaos, it is natural to wonder whether there is an analogous theory of exponential metrics associated with log-correlated Gaussian fields on $\mathbb R^d$ for arbitrary\footnote{Note that when $d=1$, the metric induced by $e^{\xi h}$ is simply given by the one-dimensional GMC measure, as any path in $\mathbb R$ is an interval.} 
$d\geq 2$, which generalizes the LQG metric. The construction of such a theory is listed as Problem 7.19 in \cite{gm-uniqueness}. 

This paper carries out the first major step toward such a theory: namely, we prove the tightness of a natural approximation scheme similar to~\eqref{eqn-lqg-metric-approx} for log-correlated Gaussian fields on $\mathbb R^d$ (in the full subcritical phase of $\xi$ values). That is, we carry out the higher-dimensional analog of~\cite{dddf-lfpp}. See Theorem~\ref{thm:tightness} below for a precise statement. We expect that it will be challenging, but possible to prove that the subsequential limit is unique (and characterized by a list of axioms similar to the ones that characterize the LQG metric in dimension two~\cite{gm-uniqueness}) by adapting the arguments in the two-dimensional case~\cite{gm-uniqueness,dg-uniqueness}. Indeed, these arguments do not use two-dimensionality in as fundamental a way as the proof of tightness in~\cite{dddf-lfpp}. See Problem~\ref{prob:uniqueness} for further discussion.

More speculatively, our limiting metric might have connections to other higher-dimensional extensions of objects related to LQG, e.g., Liouville conformal field theory in even dimensions~\cite{cercle-higher-dimension, dhks-even-dim}, the higher-dimensional analogs of the Brownian map considered in~\cite{ml-iterated-folding}, uniform samples from various classes of triangulations of higher-dimensional spheres (see, e.g.,~\cite{bz-locally-constructible,dj-3-manifolds}), higher-dimensional analogs of random planar maps constructed from trees~\cite{BC23,budd-lionni-3-spheres}, and random graphs in $\mathbb R^d$ arising from sphere packings (see, e.g.,~\cite{bc-sphere-packing,bg-rw-sphere-packing}). See Subsection~\ref{subsec:potential-relation} for more details.

The problem of constructing natural random Riemannian metrics in dimension $d\geq 3$ is also of substantial interest in theoretical physics in the context of quantum gravity (see, e.g., the books~\cite{gh-quantum-gravity,adj-quantum-geometry,rovelli-quantum-geometry}). We refer to the introductions of~\cite{BC23,budd-lionni-3-spheres} for additional relevant discussion and references.

The proofs in this paper are by necessity substantially different from those in the two-dimensional case~\cite{dddf-lfpp}. In particular, we do not have an a priori \textit{Russo-Seymour-Welsh} (RSW) type estimate (which in the two-dimensional case comes from a conformal invariance argument), and various path-joining arguments in~\cite{dddf-lfpp} do not work in higher dimensions. For these reasons, we use a fundamentally novel approach to proving tightness which bypasses any direct use of RSW estimates as well as the Efron-Stein inequality. See Subsection~\ref{subsec:outline} for details.

The results of this paper open up a number of interesting questions about random metrics on $\mathbb R^d$. See Section~\ref{sec:open-problem}  for a discussion of some open problems.

\subsection{Definitions and main result}
\label{subsec:intro-1}
We now introduce some notation and state the main result of the paper. We consider the space $\mathbb{R}^d$ with $d \geq 2$ and define the box 
\begin{equation} \label{eq:box-def}
B_r(x) := x + (-r,r)^d ,\quad \forall x \in \mathbb{R}^d, \quad \forall r>0 . 
\end{equation}
Fix a smooth function $\KK :\mathbb{R}^d \rightarrow [0,\infty)$ and $\rr >0$ such that
\begin{enumerate}
    \item \label{K-condition1} $\KK$ is radially symmetric, meaning that $\KK(x) = \KK(y)$ for any $x,y \in \mathbb{R}^d$ with the same Euclidean norm.
    \item \label{K-condition2} $\KK$ is supported\footnote{
    We expect that our arguments can be adapted to the case when $\KK$ is not compactly supported but has sufficiently rapid decay at $\infty$. This would require some added technicalities similar to the ones encountered in~\cite{dddf-lfpp}. However, the choice of $\KK$ in this paper is in some sense unimportant since, regardless of the choice of $\KK$, the fields we consider are closely related to the canonical log-correlated Gaussian field on $\mathbb{R}^d$ considered in~\cite{lgf-survey, fgf-survey} (see Remark~\ref{remark-log-correlated}). }
    in the box $B_{\rr}(0)$.
    \item \label{K-condition3} $\KK$ is normalized such that $\int_{\mathbb{R}^d} \KK(x)^2 dx = 1$.
\end{enumerate} 

We also let $W$ be a space-time white noise on $\mathbb{R}^d$. That is, $W$ is the Gaussian random generalized function on $\mathbb{R}^d \times (0,\infty)$ such that for any $f \in L^2 (\mathbb{R}^d \times (0,\infty))$, the formal integral $\int_{\mathbb{R}^d} \int_0^\infty f(y,t) W(dy,dt)$ is centered Gaussian with variance $\|f\|_{L^2}^2$. 

We consider a log-correlated Gaussian field $h$ and its approximation $h_n$, defined as follows:
\begin{equation} \label{eq:field-def}
\begin{aligned}
   h(x) &= \int_{\mathbb{R}^d} \int_0^1 \KK\big(\frac{y-x}{t}\big) t^{-\frac{d+1}{2}} W(dy,dt) \quad \mbox{and} \\
    h_n(x) &= \int_{\mathbb{R}^d} \int_{2^{-n}}^1 \KK\big(\frac{y-x}{t}\big) t^{-\frac{d+1}{2}} W(dy,dt)
\end{aligned}
\end{equation}
for $x \in \mathbb{R}^d$ and integer $n \geq 1$. From the definition of $W$, we see that $h$ and $h_n$ are centered Gaussian processes with covariance kernels
\begin{align}  \label{eq:covariance}
\mathrm{Cov}(h(x_1),h(x_2)) &= \int_0^1 \frac{1}{t} (\KK*\KK)\left( \frac{x_1-x_2}{t} \right)  \,dt \quad \mbox{and} \notag\\
\mathrm{Cov}(h_n(x_1),h_n(x_2))&= \int_{2^{-n}}^1 \frac{1}{t} (\KK*\KK)\left( \frac{x_1-x_2}{t} \right)  \,dt  \,,
\end{align}
where $\KK*\KK$ denotes the convolution. Using the representation~\eqref{eq:field-def} and the fact that $W$ is a random tempered distribution (see e.g.\ Section 2.3 of \cite{fgf-survey}), one can verify that each $h_n$ has a modification which is a smooth function (see also Proposition 2.1 of \cite{df-lqg-metric}). We henceforth assume that each $h_n$ has been replaced by such a modification. Furthermore, from~\eqref{eq:covariance} we get $\op{Var} h_n(x) = n \log 2$ for each $x\in\mathbb R^d$. The process $h$ is interpreted as a random generalized function, and is closely related to the log-correlated Gaussian field on $\mathbb R^d$ considered in~\cite{lgf-survey, fgf-survey} (see Remark~\ref{remark-log-correlated}).

Analogously\footnote{As explained in~\cite{dddf-lfpp} (see also~\cite[Section 2.1]{cg-support-thm}), in the two-dimensional case, the convolution of the planar Gaussian free field with the heat kernel (at an appropriate $n$-dependent time) has the same law as the field $h_n$ of~\eqref{eq:field-def} with $\KK(x) = \sqrt{\frac{2}{\pi}} e^{-|x|^2}$, up to adding a random continuous function. Hence~\eqref{eq:metric-def} is directly analogous to~\eqref{eqn-lqg-metric-approx}. To avoid unnecessary technical work, in this paper we require that $\KK$ is compactly supported, but we expect that our results can be fairly easily extended to the case where $\KK$ is not compactly supported but has sufficiently fast decay at $\infty$.}
to~\eqref{eqn-lqg-metric-approx}, for a parameter $\xi>0$, we define the exponential metric associated with $h_n$ as follows:
\begin{equation} \label{eq:metric-def}
    D_n(z,w) := \inf_{P : z\to w} \int_0^1 e^{\xi h_n(P(t))} |P'(t)| dt\,, \quad \forall z,w\in\mathbb R^d,
\end{equation}
where the infimum is taken over all piecewise continuously differentiable paths $P:[0,1] \rightarrow \mathbb{R}^d$ joining $z,w$. This can be interpreted as an approximation of the random metric formally given by reweighting the Euclidean lengths of paths by $e^{\xi h}$. We will be interested in (subsequential) limits of the renormalized metrics $\lambda_n^{-1} D_n$, where the normalizing constant $\lambda_n$\footnote{For technical reasons, we first work with this particular choice of normalizing constant. However, in the end, we can choose any reasonable normalizing constant, such as the median of $D_n(0,e_1)$ or $D_n(\partial B_1(0), \partial B_2(0))$.} is defined as: 
\begin{equation}
\label{eq:intro-lambda}
    \lambda_n := \mbox{median of } D_n(0,e_1; B_2(0))\, ,
\end{equation}
where $D_n(0,e_1; B_2(0))$ denotes the minimal $D_n$-length of paths joining $0$ and $e_1 := (1,0,\ldots,0)$ inside the box $B_2(0)$. 

In Section~\ref{sec:exponent}, we will prove the following.

\begin{proposition} \label{prop:exponent-intro}
For each $\xi  >0$, there exists $Q = Q(\xi) \in \mathbb{R}$ such that
\begin{equation}
\label{eq:exponent-relation-lambda}
    \lambda_n = 2^{-(1-\xi Q)n + o(n)} \quad \mbox{as } n \rightarrow \infty\,.
\end{equation}
Furthermore, $\xi \mapsto Q(\xi)$ is a continuous, non-increasing function and we have
\begin{equation}\label{eq:Q-bounds}
\frac{1}{\xi} - \sqrt{2d} \leq Q(\xi) \leq \frac{1}{\xi} + \sqrt{2} \,,\quad \forall \xi > 0 .
\end{equation}
\end{proposition}

The proof of Proposition~\ref{prop:exponent-intro} is via a subadditivity argument. Just like in the two-dimensional case, we do not know the value of $Q(\xi)$ explicitly (see Problems~\ref{prob:positive-Q} and~\ref{prob:special}). Analogously to the two-dimensional case (see~\cite[Equation (1.4)]{dg-supercritical-lfpp}), we define the \textit{critical value}
\begin{equation} \label{eqn-critical}
\xi_{\mathrm{crit}} := \sup\left\{ \xi  >0 : Q(\xi) > \sqrt{2d} \right\} .
\end{equation} 
See Remark~\ref{remark-critical} for some discussion of why this value is critical.
We note that $\xi < \xi_{\mathrm{crit}}$ if and only if $Q(\xi) > \sqrt{2d}$. The lower bound $Q(\xi) \geq \frac{1}{\xi} - \sqrt{2d}$ from~\eqref{eq:Q-bounds} implies that $\xi_{\mathrm{crit}} \geq \frac{1}{2\sqrt{2d}}$, and the upper bound $Q(\xi) \leq \frac{1}{\xi} + \sqrt{2}$ implies that $\xi_{\rm crit} \leq \frac{1}{\sqrt{2d} - \sqrt{2}} < \infty$. The main result of the paper is the tightness of our approximating metrics in the full subcritical phase.

\begin{theorem}
\label{thm:tightness}
When $\xi < \xi_{\mathrm{crit}}$, equivalently $Q(\xi) > \sqrt{2d}$, the sequence of metrics $\{\lambda_n^{-1} D_n(\cdot,\cdot)\}_{n \geq 1}$ is tight with respect to the topology of uniform convergence on compact subsets of $\mathbb{R}^d \times \mathbb{R}^d$. Furthermore, each possible subsequential limit (in distribution) is a metric on $\mathbb{R}^d$ which induces the Euclidean topology.
\end{theorem}

At an intuitive level, the subcritical condition is what makes a hierarchical construction possible. After decomposing the field into coarse- and fine-scale increments, one can use the coarse field to locate a backbone for a near-geodesic and then use the finer fields to determine its geometry at successively smaller scales. The inequality $Q(\xi)>\sqrt{2d}$ ensures that extreme fine-scale fluctuations are sufficiently sparse that the path can detour around the exceptional regions at a summable cost. Thus, the fine field refines the local geometry without repeatedly reorganizing the path at macroscopic scales. The chaining and multiscale comparison arguments below make this intuition quantitative.

\begin{remark}  \label{remark-critical} 
When $Q(\xi) < \sqrt{2d}$, the metrics $\lambda_n^{-1} D_n$ are not tight with respect to the topology of uniform convergence on compact subsets of $\mathbb R^d \times \mathbb R^d$. So, our result is optimal modulo the critical case when $Q(\xi) = \sqrt{2d}$. 
Indeed, the maximum of $h_n$ on a fixed bounded open set $U\subset\mathbb R^d$ grows like $(\sqrt{2d} +o(1)) (\log 2)  n$ as $n\to\infty$; see e.g.\ \cite{Mad15}. 
From this and the continuity properties of $h_n$ (Claim (2) of Lemma~\ref{lem:field-estimate}), we see that for any fixed $\ep > 0$, with high probability as $n \to \infty$, there exists $z\in U$ such that $h_n(w) \geq (\sqrt{2d} - \ep) (\log 2) n$ for all $w \in B_{2^{-n}}(z)$. For this choice of $z$, the definition of $D_n$ implies that the $D_n$-distance from $z$ to $\partial B_{2^{-n}}(z)$ is at least $2^{  [ (\sqrt{2d} - \ep)\xi - 1 ] n}$. By~\eqref{eq:exponent-relation-lambda}, 
\[
\lambda_n^{-1} D_n(z,\partial B_{2^{-n}}(z)) \geq 2^{( \sqrt{2d} - Q  - \ep     + o(1) ) \xi  n} .
\]
If $Q(\xi)  <\sqrt{2d}$, then for a small enough choice of $\ep$, this goes to $\infty$ as $n\to\infty$, which means that $\lambda_n^{-1} D_n$ cannot be tight with respect to the local uniform topology.  

In the two-dimensional case, it was shown in \cite{dg-supercritical-lfpp,dg-uniqueness} that the re-scaled approximating metrics converge with respect to the topology on lower semicontinuous functions for all $\xi  >0$ (including when $Q(\xi) \leq 2$). 
However, when $Q(\xi)<2$, the limiting metric does not induce the Euclidean topology on $\mathbb R^2$. Rather, there are uncountably many ``singular points" which lie at infinite distance from every other point. It is plausible that similar statements are true for general $d\geq 2$, but we do not address this in the present paper. See Problem~\ref{prob:supercritical}. 
\end{remark}

\begin{remark}  \label{remark-log-correlated}
The field $h$ in~\eqref{eq:field-def} is closely related to the log-correlated Gaussian field on $\mathbb R^d$ considered in \cite{lgf-survey, fgf-survey}. Indeed, define the random generalized function $h^\infty$ in the same manner as $h$ in~\eqref{eq:field-def}, but with $t$ integrated over $(0,\infty)$ instead of over $(0,1)$. Then, a short computation shows that for any choice of the kernel $\KK$ above, one can make sense of $h^\infty$ as a random generalized function viewed modulo additive constant\footnote{That is, $\int_{\mathbb R^d} h^\infty(x) g(x)\,dx$ makes sense whenever $g$ is smooth and compactly supported with $\int_{\mathbb R^d} g(x)\,dx = 0$.} and that $h^\infty$ agrees in law, modulo additive constant, with the log-correlated Gaussian field from \cite{lgf-survey, fgf-survey}. Furthermore, $h^\infty - h$ has a modification which is a continuous function, viewed modulo additive constant.  
This was discussed in~\cite[Section 4.1.1]{lgf-survey} and explained in detail in the two-dimensional case in~\cite[Appendix B]{afs-metric-ball} (the same proof works for any dimension). 
Due to the continuity of $h^\infty-h$, one can easily deduce from Theorem~\ref{thm:tightness} that a natural approximation scheme for the exponential metric associated with $h^\infty$ is also tight. 
\end{remark}

\subsection{Related models}
\label{subsec:potential-relation}

Since the construction of the LQG metric in~\cite{dddf-lfpp,gm-uniqueness}, there have been several additional works which prove tightness and/or uniqueness for various random fractal metrics. Examples include the supercritical LQG metric~\cite{dg-supercritical-lfpp,dg-uniqueness} (as mentioned in Remark~\ref{remark-critical}), the conformal loop ensemble chemical distance~\cite{miller-cle-metric}, and the limit of long-range percolation on $\mathbb Z^d$~\cite{baumler-long-range-perc,dfh-long-range-perc}. We also mention the \textit{directed landscape}, a random directed metric on $\mathbb R^2$ related to the KPZ universality class~\cite{dov-dl}.

An important feature of LQG is its relation with two-dimensional \textit{Liouville conformal field theory} (LCFT) rigorously constructed in \cite{dkrv-lqg-sphere} and follow-up works. The framework of LCFT can produce exact solvability results for the area and length measures associated with LQG surfaces when the underlying field is well chosen. Recently, two-dimensional LCFT has been extended to even dimensions $d\geq 4$ in the papers \cite{cercle-higher-dimension, dhks-even-dim}. Both of these works construct a log-correlated Gaussian field on a $d$-dimensional manifold whose law is re-weighted according to the so-called Liouville action. In other words, these works carry out analogs of \cite{dkrv-lqg-sphere} on certain $d$-manifolds. It should be possible to use the results of the present paper to associate a random metric with the fields considered in \cite{cercle-higher-dimension, dhks-even-dim}, at least as a subsequential limit. As in the two-dimensional case, the exponent $Q$ of Proposition~\ref{prop:exponent-intro} should correspond to the background charge in~\cite{cercle-higher-dimension} (which is also called $Q$).
 
In two dimensions, LQG is conjectured to describe the scaling limit of random planar maps. In particular, the LQG metric is believed to describe the scaling limit of the random planar maps equipped with their graph distance in, e.g., the Gromov-Hausdorff sense. This convergence has been rigorously established for uniform random planar maps toward LQG with $\gamma=\sqrt{8/3}$ ($\xi = 1/\sqrt 6$), but is open for other values of $\gamma$. More precisely, it was shown in~\cite{legall-uniqueness,miermont-brownian-map} that uniform random planar maps converge to a random metric space called the \textit{Brownian map}, and in~\cite{lqg-tbm1,lqg-tbm2} that the Brownian map is equivalent to $\sqrt{8/3}$-LQG, as a metric space. See also~\cite{hs-cardy-embedding} for a stronger topology of convergence and Section 2.4 of \cite{ddg-metric-survey} for further discussions.

It would be extremely interesting to find a natural discrete random geometry in dimension $d\geq 3$ whose scaling limit is described by one of the exponential random metrics considered in this paper (or some minor variant thereof). 

In analogy with the case of uniform triangulations in two dimensions (which converge to $\sqrt{8/3}$-LQG), a natural discrete model to consider is uniform triangulations of the $d$-dimensional sphere, with $n \in \BB N$ total $d$-simplices. 
Such triangulations appear to be very difficult to analyze. 
For example, it is a well-known open problem to determine whether the number of triangulations of the three-sphere with $n$ total tetrahedra grows exponentially or superexponentially~\cite{dj-3-manifolds,gromov-spaces}. 
%Moreover, there is no known polynomial time algorithm to determine if a given 3-dimensional triangulation has the topology of the 3-sphere. 
Moreover, simulations suggest that uniform triangulations of the three-sphere may not have interesting scaling limits when viewed as metric spaces; see, e.g.,~\cite{bk-3d-simplicial,av-3d-simplicial,abkv-3d-vacuum,ckr-3d-entropy,hty-3d-phases,hin-3d-simulation}. 
We refer to the introductions of~\cite{dj-3-manifolds,bz-locally-constructible,budd-lionni-3-spheres} and the references therein for further discussion. 

On the other hand, there are natural restricted classes of triangulations of $d$-spheres which appear to be more tractable, and whose cardinality can be shown to grow exponentially in $n$. Examples include locally constructible, constructible, shellable, and vertex-decomposable triangulations~\cite{dj-3-manifolds,bz-locally-constructible}. 
One could ask whether a uniform sample from any of these restricted classes converges in the Gromov-Hausdorff sense to the exponential metric associated with a log-correlated Gaussian field (or a field which locally looks like a log-correlated Gaussian field). 

Another potentially interesting class of discrete models are random graphs which can be represented as the tangency graph of a sphere packing in $\mathbb{R}^d$ (or the $d$-sphere). Unlike for $d=2$, there is not a simple criterion for when a graph can be represented by a sphere packing in $\mathbb{R}^d$ for $d\geq 3$. In fact, for several values of $d$, it is known that the problem of determining whether a given graph admits a sphere packing representation in $\mathbb{R}^d$ is NP hard~\cite{hk-disks-and-balls}. In dimension two, circle packings and their links to random conformal geometry are fairly well-understood (see, e.g., the survey~\cite{nachmias-circle-packing}). In higher dimensions the theory is much less well-developed and likely to be much more difficult. But, a few results can be found, e.g., in~\cite{cr-rigidity,bc-sphere-packing,lee-conformal-growth,bg-rw-sphere-packing}. One could look for a natural model of random sphere packings in $\mathbb{R}^d$ whose scaling limit is described by the exponential of a log-correlated Gaussian field. 

One can also consider graphs which admit other types of embeddings in $\mathbb{R}^d$, e.g., those which can be realized as a $d$-dimensional Delaunay triangulation, or those which can be represented as a $d$-dimensional orthogonal tiling graph in the sense of~\cite{bg-rw-sphere-packing} (which includes sphere-packable and Delaunay realizable graphs as special cases). One \textit{a priori} reason to expect that some types of orthogonal tiling graphs might be related to exponential metrics of log-correlated Gaussian fields is that orthogonal tiling graphs satisfy, in some sense, a discrete analog of conformal flatness~\cite[Section 1.4]{bg-rw-sphere-packing}. 
 
In another direction, connections between $\gamma$-LQG and random planar maps for general $\gamma\in (0,2)$ have been obtained using the framework of mating-of-trees theory \cite{wedges}; see the survey \cite{ghs-mating-survey}. The recent work \cite{BC23} presents an analog of mating-of-trees constructions in three dimensions. In a similar vein, the paper~\cite{budd-lionni-3-spheres} introduces a model of random triangulations of the three-sphere, decorated by a pair of trees, which is combinatorially tractable and has interesting geometric features. 
It is natural to wonder whether there are any mating-of-trees type constructions related to exponential metrics of log-correlated Gaussian fields in dimension $d\geq 3$. %A major obstacle is that naively mating $d \geq 3$ trees does not appear to produce a topological manifold, so direct generalizations of the mating-of-trees story for LQG in dimension two do not work. 

Recently, an analog of the Brownian map in dimension $d\geq 3$ was proposed in \cite{ml-iterated-folding}.  
In light of the aforementioned relationship between the Brownian map and $\sqrt{8/3}$-LQG, it is also natural to wonder whether this random metric space has any relation to the exponential metrics of log-correlated Gaussian fields.

\subsection{Outline}
\label{subsec:outline} 
Here, we outline the proof strategy of Theorem~\ref{thm:tightness} and describe the content of each subsequent section.

\subsubsection{Comparison to the two-dimensional case}

First, let us highlight the main differences between the method in this paper and those used in earlier works \cite{ding-dunlap-lqg-fpp, ding-dunlap-lgd, df-lqg-metric, dddf-lfpp, dg-supercritical-lfpp} to establish the tightness of approximations of exponential metrics for log-correlated fields in two dimensions. All the results in two dimensions rely crucially on RSW estimates, which give up-to-constants comparisons between quantiles of $D_n$-crossing lengths of rectangles in the ``easy direction" and the ``hard direction"; see e.g.\ Section 3 of \cite{dddf-lfpp}. The arguments to prove these estimates are based on either approximate conformal invariance or on forcing paths to cross each other, neither of which works in higher dimensions. For this reason, we will use a fundamentally different approach to prove tightness which bypasses any direct use of RSW estimates.

The first difference in our approach as compared to the two-dimensional case is that we initially use the median of the point-to-point distance, namely $\lambda_n$ from \eqref{eq:intro-lambda}, as the normalizing constant. In contrast, previous works use the median of the left-to-right crossing distance within a box as their normalizing constant (although these two medians are eventually proved to be equivalent up to constants). The point-to-point distance is typically larger than the left-to-right crossing distance, which makes it easier to upper-bound other types of distances in terms of $\lambda_n$. We choose to work with the internal point-to-point distance inside a box to ensure long-range independence, which allows us to apply percolation arguments. To implement this, we will actually work with the $q$-quantile of $D_n(0,e_1;B_2(0))$ for $q$ close to one, but not depending on $n$, in most parts of the proof.

The second difference arises from the lack of upper tail estimates in dimension $d \geq 3$. In previous works (see Sections 4 and 5 of \cite{dddf-lfpp} and Section 3 of \cite{dg-supercritical-lfpp}), upper tail estimates for the left-to-right crossing distance were derived using RSW estimates, percolation arguments, and the Efron-Stein inequality. Although the RSW argument is not applicable in our case, we can still hope to use percolation arguments to obtain an upper tail estimate. For integers $1 \leq k \leq n$, we can divide a box into $2^k$ sub-boxes and use the scaling property of $h_n$ (Lemma~\ref{lem:basic-h}) and percolation arguments to deduce an upper tail estimate for the $D_n$-distance across a hypercubic shell\footnote{A hypercubic shell is the region between two concentric boxes, which is the $d$-dimensional analog of a square annulus.}---for example, $D_n(\partial B_1(0), \partial B_2(0))$---in terms of $\lambda_{n-k}$ (in practice, we will use a large quantile rather than the median). A chaining argument similar to the two-dimensional case also allows us to upper-bound the $D_n$-diameter of a box in terms of  $(\lambda_{n - k})_{1 \leq k \leq n}$; see Proposition~\ref{prop:chain}. It turns out that a specific comparison bound between $\lambda_n$ and $\lambda_{n-k}$ for all integers $1 \leq k \leq n$, as detailed in Proposition~\ref{prop:compare}, is sufficient. This in turn ensures the tightness of the metric. Deriving this comparison is one of the novelties of the paper and is detailed in Section~\ref{sec:compare}. In that section, we will actually derive a comparison between the metrics $D_n$ and $D_{n-k}$, which may also be of independent interest.

The third difference arises from the lack of lower tail estimates in dimension $d \geq 3$ and from our choice of the normalizing constant. In previous works (see Section 4 of~\cite{dddf-lfpp}), a lower tail estimate for the left-to-right crossing distance follows from the RSW argument, which implies that each subsequential limit is a metric rather than a pseudo-metric. In our case, we will use a different and novel approach to demonstrate this. It is easy to show that for each subsequential limit, the distance across a hypercubic shell is positive with non-zero probability. The key idea is to boost this probability to one using a zero-one law argument (Lemma~\ref{lem:zero-one}). Intuitively, if the subsequential limit were a measurable function of $h$, then the zero-one law would follow directly from Kolmogorov's zero-one law, since the event that the distance is zero is a tail event of $h$. In our case, however, we do not \textit{a priori} know whether the subsequential limit is measurable with respect to $h$, so a more delicate argument is required. We refer to the proof of Lemma~\ref{lem:zero-one} for details.

\subsubsection{Detailed outline}
Next, we describe our strategy in more detail and outline the content of each section. More detailed overviews can be found at the beginning of the corresponding sections and subsections.

For the tightness argument, the reader can keep in mind two main quantitative goals. The first is the multiscale diameter estimate in Proposition~\ref{prop:chain}, which bounds the diameter of a box by a weighted sum of the quantiles $a_{n-m}^{(q)}$ over all Euclidean scales. The second is the comparison estimate of Section~\ref{sec:compare}, which relates these quantiles at different scales. Combining the two estimates turns the multiscale sum into a constant-order multiple of $a_n^{(q)}$, which is the key upper bound needed for tightness.

In Section~\ref{sec:prelim}, we provide preliminaries and fix some notation. Let $W(dx,dt)$ be the space-time white noise. Throughout the paper, we will work with the approximation of the log-correlated Gaussian field 
\[
h_{m,n}(x) = \int_{\mathbb{R}^d} \int_{2^{-n}}^{2^{-m}} \KK(\frac{x-y}{t}) t^{-\frac{d+1}{2}} W(dy, dt)\,,
\]
for integers $n > m \geq 0$ (note that $h_{0,n} = h_n$, as defined in~\eqref{eq:field-def}). Some basic properties and estimates of $h_{m,n}$ are provided in Subsection~\ref{subsec:est-log-field}. In Subsection~\ref{subsec:def-lfpp}, we define $D_{m,n}$ as the exponential metrics associated with $h_{m,n}$ and establish some basic properties of these metrics, including a Gaussian concentration bound (Lemma~\ref{lem:gauss-concentration}). Subsection~\ref{subsec:percolation} collects basic arguments about percolation with finite range of dependence, which will play a crucial role in Sections~\ref{sec:exponent} and \ref{sec:bound-distance}.

In Section~\ref{sec:exponent}, we will prove the existence of an exponent $Q$ satisfying~\eqref{eq:exponent-relation-lambda}. This follows from the subadditivity inequality: $\lambda_n \leq e^{C n^{2/3}} \lambda_m \lambda_{n-m}$ for integers $n > m \geq 1$, and the proof is similar to Proposition 2.5 of \cite{dg-supercritical-lfpp}. To establish this inequality, we construct a path connecting $0$ and $e_1$ of typical $D_n$-length in two steps. The first step is to construct a path on $2^{-m} \mathbb{Z}^d$ whose $D_m$-length can be upper-bounded. The second step is to locally modify this path so that its $D_n$-length can be controlled using a percolation argument on a refined lattice. %To use the percolation argument, we actually consider a large quantile of $D_n(0,e_1; B_2(0))$. 
In Lemma~\ref{lem:Q-lower}, we establish basic properties of $Q(\xi)$, which then completes the proof of Proposition~\ref{prop:exponent-intro}.

In Section~\ref{sec:bound-distance}, we upper-bound different types of distances. First, we present a chaining argument in Subsection~\ref{subsec:chaining}, similar to those in Section 6.3 of \cite{ding-dunlap-lqg-fpp} and Section 6.1 of \cite{df-lqg-metric}. Specifically, we construct paths of typical length at multiple scales to connect any two points in a box, which yields an upper bound for the $D_n$-diameter of the box in terms of large quantiles of $D_{n-m}(0,e_1;B_2(0))$ for $0 \leq m \leq n$. This result will be used subsequently in two places. First, we use it in Subsections~\ref{subsec:bound-diameter} and \ref{subsec:cross} to show that the medians of the box diameter and the distance across a hypercubic shell both satisfy the relation in \eqref{eq:exponent-relation-lambda}. Second, it will be used to prove the tightness of the metric in Subsection~\ref{subsec:tightness}, once a comparison between quantiles of $D_n(0,e_1;B_2(0))$ for different values of $n$ has been achieved. In Subsection~\ref{subsec:super-exponential}, we prove super-exponential concentration bounds for distances across and around hypercubic shells, which will be used in Section~\ref{sec:compare}.

In Section~\ref{sec:compare}, we compare the metrics $D_n$ and $D_{n+k}$ for integers $n \geq k \geq 1$; see Proposition~\ref{prop:compare} for the main result. We briefly describe the strategy here and refer to Subsection~\ref{subsec:sec5-strategy} for a more detailed outline. The comparison is based on controlling the behavior of the field $h_{n,n+k}$ (note that $D_{n+k}$ is obtained from $D_n$ by adding $h_{n,n+k}$ to the field). In most parts of the space, $h_{n,n+k}$ behaves well, and $D_{n+k}$ and $D_n$ satisfy the desired bound in Proposition~\ref{prop:compare}. However, there are places where $h_{n,n+k}$ does not behave well, and \textit{a priori}, it is possible that a $D_n$- or $D_{n+k}$-geodesic spends most of its time in these problematic regions. Our main effort is to control the impact of these regions on the metrics, which we do in two steps. This is in a similar spirit to the role played by Efron-Stein inequality in~\cite{dddf-lfpp, dg-supercritical-lfpp} to bound the variance of the metric, but our argument is more quantitative and applies under weaker \textit{a priori} concentration bounds. In the first step, we use a coarse-graining argument to show that, with high probability, the problematic regions can be covered by boxes at different scales. Importantly, all these boxes satisfy the condition that the $D_n$-distance around the hypercubic shell enclosing the box can be upper-bounded by the $D_n$-distance across a larger hypercubic shell. In the second step, we use this condition to show that the ill-behaved field within these boxes has a minor impact on the metric $D_n$. Specifically, paths can be modified to avoid these boxes, and their $D_n$-length increases by no more than a constant factor. Moreover, for paths entirely contained in the domain where $h_{n,n+k}$ behaves well, we can adjust the paths so that their $D_n$-length and $D_{n+k}$-length satisfy the desired comparison bound. This leads to the desired comparison between $D_n$ and $D_{n+k}$.

In Section~\ref{sec:final-proof}, we complete the proof of Theorem~\ref{thm:tightness}. The proof consists of two parts. In Subsection~\ref{subsec:tightness}, we combine results from the chaining argument in Subsection~\ref{subsec:chaining} and the comparison of quantiles from Section~\ref{sec:compare} to show the tightness of $D_n$ when normalized by the $q$-quantile of $D_n(0,e_1;B_2(0))$, where $q$ is close to one but independent of $n$. In Subsection~\ref{subsec:non-degenerate}, we show that each possible subsequential limit is a metric rather than a pseuodo-metric. From the definition of quantiles and the positive association (FKG) for $h_n$, we first show that the distance across a hypercubic shell is bounded away from zero with positive probability. We then use the locality property of the metric to prove a zero-one law (Lemma~\ref{lem:zero-one}), which allows us to increase this probability to one. Applying the result to countably many hypercubic shells shows that each subsequential limit is a metric. This in turn yields an up-to-constants comparison between the median $\lambda_n$ and the $q$-quantile of $D_n(0,e_1;B_2(0))$, which implies tightness when we normalize by $\lambda_n$ instead of the $q$-quantile. 

In Section~\ref{sec:open-problem}, we list several open problems related to the exponential metrics. Appendix~\ref{appendix:index} contains a list of notation used in the paper.

\section{Preliminaries}
\label{sec:prelim}
\subsection{Basic notation}

\subsubsection*{Numbers}

We write $\mathbb{N} = \{1,2,\ldots \}$. Without specific mention, the logarithm in this paper will be taken with respect to the base $e$. For $a \in \mathbb{R}$, we use $\lfloor a \rfloor$ to represent the largest integer not greater than $a$. For a random variable $X$, we use ${\rm Med}(X)$ to represent its median.

\subsubsection*{Metrics}
Let $(X,D)$ be a metric space. For a curve $P:[a,b] \rightarrow X$, the $D$-length of $P$ is defined as
\begin{equation*}
{\rm len}(P;D):=\sup_T \sum_{i=1}^{n} D(P(t_i), P(t_{i-1}))
\end{equation*}
where the supremum is taken over all partitions $T: a=t_0 < t_1<\ldots <t_n = b$ of $[a,b]$. The $D$-length of a curve may be infinite.

For a curve $P:[a,b] \rightarrow X$ and a set $Y \subset X$, consider the pre-image $P^{-1}(Y) \subset [a,b]$. Write the interior of $P^{-1}(Y)$ as the disjoint union of countably many open intervals $\{(a_i,b_i)\}_{i \geq 1}$. We define the restriction of $P$ to $Y$ as $P|_Y := \cup_{i \geq 1} P([a_i,b_i])$, which is the union of a family of curves, and its length is defined as 
\begin{equation}
\label{eq:restrict-length}
    {\rm len}(P|_Y; D) := \sum_{i \geq 1} {\rm len}(P([a_i,b_i]) ; D)\,.
\end{equation}
Note that $P|_Y$ and $P(P^{-1}(Y))$ are the same up to the inclusion of end points (of intervals in $P^{-1}(Y)$) or single points (i.e., each interval containing them is not a subset of $P^{-1}(Y)$). For the sets $Y$ that we will consider in this paper, the lengths of $P|_Y$ and $P(P^{-1}(Y))$ will be the same.

For $Y \subset X$, the internal metric of $D$ on $Y$ is defined as
\begin{equation}
    \label{def:int-metric}
    D(x,y; Y) := \inf_{P \subset Y}{\rm len}(P; D), \quad \forall x,y \in Y
\end{equation}
where the infimum is taken over all paths $P$ in $Y$ from $x$ to $y$. Then $D(\cdot, \cdot; Y)$ is a metric on $Y$, allowing the distance between two points to be infinite.

We say $D$ is a \textit{length metric} if for all $x,y \in X$ and $\delta>0$, there exists a curve with $D$-length at most $D(x,y) + \delta$ connecting $x$ and $y$. We say $D$ is a \textit{geodesic metric} if for all $x,y \in X$, there exists a curve with $D$-length precisely $D(x,y)$ connecting $x$ and $y$.

\subsubsection*{Subsets of Euclidean space}

In this paper, we consider the space $\mathbb{R}^d$ where $d \geq 2$ is a fixed dimension. For $z \in \mathbb{R}^d$, we write $z = (z_1,\ldots,z_d)$ for its coordinates. We use the notation $|\cdot|_1$, $|\cdot|_2$, and $|\cdot|_\infty$ to represent the $l^1$-, $l^2$-, and $l^\infty$-norms, respectively. We use $\mathfrak d_1$, $\mathfrak d_2$, and $\mathfrak d_\infty$ to denote the distances associated with these norms. Without specific mention, the distance that we use is the $l^\infty$-distance. For a set $A \subset \mathbb{R}^d$ and $r>0$, we define the $l^\infty$-neighborhood
\begin{equation*}
    B_r(A):= \{z \in \mathbb{R}^d: \mathfrak d_\infty(z,A) <r \}\,.
\end{equation*}
As in~\eqref{eq:box-def}, for $z \in \mathbb{R}^d$, we write $B_r(z) = B_r(\{z \}) = z + (-r,r)^d$ for the open box centered at $z$ with side-length $2r$. We call a domain $A \subset \mathbb{R}^d$ a \textit{hypercubic shell} if $A = B_{r_1}(x) \backslash B_{r_2}(x)$ for some $x \in \mathbb{R}^d$ and $r_1>r_2>0$.

We extend the notation of $|\cdot|_\infty$, $|\cdot|_1$, $\mathfrak d_\infty$, and $\mathfrak d_1$ to the integer lattice $\mathbb{Z}^d$. For $x \in \mathbb{Z}^d$ and an integer $n \geq 0$, we define $B_n(x)$ as the box centered at $x$ with side-length $2n$. Namely, 
\begin{equation*}
   B_n(x) := \{z \in \mathbb{Z}^d: |z-x|_\infty \leq n\}\,.
\end{equation*}
We will clarify in the context whether we are considering $x$ as a point in $\mathbb{R}^d$ or as a vertex in $\mathbb{Z}^d$. For an integer $n \geq 0$, define the set
\begin{equation}
\label{eq:def-rescaled-lattice}
    \mathscrL_n := 2^{-n}\mathbb{Z}^d \cap B_2(0)\,.
\end{equation}
In general, we consider $\mathscrL_n$ as a subset of $\mathbb{R}^d$. However, when analyzing ($*$-)paths or ($*$-)clusters on the rescaled lattice $2^{-n} \mathbb{Z}^d$, as defined in Subsection~\ref{subsec:percolation}, we view $\mathscrL_n$ as a subset of $2^{-n} \mathbb{Z}^d$. In this paper, we also consider the graph distance on the rescaled lattice $\epsilon \mathbb{Z}^d$ which is defined as $1/\epsilon$ times the $l^\infty$-distance when considering $\epsilon \mathbb{Z}^d$ as a subset of $\mathbb{R}^d$.

\subsubsection*{Convention about constants}
Constants like $c, c', C, C'$ may change from place to place, while constants with subscripts like $c_1,C_1$ remain fixed throughout the article. All constants may implicitly rely on the dimension $d$, the kernel $\KK$, $\rr$, and $\xi$. The dependence on additional variables will be indicated at the first occurrence of each constant.

\subsection{Approximation of a log-correlated Gaussian field}
\label{subsec:est-log-field}

In this subsection, we establish some basic properties of the Gaussian random functions $h_n$ introduced in Subsection~\ref{subsec:intro-1}. Let us fix a convolution kernel $\KK: \mathbb{R}^d \rightarrow [0,\infty)$ and a constant $\rr>0$ which satisfy the conditions~\ref{K-condition1}, \ref{K-condition2}, \ref{K-condition3} in Subsection~\ref{subsec:intro-1}. 
Let $W$ be a white noise on $\mathbb R^d \times (0,\infty)$ and we define $h_n$ and $h$ as in~\eqref{eq:field-def}. We will also have occasion to consider the following additional functions.

\begin{definition}
    \label{def:log-field}
For integers $n \geq m \geq 0$ and $x\in \mathbb R^d$, we define
    \begin{equation*}
    h_{m,n}(x) := h_n(x) - h_m(x)  = \int_{\mathbb{R}^d} \int_{2^{-n}}^{2^{-m}} \KK\big(\frac{y-x}{t}\big) t^{-\frac{d+1}{2}} W(dy,dt) \,.
    \end{equation*} 
    Note that $h_{0,n}(x) = h_n(x)$. 
\end{definition}

The following properties of $h_{m,n}$ follow directly from its definition and the conditions on $\KK$. We omit the proof here.
\begin{lemma}
    \label{lem:basic-h}
    For integers $n > m \geq 0$, we have
    \begin{enumerate}[(1).]
        \item $h_{m,n}$ is smooth.
        \item The law of $h_{m,n}$ is invariant under translation and rotation of $\mathbb{R}^d$. 
        \item For any $U,V \subset \mathbb{R}^d$ with $\mathfrak d_\infty(U,V) \geq 2\rr \cdot 2^{-m}$, the fields $h_{m,n}|_U$ and $h_{m,n}|_V$, which are obtained by restricting $h_{m,n}$ to the domains $U$ and $V$, are independent.
        \item The fields satisfy the scaling property: $(h_{m,n}(x))_{x \in \mathbb{R}^d} \overset{d}{=} (h_{0,n-m}(x 2^m  ))_{x \in \mathbb{R}^d} $.
    \end{enumerate}
\end{lemma}

We collect some basic estimates about the field in the following lemma.
\begin{lemma}\label{lem:field-estimate}
    \begin{enumerate}[(1).]
        \item For any integers $n >m \geq 0$ and $x \in \mathbb{R}^d$, ${\rm Var}(h_{m,n}(x)) = (n-m) \log 2 $.
        \item There exists some constant $C>0$ such that for all $n \geq 1$ and $t>0$:
        \begin{equation}
        \label{eq:lem2.3-claim2}
        \mathbb{P}\Big[ \sup_{x \in 2^{-n}B_1(0)} |\nabla h_n(x)|_\infty \geq 2^n t \Big] \leq Ce^{-\frac{t^2}{C}}.
        \end{equation}
        \item There exists some constant $C>0$ such that for all $n \geq 1$, we have
        \begin{equation*}
        \mathbb{E}\Big[\sup_{x \in B_1(0)} h_n(x)\Big] \leq n \sqrt{2d} \log 2 + C \sqrt{n}\,.
        \end{equation*}
        \item (Borell-TIS inequality) For all $u>0$ and integer $n \geq 1$, we have
        \begin{equation}
        \label{eq:borell-tis}
        \mathbb{P}\Big[ \sup_{x \in B_1(0)} h_n(x) \geq \mathbb{E}\sup_{x \in B_1(0)} h_n(x) + u \Big]\leq \exp \big(-\frac{u^2}{2 \log 2 \cdot n}\big)\,.
        \end{equation}
        \item There exists some constant $C>0$ such that for all $u>0$ and integer $n \geq 1$, we have
        \begin{equation*}
        \mathbb{P}\Big[\sup_{x \in 2^{-n} B_1(0)} h_n(x) > u \Big] \leq C \exp \big( -\frac{(u - u^{2/3})^2}{2 \log 2 \cdot n} \big) + C \exp \big( -\frac{u^{4/3}}{C} \big) \,.
        \end{equation*}
    \end{enumerate}
\end{lemma}
\begin{proof}
    We first prove Claim (1). Using the property of white noise and the identity $\int_{\mathbb{R}^d} \KK(x)^2 dx =1$ from condition~\ref{K-condition3}, we obtain
    \begin{equation*}
    \begin{aligned}
        &\quad {\rm Var}(h_{m,n}(x)) \\
        &= \mathbb{E} \bigg[\int_{\mathbb{R}^d} \int_{2^{-n}}^{2^{-m}} \KK\big(\frac{y-x}{t} \big) t^{-\frac{d+1}{2}} W(dy,dt)\\
        &\qquad \qquad \qquad \qquad \qquad \qquad \qquad  \times \int_{\mathbb{R}^d} \int_{2^{-n}}^{2^{-m}} \KK\big(\frac{y'-x}{t'} \big) t'^{-\frac{d+1}{2}} W(dy',dt') \bigg]\\
        &=\int_{\mathbb{R}^d} \int_{2^{-n}}^{2^{-m}} \KK\big(\frac{y-x}{t} \big)^2 t^{-d-1} dydt=\int_{2^{-n}}^{2^{-m}} t^{-1} dt = (n-m) \log 2  \,.
    \end{aligned}
    \end{equation*}
    
    We now prove Claim (2). Using the smoothness of $\KK$ and Fernique's theorem (see e.g.~\cite{fernique-criterion}), we have a tail estimate for $h_1$. That is, there exists a constant $A>0$ such that for all $t>0$:
    \begin{equation}
        \label{eq:lem2.3-tail1}
    \mathbb{P}\Big[ \sup_{x \in B_1(0)} |\nabla h_1(x)|_\infty \geq t \Big] \leq Ae^{-\frac{t^2}{A}}.
    \end{equation}
    By Claim (4) in Lemma~\ref{lem:basic-h}, for any integer $k$, we have $(h_{k,k+1}(x))_{x \in 2^{-k} B_1(0)} \overset{d}{=} (h_{0,1}(x 2^k))_{x \in 2^{-k}B_1(0)}$. Therefore,
    \begin{equation*}
    \sup_{x\in 2^{-k} B_1(0)} |\nabla h_{k,k+1}(x)|_\infty \overset{d}{=} 2^k  \sup_{y \in B_1(0)}|\nabla h_{0,1}(y )|_\infty\,.
    \end{equation*}
    Combining this with \eqref{eq:lem2.3-tail1}, we obtain that for all integer $k \geq 0$ and $t>0$:
    \begin{equation}
    \label{eq:lem2.3-tail2}
    \mathbb{P} \Big[\sup_{x\in 2^{-k} B_1(0)} |\nabla h_{k,k+1}(x)|_\infty \geq 2^k t \Big] \leq Ae^{-\frac{t^2}{A}}.
    \end{equation}
    Let us first prove \eqref{eq:lem2.3-claim2} in the case where $t \geq 1$. Using the facts 
    \begin{equation*}
    %\begin{aligned}
    \sup_{x \in 2^{-n}B_1(0)}|\nabla h_n(x)|_\infty \leq \sum_{k=0}^{n-1} \sup_{x \in 2^{-k}B_1(0)}|\nabla h_{k,k+1}(x)|_\infty
    %&\qquad  \qquad \qquad \mbox{and} \quad \sum_{k=0}^{n-1} 2^k \cdot 2^{\frac{n-k}{2}} \leq 4\cdot 2^n\,,
    %\end{aligned}
    \end{equation*}
    and $\sum_{k=0}^{n-1} 2^k \cdot 2^{\frac{n-k}{2}} \leq 4\cdot 2^n$, we obtain that for all integer $n\geq 1$:
    \begin{equation*}
    \begin{aligned}
        &\mathbb{P}\Big[ \sup_{x \in 2^{-n}B_1(0)} |\nabla h_n(x)|_\infty \geq 2^n t \Big] \\
        &\qquad \leq \mathbb{P}\Big[ \sum_{k=0}^{n-1} \sup_{x \in 2^{-k}B_1(0)}|\nabla h_{k,k+1}(x)|_\infty \geq \sum_{k=0}^{n-1} 2^k \cdot 2^{\frac{n-k}{2}} (t/4)\Big] \\
        &\qquad \leq \sum_{k=0}^{n-1} \mathbb{P}\Big[ \sup_{x \in 2^{-k}B_1(0)}|\nabla h_{k,k+1}(x)|_\infty \geq 2^k \cdot 2^{\frac{n-k}{2}} (t/4)\Big].
    \end{aligned}
    \end{equation*}
    Using \eqref{eq:lem2.3-tail2} and the fact that $t \geq 1$, we can choose a constant $C>0$ depending only on $A$ such that
    \begin{equation*}
    \begin{aligned}
         \mathbb{P}\Big[ \sup_{x \in 2^{-n}B_1(0)} |\nabla h_n(x)|_\infty \geq 2^n t \Big] \leq \sum_{k=0}^{n-1} A e^{-\frac{2^{n-k}(t/4)^2}{A}} \leq C e^{-\frac{t^2}{C}}.
    \end{aligned}
    \end{equation*}
    This result can be extended to all $t>0$ by enlarging the value of $C$, thereby proving Claim (2).

    Next, we prove Claim (3). Using the fact
    \begin{equation*}
    \sup_{x \in B_1(0)}h_n(x)\leq \sup_{x \in B_1(0) \cap 2^{-n}\mathbb{Z}^d} h_n(x) + 2^{-n} d \sup_{\substack{y \in B_{2^{-n}}(x)\\x \in B_1(0) \cap 2^{-n}\mathbb{Z}^d}} |\nabla h_n(y)|_\infty\,,
    \end{equation*}
    we obtain that for all integer $n\geq 1$ and $s >0 $
    \begin{equation*}
    \begin{aligned}
    &\mathbb{P} \Big[\sup_{x \in B_1(0)}h_n(x) \geq n \sqrt{2d} \log 2  + s\sqrt{n} \Big] \\
    &\qquad\leq \mathbb{P} \Big[\sup_{x \in B_1(0) \cap 2^{-n}\mathbb{Z}^d} h_n(x) \geq n \sqrt{2d} \log 2+ s\sqrt{n}/2 \Big]  \\
    &\qquad \qquad + \mathbb{P} \Big[2^{-n} d\sup_{\substack{y \in B_{2^{-n}}(x)\\x \in B_1(0) \cap 2^{-n}\mathbb{Z}^d}} |\nabla h_n(y)|_\infty \geq s\sqrt{n}/2 \Big].
    \end{aligned}
    \end{equation*}
    Using Claims (1) and (2), translation invariance of $h_n$, and the fact that $|B_1(0) \cap 2^{-n}\mathbb{Z}^d| \leq C 2^{dn}$, we have
    \begin{equation}
    \label{eq:lem2.3-claim4}
    \begin{aligned}
    &\mathbb{P} \Big[\sup_{x \in B_1(0)}h_n(x) \geq n \sqrt{2d} \log 2  + s\sqrt{n} \Big] \\
    &\qquad \leq C 2^{dn} \exp\Big(-\frac{(n \sqrt{2d} \log 2  + s\sqrt{n}/2)^2}{2\log 2 \cdot n} \Big) + C 2^{dn} \exp \Big(- \frac{s^2n}{C} \Big)\,,
    \end{aligned}
    \end{equation}
    where we enlarged the value of $C$. When $s$ is large enough (independent of $n$), the right-hand side is smaller than $Ce^{-s^2/C}$. By integrating \eqref{eq:lem2.3-claim4} with respect to $s$, we obtain Claim (3).

    Claim (4) follows from the Borell-TIS inequality (see~\cite{borell-tis1, TIS76}, and also~\cite[Theorem 2.1.1]{adler-taylor-fields}) and the fact that ${\rm Var}(h_n(x)) = n \log 2 $ as stated in Claim (1). 
    
    Finally, we prove Claim (5). Using the fact
    \begin{equation*}
        \sup_{x \in 2^{-n}B_1(0)} h_n(x) \leq h_n(0) + d2^{-n} \sup_{x \in 2^{-n}B_1(0)} |\nabla h_n(x)|_\infty\,,
    \end{equation*} 
    we obtain
    \begin{align*}
        &\mathbb{P}\Big[\sup_{x \in 2^{-n} B_1(0)} h_n(x) > u \Big] \\
        &\qquad \leq  \mathbb{P}\big[h_n(0) > u - u^{2/3}\big] + \mathbb{P}\Big[ \sup_{x \in 2^{-n}B_1(0)} 2^{-n}|\nabla h_n(x)|_\infty > u^{2/3}/d \Big]\,.
    \end{align*}
    Applying Claims (1) and (2) gives the desired result.
\end{proof}

\subsection{Definition of the exponential metric}
\label{subsec:def-lfpp}

In this subsection, we introduce the exponential metric associated with $h_{m,n}$, which is the main focus of this paper. We also establish some of its basic properties.

\begin{definition}
\label{def:LFPP}
    Fix $\xi>0$. For integers $n \geq m \geq 0$, we define the exponential metric associated with the field $h_{m,n}$ from Definition~\ref{def:log-field} as follows:
    \begin{equation*}
    D_{m,n}(z,w):=\inf_{P: z\to w} \int_0^1 e^{\xi h_{m,n}(P(t))}|P'(t)|dt\,,
    \end{equation*}
    where the infimum is taken over all piecewise continuously differentiable paths $P:[0,1] \rightarrow \mathbb{R}^d$ joining $z,w$. For an open set $U \subset \mathbb{R}^d$, we define the internal metric $D_{m,n}(\cdot,\cdot;U)$ as described in \eqref{def:int-metric}. When $m =0$, the metric $D_{0,n}$ is the same as the metric $D_n$ introduced in~\eqref{eq:metric-def}. When $m=n$, $D_{m,n}$ is equivalent to the Euclidean metric.
\end{definition}

The following lemma is a direct consequence of Claims (2) and (4) in Lemma~\ref{lem:basic-h}. We omit the proof here.
\begin{lemma}
    \label{lem:basic-LFPP}
    For integers $n \geq m \geq 0$ and any open set $U \subset \mathbb{R}^d$ (including $U = \mathbb{R}^d$), we have
    \begin{enumerate}
        \item The law of $D_{m,n}(\cdot,\cdot; U)$ is invariant under translation and rotation of $\mathbb{R}^d$. 
        \item The law of $D_{m,n}(\cdot,\cdot;U)$ satisfies the scaling property: 
        \begin{equation*}
            (D_{m,n}(x,y; U))_{x,y \in U} \overset{d}{=} 2^{-m}(D_{n-m}(2^mx,2^my, 2^mU))_{x,y \in U}\,.
        \end{equation*}
    \end{enumerate}
\end{lemma}

As a corollary of Claim (3) in Lemma~\ref{lem:basic-h}, we have that the internal metrics of $D_{m,n}$ are independent within two domains located far from each other.

\begin{lemma}
    \label{lem:dist-independent}
    For integers $n > m \geq 0$ and any open sets $U,V \subset \mathbb{R}^d$ with $\mathfrak d_\infty(U,V) \geq 2\rr \cdot 2^{-m}$, the internal metrics $D_{m,n}(\cdot,\cdot; U)$ and $D_{m,n}(\cdot,\cdot; V)$ are independent.
\end{lemma}
\begin{proof}
    The internal metric $D_{m,n}(\cdot,\cdot; U)$ is determined by $h_{m,n}|_U$, and the internal metric $D_{m,n}(\cdot,\cdot; V)$ is determined by $h_{m,n}|_V$. By Claim (3) in Lemma~\ref{lem:basic-h}, we obtain the result.
\end{proof}

Next, we prove a concentration bound for the exponential metric. The proof is similar to that of \cite[Lemma 23]{dddf-lfpp}.
\begin{lemma}
    \label{lem:gauss-concentration}
    For any open subset $U \subset \mathbb{R}^d$ (including $U = \mathbb{R}^d$), any two disjoint compact subsets $K_1,K_2 \subset U$ that are path-connected in $U$, integers $n > m \geq 0$, and $t>0$, the following concentration bound holds:
    \begin{equation}
    \label{eq:lem2.7}
    \mathbb{P} \big[|\log D_{m,n}(K_1,K_2;U) - \mathbb{E}\log D_{m,n}(K_1,K_2;U)| >t \big] \leq 2e^{-\frac{t^2}{2\xi^2 \log 2 \cdot (n-m)}}.
    \end{equation}
\end{lemma}
\begin{proof}
    We first show that $|\mathbb{E}\log D_{m,n}(K_1,K_2;U)|<\infty$. Let us begin with the upper bound. By the assumption, there exists a large constant $N$ such that $K_1$ and $K_2$ are connected by a path of Euclidean length at most $N$ in $U \cap B_N(0)$. Therefore, 
    \begin{equation*}
    \mathbb{E}\log D_{m,n}(K_1,K_2;U) \leq \mathbb{E} \log \big( N e^{\xi \sup_{x \in U \cap B_N(0)}h_{m,n}(x)} \big) <\infty\,.
    \end{equation*}
    The last inequality follows from the Gaussian tail of $\sup_{x \in U \cap B_N(0)}h_{m,n}(x)$, as indicated by Claim (2) in Lemma~\ref{lem:field-estimate}. Furthermore, there exists a large constant $M$ such that any path connecting $K_1$ and $K_2$ must have a Euclidean length of at least $\frac{1}{M}$ within $U \cap B_M(0)$. Therefore,
    \begin{equation*}
    \mathbb{E}\log D_{m,n}(K_1,K_2;U) \geq \mathbb{E} \log \big( \frac{1}{M} e^{\xi \inf_{x \in U \cap B_M(0)}h_{m,n}(x)} \big) >-\infty\,.
    \end{equation*}
    Combining the above two inequalities yields $|\mathbb{E}\log D_{m,n}(K_1,K_2;U)|<\infty$.

    We now prove \eqref{eq:lem2.7} first for a bounded open set $U$. For integer $k \geq 1$, let $D_{m,n}^{(k)}$ be the exponential metric associated with $h_{m,n}^{(k)}$, where $h_{m,n}^{(k)}$ is piecewise constant and takes the value $h_{m,n}(x)$ on each dyadic box $B_{2^{-k}}(x)$ for $x \in \mathbb{R}^d \cap 2^{-k + 1} \mathbb{Z}^d$. Then, $\sup_{x \in U}|h_{m,n}(x) - h_{m,n}^{(k)}(x)| \leq d2^{-k}\sup_{x \in B_1(U)} | \nabla h_{m,n}(x)|_\infty$. This, combined with Definition~\ref{def:LFPP}, implies that
    \begin{equation*}
    e^{-\xi d2^{-k} \sup_{x \in B_1(U)} | \nabla h_{m,n}(x)|_\infty} \leq \frac{D_{m,n}(K_1,K_2;U)}{D_{m,n}^{(k)}(K_1,K_2;U)} \leq e^{\xi d2^{-k} \sup_{x \in B_1(U)} | \nabla h_{m,n}(x)|_\infty}.
    \end{equation*}
    Together with the fact that $\mathbb{E}[\sup_{x \in B_1(U)} | \nabla h_{m,n}(x)|_\infty] <\infty$ (because it has a Gaussian tail, as indicated by Claim (2) in Lemma~\ref{lem:field-estimate}), we obtain
    \begin{equation}
    \label{eq:lem2.7-1}
    \begin{aligned}
    &\lim_{k \rightarrow \infty}D_{m,n}^{(k)}(K_1,K_2;U) = D_{m,n}(K_1,K_2;U) \quad \mbox{and} \\
    &\lim_{k \rightarrow \infty}\mathbb{E} \log D_{m,n}^{(k)}(K_1,K_2;U) = \mathbb{E} \log  D_{m,n}(K_1,K_2;U)\,.
    \end{aligned}
    \end{equation}
    By definition, $\log D_{m,n}^{(k)}(K_1,K_2;U)$ is $\xi$-Lipschitz as a function of \begin{equation*}
        (Y_1,\ldots,Y_p) := (h_{m,n}^{(k)}(x))_{x \in B_1(U) \cap 2^{-k + 1} \mathbb{Z}^d}
    \end{equation*} in terms of the $l^\infty$-norm. In addition, there exists a $p \times p$ matrix $A$ such that $(Y_1,\ldots, Y_p)^\intercal \overset{d}{=} A(X_1,\ldots,X_p)^\intercal$, where $X_1,\ldots, X_p$ are i.i.d.\ standard Gaussian random variables. By Claim (1) in Lemma~\ref{lem:field-estimate}, the $l^2$-norm of each row of $A$ equals to $\sqrt{\operatorname{Var} h_{m,n}^{(k)}(x)} = \sqrt{(n-m) \log 2 }$. Therefore, $\log D_{m,n}^{(k)}(K_1,K_2;U)$, as a function of $(X_1,\ldots,X_p)$, is $\xi\sqrt{ (n-m) \log 2 } $-Lipschitz in terms of the $l^2$-norm. By the Gaussian concentration inequality (see~\cite{borell-tis1, TIS76}, and also~\cite[Lemma 2.1.6]{adler-taylor-fields}), we have for all $t>0$:
    \begin{equation}
    \label{eq:lem2.7-2}
    \mathbb{P} \big[|\log D_{m,n}^{(k)}(K_1,K_2;U) - \mathbb{E}\log D_{m,n}^{(k)}(K_1,K_2;U)| >t \big] \leq 2e^{-\frac{t^2}{2\xi^2 \log 2 \cdot (n-m)}}.
    \end{equation}
    By sending $k$ to infinity and combining with \eqref{eq:lem2.7-1}, we obtain the desired lemma in the case where $U$ is bounded. 
    
    We can extend the result to arbitrary $U$ by considering the truncation $U \cap B_N(0)$ for integers $N \geq 1$. Since $D_{m,n}(K_1,K_2; U \cap B_N(0))$ decreases to $D_{m,n}(K_1,K_2; U)$ as $N$ goes to infinity, we have
    \begin{equation*}
    \lim_{N \rightarrow \infty}D_{m,n}(K_1,K_2;U \cap B_N(0)) = D_{m,n}(K_1,K_2;U) \quad \mbox{and}
    \end{equation*}
    \begin{equation*}
    \lim_{N \rightarrow \infty}
    \mathbb{E} \log D_{m,n}(K_1,K_2;U \cap B_N(0)) = \mathbb{E} \log D_{m,n}(K_1,K_2;U)\,.
    \end{equation*}
    Note that the inequality \eqref{eq:lem2.7-2} holds for $D_{m,n}(K_1,K_2;U \cap B_N(0))$ as long as $K_1$ and $K_2$ are path-connected in $U \cap B_N(0)$ which holds for all sufficiently large $N$. Therefore, applying \eqref{eq:lem2.7-2} with $U \cap B_N(0)$ in place of $U$ and then sending $N$ to infinity gives the desired lemma.
\end{proof}

\subsection{Percolation with finite range of dependence}
\label{subsec:percolation}

In this subsection, we consider the integer lattice $\mathbb{Z}^d$ with $d \geq 2$ and establish some results about percolation with finite range of dependence. The definitions and results from this subsection can be naturally adapted to the rescaled lattice $2^{-n} \mathbb{Z}^d$ for any integer $n \geq 1$. These results will play an important role in Sections~\ref{sec:exponent} and \ref{sec:bound-distance}.

Let $M \geq 1$ be an integer, and consider a probability measure $\mu$ on the configuration $\omega \in \{0,1\}^{\mathbb{Z}^d}$. We say that $\mu$ is \textit{$M$-dependent} if for any two subsets $U,V\subset \mathbb{Z}^d$ with $\mathfrak d_\infty(U,V) >M$, the restrictions $\omega|_U$ and $\omega|_V$ are independent. A vertex $x$ is called \textit{open} if $\omega(x) = 1$, and \textit{closed} if $\omega(x) = 0$. A \textit{path} (resp.\ \textit{$*$-path}) is a sequence of vertices $x_1,\ldots, x_n$ such that $|x_i - x_{i+1}|_1 = 1$ (resp.\ $|x_i-x_{i+1}|_\infty = 1$) for any $1\leq i \leq {n-1}$. A path is called open if all the vertices contained in it are open, and closed if all the vertices contained in it are closed. Similarly, we can define an open $*$-path and a closed $*$-path. For a subset $U \subset \mathbb{Z}^d$, we use $\partial U := \{x \in U: \exists y \in \mathbb{Z}^d \backslash U \mbox{ such that }   xy \in E \}$ to denote its interior boundary, where $E$ is the edge set of $\mathbb{Z}^d$.

We begin with a lemma about the exponential decay of the probability of long closed $*$-paths when $M$ is fixed and all the vertices have a probability close to one of being open. The proof follows from an elementary path-counting argument. 

\begin{lemma}
\label{lem:exponential-decay}
    Fix an integer $M \geq 1$. There exist two constants $c_1\in (0,1)$ and $C>0$ depending only on $M$ such that for any $M$-dependent measure $\mu$ satisfying $\inf_{x \in \mathbb{Z}^d} \mu[w(x) = 1] >c_1$, we have
    \begin{equation*}
    \mu \big[\mbox{There exists a closed }{\rm *}\mbox{-path connecting } 0 \mbox{ and } \partial B_N(0) \big] \leq Ce^{-N/C},\forall N \geq 1.
    \end{equation*}
\end{lemma}
\begin{proof}
    Let $p \in (0,1)$ be a constant to be chosen. Assume that $\mu$ is an $M$-dependent measure with $\inf_{x \in \mathbb{Z}^d} \mu[\omega(x) = 1]>p$. Let $ x_1,x_2,\ldots,x_n$ be any $*$-path connecting $0$ and $\partial B_N(0)$. Then, we have 
    \begin{equation*}
     x_1 = 0 \,, \quad x_n \in \partial B_N(0)\,, \quad \mbox{and} \quad |x_i - x_{i+1}|_\infty = 1 \quad \forall 1 \leq i \leq n-1\,.
    \end{equation*}
    We consider a subset of this path defined inductively as follows: first, take $i_1=1$, and for $j \geq 2$ define
    \begin{equation}
    \label{eq:lem2.8-saw}
    i_j := \max \{ i_{j-1} \leq k \leq n: |x_k - x_{i_{j-1}}|_\infty \leq M \} +1\,. 
    \end{equation}
    We stop the induction when $i_j = n+1$. Let us consider the obtained sequence $(y_1, \ldots, y_m) := (x_{i_1}, \ldots, x_{i_m})$. Then, $y_1 = x_{i_1} = 0$. Moreover, we have
    \begin{equation}
    \label{eq:lem2.8-saw2}
    \begin{aligned}
    &\min_{1 \leq i < j \leq m }|y_i - y_j|_\infty > M\,, \quad \max_{1 \leq j \leq m-1} |y_j - y_{j+1}|_\infty \leq M+1\,,  \\
    &\quad \mbox{and} \quad \mathfrak d_\infty(y_m,\partial B_N(0)) \leq M\,.
    \end{aligned}
    \end{equation}
    The first property follows directly from \eqref{eq:lem2.8-saw}. The second property holds because, by \eqref{eq:lem2.8-saw}, $|y_j - y_{j+1}|_\infty = |x_{i_j} - x_{i_{j+1}}|_\infty \leq |x_{i_j} - x_{i_{j+1}-1}|_\infty + 1 \leq M+1$. The last property holds because when the iteration stops, we have $|y_m - x_n|_\infty \leq M$.
    
    We now upper-bound the probability that there exists a closed sequence satisfying \eqref{eq:lem2.8-saw2}. First, we have $m \geq \frac{N}{M+1}$, which follows from the following inequality:
    \begin{equation*}
    \begin{aligned}
    N-1 = \mathfrak d_\infty(0,\partial B_N(0)) &\leq \sum_{i=1}^{m-1} |y_i - y_{i+1}|_\infty + \mathfrak d_\infty(y_m,\partial B_N(0)) \\
    &\leq (M+1)(m-1)+M.
    \end{aligned}
    \end{equation*}
    For fixed $m$, we know that the number of sequences satisfying \eqref{eq:lem2.8-saw2} is at most $(2M+3)^{d(m-1)}$ since $y_1 = 0$ and $|y_i - y_{i+1}|_\infty \leq M+1$. Furthermore, for a fixed choice of the sequence, the probability that all the vertices contained in it are closed is at most $(1-p)^m$ by the $M$-dependent property of $\mu$ and the fact that $\sup_{x \in \mathbb{Z}^d} \mu[\omega(x) = 0] < 1-p$. Therefore, when $p$ is close enough to one, we have
    \begin{equation*}
    \begin{aligned}
        &\quad \mu \big[\mbox{There exists a closed }{\rm *}\mbox{-path connecting } 0 \mbox{ and } \partial B_N(0) \big] \\
        &< \sum_{m \geq N/(M+1)} (2M+3)^{d(m-1)} \times (1-p)^m \leq Ce^{-N/C} \quad \forall N \geq 1\,. \qedhere
    \end{aligned}
    \end{equation*}
\end{proof}

Next, we give two corollaries of Lemma~\ref{lem:exponential-decay}. An open (resp.\ closed) \textit{cluster} is a connected component of open (resp.\ closed) vertices. Similarly, we define the open (resp.\ closed) \textit{$*$-cluster} which is a connected component of open (resp.\ closed) vertices where two vertices $x,y$ are considered to be neighboring each other if $|x-y|_\infty =1$. We define the \textit{diameter} of a cluster or $*$-cluster with respect to the $l^\infty$-distance on $\mathbb{Z}^d$.

\begin{lemma}
\label{lem:percolation-cluster}
    For an integer $M \geq 1$ and $c_1 = c_1(M)$ as defined in Lemma~\ref{lem:exponential-decay}, let $\mu$ be an $M$-dependent measure satisfying $\inf_{x \in \mathbb{Z}^d} \mu[w(x) = 1] >c_1$. Then, for all integers $K,N \geq 1$,
    \begin{equation*}
    \mu \big[ \mbox{All closed }{\rm *}\mbox{-clusters in }B_N(0) \mbox{ have diameter at most K}  \big] \geq 1-  C N^d e^{-K/C} ,
    \end{equation*} 
    where the constant $C$ depends on $M$, but is independent of $\mu$.
\end{lemma}
\begin{proof}
    If there exists a closed $*$-cluster in $B_N(0)$ with diameter at least $K+1$, then we can find a vertex $x \in B_N(0)$ such that $x$ is connected to $\partial B_{K+1}(x)$ with a closed $*$-path. Summing over all the possible choices of $x$ and applying Lemma~\ref{lem:exponential-decay} with $K+1$ in place of $N$, we obtain the desired result.
\end{proof}

\begin{lemma}
\label{lem:connect-inf}
    For any integer $M\geq 1$ and $\epsilon>0$, there exists a constant $c_2 = c_2(M,\epsilon) \in (0,1)$ such that for any $M$-dependent measure $\mu$ satisfying $\inf_{x \in \mathbb{Z}^d}\mu[w(x) = 1] >c_2$, we have
    \begin{equation*}
    \mu \big[\mbox{There exists an infinite open cluster containing } 0 \big] \geq 1-\epsilon\,.
    \end{equation*}
\end{lemma}
\begin{proof}
    Recall from Lemma~\ref{lem:exponential-decay} the constant $c_1$, which depends on $M$. Let $p \in (c_1,1)$ be a constant to be chosen. Let $\mu$ be an $M$-dependent measure satisfying $\inf_{x \in \mathbb{Z}^d}\mu[\omega(x) = 1]>p$. Let $N$ be a large integer to be chosen. Define the events
    \begin{equation*}
    \begin{aligned}
    \mathcal{K}_1 &:= \{ \mbox{All vertices in }B_N(0)\mbox{ are open}\}\,, \\
    \mathcal{K}_2 &:= \{ \mbox{There exists a closed }{\rm *}\mbox{-cluster enclosing }B_N(0) \}\,.
    \end{aligned}
    \end{equation*}
    By duality, we know that on the event $\mathcal{K}_1 \backslash \mathcal{K}_2$, all vertices in $B_N(0)$ are open and are connected to infinity by an open path. Hence, we only need to show that \begin{equation}
    \label{eq:lem2.10-1}
    \mu[\mathcal{K}_1 \backslash \mathcal{K}_2] \geq 1-\epsilon.
    \end{equation}

    First, we lower-bound $\mu[\mathcal{K}_1]$. Using the assumption that $\inf_{x \in \mathbb{Z}^d}\mu[\omega(x) = 1]>p$, we obtain
    \begin{equation}
    \label{eq:lem2.10-2}
    \mu[\mathcal{K}_1] \geq 1-\sum_{x \in B_N(0)} \mu[\omega(x) = 0] \geq 1- (2N+1)^d (1-p)\,.
    \end{equation}

    Next, we upper-bound $\mu[\mathcal{K}_2]$. If the event $\mathcal{K}_2$ happens, then the closed $*$-cluster must intersect the set $\{x\in \mathbb{Z}^d: x_1 \geq N+1, x_2 = x_3 =\ldots =x_d =0\}$. Let $x := (m,0,\ldots,0)$ be an intersection point where $m \geq N+1$. Then there exists a closed $*$-path from $x$ to $\partial B_{m}(x)$. Thus, applying Lemma~\ref{lem:exponential-decay} with $m$ in place of $N$, we obtain
    \begin{equation}
    \label{eq:lem2.10-3}
    \begin{aligned}
    \mu[\mathcal{K}_2] 
    &\leq \sum_{m \geq N+1} \mu \big[\mbox{There exists a closed }{\rm *}\mbox{-path from } x \mbox{ to } \partial B_m(x) \big] \\ 	&\leq \sum_{m \geq N+1} Ce^{-m/C}\leq Ce^{-N/C}.
    \end{aligned}
    \end{equation}
    Combining \eqref{eq:lem2.10-2} and \eqref{eq:lem2.10-3}, and first taking $N$ to be large and then taking $p$ close to $1$, yields \eqref{eq:lem2.10-1}. In particular, the choice of $p$ depends only on $M$ and $\epsilon$. This concludes the lemma.   
\end{proof}

\section{Existence of an exponent}
\label{sec:exponent}

In this section, we first prove the existence of an exponent $Q = Q(\xi) \in \mathbb{R}$ such that \eqref{eq:exponent-relation-lambda} holds (Proposition~\ref{prop:exponent}). This exponent governs the internal $D_n$-distance between two points in a box as $n$ grows. Furthermore, Lemma~\ref{lem:distance-any-point} extends this relation to any pair of points, and Lemma~\ref{lem:Q-lower} establishes some basic properties of $Q(\xi)$. Combining these results gives Proposition~\ref{prop:exponent-intro}.

We first introduce some notation. For each $1 \leq i \leq d$, let
\begin{equation}
\label{eq:def-ei}
e_i := \mbox{the }i \mbox{-th standard basis vector in } \mathbb{R}^d.
\end{equation}
That is, $e_i$ is a $\{0,1\}$-valued vector in $\mathbb{R}^d$ where only the $i$-th coordinate is equal to $1$. For integer $n \geq 1$ and $p \in (0,1)$, let $a_n^{(p)}$ denote the $p$-th quantile of the internal distance $D_n(0,e_1; B_2(0))$. Namely,
\begin{equation}
    \label{def:quantile}
     a_n^{(p)} := \inf \{l >0 : \mathbb{P} [D_n(0,e_1; B_2(0)) \leq l] >p \}\,.
\end{equation}
Since $D_n(0,e_1; B_2(0))$ is a continuous random variable, we have $\mathbb{P} [D_n(0,e_1; B_2(0)) \leq a_n^{(p)}] = p$. When $p = 1/2$, the number $\lambda_n$ defined in~\eqref{eq:intro-lambda} satisfies $\lambda_n = a_n^{(1/2)} = {\rm Med}(D_n(0,e_1; B_2(0)))$.

\begin{proposition}
\label{prop:exponent}
    There exists an exponent $Q = Q(\xi) \in \mathbb{R}$ such that
    \begin{equation}
    \label{eq:exponent}
    \lambda_n = 2^{-(1-\xi Q)n+o(n)} \quad \mbox{as }n\rightarrow \infty\,.
    \end{equation}
\end{proposition}

The proof of Proposition~\ref{prop:exponent} is via a subadditivity argument. We will use Lemmas~\ref{lem:compare-median-p} and \ref{lem:prop3.1-1} below. The former directly follows from the concentration bound in Lemma~\ref{lem:gauss-concentration}. The latter employs a percolation argument from Subsection~\ref{subsec:percolation} and follows an approach similar to that of \cite[Lemma 2.9]{dg-supercritical-lfpp}.

\begin{lemma}
    \label{lem:compare-median-p}
    For fixed $p_1,p_2 \in (0,1)$, there exists a constant $C>0$ depending only on $p_1$ and $p_2$ such that for all integer $n \geq 1$, we have
    \begin{equation*}
        e^{-C\sqrt{n}} a_n^{(p_1)} \leq a_n^{(p_2)} \leq e^{C \sqrt{n}} a_n^{(p_1)}.
    \end{equation*}
\end{lemma}
\begin{proof}
    Applying the concentration bound from Lemma~\ref{lem:gauss-concentration} with $K_1 = \{0\}, K_2 = \{e_1\}$, and $U = B_2(0)$ yields 
    \begin{equation*}
        \mathbb{P}\big[ |\log D_n(0,e_1;B_2(0)) - \mathbb{E} \log D_n(0,e_1;B_2(0))| \geq t \big] \leq Ce^{-\frac{t^2}{Cn}} \quad \forall t>0\,. 
    \end{equation*}
    Hence, for any fixed $p \in (0,1)$, the following inequality holds:
    \begin{equation*}
        |\log a_n^{(p)} - \mathbb{E} \log D_n(0,e_1;B_2(0)) | \leq C\sqrt{n}\,,
    \end{equation*}
    where the constant $C$ depends on $p$, but is independent of $n$. This implies the lemma.
\end{proof}

We now present a key lemma. It will imply Proposition~\ref{prop:exponent} when combined with Lemma~\ref{lem:compare-median-p}.
\begin{lemma}
\label{lem:prop3.1-1}
    There exist $c_3 \in (0,1)$ and a constant $C>0$ such that for all integers $n > m \geq 1$:
    \begin{equation}
    \label{eq:lem3.5-1}
    \lambda_n \leq  e^{Cn^{2/3}} a_m^{(c_3)} a_{n-m}^{(c_3)}\,.
    \end{equation}
\end{lemma}

We will first use a subadditivity argument to prove Proposition~\ref{prop:exponent} based on this lemma, and then provide the proof of Lemma~\ref{lem:prop3.1-1}.

\begin{proof}[Proof of Proposition~\ref{prop:exponent}]
Combining Lemmas~\ref{lem:prop3.1-1} and \ref{lem:compare-median-p}, we obtain that for all $n >m \geq 1$:
\begin{equation*}
   \lambda_n \leq  e^{Cn^{2/3}} a_m^{(c_3)} a_{n-m}^{(c_3)} \leq  e^{Cn^{2/3}} \lambda_m \lambda_{n-m}\,.
\end{equation*}
Combining this inequality with Lemma 6.4.10 in \cite{dembo-ld}, applied to $\log \lambda_n$, implies the existence of $a \in \mathbb{R}$ such that
\begin{equation*}
    \lambda_n = e^{an + o(n)} \quad \mbox{as }n \rightarrow \infty\,.
\end{equation*}
Taking $Q \in \mathbb{R}$ such that $e^a = 2^{-(1-\xi Q)}$ yields the desired result.
\end{proof}

Next, we proceed to the proof of Lemma~\ref{lem:prop3.1-1}. First, we present two auxiliary results. In Lemma~\ref{lem:h(n-m)(n)}, we provide estimates for the field $h_{n-m,n}$. Subsequently, we use these estimates in Lemma~\ref{lem:compare(n)(n-m)} to compare $a_n^{(p)}$ and $a_{n-m}^{(p)}$.

\begin{lemma}
\label{lem:h(n-m)(n)}
    There exist constants $C_1>0$ and $C>0$ such that for all integers $n > m \geq 1$:
    \begin{equation*}
        \mathbb{P} \Big[ \sup_{x \in B_2(0)} h_{n-m,n}(x) \geq C_1 \sqrt{mn} \Big] \leq Ce^{-n/C}.
    \end{equation*}
\end{lemma}
\begin{proof}
    By Claim (4) in Lemma~\ref{lem:basic-h}, we have that $\sup_{y \in B_{2^{m-n}}(0)} h_{n-m,n}(y) \overset{d}{=} \sup_{x \in B_1(0)} h_{0,m}(x)$. Therefore, 
    \begin{equation*}
        \mathbb{P} \big[\sup_{y \in B_{2^{m-n}}(0)} h_{n-m,n}(y) \geq s  \big] = \mathbb{P}\big[ \sup_{x \in B_1(0)} h_m(x) \geq s \big] \quad \forall s > 0\,. 
    \end{equation*}
    Using Claims (3) and (4) from Lemma~\ref{lem:field-estimate}, we get that for all $s \geq (1+ \sqrt{2d} \log 2 )m$:
    \begin{equation}
    \label{eq:lem3.3-2}
        \mathbb{P} \big[\sup_{y \in B_{2^{m-n}}(0)} h_{n-m,n}(y) \geq s  \big] \leq C\exp \big(-\frac{s^2}{Cm})\,.
    \end{equation}
    Hence, for all $t> 1+ \sqrt{2d} \log 2$:
    \begin{equation*}
        \begin{aligned}
             &\quad \mathbb{P} \Big[ \sup_{x \in B_2(0)} h_{n-m,n}(x) \geq t \sqrt{mn} \Big] \\
             &= \mathbb{P} \Big[ \sup_{x \in B_2(0) \cap 2^{m-n}\mathbb{Z}^d} \sup_{y \in B_{2^{m-n}}(x)} h_{n-m,n}(y) \geq t\sqrt{mn} \Big] \\
             &\leq \sum_{x \in B_2(0) \cap 2^{m-n}\mathbb{Z}^d} \mathbb{P} \Big[ \sup_{y \in B_{2^{m-n}}(0)} h_{n-m,n}(y) \geq t\sqrt{mn} \Big] \leq C2^{nd} \times \exp(-\frac{t^2mn}{Cm})\,.
        \end{aligned}
    \end{equation*}
    In the last inequality, we used \eqref{eq:lem3.3-2}, as well as the facts that $n>m$ and $|B_2(0) \cap 2^{m-n}\mathbb{Z}^d| \leq C2^{nd}$. By choosing a sufficiently large $t$, we obtain the desired result.
\end{proof}

We now provide a comparison between $a_n^{(p)}$ and $a_{n-m}^{(p)}$ based on the above lemma.

\begin{lemma}
    \label{lem:compare(n)(n-m)}
    For a fixed $p \in (0,1)$, there exists a constant $C= C(p)>0$ such that for all integers $n > m \geq 1$:
    \begin{equation*}
         e^{-C\sqrt{mn}} a_{n-m}^{(p)} \leq a_n^{(p)} \leq e^{C\sqrt{mn}} a_{n-m}^{(p)}\,.
    \end{equation*}
\end{lemma}
\begin{proof}
    Based on the definition of $a_n^{(p)}$ in~\eqref{def:quantile}, we have
    \begin{equation}
    \label{eq:lem3.4-0}
        \mathbb{P}\big[ D_{n-m}(0,e_1; B_2(0)) \geq a_{n-m}^{(p/2)}\big] = 1 - p/2\,.
    \end{equation}
    Using Lemma~\ref{lem:h(n-m)(n)} and the symmetry of $h_{m,n}$, there exists a constant $A>0$ such that for all $n>m \geq 1$:
    \begin{equation}
    \label{eq:lem3.4-1}
        \mathbb{P} \big[ \inf_{x \in B_2(0)} h_{n-m,n}(x) \geq -A\sqrt{mn} \big] > 1 - p/2\,.
    \end{equation}
    Since $h_n = h_{n-m} + h_{n-m,n}$, we have $$D_n(0,e_1;B_2(0)) \geq D_{n-m}(0,e_1;B_2(0)) e^{\xi \inf_{x \in B_2(0)} h_{n-m,n}(x)}.$$ Therefore, for all $s>0$:
    \begin{equation*}
    \begin{aligned}
        &\quad \mathbb{P} \big[ D_n(0,e_1;B_2(0)) \geq  e^{-s\sqrt{mn}} a_{n-m}^{(p/2)} \big] \\
        &\geq \mathbb{P} \Big[ \big{\{} D_{n-m}(0,e_1;B_2(0)) \geq a_{n-m}^{(p/2)} \big{\}} \cap \big{\{}\inf_{x \in B_2(0)} h_{n-m,n}(x) \geq -s\sqrt{mn} /\xi \big{\}} \Big]\,.
    \end{aligned}
    \end{equation*}
    Combining this with \eqref{eq:lem3.4-0} and \eqref{eq:lem3.4-1}, we obtain that for all $s> A \xi $, with $A$ being the constant from \eqref{eq:lem3.4-1},
    \begin{equation*}
        \begin{aligned}
            \mathbb{P} \big[ D_n(0,e_1;B_2(0)) \geq  e^{-s\sqrt{mn}} a_{n-m}^{(p/2)} \big] > 1-p/2-p/2=1-p\,.
        \end{aligned}
    \end{equation*}
    This, combined with the definition of $a_n^{(p)}$ and Lemma~\ref{lem:compare-median-p}, yields 
    \begin{equation}
    \label{eq:lem3.4-2}
        a_n^{(p)} \geq e^{-(A \xi +1)\sqrt{mn}} a_{n-m}^{(p/2)} \geq e^{-C\sqrt{mn}} a_{n-m}^{(p)}.
    \end{equation}
    Similarly, for sufficiently large $s > 0$, we can show 
     \begin{equation*}
        \begin{aligned}
         &\mathbb{P} \big[ D_n(0,e_1;B_2(0)) \leq  e^{s\sqrt{mn}} a_{n-m}^{(\frac{p+1}{2})} \big]\\ 
         &\qquad \geq \mathbb{P} \Big[ \big{\{} D_{n-m}(0,e_1;B_2(0)) \leq a_{n-m}^{(\frac{p+1}{2})} \big{\}} \cap \big{\{} \sup_{x \in B_2(0)} h_{n-m,n}(x) \leq s\sqrt{mn} /\xi \big{\}} \Big]\\
         &\qquad > \frac{p+1}{2}-\frac{1-p}{2}=p\,.
        \end{aligned}
    \end{equation*}
    This, together with the definition of $a_n^{(p)}$ and Lemma~\ref{lem:compare-median-p}, implies
    \begin{equation}
    \label{eq:lem3.4-3}
        a_n^{(p)} \leq e^{C\sqrt{mn}} a_{n-m}^{(\frac{p+1}{2})} \leq e^{C\sqrt{mn}} a_{n-m}^{(p)}.
    \end{equation}
    Combining \eqref{eq:lem3.4-2} and \eqref{eq:lem3.4-3} yields the desired result.
\end{proof}

We now turn to the proof of Lemma~\ref{lem:prop3.1-1}. The proof follows a similar approach to that of \cite[Lemma 2.9]{dg-supercritical-lfpp}. Our goal is to construct a path that connects $0$ and $e_1$ within the box $B_2(0)$, such that the $D_n$-length of this path can be upper-bounded by $a_m^{(p)}$ and $a_{n-m}^{(p)}$ with high probability provided that $p$ is sufficiently large. (We will actually use $a_{n-m-k}^{(p)}$, with $k = \lfloor (\log m)^2 \rfloor$, instead of $a_{n-m}^{(p)}$. However, by Lemma~\ref{lem:compare(n)(n-m)}, they do not differ much.) The construction will consist of four steps. In Step 1, we introduce some regularity events for the field, which all happen with high probability. In Step 2, we construct a discrete path on $\mathscrL_m$ (recall its definition from \eqref{eq:def-rescaled-lattice}) whose $D_m$-length can be upper-bounded by $a_m^{(p)}$. Step 3 involves local modifications to the discrete path so that its $D_{m+k,n}$-length can be upper-bounded. We will use a percolation argument for the rescaled lattice $2^{-m-k} \mathbb{Z}^d$ to achieve this. The introduction of the auxiliary scale $k = \lfloor (\log m)^2 \rfloor$ is mainly for this step. In Step 4, we control the $D_n$-length of the resulting path using the regularity events.

\begin{proof}[Proof of Lemma~\ref{lem:prop3.1-1}]
    Let $p \in (0,1)$ be a constant to be chosen. Define the integer
    \begin{equation}
        \label{eq:def-k}
        k := \lfloor (\log m)^2 \rfloor\,.
    \end{equation}
    We assume that 
    \begin{equation*}
        m > 100 \quad \mbox{and} \quad n > m+k\,.
    \end{equation*}Otherwise, Equation~\eqref{eq:lem3.5-1} can be deduced from Lemmas~\ref{lem:compare-median-p} and \ref{lem:compare(n)(n-m)} by choosing a sufficiently large $C$. This is because, for a fixed $p$, by Lemmas~\ref{lem:compare-median-p} and \ref{lem:compare(n)(n-m)}, we have
    \begin{equation}
    \label{eq:lem3.5-finite}
        \begin{aligned}
            &\lambda_n \leq e^{C \sqrt{n}} a_n^{(p)} \leq e^{C \sqrt{n}} a_{n-m}^{(p)} \leq e^{C \sqrt{n}}a_m^{(p)} a_{n-m}^{(p)} \quad \forall 1 \leq m  \leq 100\,,\\
            &\lambda_n \leq e^{C\sqrt{n}} a_n^{(p)} a_1^{(p)} \leq e^{C \sqrt{nk}} a_m^{(p)} a_{n-m}^{(p)} \quad \forall m <n  \leq m+k\,,
        \end{aligned}
    \end{equation}
    and $C\sqrt{n}, C\sqrt{nk} \leq C n^{2/3}$.

    Next, we will construct a path connecting $0$ and $e_1$ within $B_2(0)$. When $p$ is sufficiently close to one (not depending on $n$), the $D_n$-length of this path will be at most $e^{Cn^{2/3}} a_m^{(p)} a_{n-m-k}^{(p)}$ with probability at least $1/2$. Therefore,
    \begin{equation}
    \label{eq:lem3.5-11}
        \lambda_n \leq e^{Cn^{2/3}} a_m^{(p)} a_{n-m-k}^{(p)}\,.
    \end{equation}
    Combining this with Lemma~\ref{lem:compare(n)(n-m)}, we obtain Lemma~\ref{lem:prop3.1-1}.
    
    As announced earlier, the construction consists of four steps.
    
    \noindent\textit{Step 1: Regularity event for $h_m$ and $h_{m,m+k}$.} Define the event
    \begin{equation}
        \label{eq:def-e1}
        \mathcal{E}_1 := \big{\{}2^{-m} \sup_{x \in B_2(0)} |\nabla h_m(x)|_\infty \leq n^{2/3} \big{\}} \cap \big{\{} \sup_{x \in B_2(0)} h_{m,m+k}(x) \leq C_1 \sqrt{k(m+k)} \big{\}} \,,
    \end{equation}
    where $C_1$ is the constant defined in Lemma~\ref{lem:h(n-m)(n)}. Using the fact that $|2^{-m}\mathbb{Z}^d \cap B_2(0)| \leq C2^{md}$ and Claim (2) in Lemma~\ref{lem:field-estimate}, we obtain
    \begin{equation}
    \label{eq:est-e1-0}
    \begin{aligned}
        &\mathbb{P}\Big[2^{-m} \sup_{x \in B_2(0)} |\nabla h_m(x)|_\infty \leq n^{2/3}\Big] \\
        &\qquad \geq 1 -  \sum_{x \in 2^{-m}\mathbb{Z}^d \cap B_2(0)} \mathbb{P}\Big[\sup_{y \in B_{2^{-m}}(x)} |\nabla h_m(y)|_\infty > 2^mn^{2/3}\Big]\\
        &\qquad \geq 1-C 2^{md} \times C e^{-n^{4/3}/C} \geq 1-Ce^{-n^{4/3}/C}.
    \end{aligned}
    \end{equation}
    Combining this with Lemma~\ref{lem:h(n-m)(n)}, applied with $(m+k,k)$ in place of $(n,m)$, yields 
    \begin{equation}
        \label{eq:est-e1}
        \mathbb{P}[\mathcal{E}_1] \geq 1-Ce^{-n^{4/3}/C} - C e^{-m/C} \geq 1- C e^{-m/C}.
    \end{equation}
    \noindent\textit{Step 2: Discretize the $D_m$-geodesic between $0$ and $e_1$ on $\mathscrL_m$.} Define the event
    \begin{equation}
    \label{eq:def-e2}
        \mathcal{E}_2:= \{ D_m(0,e_1;B_2(0)) \leq a_m^{(p)} \}\,.
    \end{equation}
    By \eqref{def:quantile}, we have
    \begin{equation}
        \label{eq:est-e2}
        \mathbb{P}[\mathcal{E}_2] = p\,.
    \end{equation}
    On the event $\mathcal{E}_2$, there exists a piecewise continuously differentiable path $P:[0,1] \rightarrow B_2(0)$ from $0$ to $e_1$ such that
    \begin{equation}
    \label{eq:lem3.5-geodesic}
    {\rm len}(P;D_m) = \int_0^1 e^{\xi h_m(P(t))} |P'(t)|dt \leq 2 a_m^{(p)}.
    \end{equation}

    Recall from \eqref{eq:def-rescaled-lattice} that $\mathscrL_m = 2^{-m}\mathbb{Z}^d \cap B_2(0)$. Then, we have $0, e_1 \in \mathscrL_m$. We consider $\mathscrL_m$ as a subset of $\mathbb{R}^d$. Sometimes, we will consider ($*$-)paths or ($*$-)clusters on the rescaled lattice $2^{-m} \mathbb{Z}^d$, as defined in Subsection~\ref{subsec:percolation}, and only in these cases, we regard $\mathscrL_m$ as a subset of $2^{-m} \mathbb{Z}^d$. We now construct, on the event $\mathcal{E}_1 \cap \mathcal{E}_2$, a self-avoiding path on $\mathscrL_m$ as a discrete approximation of the path $P$. See Figure~\ref{fig:1} for an illustration.\footnote{For illustrative purposes, we depict planar graphs, but all these arguments hold for dimensions greater than two.}

\begin{figure}[H]
\centering
\includegraphics[scale=0.6]{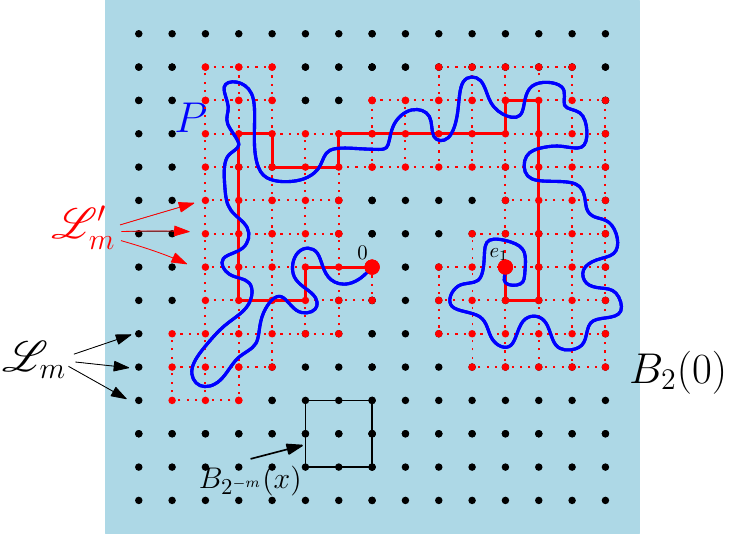}
\caption{Illustration of the sets $\mathscrL_m$ and $\mathscrL_m'$, and the path $P$ and its discrete approximation $(x_1,\ldots,x_J)$. The dotted red lines represent the edges between neighboring vertices in $\mathscrL_m'$. The path $(x_1,\ldots,x_J)$, as illustrated by the red curve, is a self-avoiding path on $\mathscrL_m'$ connecting $0$ and $e_1$.}
\label{fig:1}
\end{figure}

    Let $\mathscrL_m'$ be a subset of $\mathscrL_m$ defined as follows:
    \begin{equation*}
        \mathscrL_m' := \{ x \in \mathscrL_m : P \cap \overline{B_{2^{-m}}(x)} \neq \emptyset \}\,,
    \end{equation*}
    where $\overline{B_{2^{-m}}(x)}$ represents the closure of $B_{2^{-m}}(x)$. It follows that $0,e_1 \in \mathscrL_m'$, and there exists a discrete path in $\mathscrL_m'$ connecting them. This is because for any $x \in \mathscrL_m'$ considering the first exit time of $P$ from the box $\overline{B_{2^{-m}}(x)}$, we can find a vertex $y \in \mathscrL_m$ such that $|x-y|_1 = 2^{-m}$, and $P$ also enters the box $\overline{B_{2^{-m}}(y)}$. By doing this procedure iteratively, we obtain a discrete path in $\mathscrL_m'$ that connects $0$ and $e_1$. Taking any path in $\mathscrL_m'$ connecting $0$ and $e_1$, and applying the loop erasure procedure similar to \eqref{eq:lem2.8-saw}, yields a self-avoiding path connecting $0$ and $e_1$ in $\mathscrL_m'$. That is, there exists a self-avoiding path $0 = x_1,\ldots,x_J = e_1$ satisfying the properties that
    \begin{equation}
    \label{eq:lem3.5-path-prop}
    \begin{aligned}
        &x_i \in \mathscrL_m \quad \mbox{and} \quad P \cap \overline{B_{2^{-m}}(x_i)} \neq \emptyset \quad \forall 1 \leq i \leq J\,, \quad \mbox{and} \\
        &|x_i - x_{i+1}|_1 = 2^{-m} \quad  \forall 1 \leq i \leq J-1\,.
        \end{aligned}
    \end{equation}
    
    We now show that on the event $\mathcal{E}_1 \cap \mathcal{E}_2$, we have
    \begin{equation}
        \label{eq:lem3.5-weight}
        \sum_{j=1}^J 2^{-m} e^{\xi h_{0,m}(x_j)} \leq a_m^{(p)} e^{Cn^{2/3}}.
    \end{equation}
    This is because for each $1 \leq i \leq J$, the property $P \cap \overline{B_{2^{-m}}(x_i)} \neq \emptyset$ in \eqref{eq:lem3.5-path-prop} ensures that the path $P$ must cross the hypercubic shell $B_{2^{-m+1}}(x_i) \backslash B_{2^{-m}}(x_i)$. This segment has Euclidean length of at least $2^{-m}$. By the event $\mathcal{E}_1$, for some $C>0$, we have
    $$
    \inf_{z \in B_{2^{-m+1}}(x_i)}h_{0,m}(z) \geq h_{0,m}(x_i) -Cn^{2/3}.
    $$
    Therefore, this segment has a $D_m$-length of at least \begin{equation}
    \label{eq:lem3.5-weight-lower}  
    2^{-m} e^{\xi h_{0,m}(x_i) - Cn^{2/3}}.
    \end{equation}
    Furthermore, each point on $P$ is contained in at most $5^d$ such hypercubic shells. Combining this fact with \eqref{eq:lem3.5-weight-lower} and \eqref{eq:lem3.5-geodesic}, we obtain \eqref{eq:lem3.5-weight}.

\medskip

\noindent\textit{Step 3: Modify the path on $\mathscrL_{m+k}$.} Recall from \eqref{eq:def-rescaled-lattice} that $\mathscrL_{m+k}= 2^{-m-k} \mathbb{Z}^d \cap B_2(0)$. It follows that $\mathscrL_m \subset \mathscrL_{m+k}$. We now construct a path on $\mathscrL_{m+k}$ that closely follows the path $(x_1,\ldots,x_J)$ and has typical $D_{m+k,n}$-length; see Figure~\ref{fig:2}. We call a vertex $x \in \mathscrL_{m+k}$ $\textbf{open}$ if for all $\sigma \in \{1,-1\}$ and $1 \leq i \leq d$
\begin{equation}
\label{eq:lem3.5-def-open}
    D_{m+k,n}(x, x+ \sigma e_i 2^{-m-k}; B_{2^{-m-k+1}}(x)) \leq 2^{-m-k} a_{n-m-k}^{(p)},
\end{equation}
and \textbf{closed} otherwise. We assume that all vertices in $2^{-m-k} \mathbb{Z}^d \backslash \mathscrL_{m+k}$ are open. Using the translation and rotational invariance and the scaling property from Lemma~\ref{lem:basic-LFPP}, we have
$$
D_{m+k,n}(x, x+ \sigma e_i 2^{-m-k}; B_{2^{-m-k+1}}(x)) \overset{d}{=} 2^{-m-k}  D_{n-m-k}(0,e_1;B_2(0))\,.
$$
Combining this with the definition of $a_{n-m-k}^{(p)}$ from \eqref{def:quantile}, we obtain that for all $x \in \mathscrL_{m+k}$
\begin{equation}
\label{eq:lem3.5-est-open}
\begin{aligned}
&\quad \mathbb{P}[x \mbox{ is open}] \\
&\geq 1- \sum_{\sigma = \pm1, 1\leq i \leq d}\mathbb{P}\big[D_{m+k,n}(x, x+ \sigma e_i 2^{-m-k}; B_{2^{-m-k+1}}(x))>  2^{-m-k}a_{n-m-k}^{(p)}\big]\\
    &=1- 2d \cdot \mathbb{P}\big[D_{0,n-m-k}( 0, e_1 ; B_{2}(0) ) >  a_{n-m-k}^{(p)}\big] = 1-2d(1-p)\,.
\end{aligned}
\end{equation}
In particular, as $p$ approaches one, this probability also tends to one. Recalling the notation in Subsection~\ref{subsec:percolation}, we similarly define open (or closed) ($*$-)paths and ($*$-)clusters on the rescaled lattice $2^{-m-k} \mathbb{Z}^d$. Define the event
\begin{equation}
\label{eq:def-e3}
\begin{aligned}
\mathcal{E}_3 &:= \{\mbox{Both }0 \mbox{ and }e_1 \mbox{ are contained in infinite open clusters on }2^{-m-k} \mathbb{Z}^d,\\
&\qquad   \mbox{ and all closed }{\rm *}\mbox{-clusters have diameter at most }2^{k-2} \}\,.
\end{aligned}
\end{equation}
Here, the diameter is associated with the graph distance on the rescaled lattice $2^{-m-k} \mathbb{Z}^d$.

By the definition in \eqref{eq:lem3.5-def-open}, whether a vertex $x$ is open is determined by the field $h_{m+k, n}$ restricted to the domain $B_{2^{-m-k+1}}(x)$. So, according to Lemma~\ref{lem:dist-independent}, for two subsets $U,V \subset 2^{-m-k} \mathbb{Z}^d$ with graph distance at least $2\rr+4$, the statuses of the vertices in $U$ being open or closed are independent of the statuses of those within $V$. Therefore, $\mathbb{P}$ induces an $M$-dependent measure on $\{0,1\}^{2^{-m-k} \mathbb{Z}^d}$ (where $0$ represents closed and $1$ represents open) with $M = \lfloor 2\rr+4 \rfloor  + 1$. As a result, we can apply the percolation result in Subsection~\ref{subsec:percolation}. Using \eqref{eq:lem3.5-est-open}, 
Lemma~\ref{lem:percolation-cluster}, and Lemma~\ref{lem:connect-inf} (with $\epsilon = 0.01$), we can show the existence of $c_2' \in (0,1)$ such that when $p \geq c_2'$, the following inequality holds:
\begin{equation}
\label{eq:est-e3}
   \mathbb{P}[\mathcal{E}_3] \geq 1- 2\times 0.01 - C 2^{d(m+k)} e^{-2^{k-2}/C} \geq 1 - 0.02 - Ce^{-2^k/C}.
\end{equation}
The last inequality is due to the fact that $k \geq (\log m)^2-1$. From now on, we take
\begin{equation}
\label{eq:lem3.5-p}
p = \max \{ c_2', 0.99 \}.
\end{equation}

On the event $\mathcal{E}_3$, since all closed $*$-clusters on $2^{-m-k} \mathbb{Z}^d$ have diameter at most $2^{k-2}$, for each $x \in \mathscrL_m \subset \mathscrL_{m+k}$, there is no closed $*$-cluster on $\mathscrL_{m+k}$ that crosses the hypercubic shell $B_{2^{-m}}(x) \backslash B_{2^{-m-1}}(x)$\footnote{For $x \in \mathscrL_m$ with $\mathfrak d_\infty(x, \partial B_2(0)) = 2^{-m}$, we consider the hypercubic shell $(B_{2^{-m}}(x) \backslash B_{2^{-m-1}}(x))\cap \{y:\mathfrak d_\infty(y,\partial B_2(0)) > 2^{-m-k+1}\}$ instead. This ensures that for any $y$ in these hypercubic shells, we always have $B_{2^{-m-k+1}}(y) \subset B_2(0)$.} or encloses $B_{2^{-m-1}}(x)$. Therefore, by duality, there exists a unique open cluster on $\mathscrL_{m+k}$ that encloses $B_{2^{-m-1}}(x)$ within the hypercubic shell $B_{2^{-m}}(x) \backslash B_{2^{-m-1}}(x)$. Furthermore, the open clusters corresponding to neighboring vertices on $\mathscrL_m$ intersect, as illustrated in Figure~\ref{fig:2}. Since $0$ is contained in an infinite open cluster, we can find an open path in $\mathscrL_{m+k}$ (see the brown curves in Figure~\ref{fig:2}) that connects $0$ to its open cluster that encloses $B_{2^{-m-1}}(0)$. The same holds for $e_1$.

By joining these open paths and clusters together, and applying the loop erasure procedure, we can construct a self-avoiding open path on $\mathscrL_{m+k}$ that connects $0$ and $e_1$, closely following the sequence $(x_1,\ldots, x_J)$. Let us denote the resulting path as $0 = y_1,y_2,\ldots,y_K = e_1$. It satisfies the condition that for each $1 \leq i \leq K$:
\begin{equation}
\label{eq:lem3.5-path-prop-2}
    y_i \in \mathscrL_{m+k} \mbox{ is open}\,, \quad  \mbox{and} \quad \min_{1 \leq l \leq J} |y_i - x_l|_\infty \leq 2^{-m}.
\end{equation}

\begin{figure}[t]
\centering
\includegraphics[scale = 0.6]{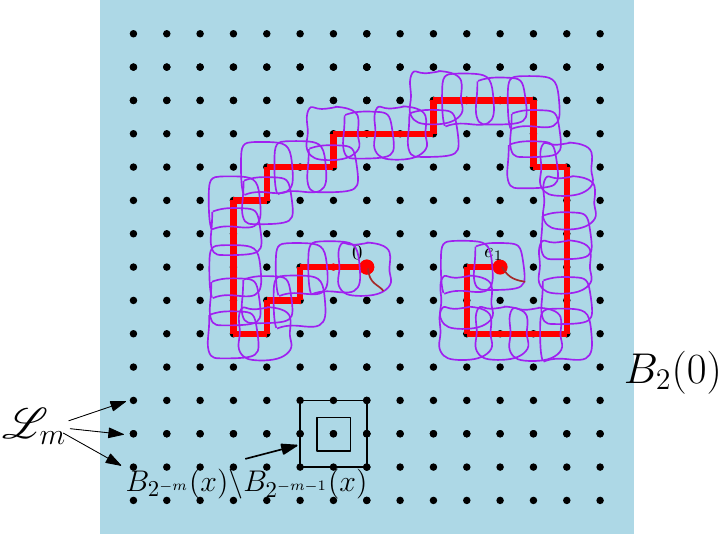}
\caption{The red path corresponds to $( x_1,x_2,\ldots,x_J )$. The open clusters on $\mathscrL_{m+k}$ that enclose $B_{2^{-m-1}}(x_i)$ are depicted in purple, and the two brown curves represent the open paths that connect $0$ and $e_1$ to their corresponding open clusters. By joining these clusters, we can construct a path on $\mathscrL_{m+k}$ connecting $0$ and $e_1$ that closely follows the red curve and has typical $D_{m+k,n}$-length.}
\label{fig:2}
\end{figure}

\noindent\textit{Step 4: Concatenate the geodesic and upper-bound the $D_n$-length.} In the final step, we join the geodesics between $y_i$ and $y_{i+1}$ for $1 \leq i \leq K-1$ and upper-bound its $D_n$-length. Assume that
\begin{equation*}
    \mathcal{E}_1 \cap \mathcal{E}_2 \cap \mathcal{E}_3 \mbox{ happens}\,.
\end{equation*}Using \eqref{eq:lem3.5-path-prop-2} and the definition of open vertices from \eqref{eq:lem3.5-def-open}, for each $1 \leq i \leq K-1$, there exists a piecewise continuously differentiable path $P_i : [0,1]  \rightarrow B_{2^{-m-k+1}}(y_i)$ that connects $y_i$ and $y_{i+1}$ and satisfies
\begin{equation}
    \label{eq:lem3.5-geodesic-2}
    {\rm len}(P_i; D_{m+k,n}) = \int_0^1 e^{\xi h_{m+k,n}(P_i(t))} |P_i'(t)|dt \leq 2^{1-m-k} a_{n-m-k}^{(p)}.
\end{equation}
By concatenating the paths $P_1,P_2,\ldots,P_{K-1}$, we obtain a path $\widetilde P$ that connects $0$ and $e_1$ within $B_2(0)$. 

We now upper-bound the $D_n$-length of $\widetilde{P}$ on the event $\mathcal{E}_1 \cap \mathcal{E}_2 \cap \mathcal{E}_3$. For each $1 \leq i \leq K$, by \eqref{eq:lem3.5-path-prop-2}, we can choose $1 \leq k_i \leq J$ such that 
\begin{equation}
\label{eq:lem3.5-neighbor-xy}
    |y_i - x_{k_i}|_\infty \leq 2^{-m}.
\end{equation}
Since $h_{0,n} = h_{0,m} + h_{m, m+k} + h_{m+k, n}$, we have
\begin{equation}
\label{eq:lem3.3-length-widetilde-P}
    \begin{aligned}
        {\rm len}(\widetilde P;D_n) 
        &= \sum_{i=1}^{K-1} \int_0^1 e^{\xi h_{0,n}(P_i(t))} |P_i'(t)|dt  \\
        &= \sum_{i=1}^{K-1} \int_0^1 e^{\xi h_{0,m}(P_i(t))}e^{\xi h_{m,m+k}(P_i(t))}e^{\xi h_{m+k,n}(P_i(t))} |P_i'(t)|dt \,.
    \end{aligned}
\end{equation}
By \eqref{eq:lem3.5-neighbor-xy} and the event $\mathcal{E}_1$ defined in \eqref{eq:def-e1}, we obtain that there exists a constant $C>0$ (not depending on $n,m$) such that for all $1 \leq i \leq K$ and $0 \leq t \leq 1$,
\begin{equation*}
    |h_{0,m}(P_i(t))- h_{0,m}(x_{k_i})| \leq Cn^{2/3}  \quad \mbox{and} \quad h_{m,m+k}(P_i(t)) \leq C\sqrt{k(m+k)}\,.
\end{equation*}
Combining this with \eqref{eq:lem3.3-length-widetilde-P} yields 
\begin{equation*}
    \begin{aligned}
        {\rm len}(\widetilde P;D_n) &\leq \sum_{i=1}^{K-1}e^{Cn^{2/3} + C\sqrt{k(m+k)}} e^{\xi h_{0,m}(x_{k_i})} \int_0^1 e^{\xi h_{m+k,n}(P_i(t))} |P_i'(t)|dt\,.
    \end{aligned}
\end{equation*}
Combining this with \eqref{eq:lem3.5-geodesic-2} and \eqref{eq:def-k}, we further have
\begin{equation*}
    {\rm len}(\widetilde P;D_n)\leq e^{Cn^{2/3}} a_{n-m-k}^{(p)} \sum_{i=1}^{K-1} 2^{-m-k} e^{\xi h_{0,m}(x_{k_i})}.
\end{equation*}
For each $x \in \mathscrL_m$, the number of vertices in $\mathscrL_{m+k}$ satisfying \eqref{eq:lem3.5-neighbor-xy} with $x_{k_i} = x$ is at most $C 2^{kd}$. Consequently, each $x$ appears in $(x_{k_i})_{1 \leq i \leq K}$ at most $C 2^{kd}$ times. Applying this fact with \eqref{eq:lem3.5-weight}, we obtain 
\begin{equation*}
\begin{aligned}
    {\rm len}(\widetilde P;D_n) 
    &\leq C 2^{kd} e^{Cn^{2/3}} a_{n-m-k}^{(p)} \sum_{j=1}^J 2^{-m} e^{\xi h_{0,m}(x_j)}  \\
    &\leq C 2^{kd} e^{Cn^{2/3}} a_{n-m-k}^{(p)} a_m^{(p)} \leq e^{Cn^{2/3}} a_{n-m-k}^{(p)} a_m^{(p)} \,.
\end{aligned}
\end{equation*}
The last inequality is due to \eqref{eq:def-k}. Therefore, on the event $\mathcal{E}_1 \cap \mathcal{E}_2 \cap \mathcal{E}_3$, we have
\begin{equation*}
    D_n(0,e_1; B_2(0)) \leq e^{Cn^{2/3}} a_{n-m-k}^{(p)} a_m^{(p)} \,.
\end{equation*}
Combining the estimates from \eqref{eq:est-e1}, \eqref{eq:est-e2}, and \eqref{eq:est-e3}, and recalling $p$ from \eqref{eq:lem3.5-p}, we conclude that for sufficiently large $m$:
\begin{equation*}
    \mathbb{P}[\mathcal{E}_1 \cap \mathcal{E}_2 \cap \mathcal{E}_3] \geq p - Ce^{-m/C} - 0.02 - Ce^{-2^k/C} \geq \frac{1}{2}\,.
\end{equation*}
Combining the above two inequalities yields Equation~\eqref{eq:lem3.5-11} when $m$ is sufficiently large. We can extend the result to small $m$ by enlarging the value of $C$ similar to \eqref{eq:lem3.5-finite}. This concludes the lemma.
\end{proof}

We now extend Proposition~\ref{prop:exponent} to the internal distance between any pair of points within a box.
\begin{lemma}
    \label{lem:distance-any-point}
    Fix $0 < r_1 <r_2$. For any $x,y \in B_{r_2 - r_1}(0)$ with $|x-y|_\infty \geq r_1$, we have
    \begin{equation*}
        {\rm Med}(D_n(x,y ; B_{r_2}(0))) = 2^{-(1-\xi Q)n + o(n) }\quad \mbox{as } n\rightarrow \infty\,.
    \end{equation*}
    Here, the $o(n)$ term only depends on $r_1$ and $r_2$, and is independent of both $x$ and $y$.
\end{lemma}
\begin{proof}
    We first prove the upper bound. Fix an integer $m \geq 1$. Using the concentration bound from Lemma~\ref{lem:gauss-concentration}, we obtain that, with probability $1-o_n(1)$,
    \begin{equation*}
        D_{m,n}(0,2^{-m}e_1; 2^{-m}B_2(0)) = 2^{o(n)} {\rm Med}(D_{m,n}(0,2^{-m}e_1; 2^{-m}B_2(0))),
    \end{equation*}
    where the $o(n)$ term can depend on $m$. Since $h_{0,n} = h_{0,m} + h_{m,n}$ for all integer $n > m$, we have
    \begin{equation*}
        \inf_{x \in 2^{-m}B_2(0)} e^{\xi h_m(x)} \leq \frac{D_{0,n}(0,2^{-m}e_1; 2^{-m}B_2(0))}{D_{m,n}(0,2^{-m}e_1; 2^{-m}B_2(0))} \leq \sup_{x \in 2^{-m}B_2(0)} e^{\xi h_m(x)} .
    \end{equation*} Furthermore, since $m$ is fixed, we have $\sup_{x \in B_1(0)} |h_m(x)| = o(n)$ with probability $1-o_n(1)$. Therefore, as $n$ tends to infinity, we have
    \begin{equation*}
        {\rm Med}(D_{0,n}(0,2^{-m}e_1; 2^{-m}B_2(0))) = 2^{o(n)}{\rm Med}(D_{m,n}(0,2^{-m}e_1; 2^{-m}B_2(0))) \,.
    \end{equation*}
    Applying the scaling relation in Lemma~\ref{lem:basic-LFPP}, we obtain that for all integer $n > m$:
    \begin{equation*}
        D_{m,n}(0,2^{-m}e_1; 2^{-m}B_2(0)) \overset{d}{=} 2^{-m}D_{0,n-m}(0,e_1;B_2(0))\,.
    \end{equation*}
    Combining the above two equations with Proposition~\ref{prop:exponent}, we obtain that as $n$ tends to infinity
    \begin{equation}
    \label{eq:lem3.6-1}
    \begin{aligned}
        {\rm Med}(D_{0,n}(0,2^{-m}e_1; 2^{-m}B_2(0))) &= 2^{o(n)}{\rm Med}(D_{0,n-m}(0,e_1;B_2(0))) \\
        &= 2^{-(1-\xi Q)n + o(n) }.
    \end{aligned}
    \end{equation}
    
    We can choose a large enough integer $m$ such that for any $x,y$ as stated in the lemma, we can connect them by a sequence of points such that any two neighboring points have an $l^2$-distance of $2^{-m}$, and the number of points is upper-bounded by a constant $C$ depending only on $r_1$ and $r_2$. Using the translation and rotational invariance from Lemma~\ref{lem:basic-LFPP}, the internal $D_n$-distance within $B_2(0)$ of neighboring points in this sequence is stochastially dominated by $D_{0,n}(0,2^{-m}e_1; 2^{-m}B_2(0))$. This sequence provides an upper bound for $D_n(x,y ; B_{r_2}(0))$ in terms of the distribution of $D_{0,n}(0,2^{-m}e_1; 2^{-m}B_2(0))$. Combining with \eqref{eq:lem3.6-1} and the concentration bound from Lemma~\ref{lem:gauss-concentration}, we obtain 
    \begin{equation*}
        {\rm Med}(D_n(x,y ; B_{r_2}(0))) \leq  2^{-(1-\xi Q)n + o(n) }\quad \mbox{as } n\rightarrow \infty\,.
    \end{equation*}
    Furthermore, the $o(n)$ term is independent of both $x$ and $y$.
    
    We now prove the lower bound. Similar to \eqref{eq:lem3.6-1}, we can show that for any fixed integer $m \geq 1$:
    \begin{equation}
    \label{eq:lem3.6-2}
        {\rm Med}(D_{0,n}(0,2^m e_1; B_{2^{m+1}}(0))) = 2^{-(1-\xi Q)n + o(n) }\quad \mbox{as } n\rightarrow \infty\,.
    \end{equation}
    This is derived from Proposition~\ref{prop:exponent} and the following scaling relation from Lemma~\ref{lem:basic-LFPP}: 
    \begin{equation*}
        D_n(0, 2^m e_1 ;B_{2^{m+1}}(0)) \overset{d}{=} 2^m D_{m,n+m}(0, e_1 ;B_2(0))\,.
    \end{equation*}

    Similar to before, we can choose a large enough integer $m$ such that for any $x,y$ as stated in the lemma, we can connect $0$ and $2^m e_1$ by a sequence of points in the box $B_{2^{m+1}}(0)$ such that any two neighboring points have an $l^2$-distance of $|x-y|_2$, and the number of points is upper-bounded by a constant $C$ depending only on $r_1$ and $r_2$. See Figure~\ref{fig:3} for an illustration. Using the translation and rotational invariance from Lemma~\ref{lem:basic-LFPP}, the internal distance within $B_{2^{m+1}}(0)$ of neighboring points in this sequence is stochastically dominated by $D_n(x,y ; B_{r_2}(0))$. This sequence provides an upper bound for $D_n(0,2^m e_1 ; B_{2^{m+1}}(0)))$ in terms of the distribution of $D_n(x,y ; B_{r_2}(0))$. Combining this with \eqref{eq:lem3.6-2} and the concentration bound from Lemma~\ref{lem:gauss-concentration}, we obtain
    \begin{equation*}
       {\rm Med}(D_n(x,y ; B_{r_2}(0))) \geq  2^{-(1-\xi Q)n + o(n) }\quad \mbox{as } n\rightarrow \infty\,.
    \end{equation*}
    This concludes the lemma. \qedhere

    \begin{figure}[H]
\centering
\includegraphics[scale = 1]{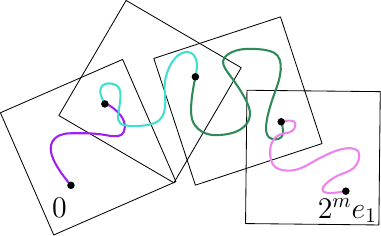}
\caption{Illustration of a path that connects $0$ and $2^m e_1$. The boxes, along with the points, represent the transformation of the triple $(x,y,B_{r_2}(0))$. The colored curves represent the geodesics between neighboring points in these boxes.}
\label{fig:3}
\end{figure}
\end{proof}

We now prove some basic properties of  $Q(\xi)$.

\begin{lemma}
\label{lem:Q-lower}
\begin{enumerate}
    \item We have $\frac{1}{\xi} - \sqrt{2d} \leq Q(\xi) \leq \frac{1}{\xi} + \sqrt{2} $ for all $\xi>0$.
    \item $\xi \mapsto Q(\xi)$ is a non-increasing, continuous function of $\xi$. 
\end{enumerate}
\end{lemma}

\begin{proof}
 We begin with the first claim. Using Claims (3) and (4) from Lemma~\ref{lem:field-estimate}, we obtain that for any fixed $A > \sqrt{2d} \log 2$:
    \begin{equation*}
        \mathbb{P}\big[ \inf_{x \in B_2(0)} h_n(x) \geq -An \big] \geq 1-e^{-n/C}.
    \end{equation*}
    On the event $\{\inf_{x \in B_2(0)} h_n(x) \geq -An \}$, we have
    \begin{equation*}
        D_n(0,e_1; B_2(0)) \geq e^{\xi \inf_{x \in B_2(0)} h_n(x)} \geq e^{-\xi A n}.
    \end{equation*}
    Combining this with Proposition~\ref{prop:exponent}, we obtain that for any $A > \sqrt{2d} \log 2$ and sufficiently large $n$:
    \begin{equation*}
        e^{-\xi A n} \leq \lambda_n = 2^{-(1-\xi Q) n+o(n)}.
    \end{equation*}
    Therefore, $Q \geq 1/\xi - A /\log 2$. As the constant $A$ can be arbitrarily close to $\sqrt{2d} \log 2$, this implies that $Q \geq 1/\xi - \sqrt{2d}$. 
    
   Next, we prove the upper bound for $Q$. Let $l$ denote the straight line connecting $0$ and $e_1$, i.e., $l$ consists of the points $\{(x,0,\ldots,0): 0 \leq x \leq 1\}$. Fix any $\epsilon>0$. Similar to \eqref{eq:lem2.3-claim4}, by Claims (1) and (2) in Lemma~\ref{lem:field-estimate}, we have
    \begin{align*}
        &\quad \mathbb{P} \Big[\sup_{x \in l}h_n(x) \geq (\sqrt{2} \log 2 +\epsilon) n \Big] \\
        &\leq \mathbb{P} \Big[\sup_{x \in l \cap 2^{-n}\mathbb{Z}^d} h_n(x) \geq (\sqrt{2} \log 2 + \epsilon / 2) n \Big] + \mathbb{P} \Big[2^{-n} d\sup_{\substack{y \in B_{2^{-n}}(x)\\x \in l \cap 2^{-n}\mathbb{Z}^d}} |\nabla h_n(y)|_\infty \geq \epsilon n / 2  \Big]\\
        &\leq 2^n \times C\exp\big(-\frac{((\sqrt{2} \log 2 + \epsilon / 2) n)^2}{2\log2 \cdot n}\big) + 2^n \times Ce^{-n^2/C} \leq Ce^{-n/C}.
    \end{align*}
    Therefore, with high probability, we have $\sup_{x \in l}h_n(x) \leq (\sqrt{2} \log 2 +\epsilon) n$. This implies that, with high probability, $$D_n(0,e_1;B_2(0)) \leq {\rm len}(l;D_n) \leq e^{\xi \sup_{x \in l}  h_n(x)} \leq e^{\xi (\sqrt{2} \log 2 +\epsilon) n}.$$ Combining this with Proposition~\ref{prop:exponent}, we get $2^{-n(1-\xi Q) + o(n)} \leq e^{\xi (\sqrt{2} \log 2 +\epsilon) n}$, hence $Q \leq \frac{1}{\xi} +\sqrt{2} + \frac{\epsilon}{\log 2}$. Since this holds for any $\epsilon>0$, we get $Q \leq \frac{1}{\xi} +\sqrt{2}$.

   Next, we prove that $\xi \mapsto Q(\xi)$ is non-increasing and continuous. For $\xi>0$ and integer $n \geq 1$, define
    \begin{equation*}
        \mathcal{D}_n^{(\xi)} := \min_{0 = x_1, \ldots, x_J = e_1} \sum_{i=1}^J e^{\xi h_n(x_i)},
    \end{equation*}where the minimum is taken over all paths in $\mathscrL_n$ connecting $0$ and $e_1$. We first show that, for a fixed $p \in (0,1)$:
    \begin{equation}
    \label{eq:lem-Q-continuous-1}
        \left( p \mbox{-quantile of } \mathcal{D}_n^{(\xi)} \right) = 2^{n \xi \cdot Q(\xi) +o(n)} \quad \mbox{as } n \rightarrow \infty\,.
    \end{equation}
    Recall from Proposition~\ref{prop:exponent} and Lemma~\ref{lem:compare-median-p} that for any fixed $p \in (0,1)$:
    \begin{equation}
    \label{eq:lem-Q-continuous-2}
       \left( p \mbox{-quantile of } D_n(0,e_1;B_2(0)) \right) = a_n^{(p)} = 2^{-n + n \xi \cdot Q(\xi) +o(n)} \quad \mbox{as } n \rightarrow \infty\,.
    \end{equation}
    Define the event 
    \[
    \mathcal{E} := \{ 2^{-n} \sup_{x \in B_2(0)} |\nabla h_n(x)|_\infty \leq n^{2/3} \} . 
    \]
    We will compare $\mathcal{D}_n^{(\xi)}$ and $D_n(0,e_1;B_2(0))$ on the event $\mathcal{E}$. Using Claim (2) of Lemma~\ref{lem:field-estimate} in a similar manner to \eqref{eq:est-e1-0}, we get that
    \begin{equation}
    \label{eq:est-mathcal-E}
        \mathbb{P}[\mathcal{E}] \geq 1 - Ce^{-n^{4/3}/C}.
    \end{equation}
    Similarly to Step 2 in the proof of Lemma~\ref{lem:prop3.1-1}, on the event $\mathcal{E}$, we can construct a discrete path $0 = x_1,\ldots, x_J = e_1$ on $\mathscrL_n$ such that
    \begin{equation*}
        \sum_{i=1}^J e^{\xi h_n(x_i)} \leq 2^n D_n(0,e_1;B_2(0)) \cdot e^{Cn^{2/3}}.
    \end{equation*}
    Combining this with \eqref{eq:lem-Q-continuous-2} and \eqref{eq:est-mathcal-E}, we obtain the upper bound part of the claim~\eqref{eq:lem-Q-continuous-1}. 
    
    The lower bound can also be deduced using the event $\mathcal{E}$, as follows. For any path $0 = x_1, \ldots, x_J = e_1$, on the event $\mathcal{E}$, we have
    \begin{equation*}
        D_n(0,e_1;B_2(0)) \leq \sum_{i=1}^{J-1} D_n(x_i,x_{i+1};B_2(0)) \leq \sum_{i=1}^{J-1} 2^{-n} e^{\xi h_n(x_i)} \cdot e^{Cn^{2/3}}.
    \end{equation*}
    This, together with \eqref{eq:lem-Q-continuous-2} and \eqref{eq:est-mathcal-E}, implies the lower bound part of \eqref{eq:lem-Q-continuous-1}.
    
    We now prove the second claim using \eqref{eq:lem-Q-continuous-1}. For any $\wt\xi > \xi$, the function $x \mapsto x^{\xi/\wt\xi}$ is concave, hence subadditive. Using this, we get $(\sum_{i=1}^J e^{\wt \xi h_n(x_i)})^{\xi/\wt\xi} \leq \sum_{i=1}^J e^{\xi h_n(x_i)}$, which implies that $(\mathcal{D}_n^{(\wt \xi)})^{\xi/\wt\xi} \leq \mathcal{D}_n^{(\xi)}$. Together with \eqref{eq:lem-Q-continuous-1}, we obtain that $Q(\wt \xi) \leq Q(\xi)$ for any $\wt\xi > \xi$, and thus $Q$ is non-increasing in $\xi$. 
    
    For any $\wt \xi > 0$, we have
    \begin{equation*}
        \exp \big( - |\wt\xi - \xi| \cdot \sup_{x \in B_2(0)} |h_n(x)| \big) \leq \frac{\mathcal{D}_n^{(\wt \xi)}}{\mathcal{D}_n^{(\xi)}} \leq \exp \big( |\wt\xi - \xi| \cdot \sup_{x \in B_2(0)} |h_n(x)| \big)\,.
    \end{equation*}
    Recall from Claims (3) and (4) of Lemma~\ref{lem:field-estimate} that $\sup_{x \in B_2(0)} |h_n(x)| \leq (1 + \sqrt{2d} \log 2)n$ with high probability. This, together with \eqref{eq:lem-Q-continuous-1}, implies that $|\xi Q(\xi) - \wt \xi Q(\wt\xi)| \leq C |\xi - \wt\xi|$, and hence $Q(\xi)$ is continuous in $\xi$. \qedhere
\end{proof}

\begin{proof}[Proof of Proposition~\ref{prop:exponent-intro}]
Combine Proposition~\ref{prop:exponent} and Lemma~\ref{lem:Q-lower}.
\end{proof}

\section{Bounds for different types of distances}
\label{sec:bound-distance}
From now on, we only consider the case where $\xi$ satisfies the condition:
\begin{equation}
\label{eq:Q>2d}
Q(\xi) > \sqrt{2d}\,,
\end{equation}
where $Q(\xi)$ is defined in Proposition~\ref{prop:exponent-intro}. As implied by Lemma~\ref{lem:Q-lower}, this set of $\xi$ includes $(0,\frac{1}{2\sqrt{2d}})$. 

In this section, we will derive bounds for different types of distances. In Subsection~\ref{subsec:chaining}, we present a chaining argument to upper-bound the diameter of a box. In Subsections~\ref{subsec:bound-diameter} and \ref{subsec:cross}, we show that both the diameter of a box and distance across a hypercubic shell decay at the same exponential rate as $\lambda_n$, up to $o(n)$ errors in the exponent, with the rate given in terms of $Q$. In Subsection~\ref{subsec:super-exponential}, we establish super-exponential concentration bounds for the distances across and around a hypercubic shell.

For integers $n \geq m \geq 0$ and an open set $U$, we define the diameter of $U$ with respect to the metric $D_{m,n}$ as follows:
\begin{equation}
    \label{def:diameter-1}
    {\rm Diam}_{m,n}(U) := \sup_{z,w \in U} D_{m,n}(z,w)\,.
\end{equation}
Similarly, for open sets $U \subset V \subset \mathbb{R}^d$, we define ${\rm Diam}_{m,n}(U;V)$ as the diameter of $U$ with respect to the internal metric $D_{m,n}(\cdot,\cdot; V)$. When $m=0$, we sometimes write ${\rm Diam}_{n}(U)$ and ${\rm Diam}_{n}(U;V)$ for simplicity.

Throughout this section, we repeatedly use the percolation estimate in Lemma~\ref{lem:percolation-cluster}. Intuitively, it provides coarse high-probability upper bounds on cluster diameters, which supply the initial geometric control needed to obtain finer estimates for the metric.

\subsection{Chaining argument}
\label{subsec:chaining}

In this subsection, we present a chaining argument that connects each pair of points in the box $B_1(0)$ using paths of typical $D_n$-length at different scales. This argument is similar to the ones in Section 6.3 of \cite{ding-dunlap-lqg-fpp} and Section 6.1 of \cite{df-lqg-metric}.
We will consider several events: $\mathscrE_{m,n}, \mathscrG_n$, and $\mathscrF_{m,n}$, as defined below. As shown in Lemma~\ref{lem:est-e-m}, all these events occur with high probability. In Proposition~\ref{prop:chain}, we use these events to upper-bound ${\rm Diam}_n(B_1(0); B_2(0))$. In Proposition~\ref{prop:chain-2}, we upper-bound the $D_n$-distance between any two Euclidean-close points in $B_1(0)$. These results will be used in Subsection~\ref{subsec:bound-diameter} to upper-bound the diameter of a fixed box, in Subsection~\ref{subsec:prove-compare} to upper-bound the diameters of many small boxes simultaneously, and finally, in Subsection~\ref{subsec:tightness} to establish tightness.

We begin with some notation. We assume that $\xi$ satisfies $Q(\xi) > \sqrt{2d}$ and fix two constants $\eta$ and $q$ satisfying
\begin{equation}
    \label{eq:def-eta-q}
    0< \eta < \xi(Q-\sqrt{2d})  \quad \mbox{and} \quad q \in (0,1)\,.
\end{equation}
Fix an integer $n \geq 1$. 
Recall from~\eqref{eq:def-ei} that $e_i$ is the $i$-th standard basis vector in $\mathbb R^d$.
Also recall from \eqref{eq:def-rescaled-lattice} that $\mathscrL_m = 2^{-m} \mathbb{Z}^d \cap B_2(0)$. For any integer $0 \leq m \leq n - 1$, we say a vertex $x \in \mathscrL_m$ is \textbf{$m$-open} if it satisfies the following condition for all $\sigma \in \{1,-1\}$ and $1 \leq i \leq d$
\begin{equation}
\label{eq:def-m-open}
    D_n(x, x+ \sigma e_i 2^{-m}; B_{2^{-m+1}}(x)) \leq 2^{-(1-\xi Q +\eta) m } a_{n-m}^{(q)}.
\end{equation}
Otherwise, we say $x$ is \textbf{$m$-closed}. We assume all the vertices in $2^{-m}\mathbb{Z}^d \backslash \mathscrL_m$ to be $m$-open. Similar to Subsection~\ref{subsec:percolation}, we define $m$-open (or $m$-closed) ($*$-)paths and ($*$-)clusters on $2^{-m} \mathbb{Z}^d$. The role of $\eta$ in the exponent of~\eqref{eq:def-m-open} is essential for the proof of tightness; see Lemma~\ref{lem:internal-B_2(0)}. Since we expect that $a_{n-m}^{(q)} = 2^{-(1 - \xi Q + o(1))m} a_n^{(q)}$ (see Corollary~\ref{cor:an-compare} for a precise statement), $\eta$ ensures that the right-hand side of~\eqref{eq:def-m-open} is much smaller than $a_n^{(q)}$ and decays at an exponential rate in $m$. In particular, it is summable in $m$.

We now define the events $\mathscrE_{m,n}$ for $0 \leq m \leq n$, the event $\mathscrG_n$, and the events $\mathscrF_{m,n}$ for $1 \leq m \leq n-1$. On the intersection of these events, each pair of points in $B_1(0)$ can be connected by concatenating paths of typical $D_n$-length at different scales, as shown in Proposition~\ref{prop:chain}. For any integer $0 \leq m \leq  n-1$, define the event
\begin{equation}
    \label{def:e-m-n}
    \begin{aligned}
    &\mathscrE_{m,n} 
    = \mathscrE_{m,n}(\eta,q)  \\
    &\qquad:= \big{\{} \mbox{All }m\mbox{-closed }{\rm *}\mbox{-clusters on } \mathscrL_m \mbox{ have diameter at most }m^2 - 101^2 \big{\}  }\,.
    \end{aligned}
\end{equation}
Here, the diameter is with respect to the graph distance on the rescaled lattice $2^{-m} \mathbb{Z}^d$. Since the cluster diameter decays exponentially (Lemma~\ref{lem:percolation-cluster}), the exponent $2$ in $m^2$ can be replaced by any number greater than 1. We additionally subtract $101^2$ (which is an arbitrary fixed large constant) to avoid certain geometric issues. In particular, when $0 \leq m \leq 100$, $\mathscrE_{m,n}$ denotes the event that there are no $m$-closed vertices on $\mathscrL_m$. By duality, we can establish the following lemma.

\begin{lemma}
    \label{lem:e-m-connect}
    There exists a unique infinite $m$-open cluster on $2^{-m}\mathbb{Z}^d$, denoted as $\mathcal{O}_m$. Furthermore, for any $100 \leq m \leq n - 1$ and on the event $\mathscrE_{m,n}$, we have the following property: For any connected domain $U \subset \mathbb{R}^d$ and any two points contained in $U \cap \mathcal{O}_m$, they can be connected by an $m$-open path on $2^{-m} \mathbb{Z}^d$ within the domain $B_{m^2/2^{m-1}}(U)$.
\end{lemma}
\begin{proof}
   The first claim follows from the definition, as all the vertices in $2^{-m} \mathbb{Z}^d \backslash B_2(0)$ are open. We now prove the second claim. Notice that on the event $\mathscrE_{m,n}$, there does not exist an $m$-closed $*$-path on $2^{-m} \mathbb{Z}^d$ that crosses the domain $B_{m^2/2^{m-1}}(U) \backslash B_{m^2/2^{m}}(U)$ or encloses $B_{m^2/2^{m}}(U)$. Therefore, by duality, there exists a unique $m$-open cluster on $2^{-m} \mathbb{Z}^d$ that encloses $B_{m^2/2^{m}}(U)$ within the domain $B_{m^2/2^{m-1}}(U) \backslash B_{m^2/2^{m}}(U)$.

   For any point contained in $U \cap \mathcal{O}_m$, there exists an $m$-open path on $2^{-m} \mathbb{Z}^d$ that connects it to the $m$-open cluster that encloses $B_{m^2/2^{m}}(U)$. Therefore, we can connect any two points in $U \cap \mathcal{O}_m$ by first connecting them to the $m$-open cluster that encloses $B_{m^2/2^{m}}(U)$ and then connecting the endpoints inside this $m$-open cluster. The resulting $m$-open path is within the domain $B_{m^2/2^{m-1}}(U)$. 
\end{proof}

For $n \geq 1$, we define the event
\begin{equation}
    \label{def:e-n-n}
    \begin{aligned}
    \mathscrG_n =\mathscrG_n(\eta) &:= \big{\{} \sup_{x \in B_3(0) } e^{\xi h_{0,n}(x)} \leq 2^{(\xi Q - \eta)n} \big{\}} \\
    &=  \big{\{} \sup_{x \in B_3(0) } h_{0,n}(x) \leq n (Q-\eta/\xi) \log 2 \big{\}} .
    \end{aligned}
\end{equation} For $1\leq m \leq n-1$, we define the event $\mathscrF_{m,n} = \mathscrF_{m,n}(\eta,q)$ to be
\begin{equation}
    \label{eq:def-f-m-n}
    \mathscrF_{m,n} := \mathscrE_{m,n} \cap \mathscrE_{m-1,n} \cap \big{\{} \forall x \in \mathscrL_m, \, \exists y \in \mathcal{O}_m \cap \mathcal{O}_{m-1} \mbox{ such that } |x-y|_\infty \leq \frac{m^2}{2^m} \big{\}}. 
\end{equation}
See Figure~\ref{fig:eventF}. The event $\mathscrF_{m,n}$, combined with Lemma~\ref{lem:e-m-connect}, ensures that we can connect any point in $\mathcal{O}_m$ to a point in $\mathcal{O}_{m-1}$ with a path of small $D_n$-length. We will justify this argument in Proposition~\ref{prop:chain}. 

\begin{figure}[H]
\centering
\includegraphics[scale = 0.7]{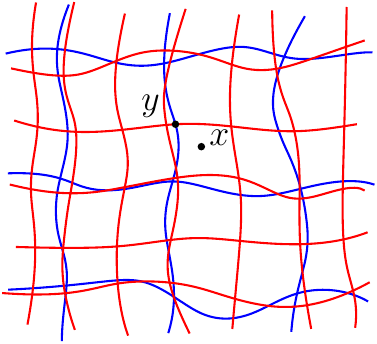}
\caption{Illustration of the event $\mathscrF_{m,n}$. On the event $\mathscrE_{m,n} \cap \mathscrE_{m-1,n}$, there are two infinite spanning clusters $\mathcal{O}_m $ and $\mathcal{O}_{m-1}$, colored red and blue, respectively. The event $\mathscrF_{m,n}$ requires that for each $x \in \mathscrL_m$, there exists $y \in \mathcal{O}_m \cap \mathcal{O}_{m-1}$ that is close to $x$.}
\label{fig:eventF}
\end{figure}

In the following lemma, we show that when $q$ is close to one, the events $\mathscrE_{m,n}$, $\mathscrG_n$, and $\mathscrF_{m,n}$ occur with high probability. The proof uses the percolation argument from Subsection~\ref{subsec:percolation}.

\begin{lemma}
    \label{lem:est-e-m}
    Fix any $\eta$ satisfying \eqref{eq:def-eta-q}. 
    \begin{enumerate}
        \item There exist constants $c_4 = c_4(\eta) \in (0,1)$ and $C = C(\eta) >0$ such that for all $q \in (c_4,1)$:
        \begin{equation}
        \label{eq:bound-e-m}
        \begin{aligned}
            &\mathbb{P}[\mathscrG_{n}] \geq 1-Ce^{-n/C} \quad \forall n \geq 1\,, \quad \quad \mathbb{P}[\mathscrE_{m,n}] \geq 1-Ce^{-m/C} \quad \forall 0 \leq m \leq n-1\,,\\
            &\mbox{and} \quad \mathbb{P}[\mathscrF_{m,n}] \geq 1-Ce^{-m/C} \quad \forall 1 \leq m \leq n-1\,.
        \end{aligned}
        \end{equation}
        \item For any $\epsilon>0$, there exists $c_5 = c_5(\eta,\epsilon) \in (0,1)$ such that for all $q \in (c_5,1)$:
        \begin{equation*}
            \mathbb{P}[\mathscrE_{m,n}] \geq 1-\epsilon \quad \forall 0 \leq m \leq n-1\,, \quad \mbox{and} \quad \mathbb{P}[\mathscrF_{m,n}] \geq 1-\epsilon \quad \forall 1 \leq m \leq n-1\,.
        \end{equation*}
    \end{enumerate}
\end{lemma}
\begin{proof}
    We begin with the first inequality in \eqref{eq:bound-e-m}. Using Claims (3) and (4) in Lemma~\ref{lem:field-estimate} and the fact that $Q-\eta/\xi>\sqrt{2d}$, we obtain 
    \begin{equation}
    \label{eq:est-mathcal-a-m}
        \mathbb{P}[\mathscrG_n] \geq 1-Ce^{-n/C} \quad \forall n \geq 1\,.
    \end{equation}

    Next, we prove the second inequality in \eqref{eq:bound-e-m}. Fix an integer $n \geq 1$, and let $0 \leq m \leq n-1$. We say a vertex $x \in \mathscrL_m$ is \textbf{$m$-good} if it satisfies
    \begin{equation}
    \label{eq:def-hat-m-open}
         D_{m,n}(x, x+ \sigma e_i 2^{-m}; B_{2^{-m+1}}(x)) \leq 2^{-m} a_{n-m}^{(q)}  \quad \forall \sigma \in \{1,-1\}\mbox{ and }1 \leq i \leq d\,.
    \end{equation}
    Otherwise, we say $x$ is \textbf{$m$-bad}. We assume all the vertices in $2^{-m} \mathbb{Z}^d \backslash \mathscrL_m$ to be $m$-good. Similarly to $\mathscrE_{m,n}$, we define $\widehat{\mathscr{E}}_{m,n}$ by replacing $m$-closed in \eqref{def:e-m-n} with $m$-bad. Since $h_n = h_m + h_{m,n}$, we get that for all $x \in \mathscrL_m$, $\sigma \in \{-1,1\}$, and $1 \leq i \leq d$:
    $$
    D_n(x, x+ \sigma e_i 2^{-m}; B_{2^{-m+1}}(x)) \leq e^{\xi \sup_{y \in B_3(0)} h_m(y)}  D_{m,n}(x, x+ \sigma e_i 2^{-m}; B_{2^{-m+1}}(x))\,.
    $$
    So, on the event $\mathscrG_m$, a vertex $x$ being $m$-good implies that it is also $m$-open. Therefore,
    \begin{equation}
    \label{eq:e-m-include}
        \mathscrE_{m,n} \supset \mathscrG_m \cap \widehat{\mathscr{E}}_{m,n}\,.
    \end{equation}
    
    We now lower-bound $\mathbb{P}[ \widehat{\mathscr{E}}_{m,n}]$ using the percolation result from Lemma~\ref{lem:percolation-cluster}. The proof follows verbatim that of \eqref{eq:est-e3}. According to the definition~\eqref{eq:def-hat-m-open} and Lemma~\ref{lem:dist-independent}, for two subsets $U,V \subset 2^{-m} \mathbb{Z}^d$ with graph distance at least $2\rr+4$ away, the statuses of the vertices in $U$ being $m$-good or $m$-bad are independent of the statuses of those within $V$. Therefore, $\mathbb{P}$ induces an $M$-dependent measure on $\{0,1\}^{2^{-m} \mathbb{Z}^d}$ (where $0$ represents $m$-bad and $1$ represents $m$-good) with $M = \lfloor 2\rr+4 \rfloor  + 1$. Similar to \eqref{eq:lem3.5-est-open}, using the translation and rotational invariance, and the scaling relation from Lemma~\ref{lem:basic-LFPP} and the definition of $a_n^{(p)}$, we can show that
    \begin{equation}
    \label{eq:est-hat-m-open}
        \mathbb{P}[x \mbox{ is }m\mbox{-good}] \geq 1-2d(1-q) \quad \mbox{for all $x \in \mathscrL_m$.}
    \end{equation}
    Hence, by Lemma~\ref{lem:percolation-cluster}, there exists $c_4' \in (0,1)$ such that for any $q>c_4'$
    \begin{equation}
    \label{eq:est-hat-e-m}
    \mathbb{P}[\widehat{\mathscr{E}}_{m,n}] \geq 1- C2^{md} e^{-m^2/C} \geq 1-Ce^{-m/C}.
    \end{equation}
    Combining \eqref{eq:e-m-include} with the estimates \eqref{eq:est-mathcal-a-m} and \eqref{eq:est-hat-e-m}, we obtain the inequality for $\mathbb{P}[\mathscrE_{m,n}]$ in \eqref{eq:bound-e-m} when $c_4 \geq c_4'$.

    We now prove the lower bound for $\mathbb{P}[\mathscrF_{m,n}]$. The proof is similar to that of $\mathscrE_{m,n}$. Let us assume that $m>200$; otherwise, we can enlarge the value of $C$. Recall the definitions of $\mathscrL_m$ from \eqref{eq:def-rescaled-lattice} and of $m$-good vertex from \eqref{eq:def-hat-m-open}. We call a vertex $x \in \mathscrL_{m-1}$ \textbf{$m$-nice} if $x$ is $(m-1)$-good and every vertex $z \in B_{2^{-m + 1}}(x) \cap \mathscrL_m$ is $m$-good. Otherwise, we say it to be \textbf{$m$-rough}. We assume all the vertices in $2^{-(m-1)} \mathbb{Z}^d \backslash \mathscrL_{m-1}$ to be $m$-nice. Similar to before, define the event
    \begin{equation*}
    \widetilde{\mathscr{E}}_{m,n} := \{ \mbox{All }m\mbox{-rough }{\rm *}\mbox{-clusters on } \mathscrL_{m-1} \mbox{ have diameter at most }m^2/2 - 101^2 \}\,,
    \end{equation*}
    where the diameter is defined associated with the graph distance on $2^{-(m-1)}\mathbb{Z}^d$. We now establish the following relation:
    \begin{equation}
    \label{eq:f-m-n-include}
        \mathscrF_{m,n} \supset \mathscrE_{m,n} \cap \mathscrE_{m-1,n} \cap \mathscrG_m \cap \mathscrG_{m-1} \cap \widetilde {\mathscr{E}}_{m,n}\,.
    \end{equation}
    Assume that all the events on the right-hand side happen. By definition, there exists a unique infinite $m$-nice cluster on $2^{-(m-1)} \mathbb{Z}^d$, denoted as $\widetilde{\mathcal{O}}_m$. We first show that $\widetilde{\mathcal{O}}_m \subset \mathcal{O}_m \cap \mathcal{O}_{m-1}$, where $\mathcal{O}_m$ and $\mathcal{O}_{m-1}$ are defined in Lemma~\ref{lem:e-m-connect}. By the definition of $m$-nice and $\mathscrG_{m-1}$, every vertex on $\widetilde{\mathcal{O}}_m$ is both $(m-1)$-good and $(m-1)$-open, which implies that $\widetilde{\mathcal{O}}_m \subset \mathcal{O}_{m-1}$. Furthermore, by the definition of $m$-nice and $\mathscrG_m$, every vertex on $\mathscrL_m$ with $l^\infty$-distance at most $2^{-m}$ from $\mathscrL_{m-1} \cap \widetilde{\mathcal{O}}_m$ is both $m$-good and $m$-open. This enables us to connect every vertex on $\mathscrL_m \cap \widetilde{\mathcal{O}}_m$ to infinity via an $m$-open path on $2^{-m} \mathbb{Z}^d$, indicating that $\widetilde{\mathcal{O}}_m \subset \mathcal{O}_m$. This proves that $\widetilde{\mathcal{O}}_m \subset \mathcal{O}_m \cap \mathcal{O}_{m-1}$. By the event $\widetilde{\mathscr{E}}_{m,n}$, for each $x \in \mathscrL_m$, there does not exist an $m$-rough $*$-cluster on $\mathscrL_{m-1}$ that surrounds $B_{m^2/2^m}(x)$. This implies the existence of $y \in \widetilde{\mathcal{O}}_m$ such that $|x-y|_\infty \leq m^2/2^m$, which satisfies the condition in the last event of \eqref{eq:def-f-m-n}. Therefore, the last event in \eqref{eq:def-f-m-n} occurs, concluding the claim~\eqref{eq:f-m-n-include}.

    We now estimate $\mathbb{P}[\widetilde{\mathscr{E}}_{m,n}]$. Similar to before, by the definition of $m$-nice and Lemma~\ref{lem:dist-independent}, $\mathbb{P}$ induces an $M$-dependent measure on $\{0,1\}^{2^{-m +1} \mathbb{Z}^d}$ (where $0$ represents $m$-rough and $1$ represents $m$-nice) with $M = \lfloor 2\rr+4 \rfloor + 1$. Moreover, using \eqref{eq:est-hat-m-open}, we obtain that for all $x \in \mathscrL_{m-1}$:
    \begin{equation*}
    \begin{aligned}
        \mathbb{P}[x \mbox{ is }m\mbox{-nice}] 
        &\geq 1- \mathbb{P}[x \mbox{ is }(m-1)\mbox{-bad}] - \sum_{z \in B_{2^{-m + 1}}(x) \cap \mathscrL_m} \mathbb{P}[z \mbox{ is }m\mbox{-bad}]  \\
        &\geq 1- 2d(1+3^d)(1-q)\,.
     \end{aligned}
    \end{equation*}
    Therefore, by Lemma~\ref{lem:percolation-cluster}, there exists $c_4'' \in (0,1)$ such that for any $q>c_4''$
    \begin{equation}
    \label{eq:est-widetilde-e-m}
    \mathbb{P}[\widetilde{\mathscr{E}}_{m,n}] \geq 1- C2^{md} e^{-m^2/C} \geq 1-Ce^{-m/C}.
    \end{equation}
    By combining the relation \eqref{eq:f-m-n-include} with the estimates \eqref{eq:est-widetilde-e-m}, and the bounds for $\mathbb{P}[\mathscrE_{m,n}]$ and $\mathbb{P}[\mathscrG_{n}]$ in \eqref{eq:bound-e-m}, we obtain the lower bound for $\mathbb{P}[\mathscrF_{m,n}]$ in \eqref{eq:bound-e-m} by taking $c_4 = \max\{c_4',c_4''\}$.

    Finally, we prove the second claim. By the first claim, the second claim holds for sufficiently large $m$. We now focus on the case where $m < A$, where $A$ is a constant depending only on $\eta$ and $\epsilon$. It suffices to show the existence of $c_5 \in (0,1)$ such that for any $n \geq 1$, $m<A$, and $q>c_5$:
    \begin{equation}
    \label{eq:lem4.2-11}
        \mathbb{P}\big[\mbox{Every vertex on }\mathscrL_m \mbox{ is }m\mbox{-open} \big] \geq 1-\epsilon/2\,.
    \end{equation}
    Using the fact that for all $x \in \mathscrL_m$, $\sigma \in \{-1,1\}$ and $1 \leq i \leq d$:
    \begin{equation*}
         D_n(x, x+ \sigma e_i 2^{-m}; B_{2^{-m+1}}(x)) \leq \sup_{z \in B_3(0)} e^{\xi h_{0,m}(z)} \cdot D_{m,n}(x, x+ \sigma e_i 2^{-m}; B_{2^{-m+1}}(x))
    \end{equation*}
    and $D_{m,n}(x, x+ \sigma e_i 2^{-m}; B_{2^{-m+1}}(x)) \overset{d}{=} 2^{-m} D_{n-m}(0, e_1 ; B_2(0))$, we obtain 
    \begin{equation*}
        \begin{aligned}
            & \mathbb{P}\big[x\mbox{ is }m\mbox{-open} \big] \\
            &\qquad \geq 1- \sum_{\sigma = \pm 1, 1 \leq i \leq d} \mathbb{P}\big[D_n(x, x+ \sigma e_i 2^{-m}; B_{2^{-m+1}}(x)) > 2^{-(1-\xi Q +\eta) m } a_{n-m}^{(q)}\big]\\
            &\qquad \geq 1- 2d \cdot \mathbb{P}\big[\sup_{z \in B_3(0)} e^{\xi h_{0,m}(z)} > H 2^{(\xi Q - \eta) m }  \big]\\
            &\qquad\qquad\qquad - 2d \cdot \mathbb{P}\big[D_{n-m}(0, e_1 ; B_2(0)) > \frac{1}{H} a_{n-m}^{(q)} \big]\,,
        \end{aligned}
    \end{equation*}
    where $H$ is any positive constant. Recall from Claim (4) in Lemma~\ref{lem:field-estimate} that $\sup_{z \in B_3(0)} h_{0,m}(z) $ has a Gaussian tail. By enlarging the value of $H$ first, and then enlarging $q$, we obtain \eqref{eq:lem4.2-11}. This implies the second claim.
\end{proof}

We now show that on the events $\mathscrG_n$ and $\mathscrF_{m,n}$ for $1 \leq m \leq n-1$, we can use paths of typical $D_n$-length to connect any two points in $B_1(0)$ within $B_2(0)$, and this will provide an upper bound for ${\rm Diam}_n(B_1(0); B_2(0))$. The approach is to first connect any point within $B_1(0)$ to the nearest point on the cluster $\mathcal{O}_{n-1}$. The $D_n$-length of the path is upper-bounded by the event $\mathscrG_n$. Next, we use an inductive procedure to connect the point obtained on $\mathcal{O}_{k}$ to the nearest point on $\mathcal{O}_{k} \cap \mathcal{O}_{k-1}$ for each $n-1 \geq k \geq 101$. The event $\mathscrF_{k,n}$ ensures that the nearest point on $\mathcal{O}_{k} \cap \mathcal{O}_{k-1}$ is not too far away with respect to the $l^\infty$-distance. Lemma~\ref{lem:e-m-connect} and the definition of $k$-open vertices from \eqref{eq:def-m-open} then allow us to control the $D_n$-distance between these two points. Finally, using the definition of $100$-open vertices from \eqref{eq:def-m-open} and the event $\mathscrE_{100,n}$, we can connect the point obtained on $\mathcal{O}_{100}$ to the origin with a path of typical $D_n$-length. Since any pair of points in $B_1(0)$ can be connected to the origin using this method, we derive an upper bound on the diameter by concatenating the paths.

\begin{proposition}
    \label{prop:chain}
    There exists a constant $C_2>0$ such that for any $\eta$, $q$ satisfying \eqref{eq:def-eta-q} and integers $n \geq 1$, on the event $(\cap_{1 \leq m \leq n-1} \mathscrF_{m,n}) \cap \mathscrG_n$, we have
    \begin{equation}
    \label{eq:prop-chain}
        {\rm Diam}_n(B_1(0); B_2(0)) \leq C_2 \sum_{m = 1}^n m^{2d} 2^{-(1-\xi Q +\eta)m} a_{n-m}^{(q)}\,.
    \end{equation}
    Here, we use the convention that $a_{0}^{(p)} = 1$ for any $p \in (0,1)$.
\end{proposition}
\begin{proof}
    Fix any $\eta$ and $q$ that satisfy \eqref{eq:def-eta-q}. The constants $C$ in this proof will not depend on $\eta$ or $q$. Assume that \begin{equation*}
    \mbox{the event }(\cap_{1 \leq m \leq n-1} \mathscrF_{m,n}) \cap \mathscrG_n \mbox{ happens}\,.
    \end{equation*}By the definition of $\mathscrF_{m,n}$ from \eqref{eq:def-f-m-n}, the event $\mathscrE_{m,n}$ occurs for all $0 \leq m \leq n - 1$. Since we can connect any pair of points to the origin and then concatenate the paths, it suffices to show that
    \begin{equation}
    \label{eq:bound-x-0}
        D_n(x,0; B_2(0)) \leq \frac{1}{2}C_2 \sum_{m = 1}^n m^{2d} 2^{-(1-\xi Q +\eta)m} a_{n-m}^{(q)} \quad \forall x \in B_1(0)\,.
    \end{equation}
    
    Fix $x \in B_1(0)$. For $n \leq 200$, we can establish \eqref{eq:bound-x-0} by using the event $\mathscrG_n$ and the following bound:
    \begin{equation*}
        D_n(x,0;B_2(0)) \leq \sqrt{d} \sup_{z \in B_2(0)} e^{\xi h_{0,n}(z)}.
    \end{equation*}
    Let us assume that $n>200$, and construct a path of typical $D_n$-length from $x$ to $0$. Recall from Lemma~\ref{lem:e-m-connect} that $\mathcal{O}_k$ is the infinite $k$-open cluster on $2^{-k} \mathbb{Z}^d$. We will define a sequence of points $(x_k)$ for $n-1 \geq k \geq 100$ inductively, and all these points will satisfy:
    \begin{equation*}
        x_k \in \mathcal{O}_k \cap B_{3/2}(0) \quad \forall n-1 \geq k \geq 100\,.
    \end{equation*}
    The $D_n$-distance between $0$ and $x$ can be bounded by summing the $D_n$-distances between neighboring points in this sequence.
    
    \noindent\textit{Step 1: Connect $x$ to $x_{n-1}$.} First, we select $x_{n-1} \in \mathcal{O}_{n-1}$ such that 
    \begin{equation}
    \label{eq:prop4.3-step1-0}
        |x-x_{n-1}|_\infty \leq n^2 2^{-n+1}.
    \end{equation}
    Such a point $x_{n-1}$ exists because on the event $\mathscrE_{n-1,n}$, there does not exist an $(n-1)$-closed $*$-cluster on $\mathscrL_{n-1}$ that surrounds $B_{n^2/2^{n-1}}(x)$.
    %The event $\mathscrE_{n-1,n}$ guarantees the existence of such an $x_{n-1}$. This is because, by a similar duality argument to Lemma~\ref{lem:e-m-connect}, we can deduce from the event $\mathscrE_{n-1,n}$ that there exists a unique $(n-1)$-open cluster on $2^{-n+1} \mathbb{Z}^d$ within the domain $B_{n^2/2^{n-2}}(x) \backslash B_{n^2/2^{n-1}}(x)$ that encloses $B_{n^2/2^{n-1}}(x)$. Furthermore, this cluster can be connected to infinity via an $(n-1)$-open path on $2^{-n+1} \mathbb{Z}^d$, and thus it is a subset of $\mathcal{O}_{n-1}$. 
    Then we connect $x$ and $x_{n-1}$ by a straight line. By the event $\mathscrG_n$, we have
    \begin{equation}
    \label{eq:prop4.3-step1}
        D_n(x,x_{n-1};B_2(0)) \leq \sqrt{d} |x - x_{n-1}|_\infty \sup_{z \in B_2(0)} e^{\xi h_{0,n}(z)} \leq Cn^{2} 2^{-(1-\xi Q +\eta)n} a_{0}^{(q)}.
    \end{equation}
    
    \noindent\textit{Step 2: Connect $x_k$ to $x_{k-1}$ inductively for $n-1 \geq k \geq 101$.}
    Suppose that $x_k \in \mathcal{O}_k$ has been defined and $k \geq 101$. We now choose $x_{k-1}$ and control the $D_n$-distance between $x_k$ and $x_{k-1}$. By the event $\mathscrF_{k,n}$ defined in \eqref{eq:def-f-m-n}, we can choose 
    \begin{equation}
    \label{eq:prop4.3-step2-0}
        x_{k-1} \in \mathcal{O}_k \cap \mathcal{O}_{k-1} \quad \mbox{such that} \quad |x_{k-1} - x_k|_\infty \leq k^2/2^k.
    \end{equation}
    By Lemma~\ref{lem:e-m-connect}, applied with $m=k$ and $U = B_{2k^2/2^k}(x_k)$, there exists a $k$-open path on $2^{-k}\mathbb{Z}^d$ contained in the domain $B_{k^2/2^{k-1}}(U)$ that connects $x_k$ and $x_{k-1}$, which will be denoted as $x_k = z_1,z_2,\ldots,z_K = x_{k-1}$. Since this path is contained in $B_{k^2/2^{k-1}}(U)$, we have $K \leq Ck^{2d}$. Combined with the definition of $k$-open from \eqref{eq:def-m-open}, we obtain
    \begin{equation}
    \label{eq:prop4.3-step2}
        D_n(x_k,x_{k-1};B_2(0)) \leq \sum_{i=1}^{K-1} D_n(z_i,z_{i+1}; B_{2^{-k+1}}(z_i)) \leq Ck^{2d} 2^{-(1-\xi Q +\eta)k} a_{n-k}^{(q)} \,.
    \end{equation}
    
    \noindent\textit{Step 3: Connect $x_{100}$ to $0$.} We now control the $D_n$-distance from $x_{100}$ to $0$ using the event $\mathscrE_{100,n}$. By \eqref{eq:prop4.3-step1-0} and \eqref{eq:prop4.3-step2-0}, we have $x_{100} \in B_{3/2}(0)$. On the event $\mathscrE_{100,n}$, all the vertices in $\mathscrL_{100}$ are $100$-open. Therefore, there exists a $100$-open path that connects $x_{100}$ to $0$ on $2^{-100}\mathbb{Z}^d \cap B_{3/2}(0)$. By the definition of $100$-open from \eqref{eq:def-m-open}, we obtain
    \begin{equation}
    \label{eq:prop4.3-step3}
        D_n(x_{100},0;B_2(0)) \leq C a_{n-100}^{(q)}.
    \end{equation}

    Combining \eqref{eq:prop4.3-step1}, \eqref{eq:prop4.3-step2}, and \eqref{eq:prop4.3-step3}, we obtain 
    \begin{equation*}
    \begin{aligned}
        &D_n(x,0;B_2(0)) \\
        &\qquad \leq D_n(x,x_{n-1};B_2(0)) +\sum_{k=101}^{n-1}D_n(x_k,x_{k-1};B_2(0)) + D_n(x_{100},0;B_2(0)) \\
        &\qquad \leq C \sum_{m = 1}^n m^{2d} 2^{-(1-\xi Q +\eta)m} a_{n-m}^{(q)}.
    \end{aligned}
    \end{equation*}
    This proves \eqref{eq:bound-x-0} and thus yields the proposition.
\end{proof}

\begin{remark}
\label{rmk:chain}
    As a direct consequence of Proposition~\ref{prop:exponent}, we know that the right-hand side of \eqref{eq:prop-chain} is at most $2^{-(1-\xi Q)n +o(n)}$. Using this fact and the estimates from Lemma~\ref{lem:est-e-m}, we will establish in Proposition~\ref{prop:diameter} that the median of the diameter of a box corresponds to this order. Nevertheless, to demonstrate the tightness of the metric $D_n$ when normalized by $a_n^{(q)}$, we must upper-bound the right-hand side of \eqref{eq:prop-chain} by $a_n^{(q)}$, with at most a constant-order multiplicative error. Therefore, a more accurate comparison between $a_n^{(q)}$ and $a_{n-m}^{(q)}$ than the one from Proposition~\ref{prop:exponent} is necessary. This will be the main focus of Section~\ref{sec:compare}. Such a comparison is provided in Corollary~\ref{cor:an-compare}. 
\end{remark}

We also establish a variant of the above proposition, which provides an upper bound on the distance between any two Euclidean-close points.
\begin{proposition}
    \label{prop:chain-2}
    There exists a constant $C_3>0$ such that for any $\eta$, $q$ satisfying \eqref{eq:def-eta-q} and integers $n \geq r >100$, on the event $(\cap_{r \leq m \leq n -1} \mathscrF_{m,n}) \cap \mathscrG_n$, we have
    \begin{equation*}
        \sup_{\substack{x,y \in B_1(0)\\ |x-y|_\infty \leq 2^{-r}}} D_n(x,y;B_2(0)) \leq C_3 \sum_{m = r}^n m^{2d} 2^{-(1-\xi Q +\eta)m} a_{n-m}^{(q)}.
    \end{equation*}
\end{proposition}
\begin{proof}
    Fix $\eta$ and $q$ that satisfy \eqref{eq:def-eta-q}. The constants $C$ in this proof will not depend on $\eta$ or $q$. Let $n \geq r >100$, and let $x,y \in B_1(0)$ such that $|x-y|_\infty \leq 2^{-r}$. We construct the sequence of points $(x_k)$ and $(y_k)$ as in Proposition~\ref{prop:chain} such that 
    \begin{equation*}
        x_k, y_k \in \mathcal{O}_k \cap B_{3/2}(0) \quad \mbox{for all } r \leq k \leq n-1\,.
    \end{equation*}
    By \eqref{eq:prop4.3-step1-0} and \eqref{eq:prop4.3-step2-0}, we obtain
    \begin{equation*}
    \begin{aligned}
        |x_r - x|_\infty &\leq |x - x_{n-1}|_\infty +\sum_{m=r+1}^{n-1} |x_m -x_{m-1}|_\infty \\
        &\leq n^22^{-n+1} + \sum_{m=r+1}^{n-1} m^2 2^{-m} \leq Cr^2 2^{-r}.
    \end{aligned}
    \end{equation*}
    Similarly, we have $|y_r - y|_\infty \leq Cr^2 2^{-r}$. Therefore,
    \begin{equation*}
        |x_r - y_r|_\infty \leq |x_r-x|_\infty+|x-y|_\infty + |y_r - y|_\infty  \leq Cr^2 2^{-r}.
    \end{equation*}
    Note that $x_r, y_r$ are both in $\mathcal{O}_r$. Similar to the derivation of \eqref{eq:prop4.3-step2}, we can use the event $\mathscrE_{r,n}$ and apply Lemma~\ref{lem:e-m-connect} with $U = B_{Cr^22^{-r}}(x_r)$ and $r$ in place of $m$ to obtain 
    \begin{equation*}
        D_n(x_r,y_r; B_2(0)) \leq C r^{2d} 2^{-(1-\xi Q +\eta)r} a_{n-r}^{(q)} \,.
    \end{equation*}
    Combining this with \eqref{eq:prop4.3-step1} and \eqref{eq:prop4.3-step2} yields
    \begin{equation*}
    \begin{aligned}
        D_n(x,y;B_2(0)) &\leq  D_n(x,x_{n-1};B_2(0)) + D_n(y,y_{n-1}; B_2(0))   + D_n(x_r,y_r;B_2(0))\\
        &\quad + \sum_{k=r+1}^{n-1} \Big(D_n(x_k,x_{k-1};B_2(0))+ D_n(y_k,y_{k-1};B_2(0)) \Big)\\
        &\leq C \sum_{m = r}^n m^{2d} 2^{-(1-\xi Q +\eta)m} a_{n-m}^{(q)} \,.
    \end{aligned}
    \end{equation*}
    This holds for any pair of $x,y$, and thus gives the desired proposition.
\end{proof}

\subsection{Bounds for the box diameter}
\label{subsec:bound-diameter}

In this subsection, we study the box diameter. Lemma~\ref{lem:diamater-rough-concentration} provides a Gaussian concentration inequality for the diameter of a general set. In Proposition~\ref{prop:diameter}, we use the results from Subsection~\ref{subsec:chaining} to show that the median of the diameter of a box decays at a rate given by $Q$.

\begin{lemma}
   \label{lem:diamater-rough-concentration}
   Fix a connected open set $V$ and a bounded connected open set $U$ such that $\overline{U} \subset V$. For all integers $n>m \geq 0$, the following concentration bound holds:
   \begin{equation*}
       \mathbb{P} \big[|\log {\rm Diam}_{m,n}(U;V) - \mathbb{E}\log {\rm Diam}_{m,n}(U;V)| >t \big] \leq 2e^{-\frac{t^2}{2\xi^2 \log 2 \cdot (n-m)}} \quad \forall t>0\,.
   \end{equation*}
\end{lemma}
\begin{proof}
    The proof follows verbatim that of Lemma~\ref{lem:gauss-concentration}. Hence, we omit it here.
\end{proof}

We now use results from Subsection~\ref{subsec:chaining} to show that when $Q(\xi)>\sqrt{2d}$, the median of the diameter ${\rm Med}({\rm Diam}_n (B_1(0);B_2(0)))$ satisfies the same decay rate as $\lambda_n$ (recall from Proposition~\ref{prop:exponent}), up to $o(n)$ errors in the exponent. This result, together with the above concentration inequality, implies that ${\rm Diam}_n (B_1(0);B_2(0)) = 2^{-(1-\xi Q)n + o(n)}$ with high probability.

\begin{proposition}
    \label{prop:diameter}
    For any $\xi>0$ with $Q(\xi)>\sqrt{2d}$, we have
    \begin{equation*}
        {\rm Med}({\rm Diam}_n (B_1(0);B_2(0))) = 2^{-(1-\xi Q)n + o(n)} \quad \mbox{as } n \rightarrow \infty\,.
    \end{equation*}
\end{proposition}
\begin{proof}
    Using the fact that ${\rm Diam}_n (B_1(0);B_2(0)) \geq D_n(0,e_1;B_2(0))$ and Proposition~\ref{prop:exponent}, we obtain
    \begin{equation}
    \label{eq:prop4.6-lower}
        {\rm Med}({\rm Diam}_n (B_1(0);B_2(0))) \geq \lambda_n = 2^{-(1-\xi Q)n +o(n)} \quad \mbox{as } n\rightarrow \infty\,.
    \end{equation}
    
    Next, we prove the upper bound. Fix any $\eta$ that satisfies \eqref{eq:def-eta-q}. Let $A>0$ and $\epsilon \in (0,1)$ be two constants to be chosen. Applying Lemma~\ref{lem:est-e-m} with the above choice of $\eta$ and $\epsilon$, we obtain that for any $q > \max\{c_4(\eta),c_5(\eta, \epsilon)\}$:
    \begin{equation}
    \label{eq:prop-diameter-prob}
    \begin{aligned}
    \mathbb{P}\Big[(\cap_{1 \leq m \leq n -1} \mathscrF_{m,n}) \cap \mathscrG_n \Big] &\geq 1 -  \mathbb{P}[\mathscrG_n^c] -\sum_{1\leq m \leq A} \mathbb{P}[\mathscrF_{m,n}^c] - \sum_{A < m \leq n-1 }\mathbb{P}[\mathscrF_{m,n}^c] \\
    &\geq 1- Ce^{-n/C} - A \epsilon - \sum_{A < m \leq n-1} Ce^{-m/C}.
    \end{aligned}
    \end{equation}
    We can choose a large value for $A$ and then select a small $\epsilon$ such that the right-hand side of \eqref{eq:prop-diameter-prob} is at least $1/2$ for all sufficiently large $n$. Therefore, there exists a constant $c = c(\eta) \in (0,1)$ such that for any $q > c(\eta)$ and all sufficiently large $n$:
    \begin{equation*}
    \mathbb{P}\Big[(\cap_{1 \leq m \leq n -1} \mathscrF_{m,n}) \cap \mathscrG_n \Big] \geq \frac{1}{2}.
    \end{equation*}
    Let us fix any $q > c(\eta)$. Combining the above inequality with Proposition~\ref{prop:chain}, we obtain that for all sufficiently large $n$:
    \begin{equation*}
        {\rm Med}({\rm Diam}_n (B_1(0);B_2(0)))\leq C_2 \sum_{m = 1}^n m^{2d} 2^{-(1-\xi Q +\eta)m} a_{n-m}^{(q)}\,.
    \end{equation*}
    Using Proposition~\ref{prop:exponent} and Lemma~\ref{lem:compare-median-p}, there exists a large constant $C = C(q, \eta) > 0$ such that 
    \begin{equation*}
        a_n^{(q)} \leq C 2^{-(1-\xi Q -\eta/2)n} \quad \forall n \geq 1 \,.
    \end{equation*}
    Therefore,
    \begin{equation*}
    \begin{aligned}
        &\quad {\rm Med}({\rm Diam}_n (B_1(0);B_2(0))) \leq C \sum_{m = 1}^n m^{2d} 2^{-(1-\xi Q +\eta)m} 2^{-(1-\xi Q -\eta/2)(n-m)} \\
        &= C 2^{-(1-\xi Q-\eta/2)n} \sum_{m = 1}^n m^{2d} 2^{-\eta m/2} \leq C 2^{-(1-\xi Q-\eta/2)n}.
    \end{aligned}
    \end{equation*}
    Since this holds for any $\eta \in (0, \xi(Q-\sqrt{2d}))$, we obtain 
    \begin{equation*}
        {\rm Med}({\rm Diam}_n (B_1(0);B_2(0))) \leq 2^{-(1-\xi Q )n + o(n)} \quad \mbox{as} \quad n\rightarrow \infty .
    \end{equation*}
    Combining this with \eqref{eq:prop4.6-lower} yields the desired result.
\end{proof}

\subsection{Bounds for distances across hypercubic shells}
\label{subsec:cross}

In this subsection, we study the distance across a hypercubic shell, which will be defined below. Proposition~\ref{prop:box-cross} will show that its median satisfies the same decay rate as $\lambda_n$, given in terms of $Q$, up to $o(n)$ errors in the exponent.

For integers $n \geq m \geq 0$ and a hypercubic shell $A= B_{r_1}(x)\backslash B_{r_2}(x)$ where $x \in \mathbb{R}^d$ and $r_1>r_2>0$, we define the distance across $A$ as:
\begin{equation}
\label{eq:def-box-cross}
    D_{m,n}(\mbox{across }A):= D_{m,n}(\partial B_{r_1}(x),\partial B_{r_2}(x))\,.
\end{equation}
Note that this distance only depends on the internal metric $D_{m,n}(\cdot,\cdot;B_{r_1}(x))$. When $m=0$, we sometimes abbreviate it as $D_{n}(\mbox{across }A)$. We also know that this distance satisfies the concentration bound in Lemma~\ref{lem:gauss-concentration} with $(K_1, K_2) = (\partial B_{r_1}(x), \partial B_{r_2}(x))$.

%We now show that the median of $D_n(\mbox{across }B_2(0) \backslash B_1(0))$ shares the same asymptotic exponent as $\lambda_n$ when $Q(\xi) > \sqrt{2d}$.
\begin{proposition}
    \label{prop:box-cross}
    For any $\xi>0$ with $Q(\xi)>\sqrt{2d}$, we have
    \begin{equation*}
        {\rm Med}(D_n(\mbox{across }B_2(0) \backslash B_1(0))) = 2^{-(1-\xi Q)n + o(n)} \quad \mbox{as } n \rightarrow \infty\,.
    \end{equation*}
\end{proposition}
\begin{proof}
    The upper bound follows directly from Lemma~\ref{lem:distance-any-point} and the fact that $D_n(\mbox{across }B_2(0) \backslash B_1(0)) \leq D_n(e_1,2e_1) \leq D_n(e_1,2e_1; B_3(0))$.

    Next, we prove the lower bound. Fix $\epsilon>0$, which will eventually tend to zero. Let the integer $n \geq 1$ and $m = \lfloor \epsilon n \rfloor$. Define the sets
    \begin{equation}
    \label{eq:prop-box-cross-def-cover}
        \{x_1, \ldots, x_{J}\} = 2^{-m + 1} \mathbb{Z}^d \cap \partial B_2(0) \quad \mbox{and} \quad \{y_1, \ldots, y_{K}\} = 2^{-m + 1} \mathbb{Z}^d \cap \partial B_1(0)\,.
    \end{equation}Then, these points satisfy the following three conditions:
    \begin{enumerate}
        \item \label{prop4.8-condition1} $x_i \in \partial B_2(0)$ for any $1 \leq i \leq J$, and $y_i \in \partial B_1(0)$ for any $1 \leq i \leq K$.
        \item \label{prop4.8-condition2}$\partial B_2(0) \subset \cup_{1 \leq i \leq J} \overline{B_{2^{-m}}(x_i)}$ and $\partial B_1(0) \subset \cup_{1 \leq i \leq K} \overline{B_{2^{-m}}(y_i)}$; see Figure~\ref{fig:5} for an illustration.
        \item \label{prop4.8-condition3}$J < C2^{(d-1)m}$ and $K < C2^{(d-1)m}$ for some constant $C$ that is independent of $\epsilon$ and $n$.
    \end{enumerate} 

    \begin{figure}[H]
\centering
\includegraphics[scale=0.6]{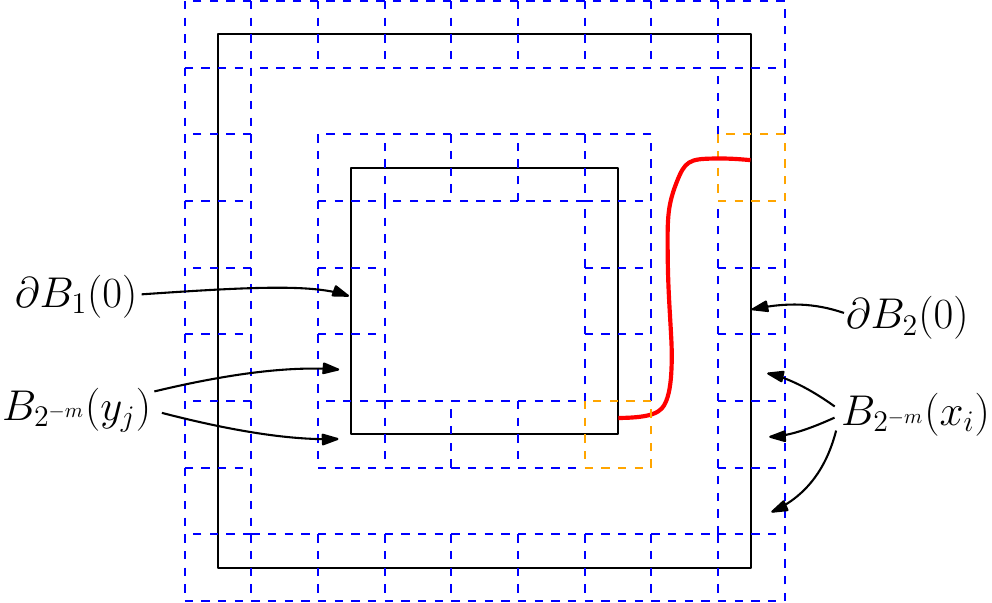}
\caption{Illustration of the boxes $B_{2^{-m}}(x_i)$ for $1 \leq i \leq J$, and $B_{2^{-m}}(y_j)$ for $1 \leq j \leq K$. Any path connecting $\partial B_1(0)$ and $\partial B_2(0)$, highlighted in red, must connect the boxes $B_{2^{-m}}(x_i)$ and $B_{2^{-m}}(y_j)$ for some values of $i$ and $j$. These two boxes are highlighted with orange dashed lines.}
\label{fig:5}
\end{figure}

    By Lemma~\ref{lem:distance-any-point} and condition~\ref{prop4.8-condition1}, we obtain that for all $1 \leq i \leq J$ and $1 \leq j \leq K$:
    \begin{equation}
    \label{eq:prop4.8-1}
        {\rm Med}(D_n(x_i, y_j; B_3(0))) = 2^{-(1-\xi Q)n+o(n)} \quad \mbox{as }n \rightarrow \infty\,.
    \end{equation}
    In particular, the $o(n)$ term does not depend on $i$ and $j$. 
    
    We claim that
    \begin{equation}
    \label{eq:prop4.8-2}
       {\rm Med}({\rm Diam}_n(B_{2^{-m}}(0); B_{2^{-m+1}}(0))) = 2^{-(1-\xi Q)(1-\epsilon)n-\epsilon n + o(n)}.
    \end{equation}
    This follows from Proposition~\ref{prop:diameter} and the scaling relation from Lemma~\ref{lem:basic-LFPP}, as we now elaborate. Since $h_n = h_m + h_{m,n}$, we have
    \begin{equation*}
    \begin{aligned}
        \inf_{z \in B_{2^{-m+1}}(0)}e^{ \xi h_{0,m}(z)}\leq \frac{{\rm Diam}_n(B_{2^{-m}}(0); B_{2^{-m+1}}(0))}{{\rm Diam}_{m,n}(B_{2^{-m}}(0); B_{2^{-m+1}}(0))}\leq \sup_{z \in B_{2^{-m+1}}(0)}e^{ \xi h_{0,m}(z)}.
    \end{aligned}
    \end{equation*}
    Using Claim (5) from Lemma~\ref{lem:field-estimate}, we see that with probability $1-o_m(1)$, the inequality $|h_{0,m}(z)| \leq C m^{2/3}$ holds for all $z \in B_{2^{-m+1}}(0)$. Combining these results with the concentration bound from Lemma~\ref{lem:diamater-rough-concentration} gives
    \begin{equation*}
    \begin{aligned}
       {\rm Med}({\rm Diam}_n(B_{2^{-m}}(0); B_{2^{-m+1}}(0))) = 2^{o(n)}{\rm Med}({\rm Diam}_{m,n}(B_{2^{-m}}(0); B_{2^{-m+1}}(0)))  \,.
    \end{aligned}
    \end{equation*}
    Using the scaling relation from Lemma~\ref{lem:basic-LFPP}, we have 
    \begin{equation*}
        {\rm Diam}_{m,n}(B_{2^{-m}}(0); B_{2^{-m+1}}(0)) \overset{d}{=} 2^{-m} {\rm Diam}_{n-m}(B_1(0); B_2(0))\,.
    \end{equation*}
    Combining the above two identities with Proposition~\ref{prop:diameter}, we obtain \eqref{eq:prop4.8-2}.
    
    By the triangle inequality, we have
    \begin{equation*}
    \begin{aligned}
        &\quad D_n(B_{2^{-m}}(x_i), B_{2^{-m}}(y_j); B_3(0)) \\
        &\geq D_n(x_i,y_j; B_3(0)) - {\rm Diam}_n(B_{2^{-m}}(x_i); B_{2^{-m+1}}(x_i)) \\
        &\quad - {\rm Diam}_n(B_{2^{-m}}(y_j); B_{2^{-m+1}}(y_j))
    \end{aligned}
    \end{equation*}
    Using \eqref{eq:prop4.8-1} and \eqref{eq:prop4.8-2}, together with the concentration bounds from Lemmas~\ref{lem:gauss-concentration} and \ref{lem:diamater-rough-concentration} for $D_n(x_i,y_j; B_3(0))$ as well as for ${\rm Diam}_n(B_{2^{-m}}(0); B_{2^{-m+1}}(0))$, respectively, we obtain
    \begin{equation*}
    \begin{aligned}
        {\rm Med}(D_n(B_{2^{-m}}(x_i), B_{2^{-m}}(y_j); B_3(0))) 
        &\geq 2^{-(1-\xi Q)n+o(n)} - 2^{-(1-\xi Q)(1-\epsilon)n-\epsilon n + o(n)} \\
        &= 2^{-(1-\xi Q)n+o(n)}.
    \end{aligned}
    \end{equation*}
    Here, we also used the assumption that $Q>0$. Applying the concentration bound from Lemma~\ref{lem:gauss-concentration} with $(K_1,K_2,U) = (B_{2^{-m}}(x_i), B_{2^{-m}}(y_j), B_3(0))$, we obtain that for some constant $A>0$, independent of $\epsilon$ and $n$, the following inequality holds for all sufficiently large $n$:
    \begin{equation*}
        \mathbb{P}\big[D_n(B_{2^{-m}}(x_i), B_{2^{-m}}(y_j); B_3(0))) \geq 2^{-(1-\xi Q)n - \epsilon^{1/3}n }\big] \geq 1-Ae^{-\epsilon^{2/3}n/A}.
    \end{equation*}
    Using the above inequality with condition~\ref{prop4.8-condition3} below \eqref{eq:prop-box-cross-def-cover} and $m = \lfloor \epsilon n \rfloor$, we obtain that for all sufficiently large $n$:
    \begin{equation*}
    \begin{aligned}
        &\quad \mathbb{P}\big[D_n(B_{2^{-m}}(x_i), B_{2^{-m}}(y_j); B_3(0))) \geq 2^{-(1-\xi Q)n - \epsilon^{1/3}n } \mbox{ } \forall 1 \leq i \leq J, 1 \leq j \leq K\big] \\
        &\geq 1-J \times K \times A e^{-\epsilon^{2/3}n/A} \geq 1 - C e^{2\epsilon(d-1) \log 2 \cdot n - \epsilon^{2/3}n/A}.
    \end{aligned}
    \end{equation*}
    Assume that $2\epsilon  (d-1) \log 2< \epsilon^{2/3} / (2A)$, which holds for all sufficiently small $\epsilon$. Then, we have
    \begin{equation}
    \label{eq:prop4.8-3}
    \begin{aligned}
     &\mathbb{P}\big[D_n(B_{2^{-m}}(x_i), B_{2^{-m}}(y_j); B_3(0))) \geq 2^{-(1-\xi Q)n - \epsilon^{1/3}n } \mbox{ } \forall 1 \leq i \leq J, 1 \leq j \leq K\big] \\
     &\qquad\qquad\geq 1-Ce^{-\epsilon^{2/3} n/C} \,.
     \end{aligned}
    \end{equation}
    
    By condition~\ref{prop4.8-condition2} below \eqref{eq:prop-box-cross-def-cover}, we know that any path that crosses $B_2(0)\backslash B_1(0)$ must connect the boxes $B_{2^{-m}}(x_i)$ and $B_{2^{-m}}(y_j)$ for some pair of $i$ and $j$, as depicted in Figure~\ref{fig:5}. Therefore,
    \begin{equation*}
        D_n(\mbox{across }B_2(0) \backslash B_1(0))  \geq \min_{1 \leq i \leq  J, 1 \leq j \leq K} D_n(B_{2^{-m}}(x_i), B_{2^{-m}}(y_j); B_3(0)))\,.
    \end{equation*}
    Combining this with \eqref{eq:prop4.8-3}, we obtain that for all sufficiently large $n$:
    \begin{equation*}
        \mathbb{P}\big[ D_n(\mbox{across }B_2(0) \backslash B_1(0)) \geq 2^{-(1-\xi Q)n - \epsilon^{1/3}n } \big] \geq 1-Ce^{-\epsilon^{2/3}n/C}.
    \end{equation*}
    Therefore, for all sufficiently large $n$:
    \begin{equation*}
        {\rm Med}(D_n(\mbox{across }B_2(0) \backslash B_1(0))) \geq 2^{-(1-\xi Q)n - \epsilon^{1/3}n}.
    \end{equation*}
    Since this inequality holds for any small enough $\epsilon$, we get the corresponding lower bound. This proves the proposition. \qedhere

\end{proof}

\subsection{Concentration bounds for distances across and around hypercubic shells}
\label{subsec:super-exponential}

In this subsection, we establish super-exponential concentration bounds for the distances across and around hypercubic shells in Lemmas~\ref{lem:super-exponential-1} and \ref{lem:super-exponential-2}. The distance around a hypercubic shell will be defined just below.

For integers $n \geq m \geq 0$ and a hypercubic shell $A = B_{r_1}(x) \backslash B_{r_2}(x)$ where $x \in \mathbb{R}^d$ and $r_1 > r_2 > 0$, the distance around $A$ is defined as\footnote{Our definition of the distance around a hypercubic shell is a natural generalization of the distance around an annulus in two dimensions, which is defined as the minimal length of a path that disconnects the inner and outer boundaries of the annulus (see e.g.\ \cite[Definition 2.1]{dg-supercritical-lfpp}). There are two main reasons why this is a natural generalization. First, it allows us to connect any two paths that cross $A$ using a path whose $D_n$-length can be bounded by $D_{n}(\mbox{around }A)$. This is similar to the role played by the distance around an annulus in two dimensions. Second, in contrast to the point-to-point distance, we can establish super-exponential concentration bounds for $D_{n}(\mbox{around }A)$, as shown in Lemmas~\ref{lem:super-exponential-1} and \ref{lem:super-exponential-2}, which will be crucial in Section~\ref{sec:compare}. Bounds of this type (though in a stronger form) also play an important role in \cite{dddf-lfpp, dg-supercritical-lfpp}; see Section 4 of \cite{dddf-lfpp} and Proposition 2.4 of \cite{dg-supercritical-lfpp}.}
\begin{equation}
    \label{eq:def-around}
    D_{m,n}(\mbox{around }A) := \sup_B \sup_{l_1,l_2 \subset B} D_{m,n}(l_1,l_2;A \cap B)  \,,
\end{equation}
where $B$ ranges over all the boxes that contain a path crossing $A$, and $l_1$ and $l_2$ range over all piecewise continuously differentiable paths that cross $A$ and are contained in $B$.\footnote{The inclusion of $B$ in \eqref{eq:def-around} is for some technical reasons in Section~\ref{sec:compare}. In that section, we will use the distance around a box to find a detour that bypasses the box $B_{r_2}(x)$ within the prescribed domain $B_2(0)$. To ensure that the detour stays within $B_2(0)$, we introduce $B$ here and will take $B = B_2(0)$ throughout Section~\ref{sec:compare}.} See Figure~\ref{fig:sec4.4} for an illustration. The distance only depends on the internal metric $D_{m,n}(\cdot,\cdot; B_{r_1}(x))$. For $m=0$, we sometimes abbreviate it as $D_{n}(\mbox{around }A)$. This distance will be used in Section~\ref{sec:compare}.

\begin{figure}[H]
\centering
\includegraphics[scale = 0.6]{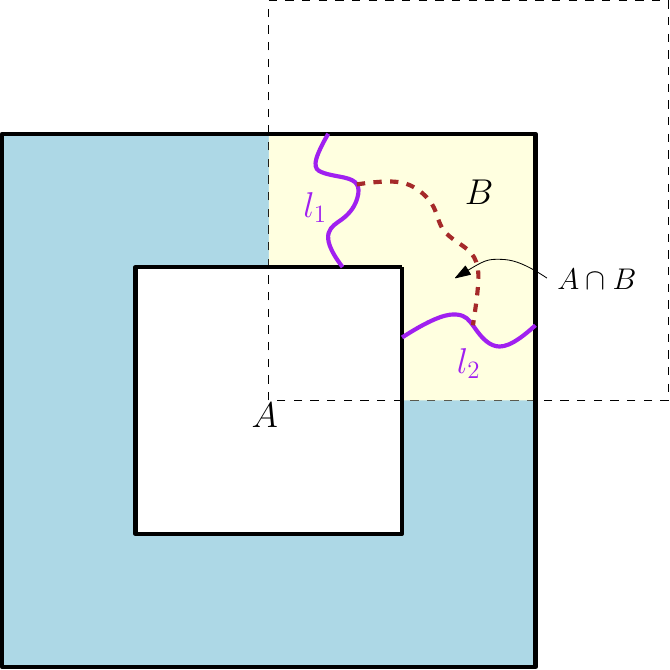}
\caption{Illustration of $D_{n}(\mbox{around }A)$. The blue domain represents $A$, and the yellow domain represents $A \cap B$. The two purple curves that cross $A$ within the domain $B$ represent $l_1$ and $l_2$. The dashed brown curve connects $l_1$ and $l_2$ within the domain $A \cap B$. The distance $D_{n}(\mbox{around }A)$ is defined by taking the minimum length of such a dashed brown curve, then the maximum over all possible choices of $l_1$, $l_2$, and the box $B$ which contains a crossing of $A$.}
\label{fig:sec4.4}
\end{figure}

In the following lemma, we show that both the distances across and around a hypercubic shell are upper-bounded by the typical order, up to a multiplicative error of $2^{\epsilon n}$, with super-exponentially high probability, for any fixed $\epsilon$. A corresponding lower bound will be established in Lemma~\ref{lem:super-exponential-2}. Note that Lemma~\ref{lem:gauss-concentration} can only give an exponential bound. To achieve a super-exponential bound, we will use a percolation argument.

\begin{lemma}
    \label{lem:super-exponential-1}
    For any $\epsilon>0$ and $r_1>r_2>0$, there exists a constant $C = C(\epsilon, r_1, r_2) >0$ such that for all integer $n \geq 1$:
    \begin{equation*}
    \begin{aligned}
        &\mathbb{P}[D_n(\mbox{across }B_{r_1}(0) \backslash B_{r_2}(0)) < 2^{-(1-\xi Q)n + \epsilon n}] > 1- Ce^{-n\log n},\\
        &\mathbb{P}[D_n(\mbox{around }B_{r_1}(0) \backslash B_{r_2}(0)) < 2^{-(1-\xi Q)n + \epsilon n}] > 1- Ce^{-n\log n}.
    \end{aligned}
    \end{equation*}
\end{lemma}
\begin{proof}
    Fix $\epsilon>0$ and $r>r_1>r_2>0$. Assume that $n > \max \{100, \frac{1}{r-r_1}, \frac{1}{r_1 - r_2} \}$ and let
    \begin{equation*}
        m = \lfloor n/(\log n )^2 \rfloor \,.
    \end{equation*} Define the event
    \begin{equation*}
        \mathcal{G}_1:=\big{\{} \sup_{x \in B_r(0)} e^{\xi h_{0,m}(x)} < 2^{\epsilon n/3} \big{\}} = \big{\{} \sup_{x \in B_r(0)} h_{0,m}(x) <\frac{\epsilon n \log 2 }{3\xi}  \big{\}}\,.
    \end{equation*}
    By Claims (3) and (4) in Lemma~\ref{lem:field-estimate}, we obtain
    \begin{equation}
    \label{eq:lem4.9-est-1}
        \mathbb{P}[\mathcal{G}_1] \geq 1- C e^{-\frac{n^2}{Cm}} \geq 1- Ce^{-n \log n}.
    \end{equation}
    Define the set
    \begin{equation}
    \label{eq:def-mathscr-L}
        \widetilde{\mathscr{L}}_m:= 2^{-m} \mathbb{Z}^d \cap B_{r_1}(0) \backslash B_{r_2}(0)\,.
    \end{equation}
    Similar to $\mathscrL_m$ defined in \eqref{eq:def-rescaled-lattice}, we consider $\widetilde{\mathscr{L}}_m$ as a subset of $\mathbb{R}^d$. Sometimes, we will consider ($*$-)paths or ($*$-)clusters on the rescaled lattice $2^{-m} \mathbb{Z}^d$, and only in these cases, we regard $\widetilde{\mathscr{L}}_m$ as a subset of $2^{-m} \mathbb{Z}^d$. We consider a vertex $x \in \widetilde{\mathscr{L}}_m$ to be \textbf{open} if it satisfies
    \begin{equation}
    \label{eq:lem4.9-def-open}
        {\rm Diam}_{m,n}(B_{2^{-m}}(x); B_{2^{-m+1}}(x)) \leq 2^{-(1-\xi Q)n + \epsilon n /3}.
    \end{equation}
    Otherwise, we say it is \textbf{closed}. We consider all the vertices in $2^{-m} \mathbb{Z}^d \backslash \widetilde{\mathscr{L}}_m$ to be open. Using Lemma~\ref{lem:basic-LFPP} and Proposition~\ref{prop:diameter}, we obtain
    \begin{equation*}
    \begin{aligned}
        {\rm Med}({\rm Diam}_{m,n}(B_{2^{-m}}(x); B_{2^{-m+1}}(x))) 
        &= 2^{-m}{\rm Med}({\rm Diam}_{n-m}(B_1(0);B_2(0))  \\
        &= 2^{-(1-\xi Q)n + o(n)}.
    \end{aligned}
    \end{equation*}
    Combining this with the concentration bound from Lemma~\ref{lem:diamater-rough-concentration} yields that for all $x \in \widetilde{\mathscr{L}}_m$:
    \begin{equation}
    \label{eq:lem4.9-p-open}
        \mathbb{P}[x \mbox{ is open} ] \geq 1-Ce^{-n/C}.
    \end{equation}
    In particular, this probability is close to one when $n$ becomes large. According to the definition~\eqref{eq:lem4.9-def-open} and Lemma~\ref{lem:dist-independent}, for two subsets $U,V \subset 2^{-m} \mathbb{Z}^d$ with graph distance at least $2\rr+4$ away, the statuses of the vertices in $U$ being open or closed are independent of the statuses of those within $V$. Hence, $\mathbb{P}$ induces an $M$-dependent probability measure on $\{0,1\}^{2^{-m} \mathbb{Z}^d}$ with $M = \lfloor 2\rr+4 \rfloor +1$. Define the event
    \begin{equation*}
        \mathcal{G}_2 := \{ \mbox{All closed }{\rm *}\mbox{-clusters on } 2^{-m} \mathbb{Z}^d \mbox{ have diameter at most }(r_1 - r_2)2^{m-1} \}\,.
    \end{equation*}
    Here, the diameter is associated with the graph distance on the rescaled lattice $2^{-m} \mathbb{Z}^d$. Using \eqref{eq:lem4.9-p-open} and Lemma~\ref{lem:percolation-cluster}, we obtain that for all sufficiently large $n$:
    \begin{equation}
    \label{eq:lem4.9-est-2}
        \mathbb{P}[\mathcal{G}_2] \geq 1-C2^{md} e^{-2^m/C} \geq 1-Ce^{-n\log n}.
    \end{equation}
    
    We now show that on the event $\mathcal{G}_1 \cap \mathcal{G}_2$, for all sufficiently large $n$:
    \begin{equation}
    \label{eq:lem4.9-upper-bound}
    \begin{aligned}
    &D_n(\mbox{across }B_{r_1}(0) \backslash B_{r_2}(0)) \leq 2^{-(1-\xi Q)n + \epsilon n} \quad \mbox{and} \\
    &D_n(\mbox{around }B_{r_1}(0) \backslash B_{r_2}(0)) \leq 2^{-(1-\xi Q)n + \epsilon n}.
    \end{aligned}
    \end{equation}
    Suppose that the event $\mathcal{G}_1 \cap \mathcal{G}_2$ happens. We first upper-bound $D_n(\mbox{across }B_{r_1}(0) \backslash B_{r_2}(0))$. By the event $\mathcal{G}_2$ and duality, there exists an open path on $2^{-m} \mathbb{Z}^d$ that crosses $B_{r_1}(0) \backslash B_{r_2}(0)$, in the sense that the end points of this path have $l^\infty$-distance at most $2^{-m}$ to $\partial B_{r_1}(0)$ and $\partial B_{r_2}(0)$, respectively. Furthermore, we can constrain this path to be within $\widetilde{\mathscr{L}}_m$. We denote this path by $(x_1,x_2,\ldots, x_J)$. Then we have $J \leq C2^{md}$ and 
    \begin{equation*}
        D_n(\mbox{across }B_{r_1}(0) \backslash B_{r_2}(0)) \leq \sum_{i=1}^{J} {\rm Diam}_n(B_{2^{-m}}(x_i); B_{2^{-m+1}}(x_i))\,.
    \end{equation*}
    Using the fact that $h_n = h_m + h_{m,n}$ and the event $\mathcal{G}_1$, we further have
    \begin{equation*}
    \begin{aligned}
        &D_n(\mbox{across }B_{r_1}(0) \backslash B_{r_2}(0))\\ &\qquad \leq \sup_{x \in B_r(0)} e^{\xi h_{0,m}(x)}\sum_{i=1}^{J} {\rm Diam}_{m,n}(B_{2^{-m}}(x_i); B_{2^{-m+1}}(x_i)) \\
        &\qquad \qquad \leq 2^{\epsilon n/3} \sum_{i=1}^{J} {\rm Diam}_{m,n}(B_{2^{-m}}(x_i); B_{2^{-m+1}}(x_i))\,.
    \end{aligned}
    \end{equation*}
    Using the definition of open vertices from \eqref{eq:lem4.9-def-open} and the fact that $J \leq C2^{md}$, we obtain that for all sufficiently large $n$:
    \begin{equation}
    \label{eq:lem4.9-upper-bound-101}
        \begin{aligned}
            D_n(\mbox{across }B_{r_1}(0) \backslash B_{r_2}(0)) \leq 2^{\epsilon n/3} \times C 2^{md} \times 2^{-(1-\xi Q)n + \epsilon n /3} \leq 2^{-(1-\xi Q)n + \epsilon n}.
        \end{aligned}
    \end{equation}
    This justifies the bound for $D_n(\mbox{across }B_{r_1}(0) \backslash B_{r_2}(0))$ in the claim~\eqref{eq:lem4.9-upper-bound}.

    We now bound $D_n(\mbox{around }B_{r_1}(0) \backslash B_{r_2}(0))$. The proof is similar to before. By the event $\mathcal{G}_2$, there exists an open cluster in $\widetilde{\mathscr{L}}_m$ that encloses the domain $B_{r_2}(0)$. Fix a box $B$ that contains a path crossing $B_{r_1}(0) \backslash B_{r_2}(0)$. For any $l_1$ and $l_2$ that cross $B_{r_1}(0) \backslash B_{r_2}(0)$ and are contained within $B$, they must intersect the boxes $B_{2^{-m}}(x_1)$ and $B_{2^{-m}}(x_2)$, respectively, for some $x_1$ and $x_2$ in this open cluster. This allows us to connect $l_1$ and $l_2$ via a discrete path on $\widetilde{\mathscr{L}}_m$ inside this open cluster. Furthermore, we can require that this discrete path is contained within $B$. Similar to \eqref{eq:lem4.9-upper-bound-101}, using the event $\mathcal{G}_1$ and the definition of open vertices from \eqref{eq:lem4.9-def-open}, we obtain that for all sufficiently large $n$: \footnote{Careful readers may worry that for some $x$ on this discrete path, the box $B_{2^{-m+1}}(x)$ may not be entirely contained within $B$. To address this issue, we can modify \eqref{eq:lem4.9-def-open} to include the condition that ${\rm Diam}_{m,n}(B_{2^{-m}}(x) \cap B; B_{2^{-m+1}}(x) \cap B) \leq 2^{-(1-\xi Q)n + \epsilon n /3}$ for any such box $B$. By adapting the arguments in Subsection~\ref{subsec:chaining} and considering percolation clusters in the half-space or smaller sections of the space, we can show that $\sup_{B}{\rm Diam}_n(B_1(0) \cap B; B_2(0) \cap B) = 2^{-n(1-\xi Q) +o(n)}$, where $B$ ranges over all boxes of side-length at least $10$ that contain 0. Thus, we can still prove \eqref{eq:lem4.9-p-open} under this modified condition. The rest of the proof remains the same.}
    \begin{equation*}
    \begin{aligned}
            D_n(l_1,l_2; B \cap B_{r_1}(0) \backslash B_{r_2}(0) ) \leq 2^{\epsilon n/3} \times C 2^{md} \times 2^{-(1-\xi Q)n + \epsilon n /3} \leq 2^{-(1-\xi Q)n + \epsilon n}.
    \end{aligned}
    \end{equation*}
    This inequality holds uniformly for any pair of paths, $l_1$ and $l_2$, and any $B$. Hence, we get the claim~\eqref{eq:lem4.9-upper-bound}. Combining this claim with the estimates \eqref{eq:lem4.9-est-1} and \eqref{eq:lem4.9-est-2}, we obtain the desired lemma.
\end{proof}

In the following lemma, we establish a lower bound for the distances across and around a hypercubic shell which holds with super-exponentially high probability. The proof uses a percolation argument and is similar to that of Lemma~\ref{lem:super-exponential-1}.

\begin{lemma}
\label{lem:super-exponential-2}
    For any $\epsilon>0$ and $r_1>r_2>0$, there exists a constant $C = C(\epsilon, r_1, r_2) >0$ such that for all integer $n \geq 1$
    \begin{equation*}
    \begin{aligned}
        &\mathbb{P}[D_n(\mbox{across }B_{r_1}(0) \backslash B_{r_2}(0)) > 2^{-(1-\xi Q)n - \epsilon n}] > 1- Ce^{-n\log n},\\
        &\mathbb{P}[D_n(\mbox{around }B_{r_1}(0) \backslash B_{r_2}(0)) > 2^{-(1-\xi Q)n - \epsilon n}] > 1- Ce^{-n\log n}.
    \end{aligned}
    \end{equation*}
\end{lemma}
\begin{proof}
    Fix $\epsilon>0$ and $r>r_1>r_2>0$. Assume that $n > \max \{100, \frac{1}{r-r_1}, \frac{1}{r_1 - r_2} \}$ and let $m = \lfloor n/(\log n )^2 \rfloor$. Define the event
    \begin{equation*}
        \mathcal{G}_3:=\big{\{} \inf_{x \in B_r(0)} e^{\xi h_{0,m}(x)} > 2^{-\epsilon n/3} \big{\}} = \big{\{} \inf_{x \in B_r(0)} h_{0,m}(x) >-  \frac{\epsilon n \log 2 }{3\xi}  \big{\}}\,.
    \end{equation*}
    By Claims (3) and (4) in Lemma~\ref{lem:field-estimate}, we obtain
    \begin{equation}
    \label{eq:lem4.10-est-3}
        \mathbb{P}[\mathcal{G}_3] \geq 1- C e^{-\frac{n^2}{Cm}} \geq 1- Ce^{-n \log n}.
    \end{equation}
    Recall from \eqref{eq:def-mathscr-L} that $\widetilde{\mathscr{L}}_m= 2^{-m} \mathbb{Z}^d \cap B_{r_1}(0) \backslash B_{r_2}(0)$. We now consider a vertex $x \in \widetilde{\mathscr{L}}_m$ to be \textbf{open} if it satisfies
    \begin{equation}
    \label{eq:lem4.10-def-open-2}
        D_{m,n}(\mbox{across }B_{2^{-m+1}}(x) \backslash B_{2^{-m}}(x)) \geq 2^{-(1-\xi Q)n - \epsilon n /3}.
    \end{equation}
    Otherwise, we say it is \textbf{closed}. Note that this definition of open vertices is different from the one in Lemma~\ref{lem:super-exponential-1}. According to Lemma~\ref{lem:dist-independent}, this induces an $M$-dependent measure on $\{ 0 ,1 \}^{2^{-m} \mathbb{Z}^d}$ with $M = \lfloor 2\rr+4 \rfloor +1$. Using Lemma~\ref{lem:basic-LFPP} and Proposition~\ref{prop:box-cross}, we obtain
    \begin{equation*}
    \begin{aligned}
        &\quad {\rm Med}(D_{m,n}(\mbox{across }B_{2^{-m+1}}(x) \backslash B_{2^{-m}}(x))) \\
        &= 2^{-m} {\rm Med}(D_{n-m}(\mbox{across }B_2(0) \backslash B_1(0)) 
        = 2^{-(1-\xi Q)n +o(n)}.
     \end{aligned}
    \end{equation*}
    Using Lemma~\ref{lem:gauss-concentration} with $(K_1,K_2) = (\partial B_{2^{-m+1}}(x), \partial B_{2^{-m}}(x))$, we obtain that for all $x \in \widetilde{\mathscr{L}}_m$:
    \begin{equation}
    \label{eq:lem4.10-open-probab}
    \begin{aligned}
         \mathbb{P}[x \mbox{ is open} ] &= \mathbb{P} \big[D_{m,n}(\mbox{across }B_{2^{-m+1}}(x) \backslash B_{2^{-m}}(x)) \geq 2^{-(1-\xi Q)n - \epsilon n /3}\big]\\
         &\geq 1-Ce^{-n/C}.
    \end{aligned}
    \end{equation}
    Define the event
    \begin{equation*}
        \mathcal{G}_4 := \{ \mbox{All closed }{\rm *}\mbox{-clusters on } 2^{-m} \mathbb{Z}^d \mbox{ have diameter at most }(r_2 - r_1)2^{m-1} \}\,.
    \end{equation*} Using \eqref{eq:lem4.10-open-probab} and Lemma~\ref{lem:percolation-cluster}, we get that for all sufficiently large $n$:
    \begin{equation}
    \label{eq:lem4.10-est-4}
        \mathbb{P}[\mathcal{G}_4] \geq 1-C2^{md} e^{-2^m/C} \geq 1-Ce^{-n\log n}.
    \end{equation}
    
    We now show that on the event $\mathcal{G}_3 \cap \mathcal{G}_4$,
    \begin{equation}
    \label{eq:lem4.10-lower-bound}
    \begin{aligned}
     &D_n(\mbox{across }B_{r_1}(0) \backslash B_{r_2}(0)) > 2^{-(1-\xi Q)n - \epsilon n} \quad \mbox{and} \\
     &D_n(\mbox{around }B_{r_1}(0) \backslash B_{r_2}(0)) > 2^{-(1-\xi Q)n - \epsilon n}.
     \end{aligned}
    \end{equation}
    Suppose that the event $\mathcal{G}_3 \cap \mathcal{G}_4$ happens. We first lower-bound $D_n(\mbox{across }B_{r_1}(0) \backslash B_{r_2}(0))$. By the event $\mathcal{G}_4$ and duality, there is an open cluster in $\widetilde{\mathscr{L}}_m$ that encloses the domain $B_{r_1}(0)$. Note that any path crossing $B_{r_2}(0)\backslash B_{r_1}(0)$ must enter the box $B_{2^{-m}}(x)$ for some $x$ in this open cluster. Therefore, using the definition of open vertices from \eqref{eq:lem4.10-def-open-2} and the event $\mathcal{G}_3$, we deduce that
    \begin{equation}
    \label{eq:lem4.10-lower-bound-101}
    \begin{aligned}
        & D_n(\mbox{across }B_{r_1}(0) \backslash B_{r_2}(0)) \geq \inf_{x \in \widetilde{\mathscr{L}}_m \mbox{ }{\rm is}\mbox{ }{\rm open}}D_{n}(\mbox{across }B_{2^{-m+1}}(x) \backslash B_{2^{-m}}(x)) \\
        &\quad\geq \inf_{z \in B_r(0)} e^{\xi h_m(z)} \times \inf_{x \in \widetilde{\mathscr{L}}_m \mbox{ }{\rm is}\mbox{ }{\rm open}} D_{m,n}(\mbox{across }B_{2^{-m+1}}(x) \backslash B_{2^{-m}}(x))\\
        &\qquad > 2^{-(1-\xi Q)n - \epsilon n}.
    \end{aligned}
    \end{equation}
    
    Next, we lower-bound $D_n(\mbox{around }B_{r_1}(0) \backslash B_{r_2}(0))$. Recall from \eqref{eq:def-ei} the definition of the point $e_1$. Let $l_1$ and $l_2$ be two straight curves connecting $r_1 e_1$ to $r_2 e_1$ and $-r_1 e_1$ to $-r_2 e_1$, respectively. By the event $\mathcal{G}_4$ and duality, there exists an open cluster in $\widetilde{\mathscr{L}}_m$ that separates these curves. Namely, any path in $B_{r_1}(0) \backslash B_{r_2}(0)$ connecting $l_1$ and $l_2$ must intersect the box $B_{2^{-m}}(x)$ for some vertex $x$ in this open cluster. Therefore, similar to \eqref{eq:lem4.10-lower-bound-101}, we obtain 
    \begin{equation*}
    \begin{aligned}
        &\quad D_n(\mbox{around }B_{r_1}(0) \backslash B_{r_2}(0)) \geq D_n(l_1,l_2; B_{r_1}(0) \backslash B_{r_2}(0)) \\
        &\geq \inf_{x \in \widetilde{\mathscr{L}}_m \mbox{ }{\rm is}\mbox{ }{\rm open}} D_{n}(\mbox{across }B_{2^{-m+1}}(x) \backslash B_{2^{-m}}(x)) > 2^{-(1-\xi Q)n - \epsilon n}.
    \end{aligned}
    \end{equation*}
    This justifies the claim~\eqref{eq:lem4.10-lower-bound}. Combining the claim~\eqref{eq:lem4.10-lower-bound} with the estimates \eqref{eq:lem4.10-est-3} and \eqref{eq:lem4.10-est-4} yields the desired lemma.
\end{proof}

\section{Comparison between different scales}
\label{sec:compare}
In this section, we continue to assume that $\xi$ satisfies $Q(\xi) > \sqrt{2d}$. Our main result is Proposition~\ref{prop:compare}, where we compare the metrics $D_n$ and $D_{n+k}$ for integers $n, k \geq 1$. As a consequence, we establish a relation between the quantiles $a_n^{(q)}$ for different values of $n$ in Corollary~\ref{cor:an-compare}.

\begin{proposition}
\label{prop:compare}
    Suppose that $\xi$ satisfies $Q(\xi)>\sqrt{2d}$. For any $\epsilon>0$, there exists a constant $c = c(\epsilon) \in (0,1)$ such that for all integers $n \geq 1$ and $1 \leq k \leq c n$, with probability greater than $1 - 2^{-c n}$, the following bound holds for all pair of points $z,w \in B_1(0)$ with $|z-w|_\infty \geq 2^{-c n }$:
    \begin{equation}
    \label{eq:prop-compare-0}
        c 2^{-(1-\xi Q + \epsilon) k } < \frac{D_{n+k}(z,w ; B_2(0))}{D_{n}(z,w; B_2(0))} < \frac{1}{c} 2^{-(1-\xi Q - \epsilon) k}.
    \end{equation}
\end{proposition}

As a consequence, we derive a comparison between $a_n^{(q)}$ for different values of $n$. As mentioned in Remark~\ref{rmk:chain}, combining this comparison with the results from Subsection~\ref{subsec:chaining} implies the tightness of the metric $D_n$ when normalized by $a_n^{(q)}$ for $q$ close to one (see Subsection~\ref{subsec:tightness} for details).

\begin{corollary}
\label{cor:an-compare}
    Given $\epsilon>0$ and $q \in (0,1)$, there exists a constant $c = c(\epsilon, q) >0$ such that for all integers $n \geq k \geq 1$:
    \begin{equation}
    \label{eq:cor-an-compare}
        c 2^{(1-\xi Q -\epsilon)k} a_n^{(q-2^{-cn})} \leq a_{n-k}^{(q)} \leq \frac{1}{c} 2^{(1-\xi Q +\epsilon)k} a_n^{(q+2^{-cn})}.
    \end{equation}
\end{corollary}
\begin{proof}
    Applying Proposition~\ref{prop:compare} with the above choice of $\epsilon, k$, and $n-k$ in place of $n$, we obtain that for all $n \geq 1$ and $1 \leq k \leq c n$:
    \begin{equation*}
        \mathbb{P}\big[ D_{n-k}(0,e_1;B_2(0)) \leq \frac{1}{c} 2^{(1-\xi Q +\epsilon) k } D_{n}(0,e_1;B_2(0)) \big] \geq 1-2^{-cn}.
    \end{equation*}
    Combining this with the definition of $a_n^{(q)}$ from \eqref{def:quantile}, we obtain that for all $n \geq 1$ and $1 \leq k \leq c n$:
    \begin{equation*}
        a_{n-k}^{(q)} \leq \frac{1}{c} 2^{(1-\xi Q +\epsilon) k } a_{n}^{(q + 2^{-cn})}.
    \end{equation*}
    We can extend this result to all $1 \leq k \leq n$ by applying Proposition~\ref{prop:exponent} and decreasing the value of $c$. This yields the second inequality in \eqref{eq:cor-an-compare}. The other one can be obtained in a similar way. 
\end{proof}

\subsection{Proof strategy}
\label{subsec:sec5-strategy}

Here we outline the proof strategy of Proposition~\ref{prop:compare} and describe the structure of Section~\ref{sec:compare}. 

Note that the metric $D_{n+k}$ is obtained from $D_n$ by adding the field $h_{n,n+k}$ to $h_n$. It is straightforward to check that $D_{n+k}$ and $D_n$ satisfy the desired relation in \eqref{eq:prop-compare-0} when the field $h_{n,n+k}$ is well-behaved in the sense that for each box of side-length $2^{-n}$, both its $D_{n,n+k}$-diameter and the $D_{n,n+k}$-distance across a hypercubic shell enclosing the box can be bounded by $\lambda_k$, up to a multiplicative error of $2^{\epsilon k}$, for some small but fixed $\epsilon>0$. For precise definitions, we refer to conditions~\ref{cond-a1-new-1}, \ref{cond-a1-new-2}, and \ref{cond-a1-new-3} in Proposition~\ref{prop:cover} below. However, there exist regions (likely even along geodesics) where $h_{n,n+k}$ does not behave well, and the primary focus of this section is to address these regions. The main idea is that with high probability, these problematic regions are sparsely distributed in the space. Therefore, we expect that the length metric $D_n$ will not increase much if we reroute the path to avoid these regions. Furthermore, the modified metrics obtained by bypassing these problematic regions, should satisfy the relation in \eqref{eq:prop-compare-0}, as the path only traverses the regions where $h_{n,n+k}$ is well-behaved.

We now describe the proof strategy in more detail. The proof consists of two steps. The first step is to show that, with high probability, we can find boxes at different scales to cover the problematic regions where the field $h_{n,n+k}$ does not behave well. We refer to Proposition~\ref{prop:cover} for a precise statement. These boxes will satisfy the condition that the $D_n$-distance around a hypercubic shell enclosing the box can be upper-bounded by the $D_n$-distance across a larger hypercubic shell; see condition~\ref{cond-new-3} of Proposition~\ref{prop:cover} and also Figure~\ref{fig:nice-box} for an illustration. This condition will be used in the second step to argue that the metric $D_n$ will increase by at most a constant factor when we reroute the path to avoid these boxes. 

The construction of these boxes is via a coarse-graining argument, presented in Subsections~\ref{subsec:coarse-grain}--\ref{subsec:cover}. In particular, we introduce the notions of nice and bad boxes in Definition~\ref{def:nice-boxes}. A key feature of this definition is that the statuses of being nice or bad for two distant boxes are independent. Using this independence property, we will show in Lemma~\ref{lem:est-qm} that the probabilities of bad boxes decay rapidly to zero. Finally, Lemma~\ref{lem:est-j12} and Proposition~\ref{prop:cover} show that, with high probability, we can find boxes at various scales to cover the problematic region, and in the remaining region the field $h_{n,n+k}$ is well-behaved.

The second step, carried out in Subsection~\ref{subsec:prove-compare}, completes the proof of Proposition~\ref{prop:compare} using the covering constructed in Proposition~\ref{prop:cover}. Let $\mathcalU$ denote the union of the boxes from Proposition~\ref{prop:cover}; see \eqref{eq:def-mathcal-U} for its precise definition. We will prove the following two claims:
\begin{itemize}
    \item For any path $P$ in $B_2(0)$, we can modify it to avoid the domain $\mathcalU$ and the $D_n$-length of the modified path can be upper-bounded the $D_n$-length of $P$ outside $\mathcalU$; see Lemma~\ref{lem:bypassU}. 
    
    \item For any path contained in $B_2(0) \backslash \mathcalU$, by adjusting the path, its $D_n$-length and $D_{n+k}$-length satisfy the desired bound in \eqref{eq:prop-compare-0}; see Lemma~\ref{lem:modify-outU} for a precise statement.
    
\end{itemize}

Combining these two statements yields Proposition~\ref{prop:compare}. Specifically, the second inequality in~\eqref{eq:prop-compare-0} is straightforward: we start with the $D_n$-geodesic $P$, modify it to bypass $\mathcalU$ and then further modify it within $B_2(0) \backslash \mathcalU$ so that the $D_{n+k}$-length of the resulting path is upper-bounded by the $D_n$-length of $P$. To prove the first inequality in~\eqref{eq:prop-compare-0}, we begin with the $D_{n+k}$-geodesic $P$ and modify it into a path $P'$ whose $D_n$-length outside $\mathcalU$ is upper-bounded by the $D_{n+k}$-length of $P$. We can only control the $D_n$-length of $P'$ outside $\mathcalU$ (which is also why we consider the $D_n$-length outside $\mathcalU$ in Lemma~\ref{lem:bypassU}), since the $D_n$-length of $P$ could be concentrated in $\mathcalU$, in which case any local modification would not work. We then apply Lemma~\ref{lem:bypassU} to further modify $P'$ so that the $D_n$-length of the resulting path is upper-bounded by the $D_n$-length of $P'$ outside $\mathcalU$, which implies the first inequality in~\eqref{eq:prop-compare-0}. We refer to the proof of Proposition~\ref{prop:compare-1} for more details.

The structure of this section is as follows: Subsection~\ref{subsec:coarse-grain} introduces the notation for the coarse-graining argument including definitions of nice and bad boxes. In Subsection~\ref{subsec:est-bad}, we estimate the probabilities of bad boxes. Subsection~\ref{subsec:cover} is devoted to the construction of the covering. Finally, in Subsection~\ref{subsec:prove-compare}, we prove Proposition~\ref{prop:compare} using the covering from Proposition~\ref{prop:cover}.

\subsection{Coarse-graining argument}
\label{subsec:coarse-grain}

In this subsection, we introduce the notation for the coarse-graining argument. We start with the notation concerning constants and the coarse-graining scale. Fix $\epsilon>0$, and let 
\begin{equation}
\label{eq:def-lambda}
    \alphaC = Q^2/4 - d/2 > 0 \,, \quad \lambdaC = \frac{\alphaC}{4(\alphaC + d)}\,, \quad \mbox{and} \quad \RC = \lfloor 100 \rr \rfloor + 100 \, ,
\end{equation}
where we recall from Subsection~\ref{subsec:intro-1} that $\rr$ is the $l^\infty$-diameter of the support of the convolution kernel in the definition of $h$.
We fix two integers to be determined later:
\begin{equation*}
    L>100\quad \mbox{and} \quad K>\max \big{\{} 100^d \RC, (16d \xi+1)^2 \big{\}}\,.
\end{equation*}
In the end, we will first choose $K$ to be large, and then select $L$ to be large. Let $n$ and $k$ be two positive integers. We will assume that $n>100$ and is sufficiently large (which may depend on $L$ and $K$). Throughout this section, the constants $c$ and $C$ are independent of both $n$ and $k$, and all the sets are considered as subsets of $\mathbb{R}^d$.

We define a sequence of integers $\{ \mathsf{a}_i \}_{i \geq 1}$ such that
\begin{equation}
\label{eq:def-ai}
    \mathsf{a}_1 = K^4, \quad \mbox{and} \quad \mathsf{a}_{i+1} = \lfloor (1+\lambdaC) \mathsf{a}_i \rfloor + 1 \qquad \forall i \geq 1\,.
\end{equation}
Define the integer
\begin{equation} \label{eq:def:mm}
    \mm := \sup \{ i \geq 1: \mathsf{a}_i \leq (1-\lambdaC) n \}\,.
\end{equation}
Then, we have
\begin{equation}
\label{eq:def-am}
    \mm \geq c\log n\,, \quad \mbox{and} \quad (1-2 \lambdaC)n \leq \mathsf{a}_\mm \leq (1-\lambdaC) n\,.
\end{equation}
The second inequality is due to the fact that when $n$ is sufficiently large, $\mathsf{a}_\mm \geq \frac{1-\lambdaC}{1+\lambdaC} n -1 \geq (1-2\lambdaC)n$. Let us define the sets
\begin{equation} \label{eq:def-mathscrY}
   \mathscrY_0 := 2^{-n} \mathbb{Z}^d, \quad \mbox{and} \quad  \mathscrY_i := 2^{-n+\mathsf{a}_i} \mathbb{Z}^d \quad \forall 1 \leq i \leq \mm\,.
\end{equation}
For $0 \leq i \leq \mm$, we consider $\mathscrY_i$ as a subset of $\mathbb{R}^d$. We call a box $B_r(x)$ in $\mathbb{R}^d$ an \textbf{$\mathsf{a}_i$-box} if $r = 2^{-n+\mathsf{a}_i}$ and $x \in \mathscrY_i$.

We will consider the decomposition of $h_{0,n}$ and $h_{0,n+k}$ as follows:
\begin{equation}
\label{eq:field-decompose-sec5}
\begin{aligned}
    &h_{0,n} = h_{0,n-\mathsf{a}_\mm} + h_{n-\mathsf{a}_\mm,n-\mathsf{a}_{\mm-1}} + \ldots + h_{n-\mathsf{a}_1,n} \,, \quad \mbox{and}\\
    &h_{0,n+k} = h_{0,n-\mathsf{a}_\mm} + h_{n-\mathsf{a}_\mm,n-\mathsf{a}_{\mm-1}} + \ldots + h_{n-\mathsf{a}_1,n} + h_{n,n+k}\,.
\end{aligned}
\end{equation}
In particular, the field $h_{n+k}$ can be seen as $h_n$ plus an additional field $h_{n,n+k}$. 

We now introduce the notation for \textbf{nice $\mathsf{a}_i$-boxes} and \textbf{bad $\mathsf{a}_i$-boxes} inductively for $1 \leq i \leq \mm$. In particular, the statuses of being nice or bad for two $\mathsf{a}_i$-boxes are independent, if they are at $|\cdot|_\infty$-distance at least $\RC \cdot 2^{-n+\mathsf{a}_i}$ away. Recall from Section~\ref{sec:bound-distance} the definitions of ${\rm Diam}_n(A), D_n(\mbox{across }A)$, and $D_n(\mbox{around }A)$.

Here is the intuition behind the definition. At the smallest coarse-graining scale, a nice box is one in which the field varies in a controlled way and the fine metric has the expected upper and lower bounds. At each subsequent scale, a box is declared nice if its field increment is regular and either all smaller boxes inside it are nice or all bad smaller boxes are confined to one localized region around which a sufficiently cheap detour is available. Thus, a nice box permits us either to use the fine-scale geometry directly or to isolate and bypass its exceptional part. One expects this recursive scheme to cover the relevant region because the smallest boxes are nice with probability arbitrarily close to one when $L$ is large, while finite-range independence makes the probability of a bad box decay rapidly with its scale; see Lemmas~\ref{lem:est-q1} and~\ref{lem:est-qm}. At the largest scale $\mathsf{a}_\mm$, which is proportional to $n$, this decay beats the number of boxes by a union bound. Descending through the scales then leaves every point either in a nice box or in one of the localized regions selected for the covering in Proposition~\ref{prop:cover}.

\begin{definition}
    \label{def:nice-boxes}
    \begin{enumerate}
        \item An $\mathsf{a}_1$-box $B_{2^{-n+\mathsf{a}_1}}(x)$ is called \textbf{nice} if it satisfies the following three conditions:
        \begin{enumerate}[(1)]
            \item\label{condition-a1-nice-a} $2^{-n} |\nabla h_{n-\mathsf{a}_1,n}(z)|_\infty \leq L$ for every $z \in B_{2^{-n+\mathsf{a}_1}}(x)$.
            \item\label{condition-a1-nice-b} ${\rm Diam}_{n,n+k} (B_{2^{-n}}(y); B_{2 \cdot 2^{-n}}(y)) < L 2^{-n - (1-\xi Q - \epsilon) k}$ for every $y \in \mathscrY_0 \cap B_{2^{-n+\mathsf{a}_1}}(x)$.
            \item\label{condition-a1-nice-c} $D_{n,n+k} (\mbox{across }B_{2 \cdot 2^{-n}}(y) \backslash B_{2^{-n}}(y)) >  \frac{1}{L}2^{-n - (1-\xi Q + \epsilon) k}$ for every $y \in \mathscrY_0 \cap B_{2^{-n+\mathsf{a}_1}}(x)$.
        \end{enumerate}
        Otherwise, it is called \textbf{bad}.
        \item For $2 \leq j \leq \mm$, given the definition of nice and bad $\mathsf{a}_{j-1}$-boxes, we say an $\mathsf{a}_j$-box $B_{2^{-n+\mathsf{a}_j}}(x)$ is \textbf{nice} if the following condition \ref{condition-nice-a} holds and one of the conditions \ref{condition-nice-b} or \ref{condition-nice-c} hold:
        \begin{enumerate}[(a)]
            \item \label{condition-nice-a} $2^{-n+\mathsf{a}_{j-1}} |\nabla h_{n-\mathsf{a}_j,n-\mathsf{a}_{j-1}}(z)|_\infty \leq K \sqrt{\mathsf{a}_{j-1}}$ for every $z \in B_{4 \cdot 2^{-n+\mathsf{a}_j}}(x)$.  
            \item \label{condition-nice-b} The $\mathsf{a}_{j-1}$-boxes contained in the $\mathsf{a}_j$-box $B_{2^{-n+\mathsf{a}_j}}(x)$ are all nice.
            \item \label{condition-nice-c} There are bad $\mathsf{a}_{j-1}$-boxes contained in $B_{2^{-n+\mathsf{a}_j}}(x)$, but there exists  $z \in \mathscrY_{j-1} \cap B_{2^{-n+\mathsf{a}_j}}(x)$ such that all of these bad $\mathsf{a}_{j-1}$-boxes are contained in the box $B_{\RC \cdot 2^{-n+\mathsf{a}_{j-1}}}(z)$. Moreover, the following inequality holds:
            \begin{equation}
            \label{eq:scale-control}
            \begin{aligned}
                &\quad D_{n-\mathsf{a}_j,n}(\mbox{across }B_{4 \cdot 2^{-n+\mathsf{a}_j}}(x) \backslash B_{2 \cdot 2^{-n+\mathsf{a}_j}}(x)) \\
                &> \exp(K^{3/2} \sqrt{\mathsf{a}_j}) D_{n-\mathsf{a}_j,n}(\mbox{around }B_{4\RC \cdot 2^{-n+\mathsf{a}_{j-1}}}(z) \backslash B_{2\RC \cdot 2^{-n+\mathsf{a}_{j-1}}}(z))\,.
            \end{aligned}
            \end{equation}
        \end{enumerate}
        Otherwise, it is called \textbf{bad}. 
    \end{enumerate}
\end{definition}

\begin{figure}[H]
\centering
\includegraphics[scale = 0.5]{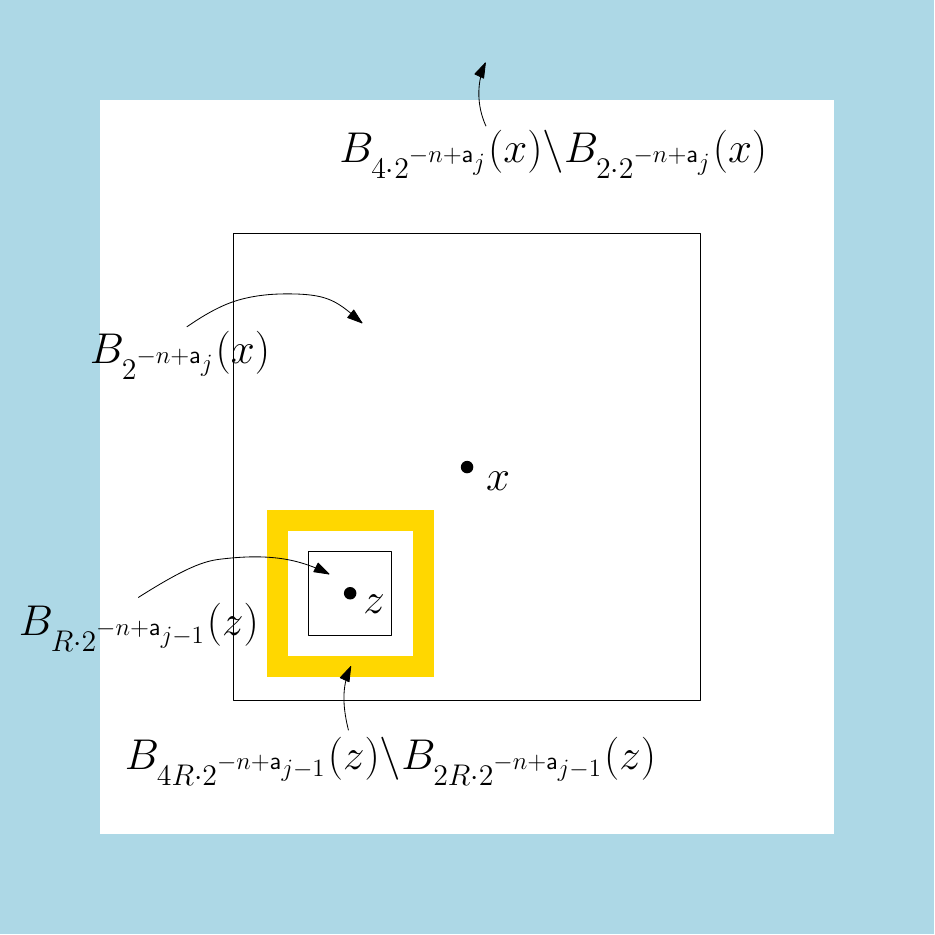}
\caption{Illustration of condition~\ref{condition-nice-c} from Definition~\ref{def:nice-boxes}. All the $a_{j-1}$-boxes contained in $B_{2^{-n + \mathsf{a}_j}}(x)$ but not in $B_{\RC \cdot 2^{-n + \mathsf{a}_{j-1}}}(z)$ are nice. Moreover, the $D_{n-\mathsf{a}_j,n}$-distance around the hypercubic shell represented by the yellow region can be upper-bounded by the $D_{n-\mathsf{a}_j,n}$-distance across the hypercubic shell represented by the blue region.}
\label{fig:nice-box}
\end{figure}

\begin{remark}[Comments on the definitions of nice boxes]
\label{rmk:nice-box}
\begin{enumerate}
\item Using the definition and an induction argument, we can show that the event that an $\mathsf{a}_i$-box $B_{2^{-n+\mathsf{a}_i}}(x)$ is nice is measurable with respect to the $\sigma$-algebra $\sigma( h_{a,b}|_{B_{4 \cdot 2^{-n+\mathsf{a}_i}}(x)} : n-\mathsf{a}_i \leq a \leq b \leq n+k)$. Together with Claim (3) in Lemma~\ref{lem:basic-h}, this shows that the statuses of being nice or bad for two $\mathsf{a}_i$-boxes are independent, if they are at $|\cdot|_\infty$-distance at least $\RC \cdot 2^{-n+\mathsf{a}_i}$ away. 

\item In condition~\ref{condition-nice-c}, the factor ``$\RC$'' in $B_{\RC \cdot 2^{-n+\mathsf{a}_{j-1}}}(z)$ ensures that, for any two $\mathsf{a}_{j-1}$-boxes that cannot be covered by such a box, their statuses of being nice or bad are independent. In addition, for a fixed choice of $z \in \mathscrY_{j-1} \cap B_{2^{-n+\mathsf{a}_j}}(x)$, both of the distances considered in \eqref{eq:scale-control} are independent of $\sigma( h_{a,b}|_{B_{(\RC+4) \cdot 2^{-n+\mathsf{a}_{j-1}}}(z)} : n-\mathsf{a}_{j-1} \leq a \leq b \leq n+k)$. Here, we consider the hypercubic shell $B_{4 \cdot 2^{-n+\mathsf{a}_j}}(x) \backslash B_{2 \cdot 2^{-n+\mathsf{a}_j}}(x)$ to ensure that any path entering $B_{\RC \cdot 2^{-n+\mathsf{a}_{j-1}}}(z)$ with start and end points that are at $|\cdot|_\infty$-distance at least $8\cdot 2^{-n+\mathsf{a}_j}$ away, must cross this hypercubic shell.

\item Conditions~\ref{condition-a1-nice-a} and~\ref{condition-nice-a} are finally used to control $\nabla h_n$. Rather than $\nabla h_n$, we consider the gradients $\nabla h_{n-\mathsf{a}_1,n}$ in condition~\ref{condition-a1-nice-a} and $\nabla h_{n-\mathsf{a}_j,n-\mathsf{a}_{j-1}}$ in condition~\ref{condition-nice-a} to ensure the independence properties. Eventually, we will use these bounds, together with the bound for $\nabla h_{n, n-\mathsf{a}_\mm}$, to estimate $\nabla h_n$; see condition~\ref{cond-a1-new-1} of Proposition~\ref{prop:cover}. For a similar purpose, we consider the metric $ D_{n-\mathsf{a}_j,n}$ in \eqref{eq:scale-control}. Eventually, we will combine \eqref{eq:scale-control} with the bound for $\nabla h_{n, n-\mathsf{a}_j}$ to derive a similar inequality for the metric $D_n$, as shown in \eqref{eq:scale-control-new} in Proposition~\ref{prop:cover}.

\end{enumerate}
    
\end{remark}

\subsection{Probabilities of bad boxes}
\label{subsec:est-bad}
In this subsection, we estimate the probabilities of the bad boxes. We will prove in Lemma~\ref{lem:est-qm} that these probabilities decay rapidly to zero. For $1 \leq j \leq \mm$, define the probability 
\begin{equation}
\label{eq:def-qj}
    \qq_j  = \qq_j(L,K,n,k)  := \mathbb{P}\big[\mbox{the }\mathsf{a}_j\mbox{-box }B_{2^{-n+\mathsf{a}_j}}(x) \mbox{ is bad}\big]
\end{equation}
(noting that the probability does not depend on $x$). We first show that $\qq_1$ tends to zero uniformly as $L$ becomes large.
\begin{lemma}
\label{lem:est-q1}
    For any fixed $K>100$, we have
    \begin{equation*}
        \lim_{L \rightarrow \infty} \sup_{n,k \geq 1} \qq_1 = 0\,.
    \end{equation*}
\end{lemma}
\begin{proof}
    Fix $K>100$. It suffices to show that, uniformly for any $n$ and $k$, conditions~\ref{condition-a1-nice-a}, \ref{condition-a1-nice-b}, and \ref{condition-a1-nice-c} in Definition~\ref{def:nice-boxes} happen with probability close to $1$ as $L$ becomes large. We begin with condition~\ref{condition-a1-nice-a}. Using the translation invariance and scaling property from Lemma~\ref{lem:basic-h}, we obtain
    \begin{equation*}
        \mathbb{P}\Big[\sup_{z \in B_{2^{-n+\mathsf{a}_1}}(x)} 2^{-n}|\nabla h_{n-\mathsf{a}_1,n}(z)|_\infty \leq L\Big] = \mathbb{P}\Big[\sup_{z \in B_1(0)} 2^{-\mathsf{a}_1}|\nabla h_{\mathsf{a}_1}(z)|_\infty \leq L\Big]\,.
    \end{equation*}
    By Claim (2) from Lemma~\ref{lem:field-estimate}, the right-hand side tends to 1 as $L$ becomes large since the random variable inside the bracket has a Gaussian tail. Therefore, condition~\ref{condition-a1-nice-a} happens with probability close to $1$ as $L$ becomes large. We now lower-bound the probabilities of the conditions \ref{condition-a1-nice-b} and \ref{condition-a1-nice-c}. Using the translation invariance and scaling property from Lemma~\ref{lem:basic-LFPP}, we obtain
    \begin{equation*}
    \begin{aligned}
        &{\rm Diam}_{n,n+k} (B_{2^{-n}}(y); B_{2 \cdot 2^{-n}}(y)) \overset{d}{=} 2^{-n}{\rm Diam}_{k} (B_1(0); B_2(0))\,, \quad \mbox{and}\\
        &D_{n,n+k} (\mbox{across }B_{2 \cdot 2^{-n}}(y) \backslash B_{2^{-n}}(y))\overset{d}{=} 2^{-n} D_k (\mbox{across }B_2(0) \backslash B(0))\,.
    \end{aligned}
    \end{equation*}
    Combining these with Propositions~\ref{prop:diameter} and \ref{prop:box-cross}, and noting that $|\mathscrY_0 \cap B_{2^{-n+\mathsf{a}_1}}(x)| \leq C$, we obtain that conditions~\ref{condition-a1-nice-b} and \ref{condition-a1-nice-c} both happen with probability close to $1$ as $L$ becomes large. Here, we also used the concentration bounds from Lemmas~\ref{lem:diamater-rough-concentration} and \ref{lem:gauss-concentration}. This concludes the proof.
\end{proof}

Next, we upper-bound $\qq_j$. We first prove a lemma, which is a consequence of the scaling property from Lemma~\ref{lem:basic-LFPP} and the super-exponential concentration bounds in Subsection~\ref{subsec:super-exponential}.

\begin{lemma}
\label{lem:5.4}
    Fix $\beta \in (d, Q^2/2)$. For all sufficiently large integer $K$ (which may depend on $\beta$), the following inequality holds for all integers $n,k \geq 1$, $2 \leq j \leq \mm$, and for all $x,z \in \mathbb{R}^d$:
    \begin{equation*}
        \begin{aligned}
            &\mathbb{P}\Big[D_{n-\mathsf{a}_j,n}(\mbox{across }B_{4 \cdot 2^{-n+\mathsf{a}_j}}(x) \backslash B_{2 \cdot 2^{-n+\mathsf{a}_j}}(x)) \\
            &\qquad \quad > \exp(K^{3/2} \sqrt{\mathsf{a}_j}) D_{n-\mathsf{a}_j,n}(\mbox{around }B_{4\RC \cdot 2^{-n+\mathsf{a}_{j-1}}}(z) \backslash B_{2\RC \cdot 2^{-n+\mathsf{a}_{j-1}}}(z)) \Big] \\
            &\quad \geq 1- 2^{-\beta \lambdaC \mathsf{a}_{j-1}}.
        \end{aligned}
    \end{equation*}
\end{lemma}
\begin{proof}
    Let $\theta$ and $\delta$ be two positive constants to be chosen. Define the random variables $D_1, D_2$, and $H$ as follows:
    \begin{equation*}
    \begin{aligned}
        &D_1 := D_{n-\mathsf{a}_j,n}(\mbox{across }B_{4 \cdot 2^{-n+\mathsf{a}_j}}(x) \backslash B_{2 \cdot 2^{-n+\mathsf{a}_j}}(x))\,,\\
        &D_2 := D_{n-\mathsf{a}_{j-1},n}(\mbox{around }B_{4\RC \cdot 2^{-n+\mathsf{a}_{j-1}}}(z) \backslash B_{2\RC \cdot 2^{-n+\mathsf{a}_{j-1}}}(z))\,,\\
        &H := \sup_{w \in B_{4\RC \cdot 2^{-n+\mathsf{a}_{j-1}}}(z)}h_{n-\mathsf{a}_j, n-\mathsf{a}_{j-1}}(w)\,.
    \end{aligned}
    \end{equation*}
    
    Using the fact that
    \begin{equation*}
       D_{n-\mathsf{a}_j,n}(\mbox{around }B_{4\RC \cdot 2^{-n+\mathsf{a}_{j-1}}}(z) \backslash B_{2\RC \cdot 2^{-n+\mathsf{a}_{j-1}}}(z)) \leq D_2 \times e^{\xi H},
    \end{equation*}
    we obtain
    \begin{equation}
    \label{eq:lem5.4-1}
    \begin{aligned}
        &\quad \mathbb{P}\big[D_1 > \exp(K^{3/2} \sqrt{\mathsf{a}_j}) D_{n-\mathsf{a}_j,n}(\mbox{around }B_{4\RC \cdot 2^{-n+\mathsf{a}_{j-1}}}(z) \backslash B_{2\RC \cdot 2^{-n+\mathsf{a}_{j-1}}}(z)) \big]\\
        &\geq \mathbb{P}\Big[ \big{\{} D_1 > 2^{-n + (\xi Q - \theta) \mathsf{a}_j} \big{\}} \cap \big{\{} D_2 <  2^{-n + (\xi Q + \theta) \mathsf{a}_{j-1} } \big{\}} \cap \big{\{} H <  s_j  \big{\}} \Big]\, 
    \end{aligned}
    \end{equation}
    where
    \[
    s_j = Q(\mathsf{a}_j - \mathsf{a}_{j-1}) \log 2 - \frac{\theta}{\xi} (\mathsf{a}_j + \mathsf{a}_{j-1}) \log 2 - K^{3/2} \sqrt{\mathsf{a}_j} / \xi .
    \]
    By \eqref{eq:def-ai}, we have $(1+\lambdaC)\mathsf{a}_{j-1} \leq  \mathsf{a}_j \leq (1+\lambdaC)\mathsf{a}_{j-1} + 1$. Therefore, we can choose a sufficiently small $\theta = \theta(\delta)>0$ such that, for all sufficiently large $K$ (which may depend on $\delta$ and $\theta$):
    \begin{equation}
        \label{eq:lem5.4-s}
        s_j > (Q \lambdaC \log 2  - \delta) \mathsf{a}_{j-1} \quad \forall j \geq 2\,.
    \end{equation}Here, we also used the fact that $\mathsf{a}_{j-1} \geq \mathsf{a}_1 = K^4$. 

    Using the translation invariance and scaling property from Lemma~\ref{lem:basic-LFPP}, we obtain
    \begin{equation*}
    \begin{aligned}
        &D_1 \overset{d}{=} 2^{-n+\mathsf{a}_j} D_{\mathsf{a}_j}(\mbox{across } B_4(0) \backslash B_2(0)) \quad \mbox{and}\\
        &D_2 \overset{d}{=}  2^{-n+\mathsf{a}_{j-1}} D_{\mathsf{a}_{j-1}}(\mbox{around }B_{4\RC}(0) \backslash B_{2\RC}(0))\,.
        \end{aligned}
    \end{equation*}
    Combining these with Lemmas~\ref{lem:super-exponential-1} and \ref{lem:super-exponential-2} yields
    \begin{equation}
    \label{eq:lem5.4-2}
    \begin{aligned}
        &\mathbb{P}\Big[D_1 > 2^{-n + (\xi Q - \theta) \mathsf{a}_j} \Big] > 1- C2^{-\mathsf{a}_j \log \mathsf{a}_j} \quad \mbox{and} \\
        &\mathbb{P}\Big[D_2 < 2^{-n + (\xi Q + \theta) \mathsf{a}_{j-1} } \Big] > 1- C2^{-\mathsf{a}_{j-1} \log \mathsf{a}_{j-1}}.
        \end{aligned}
    \end{equation}

    Using the translation invariance and the scaling property from Lemma~\ref{lem:basic-h}, we obtain
    \begin{equation}
    \label{eq:lem5.4-3-0}
        H \overset{d}{=} \sup_{w \in B_{4\RC \cdot 2^{\mathsf{a}_{j-1}-\mathsf{a}_j}}(0)} h_{\mathsf{a}_j - \mathsf{a}_{j-1}}(w)\,.
    \end{equation}
    Combining \eqref{eq:lem5.4-3-0} with \eqref{eq:lem5.4-s} and Claim (5) from Lemma~\ref{lem:field-estimate} for $n = \mathsf{a}_j - \mathsf{a}_{j-1}$, we obtain that for all sufficiently large $K$:
    \begin{equation*}
    \begin{aligned}
        \mathbb{P}[H < s_j] 
        &\geq \mathbb{P}[H < t \mathsf{a}_{j-1} ] \\ 
        &\geq 1- C \exp \Big(- \frac{(t \mathsf a_{j-1} - (t \mathsf a_{j-1})^{2/3})^2 }{2 \log 2 \cdot(\mathsf{a}_j - \mathsf{a}_{j-1})} \Big) - C \exp \Big(- \frac{(t \mathsf a_{j-1})^{4/3}}{C} \Big)\,,
        \end{aligned}
    \end{equation*}
    where $t = Q \lambdaC \log 2 - \delta$ and $C$ is a constant independent of $\delta$, $K$, and $j$. Since $\beta < Q^2/2$, which implies that $\sqrt{2 \beta} < Q$, we can choose a sufficiently small $\delta$ such that $t > \sqrt{2 \beta} \lambda \log 2$. Combining this with the fact that $(1+\lambdaC) \mathsf{a}_{j-1} \leq \mathsf{a}_j \leq (1+\lambdaC) \mathsf{a}_{j-1} +1$, we obtain that for all sufficiently large $K$:
    \begin{equation}
    \label{eq:lem5.4-3}
        \mathbb{P}[H <  s_j ] \geq 1- 2^{-\beta\lambdaC \mathsf{a}_{j-1}} / 2\,.
    \end{equation}
    Combining \eqref{eq:lem5.4-1} with the estimates \eqref{eq:lem5.4-2} and \eqref{eq:lem5.4-3} yields the desired result.
\end{proof}

We now upper-bound $\qq_j$ in terms of $\qq_{j-1}$.

\begin{lemma}
\label{lem:est-qj}
    For each fixed $\beta \in (d, Q^2/2)$, there exists a constant $C = C(\beta)>0$ such that for all $K \geq C$ and $L>100$, the following inequality holds for all integers $n,k \geq 1$ and $2 \leq j \leq \mm$:
    \begin{equation*}
        \qq_j \leq C \Big( e^{-K\mathsf{a}_j} + 2^{2d \lambdaC \mathsf{a}_{j-1}} \qq_{j-1}^2 + 2^{(d - \beta)\lambdaC \mathsf{a}_{j-1}} \qq_{j-1} \Big)\,.
    \end{equation*}
\end{lemma}
\begin{proof}
    By Definition~\ref{def:nice-boxes}, an $\mathsf{a}_j$-box $B_{2^{-n+\mathsf{a}_j}}(x)$ is bad if either condition~\ref{condition-nice-a} does not hold, or if both conditions~\ref{condition-nice-b} and \ref{condition-nice-c} fail. We first estimate the probability that condition~\ref{condition-nice-a} fails, which will be denoted by $I$. Then, we have
    \begin{equation*}
        I = \mathbb{P}\Big[ \sup_{z \in B_{4 \cdot 2^{-n+\mathsf{a}_j}}(x)}2^{-n+\mathsf{a}_{j-1}} |\nabla h_{n-\mathsf{a}_j,n-\mathsf{a}_{j-1}}(z)|_\infty > K \sqrt{\mathsf{a}_{j-1}}\Big]\,.
    \end{equation*}
    By the translation and scaling property from Lemma~\ref{lem:basic-h}, we obtain
    \begin{equation*}
        I = \mathbb{P}\Big[ \sup_{z \in B_4(0)} 2^{-(\mathsf{a}_j-\mathsf{a}_{j-1})} |\nabla h_{\mathsf{a}_j-\mathsf{a}_{j-1}}(z)|_\infty > K \sqrt{\mathsf{a}_{j-1}} \Big]\,.
    \end{equation*}
    Applying Claim (2) from Lemma~\ref{lem:field-estimate}, with $n = \mathsf{a}_j - \mathsf{a}_{j-1}$, and noting that $|2^{\mathsf{a}_{j-1} - \mathsf{a}_j }\mathbb{Z}^d \cap B_4(0)| \leq C2^{d \lambdaC \mathsf{a}_{j-1}}$, we obtain
    \begin{equation*}
    \begin{aligned}
        I 
        &\leq \sum_{y \in 2^{\mathsf{a}_{j-1} - \mathsf{a}_j }\mathbb{Z}^d \cap B_4(0) } \mathbb{P}\Big[ \sup_{z \in B_{2^{\mathsf{a}_{j-1} - \mathsf{a}_j }}(y)} 2^{\mathsf{a}_{j-1} - \mathsf{a}_j } |\nabla h_{\mathsf{a}_{j} - \mathsf{a}_{j-1} }(z)|_\infty > K \sqrt{\mathsf{a}_{j-1}} \Big] \\
        &\leq C2^{d\lambdaC \mathsf{a}_{j-1}} \times C e^{-K^2\mathsf{a}_{j-1}/C}.
        \end{aligned}
    \end{equation*}
    Therefore, for all sufficiently large $K$, we get
    \begin{equation}
    \label{eq:lem5.5-1}
        \mathbb{P}\big[ \mbox{condition }\ref{condition-nice-a} \mbox{ fails}\big] = I \leq C e^{-K \mathsf{a}_j}.
    \end{equation}

    Next, we upper-bound the probability that both conditions~\ref{condition-nice-b} and \ref{condition-nice-c} fail. By definition, there are two possible cases:
    \begin{enumerate}
        \item There exist two bad $\mathsf{a}_{j-1}$ boxes contained in $B_{2^{-n+\mathsf{a}_j}}(x)$, and they are at $|\cdot|_\infty$-distance at least $\RC \cdot 2^{-n+\mathsf{a}_{j-1}}$ apart from each other.
        \item For some $z \in \mathscrY_{j-1} \cap B_{2^{-n+\mathsf{a}_j}}(x)$, there exists at least one bad $\mathsf{a}_{j-1}$-box inside the box $B_{\RC \cdot 2^{-n+\mathsf{a}_{j-1}}}(z)$, and the inequality~\eqref{eq:scale-control} fails for $z$.
    \end{enumerate}
    We denote the event that the first case happens by $\mathcal{B}_1$ and the event that the second case happens by $\mathcal{B}_2$. 

    We first upper-bound $\mathbb{P}[\mathcal{B}_1]$. As noted in Remark~\ref{rmk:nice-box}, for any two $\mathsf{a}_{j-1}$-boxes that are at least $\RC \cdot 2^{-n+\mathsf{a}_{j-1}}$ away from each other, their statuses of being nice or bad are independent. Therefore,
    \begin{equation}
    \label{eq:lem5.5-2}
    \begin{aligned}
        &\quad \mathbb{P}[ \mathcal{B}_1] \\
        &\leq \sum_{\substack{x_1,x_2 \in \mathscrY_{j-1} \cap B_{2^{-n+\mathsf{a}_j}}(x)\\ |x_1 - x_2|_\infty \geq \RC \cdot  2^{-n+\mathsf{a}_{j-1}} }}\mathbb{P}\Big[B_{2^{-n+\mathsf{a}_{j-1}}}(x_1) \mbox{ is bad}\Big] \cdot  \mathbb{P}\Big[B_{2^{-n+\mathsf{a}_{j-1}}}(x_2) \mbox{ is bad} \Big] \\
        &\leq C 2^{2d\lambdaC \mathsf{a}_{j-1}} \qq_{j-1}^2.
    \end{aligned}
    \end{equation}
    The last inequality follows from $|\mathscrY_{j-1} \cap B_{2^{-n+\mathsf{a}_j}}(x)| \leq C 2^{d(\mathsf{a}_j - \mathsf{a}_{j-1})} \leq C 2^{ d\lambdaC \mathsf{a}_{j-1}}$ and the definition of $\qq_{j-1}$ from \eqref{eq:def-qj}.

    Next, we upper-bound $\mathbb{P}[\mathcal{B}_2]$. Using the definition of $\qq_{j-1}$ from \eqref{eq:def-qj}, we have
    \begin{equation}
    \label{eq:lem5.5-3-0}
        \mathbb{P}\big[\exists y \in \mathscrY_{j-1} \cap B_{\RC \cdot 2^{-n+\mathsf{a}_{j-1}}}(z) \mbox{ such that }B_{2^{-n+\mathsf{a}_{j-1}}}(y) \mbox{ is bad}] \leq C \qq_{j-1}\,.
    \end{equation}
    Note that both the distances considered in \eqref{eq:scale-control} are independent of the $\sigma$-algebra $\sigma( h_{a,b}|_{B_{(\RC+4) \cdot 2^{-n+\mathsf{a}_{j-1}}}(z)} : n-\mathsf{a}_{j-1} \leq a \leq b \leq n+k)$. Therefore, by Remark~\ref{rmk:nice-box}, they are also independent of the statuses of the $\mathsf{a}_{j-1}$-boxes contained in $B_{\RC \cdot 2^{-n+\mathsf{a}_{j-1}}}(z)$. Combining this independence argument with \eqref{eq:lem5.5-3-0} and Lemma~\ref{lem:5.4} with the same choice of $\beta$, we obtain that for all sufficiently large $K$:
    \begin{equation}
    \label{eq:lem5.5-3}
        \mathbb{P}[ \mathcal{B}_2] \leq |\mathscrY_{j-1} \cap B_{2^{-n+\mathsf{a}_j}}(x)| \times C \qq_{j-1} \times 2^{-\beta \lambdaC \mathsf{a}_{j-1}} \leq C 2^{(d - \beta)\lambdaC \mathsf{a}_{j-1}} \qq_{j-1}\,.
    \end{equation}
    In the last inequality, we used $|\mathscrY_{j-1} \cap B_{2^{-n+\mathsf{a}_j}}(x)| \leq C2^{d \lambdaC \mathsf{a}_{j-1}}$. Combining \eqref{eq:lem5.5-1}, \eqref{eq:lem5.5-2}, and \eqref{eq:lem5.5-3} yields the desired result.
\end{proof}

In the following lemma, we derive an upper bound for $\qq_j$ based on Lemmas~\ref{lem:est-q1} and \ref{lem:est-qj}. Recall from \eqref{eq:def-lambda} the definition of the constants $\alphaC$ and $\lambdaC$.

\begin{lemma}
    \label{lem:est-qm}
    There exist a constant $C>0$ and an increasing function $L_0: \mathbb{N} \rightarrow \mathbb{N}$ such that for all $K >C$ and $L>L_0(K)$, we have
    \begin{equation*}
        \qq_j \leq 2^{-\alphaC \mathsf{a}_j} \quad \forall n,k \geq 1 \mbox{ and }1 \leq j \leq \mm\,.
    \end{equation*}
\end{lemma}
\begin{proof}
By Lemma~\ref{lem:est-q1}, for any fixed $K>100$, there exists $L(K)>0$ such that for all $L > L(K)$,
\begin{equation}
\label{eq:lem5.6-1}
    \qq_1 \leq 2^{-\alphaC \mathsf{a}_1} \quad \forall n,k \geq 1\,.
\end{equation}
Since $\alphaC < Q^2/2-d$, we can choose a constant $\beta \in (d+\alphaC, Q^2/2)$. We now prove the lemma using an induction argument. Suppose that the inequality $\qq_{j-1} \leq 2^{-\alphaC \mathsf{a}_{j-1}}$ holds. Using Lemma~\ref{lem:est-qj}, we obtain that for all $K > C$ and $L>100$,
\begin{equation*}
    \begin{aligned}
        \qq_j \leq C \Big( e^{-K\mathsf{a}_j} + 2^{(2d \lambdaC-2\alphaC) \mathsf{a}_{j-1}}  + 2^{((d - \beta)\lambdaC -\alphaC )\mathsf{a}_{j-1}}\Big)\,.
    \end{aligned}
\end{equation*}
Recall the definition of $\lambdaC$ from \eqref{eq:def-lambda}. Since $\lambdaC < \frac{\alphaC}{2d +\alphaC}$, we obtain $(2d + \alphaC) \lambdaC - 2 \alphaC < -\alphaC$, which implies $2d\lambdaC - 2\alphaC < -\alphaC(1+\lambdaC)$. Since $d - \beta < - \alphaC$, we obtain $(d-\beta) \lambdaC - \alphaC < -\alphaC (1+\lambdaC)$. Moreover, recall from \eqref{eq:def-ai} that $\mathsf{a}_j \leq (1+\lambdaC)\mathsf{a}_{j-1}+1$. Therefore, we can choose a sufficiently large constant $C'$ (independent of $n,k$, and $j$) such that for all $K>C'$, $L>100$, and $2 \leq j \leq \mm$:
\begin{equation*}
    \qq_{j-1} \leq 2^{-\alphaC \mathsf{a}_{j-1}} \quad \mbox{implies that } \quad  \qq_j \leq 2^{-\alphaC \mathsf{a}_j}.
\end{equation*}
By \eqref{eq:lem5.6-1} and an induction argument, we obtain that $\qq_j \leq 2^{-\alphaC \mathsf{a}_j}$ for all $1 \leq j \leq \mm$.
\end{proof}

\subsection{Construction of a covering}
\label{subsec:cover}

In this subsection, we construct a sequence of boxes at various scales in Proposition~\ref{prop:cover}. These boxes satisfy a series of conditions which will play an important role in the proof of Proposition~\ref{prop:compare} in Subsection~\ref{subsec:prove-compare}. We first introduce two events and show in Lemma~\ref{lem:est-j12} that both of these events happen with high probability. Define the events as follows (recall the definition of ``nice" from Definition~\ref{def:nice-boxes}):
\begin{equation}
\label{eq:def-j12}
\begin{aligned}
    &\mathcalJ_1 = \mathcalJ_1(L,K,n,k) := \big{\{}\mbox{Each }\mathsf{a}_\mm \mbox{-box contained in }B_2(0) \mbox{ is nice} \big{\}}\,, \\
    &\mathcalJ_2 = \mathcalJ_2(L,K,n,k) := \big{\{} \sup_{z \in B_3(0)} 2^{-n + \mathsf{a}_\mm} |\nabla h_{0, n-\mathsf{a}_\mm}(z)|_\infty \leq K \sqrt{\mathsf{a}_\mm} \big{\}} \,.
\end{aligned}
\end{equation}

\begin{lemma}
\label{lem:est-j12}
There exists a constant $C>0$ such that for all $K>C$ and $L>L_0(K)$ (recall from Lemma~\ref{lem:est-qm}), we have
    \begin{equation*}
        \mathbb{P}[\mathcalJ_1 \cap \mathcalJ_2] \geq 1-C2^{-n/C} \quad \forall n,k \geq 1\,.
    \end{equation*}
\end{lemma}
\begin{proof}
    Using Lemma~\ref{lem:est-qm}, we obtain that for $K>C$ and $L>L_0(K)$,
    \begin{equation*}
        \mathbb{P}[\mathcalJ_1] \geq 1- \qq_m \times | \mathscrY_\mm \cap B_2(0)| \geq 1-C 2^{-\alphaC \mathsf{a}_\mm} 2^{d(n-\mathsf{a}_\mm)}.
    \end{equation*}
    Recall from \eqref{eq:def-am} that $\mathsf{a}_\mm \geq (1-2\lambdaC) n $ and recall the choice of $\lambdaC$ from \eqref{eq:def-lambda}. Since $\lambdaC < \frac{\alphaC}{2 \alphaC + 2d}$, we have $-(\alphaC + d)(1- 2 \lambdaC) + d < 0$. Therefore,
    \begin{equation*}
        \mathbb{P}[\mathcalJ_1] \geq 1 - C 2^{-(\alphaC + d)(1-2 \lambdaC)n + dn} \geq 1-C2^{-n/C}.
    \end{equation*}
    Applying Claim (2) in Lemma~\ref{lem:field-estimate}, with $n-\mathsf{a}_\mm$ in place of $n$, we obtain that for all sufficiently large $K$:
    \begin{equation*}
    \begin{aligned}
        \mathbb{P}[\mathcalJ_2] &\geq 1 - \sum_{y \in \mathscrY_\mm \cap B_3(0)} \mathbb{P} \Big[ \sup_{z \in B_{2^{-n+\mathsf{a}_\mm}}(y)} 2^{-n + \mathsf{a}_\mm} |\nabla h_{0, n-\mathsf{a}_\mm}(z)|_\infty > K \sqrt{\mathsf{a}_\mm} \Big] \\
        &\geq 1- C2^{d(n-\mathsf{a}_\mm)} \times C e^{-K^2 \mathsf{a}_\mm / C} \geq 1-C2^{-n/C}.
    \end{aligned}
    \end{equation*}
    In the last inequality, we used the fact that $\mathsf{a}_\mm \geq (1-2\lambdaC)n$ and $K$ is sufficiently large. Combining the above two inequalities yields the lemma.
\end{proof}

In the following proposition, on the event $\mathcalJ_1 \cap \mathcalJ_2$, we will construct a sequence of boxes at various scales described by their centers $\mathcalX_1,\mathcalX_2,\ldots, \mathcalX_{\mm-1}$. These boxes will satisfy a series of conditions as listed below. Condition~\ref{cond-new-3} will be used in Lemma~\ref{lem:bypassU} to demonstrate that the length metric $D_n$ increases by at most a constant factor when rerouting the path to avoid these boxes. Conditions~\ref{cond-a1-new-1}, \ref{cond-a1-new-2}, and \ref{cond-a1-new-3} for $\mathsf{a}_1$-boxes will be used in Lemma~\ref{lem:modify-outU}. This will establish that for any path inside the region $B_2(0) \backslash \mathcalU$ (see \eqref{eq:def-mathcal-U} for the definition of $\mathcalU$), by adjusting the path, its $D_n$-length and $D_{n+k}$-length satisfy the desired bound in Proposition~\ref{prop:compare}.

The proof of Proposition~\ref{prop:cover} uses an induction argument based on the event $\mathcalJ_1$ and the definitions of nice boxes. As mentioned in Remark~\ref{rmk:nice-box}, we will use the event $\mathcalJ_2$, combined with the bounds for $\nabla h_{n-\mathsf{a}_1,n}$ and $\nabla h_{n-\mathsf{a}_j,n-\mathsf{a}_{j-1}}$, to derive bounds on $\nabla h_n$ and the properties of the metric $D_n$ using Equation~\eqref{eq:scale-control}. Readers can keep in mind that Proposition~\ref{prop:cover} and the estimates from Lemma~\ref{lem:est-j12} serve as the only two inputs from the coarse-graining argument when proving Proposition~\ref{prop:compare} in Subsection~\ref{subsec:prove-compare}. In particular, we will no longer use the notion of nice or bad $\mathsf{a}_j$-boxes there.

\begin{proposition}
\label{prop:cover}
On the event $\mathcalJ_1 \cap \mathcalJ_2$, we can construct a sequence of sets, $\mathcalX_1,\mathcalX_2,\ldots, \mathcalX_{\mm-1}$, that satisfy the following properties for every $1 \leq j \leq \mm-1$: 
\begin{enumerate}[(i)]
    \item\label{cond-new-1} $\mathcalX_j \subset \mathscrY_j \cap B_2(0)$.
    \item\label{cond-new-2} For each pair of distinct points $z , w \in \mathcalX_j$, we have $|z-w|_\infty > 2^{-n + \mathsf{a}_{j+1}}/4$.
    \item\label{cond-new-3} For every $z \in \mathcalX_j$, there exists $x \in \mathscrY_{j+1} \cap B_2(0)$ such that $|x-z|_\infty \leq 2^{-n + \mathsf{a}_{j+1}}$ and
    \begin{equation}
        \label{eq:scale-control-new}
        \begin{aligned}
         &D_n\big(\mbox{across }B_{4 \cdot 2^{-n+\mathsf{a}_{j+1}}}(x) \backslash B_{2 \cdot 2^{-n+\mathsf{a}_{j+1}}}(x)\big)\\ 
         &\qquad\qquad> \exp(K \sqrt{\mathsf{a}_j}) D_n\big(\mbox{around }B_{4\RC \cdot 2^{-n+\mathsf{a}_j}}(z) \backslash B_{2\RC \cdot 2^{-n+\mathsf{a}_j}}(z)\big)\,.
        \end{aligned}
    \end{equation}
\end{enumerate}
Define the domain
\begin{equation}
\label{eq:def-mathcal-U}
    \mathcalU := \cup_{1 \leq j \leq \mm-1} \cup_{z \in \mathcalX_j}  B_{\RC \cdot 2^{-n+\mathsf{a}_j}}(z)\,.
\end{equation}
Moreover, for every $\mathsf{a}_1$-box $B_{2^{-n+\mathsf{a}_1}}(w)$ that is contained in $B_2(0)$ but not in $\mathcalU$, it satisfies the following conditions:
\begin{enumerate}[(A)]
    \item\label{cond-a1-new-1} $2^{-n} |\nabla h_n(y)|_\infty \leq L + 1$ for every $y \in B_{2^{-n+\mathsf{a}_1}}(w)$.
    \item\label{cond-a1-new-2} ${\rm Diam}_{n,n+k} (B_{2^{-n}}(y); B_{2 \cdot 2^{-n}}(y)) < L 2^{-n - (1-\xi Q - \epsilon) k}$ for every $y \in \mathscrY_0 \cap B_{2^{-n+\mathsf{a}_1}}(w)$.
    \item\label{cond-a1-new-3} $D_{n,n+k} (\mbox{across }B_{2 \cdot 2^{-n}}(y) \backslash B_{2^{-n}}(y)) >  \frac{1}{L}2^{-n - (1-\xi Q + \epsilon) k}$ for every $y \in \mathscrY_0 \cap B_{2^{-n+\mathsf{a}_1}}(w)$.
\end{enumerate}

\end{proposition}

\begin{proof}

Suppose that the event $\mathcalJ_1 \cap \mathcalJ_2$ happens. We will define the sets $\mathcalX_j$ inductively for $j$ going from $\mm$ to $1$. For any $1 \leq j \leq \mm - 1$, the set $\mathcalX_j$ will satisfy conditions~\ref{cond-new-1}, \ref{cond-new-2}, and \ref{cond-new-3}. Let $U_p := \cup_{p \leq r \leq \mm } \cup_{z \in \mathcalX_r}  B_{\RC \cdot 2^{-n+\mathsf{a}_r}}(z)$. In addition, we require that for any $1 \leq p \leq \mm$, they also satisfy the assumption:
\begin{equation}
\label{eq:prop5.8-asmp}
    \mbox{every }\mathsf{a}_p \mbox{-box contained in }B_2(0), \mbox{ but not in }U_p, \mbox{ is nice}.
\end{equation}Setting $p=1$ in this assumption implies that conditions~\ref{cond-a1-new-2} and \ref{cond-a1-new-3} hold for each $\mathsf{a}_1$-box contained in $B_2(0)$ but not in $U_1 = \mathcalU$. Condition~\ref{cond-a1-new-1} will be derived from the bounds of $\nabla h_{n-\mathsf{a}_\mm}$, $\nabla h_{n-\mathsf{a}_1,n}$, and $\nabla h_{n-\mathsf{a}_j,n-\mathsf{a}_{j-1}}$ for $2 \leq j \leq \mm$.

Let $\mathcalX_\mm = \emptyset$. By the event $\mathcalJ_1$, the assumption \eqref{eq:prop5.8-asmp} holds for $p=\mm$. Assume that $\mathcalX_\mm,\ldots,\mathcalX_{j+1}$ have been defined and that the assumption \eqref{eq:prop5.8-asmp} holds for every $j + 1 \leq p \leq \mm$. We now construct $\mathcalX_j$ and verify that it satisfies conditions~\ref{cond-new-1}, \ref{cond-new-2}, \ref{cond-new-3}, as well as the assumption \eqref{eq:prop5.8-asmp} for $p = j$. Consider any bad $\mathsf{a}_j$-box that is contained in $B_2(0)$ but not in $U_{j+1}$. Then there exists an $\mathsf{a}_{j+1}$-box that contains this $\mathsf{a}_j$-box and is contained in $B_2(0)$ but not in $U_{j+1}$. Moreover, we can require that the $\mathsf{a}_j$-box is at $|\cdot|_\infty$-distance at least $2^{-n+\mathsf{a}_{j+1}}/2 - 2^{-n+\mathsf{a}_j}$ away from the boundary of this $\mathsf{a}_{j+1}$-box that is inside $B_2(0)$. By the assumption \eqref{eq:prop5.8-asmp} for $p = j + 1$, this $\mathsf{a}_{j+1}$-box is nice. Therefore, according to condition~\ref{condition-nice-c} from Definition~\ref{def:nice-boxes}, for some $z \in \mathscrY_j \cap B_2(0)$, we can use the box $B_{\RC \cdot 2^{-n+\mathsf{a}_j}}(z)$ to cover this bad $\mathsf{a}_j$-box. Furthermore, other bad $\mathsf{a}_j$-boxes not covered by $B_{\RC \cdot 2^{-n+\mathsf{a}_j}}(z)$ are at least $2^{-n+\mathsf{a}_{j+1}}/3$ away from the point $z$. This is due to the fact that these bad $\mathsf{a}_j$-boxes cannot be contained in this nice $\mathsf{a}_{j+1}$-box, and by construction, $z$ is at least $2^{-n+\mathsf{a}_{j+1}}/2.5$ away from the boundary of this $\mathsf{a}_{j+1}$-box. We refer to Figure~\ref{fig:nice-box} for an illustration. We define $\mathcalX_j$ to be the union of such $z$'s.

Next, we verify that $\mathcalX_j$ satisfies conditions~\ref{cond-new-1}, \ref{cond-new-2}, and \ref{cond-new-3}, as well as the induction assumption~\eqref{eq:prop5.8-asmp} for $p = j$. Conditions~\ref{cond-new-1} and \ref{cond-new-2}, and the assumption are straightforward from the construction. We now verify condition~\ref{cond-new-3}. Consider any $z \in \mathcalX_j$, and let $x$ be the center of the associated $\mathsf{a}_{j+1}$-box to the box $B_{\RC \cdot 2^{-n+\mathsf{a}_j}}(z)$. By construction, we have $x \in \mathscrY_{j+1} \cap B_2(0)$, $|x-z|_\infty \leq 2^{-n + \mathsf{a}_{j+1}}$, and Equation~\eqref{eq:scale-control} holds for this pair of $(x,z)$. That is,
\begin{equation}
\label{eq:prop5.8-1001}
\begin{aligned}
    &D_{n - \mathsf{a}_{j+1}, n} \big(\mbox{across }B_{4 \cdot 2^{-n+\mathsf{a}_{j+1}}}(x) \backslash B_{2 \cdot 2^{-n+\mathsf{a}_{j+1}}}(x)\big) \\
    &\qquad > \exp(K^{3/2} \sqrt{\mathsf{a}_{j+1}}) D_{n - \mathsf{a}_{j+1}, n}\big(\mbox{around }B_{4\RC \cdot 2^{-n+\mathsf{a}_j}}(z) \backslash B_{2\RC \cdot 2^{-n+\mathsf{a}_j}}(z)\big)\,.
\end{aligned}
\end{equation}
Since the $\mathsf{a}_{j+1}$-box centered at $x$ is not contained in $U_{j+1}$, for each $j+2 \leq l \leq \mm$, there exists an $\mathsf{a}_l$-box that contains $B_{2^{-n+\mathsf{a}_j}}(x)$ and is contained within $B_2(0)$ but not in $U_l$. By assumption \eqref{eq:prop5.8-asmp} with $p=l$, this $\mathsf{a}_l$-box must be nice. Using condition~\ref{condition-nice-a} from Definition~\ref{def:nice-boxes}, we obtain that for all $j+2 \leq l \leq \mm$:
\begin{equation*}
    2^{-n + \mathsf{a}_{l-1}}| \nabla h_{n-\mathsf{a}_l, n-\mathsf{a}_{l-1}}(w)|_\infty \leq K \sqrt{\mathsf{a}_{l-1}} \quad \forall w \in B_{4 \cdot 2^{-n+\mathsf{a}_{j+1}}}(x)\,.
\end{equation*}
Furthermore, by the event $\mathcalJ_2$, we have $2^{-n + \mathsf{a}_\mm} |\nabla h_{0, n-\mathsf{a}_\mm}(w)|_\infty \leq K \sqrt{\mathsf{a}_\mm}$ for all $w \in B_{4 \cdot 2^{-n+\mathsf{a}_{j+1}}}(x)$. Summing these inequalities yields that for all $w \in B_{4 \cdot 2^{-n+\mathsf{a}_{j+1}}}(x)$,
\begin{equation}
\label{eq:prop5.8-1}
\begin{aligned}
|\nabla h_{0, n-\mathsf{a}_{j+1}}(w)|_\infty \leq \sum_{l = j+2}^\mm 2^{n - \mathsf{a}_{l-1}} K \sqrt{\mathsf{a}_{l-1}} + 2^{n-\mathsf{a}_\mm} K \sqrt{\mathsf{a}_\mm} \leq 2 \cdot 2^{n-\mathsf{a}_{j+1}} K \sqrt{\mathsf{a}_{j+1}}\,,
\end{aligned}
\end{equation}where the last inequality follows from the fact that $\sum_{l \geq s} 2^{-l} \sqrt{l} \leq 2 \cdot 2^{-s} \sqrt{s}$ for any $s \geq 100$. This implies that $h_{n-\mathsf{a}_{j+1}}(x) - 8d K   \sqrt{\mathsf{a}_{j+1}} \leq h_{n-\mathsf{a}_{j+1}}(w) \leq h_{n-\mathsf{a}_{j+1}}(x) + 8d K   \sqrt{\mathsf{a}_{j+1}} $ for all $w \in B_{4 \cdot 2^{-n+\mathsf{a}_{j+1}}}(x)$. Therefore,
\begin{equation*}
\begin{aligned}
    e^{\xi h_{n-\mathsf{a}_{j+1}}(x) - 8d \xi K  \sqrt{\mathsf{a}_{j+1}}}
    &\leq \frac{D_n(\mbox{across }B_{4 \cdot 2^{-n+\mathsf{a}_{j+1}}}(x) \backslash B_{2 \cdot 2^{-n+\mathsf{a}_{j+1}}}(x))}{D_{n-\mathsf{a}_{j+1},n}(\mbox{across }B_{4 \cdot 2^{-n+\mathsf{a}_{j+1}}}(x) \backslash B_{2 \cdot 2^{-n+\mathsf{a}_{j+1}}}(x))}  \\
    &\qquad\qquad\qquad\leq e^{\xi h_{n-\mathsf{a}_{j+1}}(x) + 8d \xi K  \sqrt{\mathsf{a}_{j+1}}} \quad \mbox{and}
\end{aligned}
\end{equation*}
\begin{equation*}
\begin{aligned}
    e^{\xi h_{n-\mathsf{a}_{j+1}}(x) - 8d \xi K  \sqrt{\mathsf{a}_{j+1}}} 
    &\leq \frac{D_n(\mbox{around }B_{4\RC \cdot 2^{-n+\mathsf{a}_j}}(z) \backslash B_{2\RC \cdot 2^{-n+\mathsf{a}_j}}(z))}{D_{n-\mathsf{a}_{j+1},n}(\mbox{around }B_{4\RC \cdot 2^{-n+\mathsf{a}_j}}(z) \backslash B_{2\RC \cdot 2^{-n+\mathsf{a}_j}}(z))}\\
    &\qquad\qquad\qquad\leq e^{\xi h_{n-\mathsf{a}_{j+1}}(x) + 8d \xi K  \sqrt{\mathsf{a}_{j+1}}}.
    \end{aligned}
\end{equation*}
Combining these results with \eqref{eq:prop5.8-1001} and $K^{3/2} \sqrt{\mathsf{a}_{j+1}} - 16d \xi K \sqrt{\mathsf{a}_{j+1}} > K \sqrt{\mathsf{a}_j}$ (which follows from $K > (16 d \xi +1)^2$ and $\mathsf{a}_{j+1} > \mathsf{a}_j$), we obtain \eqref{eq:scale-control-new}. This verifies condition~\ref{cond-new-3}.

Using an induction argument, we can construct the sequence $\mathcalX_\mm, \ldots, \mathcalX_1$ satisfying conditions~\ref{cond-new-1}, \ref{cond-new-2}, and \ref{cond-new-3} for all $1 \leq j \leq \mm-1$, and satisfying \eqref{eq:prop5.8-asmp} for all $1 \leq p \leq \mm$. We now show that conditions~\ref{cond-a1-new-1}, \ref{cond-a1-new-2}, and \ref{cond-a1-new-3} hold for every $\mathsf{a}_1$-box $B_{2^{-n+\mathsf{a}_1}}(x)$ contained in $B_2(0)$ but not in $\mathcalU = U_1$. By \eqref{eq:prop5.8-asmp} with $p=1$, such $\mathsf{a}_1$-boxes are nice. Hence, conditions~\ref{cond-a1-new-2} and \ref{cond-a1-new-3} hold, as they are the same as conditions~\ref{condition-a1-nice-b} and \ref{condition-a1-nice-c} from Definition~\ref{def:nice-boxes}, respectively. We now prove condition~\ref{cond-a1-new-1}. Similarly to \eqref{eq:prop5.8-1}, we obtain that for all $w \in B_{4 \cdot 2^{-n+\mathsf{a}_1}}(x)$,
\begin{equation*}
\begin{aligned}
    |\nabla h_{0, n-\mathsf{a}_1}(w)|_\infty \leq \sum_{l = 2}^\mm 2^{n - \mathsf{a}_{l-1}} K \sqrt{\mathsf{a}_{l-1}} + 2^{n-\mathsf{a}_\mm} K \sqrt{\mathsf{a}_\mm} \leq 2 \cdot 2^{n-\mathsf{a}_1} K \sqrt{\mathsf{a}_1} \leq 2^n.
\end{aligned}
\end{equation*}
The last inequality is due to the fact that $\mathsf{a}_1 = K^4$ and $K>100$. Combining this with condition~\ref{condition-a1-nice-a} from Definition~\ref{def:nice-boxes}, we obtain condition~\ref{cond-a1-new-1}. This completes the proof. \qedhere
\end{proof}

\subsection{Proof of Proposition~\ref{prop:compare}}
\label{subsec:prove-compare}

In this subsection, we finish the proof of Proposition~\ref{prop:compare}. We assume that the events $\mathcalJ_1$ and $\mathcalJ_2$ defined in \eqref{eq:def-j12} happen. Recall the sequence of sets $\mathcalX_1, \ldots, \mathcalX_{\mm-1}$ and the domain $\mathcalU$ constructed in Proposition~\ref{prop:cover}. 

We first show that for any continuous path within $B_2(0)$, we can modify it to bypass the domain $\mathcalU$, while only slightly changing the start and end points in terms of the $|\cdot|_\infty$-distance. Moreover, the $D_n$-length of this modified path can be upper-bounded by that of the original one up to constants. Our strategy is to apply \eqref{eq:scale-control-new} in an inductive manner to iteratively reroute the path around the boxes $B_{\RC \cdot 2^{-n+\mathsf{a}_j}}(z)$ for $z \in \mathcalX_j$. Note that it is possible for the start and end points to be contained in $\mathcalU$, in which case they have to be adjusted. In addition, the following lemma is trivial when the path has $|\cdot|_\infty$-length at most $10 \cdot 2^{-n+\mathsf{a}_\mm}$---in particular, for paths entirely contained in $\mathcalU$---in which case we simply take the modified path to be empty.

\begin{lemma}
    \label{lem:bypassU}
    There exists a constant $C$ that depends only on $K$ such that on the event $\mathcalJ_1 \cap \mathcalJ_2$, for any continuous path $P:[0,t] \rightarrow B_2(0)$, there exists a continuous path $\widetilde P$ in $B_2(0) \backslash \mathcalU$ connecting $B_{10 \cdot 2^{-n+\mathsf{a}_\mm}}(P(0))$ and $B_{10 \cdot 2^{-n+\mathsf{a}_\mm}}(P(t))$ that satisfies the inequality:
    \begin{equation}
    \label{eq:lem5.9}
        {\rm len}(\widetilde P ; D_n) \leq C \cdot {\rm len}(P|_{B_2(0) \backslash \mathcalU}; D_n)\,,
    \end{equation}
    where ${\rm len}(P|_{B_2(0) \backslash \mathcalU}; D_n)$ denotes the $D_n$-length of $P$ restricted to the region $B_2(0) \backslash \mathcalU$, as defined in \eqref{eq:restrict-length}.
\end{lemma}
\begin{proof}
Assume that $|P(0) - P(t)|_\infty > 10 \cdot 2^{-n+\mathsf{a}_\mm}$; otherwise, the lemma becomes trivial. Define the domain 
\begin{equation}
\label{eq:def-S}
    S:= B_2(0) \backslash (B_{2^{-n+\mathsf{a}_\mm}}(P(0)) \cup B_{2^{-n+\mathsf{a}_\mm}}(P(t)))\,.
\end{equation}For $1 \leq j \leq \mm-1$, we define the domain
\begin{equation*}
    U_j := \cup_{j \leq p \leq \mm-1}\cup_{z \in \mathcalX_p}  B_{\RC \cdot 2^{-n+\mathsf{a}_p}}(z) \,.
\end{equation*}
Recall from condition~\ref{cond-new-2} in Proposition~\ref{prop:cover} that the $l^\infty$-distance between any two different points in $\mathcalX_p$ is at least $2^{-n + \mathsf{a}_{p+1}}/4$.

We now construct a sequence of paths $\widetilde P^{(\mm)}, \widetilde P^{(\mm-1)}, \ldots, \widetilde P^{(1)}$ inductively, with $\widetilde P^{(\mm)} = P$, such that for each $\mm-1 \geq j \geq 1$:
\begin{enumerate}
    \item\label{lem5.9-cond-1} $\widetilde P^{(j)}$ connects $P(0)$ and $P(t)$ within $B_2(0)$.
    \item\label{lem5.9-cond-2} The path $\widetilde P^{(j)}$ is contained in $B_{8\RC \cdot 2^{-n + \mathsf{a}_j}}(\widetilde P^{(j+1)})$, and $\mathfrak d_\infty(\widetilde P^{(j)}, S \cap U_j) \geq \frac{1}{2}\RC \cdot 2^{-n+\mathsf{a}_j}$.
    \item\label{lem5.9-cond-3} There exists a set $\mathcalW_j \subset \mathcalX_j$, and for each $z \in \mathcalW_j$, there exist an $x_z$ with $|x_z-z|_\infty \leq 2^{-n+\mathsf{a}_{j+1}}$ and a curve $I_z^{(j)}$ in the hypercubic shell $B_{4\RC \cdot 2^{-n+\mathsf{a}_j}}(z) \backslash B_{2\RC \cdot 2^{-n+\mathsf{a}_j}}(z)$ such that
    \begin{equation}
    \label{eq:lem5.9-1}
        {\rm len}(I_z^{(j)}; D_n) < 2\exp(- K \sqrt{\mathsf{a}_j}) D_n\big(\mbox{across }B_{4 \cdot 2^{-n+\mathsf{a}_{j+1}}}(x_z) \backslash B_{2 \cdot 2^{-n+\mathsf{a}_{j+1}}}(x_z)\big)\,.
    \end{equation}
    Moreover, we have
    \begin{equation}
    \label{eq:lem5.9-1-0}
        \widetilde P^{(j)} \subset  \cup_{z \in \mathcalW_j} I_z^{(j)} \cup \widetilde P^{(j+1)}|_{B_2(0) \backslash U_{j}}\,.
    \end{equation}
\end{enumerate}
Recall from \eqref{eq:restrict-length} the definition of the restriction of a curve to a set. Equation~\eqref{eq:lem5.9-1-0} indicates that the curve on the left-hand side can be covered by the curves on the right-hand side. Furthermore, by comparing the $D_n$-lengths of both sides, the length of the left-hand curve is not greater than the total lengths of the right-hand curves. Equations~\eqref{eq:lem5.9-1} and \eqref{eq:lem5.9-1-0} will eventually be used to show that the $D_n$-length of $\widetilde P^{(1)}$, after removing the parts outside of the region $S$ from~\eqref{eq:def-S}, can be upper-bounded by that of $\widetilde P^{(\mm)} = P$.

Let $\widetilde P^{(\mm)} = P$. For any $1 \leq j \leq \mm-1$, assume that the path $\widetilde P^{(j+1)}$ has been constructed and satisfies conditions~\ref{lem5.9-cond-1}, \ref{lem5.9-cond-2}, and \ref{lem5.9-cond-3}. We now construct $\widetilde P^{(j)}$. Our strategy is that each time the path $\widetilde P^{(j+1)}$ is close to a box $B_{\RC \cdot 2^{-n+\mathsf{a}_j}}(z)$ for some $z \in \mathcalX_j$, we will replace the subpath with a detour $I_z^{(j)}$. This detour traverses a larger hypercubic shell $B_{4\RC \cdot 2^{-n+\mathsf{a}_j}}(z) \backslash B_{2\RC \cdot 2^{-n+\mathsf{a}_j}}(z)$, and avoids the box $B_{\RC \cdot 2^{-n+\mathsf{a}_j}}(z)$. We refer to Figure~\ref{fig:short-cut} for an illustration. Equation~\eqref{eq:lem5.9-1} will be a consequence of condition~\ref{cond-new-3} from Proposition~\ref{prop:cover}. 

\begin{figure}[H]
\centering
\includegraphics[scale = 0.5]{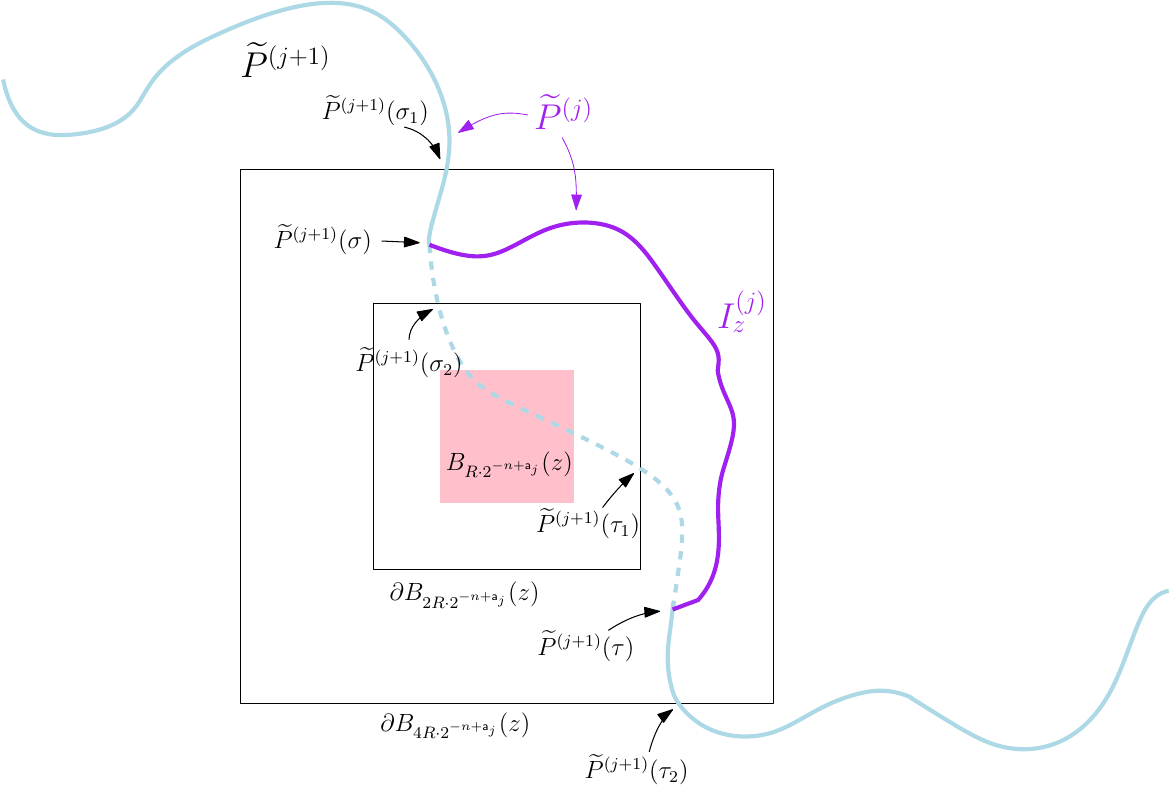}
\caption{Illustration of the paths $\widetilde P^{(j+1)}$ and $\widetilde P^{(j)}$. The pink region corresponds to $B_{\RC \cdot 2^{-n+\mathsf{a}_j}}(z)$. The cyan curve, including both the solid and dashed segments, represents the path $\widetilde P^{(j+1)}$. The purple curve depicts the detour $I_z^{(j)}$, introduced when the curve $\widetilde P^{(j+1)}$ enters the box $B_{\frac{3}{2}\RC \cdot 2^{-n + \mathsf{a}_j}}(z)$. We iteratively replace the cyan curve by the union of the solid part of the cyan and purple curves until we get a path which does not enter $B_{\frac{3}{2}\RC \cdot 2^{-n + \mathsf{a}_j}}(z)$ for any $z\in \mathcalXp_j$. The resulting path is defined to be $\widetilde P^{(j)}$.}
\label{fig:short-cut}
\end{figure}

Let 
\begin{equation} \label{eq:def-mathcalXp}
\mathcalXp_j := \{z \in \mathcalX_j : B_{\RC \cdot 2^{-n+\mathsf{a}_j}}(z) \cap S \neq \emptyset\}\,.
\end{equation}
For each $z \in \mathcalXp_j$, by condition~\ref{cond-new-3} from Proposition~\ref{prop:cover}, there exists an $x_z$ such that $|x_z - z|_\infty \leq 2^{-n + \mathsf{a}_{j+1}}$ and it satisfies the following condition:
   \begin{equation}
        \label{eq:lem5.9-2}
        \begin{aligned}
             &D_n\big(\mbox{around }B_{4\RC \cdot 2^{-n+\mathsf{a}_j}}(z) \backslash B_{2\RC \cdot 2^{-n+\mathsf{a}_j}}(z)\big) \\
             &\qquad\qquad< \exp(- K \sqrt{\mathsf{a}_j}) D_n\big(\mbox{across }B_{4 \cdot 2^{-n+\mathsf{a}_{j+1}}}(x_z) \backslash B_{2 \cdot 2^{-n+\mathsf{a}_{j+1}}}(x_z)\big)\,.
        \end{aligned}
    \end{equation}
Suppose that the path $\widetilde P^{(j+1)}$ enters the box $B_{\frac{3}{2}\RC \cdot 2^{-n+\mathsf{a}_j}}(z)$. Let us parametrize the path $\widetilde P^{(j+1)}$ by $[0,\widetilde t]$, and consider the subsequent times:
\begin{equation*}
    \begin{aligned}
        \sigma_2 &:= \inf \{s >0 : \widetilde P^{(j+1)}(s) \in \partial B_{2 \RC \cdot 2^{-n+\mathsf{a}_j}}(z) \}\,,\\ 
        \sigma_1 &:= \sup \{s <\sigma_2 : \widetilde P^{(j+1)}(s) \in \partial B_{4 \RC \cdot 2^{-n+\mathsf{a}_j}}(z) \}\,,\\
        \tau_1 &:= \sup \{ s > \sigma_2: \widetilde P^{(j+1)}(s) \in \partial B_{2 \RC \cdot 2^{-n+\mathsf{a}_j}}(z) \}\,,\\
         \tau_2 &:= \inf \{ s>\tau_1: \widetilde P^{(j+1)}(s) \in \partial B_{4 \RC \cdot 2^{-n+\mathsf{a}_j}}(z) \}\,.
    \end{aligned}   
\end{equation*}
Then, we have $0 < \sigma_1 < \sigma_2 < \tau_1 <\tau_2 < \widetilde t$ which follows from the definitions of $S$ in~\eqref{eq:def-S} and  $\mathcalXp_j$. By definition, the paths $\widetilde P^{(j+1)}[0,\sigma_1]$ and $\widetilde P^{(j+1)}[\tau_2, \widetilde t]$ do not intersect $B_{\frac{3}{2}\RC \cdot 2^{-n+\mathsf{a}_j}}(z)$. Furthermore, the paths $\widetilde P^{(j+1)}[\sigma_1, \sigma_2]$ and $\widetilde P^{(j+1)}[\tau_1, \tau_2]$ cross the hypercubic shell $B_{4\RC \cdot 2^{-n+\mathsf{a}_j}}(z) \backslash B_{2\RC \cdot 2^{-n+\mathsf{a}_j}}(z)$. By the definition of $D_n(\mbox{around } A)$ from \eqref{eq:def-around} with $B = B_2(0)$, there exists a continuous path $I_z^{(j)}$ and times $\sigma \in [\sigma_1, \sigma_2]$ and $\tau \in [\tau_1, \tau_2]$, such that this path connects $\widetilde P^{(j+1)} (\sigma)$ and $\widetilde P^{(j+1)} (\tau)$ within $(B_{4\RC \cdot 2^{-n+\mathsf{a}_j}}(z) \backslash B_{2\RC \cdot 2^{-n+\mathsf{a}_j}}(z)) \cap B_2(0)$ and satisfies that 
\[
{\rm len}(I_z^{(j)} ; D_n) \leq 2 D_n(\mbox{around }B_{4\RC \cdot 2^{-n+\mathsf{a}_j}}(z) \backslash B_{2\RC \cdot 2^{-n+\mathsf{a}_j}}(z))\,.
\]
 Combining this with \eqref{eq:lem5.9-2}, we obtain
\begin{equation}
    \label{eq:lem5.9-3}
        {\rm len}(I_z^{(j)} ; D_n) < 2 \exp(- K \sqrt{\mathsf{a}_j}) D_n(\mbox{across }B_{4 \cdot 2^{-n+\mathsf{a}_{j+1}}}(x_z) \backslash B_{2 \cdot 2^{-n+\mathsf{a}_{j+1}}}(x_z))\,.
\end{equation}

We now reroute the path $\widetilde P^{(j+1)}$ around one point in $\mathcalXp_j$ at a time, until the resulting path does not get close to any points in $\mathcalXp_j$. That is, let $P_0 = \widetilde P^{(j+1)}$. Inductively, if $k \geq 1$ and $P_{k-1}$ has been defined, we arbitrarily select a point $z$ in $\mathcalXp_j$ such that $P_{k-1}$ enters the box $B_{\frac{3}{2}\RC \cdot 2^{-n+\mathsf{a}_j}}(z)$. We then reroute $P_{k-1}$ to get a new path $P_k$ by replacing the subpath $P_{k-1}[\sigma, \tau]$ (defined analogously to $\widetilde P^{(j+1)}[\sigma, \tau]$) in $P_{k-1}$ with a detour (analogous to $I_z^{(j)}$) that satisfies \eqref{eq:lem5.9-3}. Thus, the new path $P_k$ does not enter $B_{\frac{3}{2}\RC \cdot 2^{-n+\mathsf{a}_j}}(z)$. Since the points in $\mathcalXp_j \subset \mathcalX_j$ are spaced at $|\cdot|_\infty$-distance at least $2^{-n + \mathsf{a}_{j+1}}/4 >10\RC \cdot 2^{-n+\mathsf{a}_j}$ away from each other, we do the rerouting for each $z$ at most once. We will continue to use $I_z^{(j)}$ to denote the detour associated with $z$, albeit with a slight abuse of notation. Since there are only finitely many possibilities for $z \in \mathcalXp_j$, this procedure will eventually terminate. Consequently, we obtain a path that does not enter the box $B_{\frac{3}{2}\RC \cdot 2^{-n+\mathsf{a}_j}}(z)$ for any $z \in \mathcalXp_j$. We denote the resulting path by $\widetilde P^{(j)}$ and use $\mathcalW_j$ to denote the subset of points $z\in \mathcalXp_j$ for which we have added an associated detour $I_z^{(j)}$. It is possible that for some $z \in \mathcalW_j$, its corresponding detour $I_z^{(j)}$ has been removed from the path at a subsequent stage of the above iterative procedure to construct $\wt P^{(j)}$, but we still consider such points to be in $\mathcalW_j$. 

We now verify that the path $ \widetilde P^{(j)}$ satisfies conditions~\ref{lem5.9-cond-1}, \ref{lem5.9-cond-2}, and \ref{lem5.9-cond-3}. Conditions~\ref{lem5.9-cond-1} and \ref{lem5.9-cond-3} are straightforward from the definition and \eqref{eq:lem5.9-3}. From our construction, it follows that $\widetilde P^{(j)}$ is contained in $B_{8\RC \cdot 2^{-n + \mathsf{a}_j}}(\widetilde P^{(j+1)})$, and 
\begin{equation*}
\mathfrak d_\infty(\widetilde P^{(j)}, \cup_{z \in \mathcalX_j} B_{\RC \cdot 2^{-n+\mathsf{a}_j}}(z) \cap S) \geq \frac{1}{2}\RC \cdot 2^{-n+\mathsf{a}_j} . 
\end{equation*}
Combining these with the fact that $\mathfrak d_\infty(\widetilde P^{(j+1)}, S \cap U_{j+1}) \geq \frac{1}{2}\RC \cdot 2^{-n+\mathsf{a}_{j+1}}$, which follows from condition~\ref{lem5.9-cond-2} with $j+1$, we obtain
\begin{equation*}
    \mathfrak d_\infty(\widetilde P^{(j)}, S \cap U_j ) \geq \frac{1}{2}\RC \cdot 2^{-n+\mathsf{a}_j}. 
\end{equation*}
This verifies condition~\ref{lem5.9-cond-2}.

In this way, we can inductively construct the paths $\widetilde P^{(1)}, \ldots, \widetilde P^{(\mm-1)}$ satisfying conditions~\ref{lem5.9-cond-1}, \ref{lem5.9-cond-2}, and \ref{lem5.9-cond-3}. We now restrict $\widetilde P^{(1)}$ to the domain $S$ defined in~\eqref{eq:def-S}, yielding a path $\widetilde P$ that connects $B_{2^{-n+\mathsf{a}_\mm}}(P(0))$ and $B_{2^{-n+\mathsf{a}_\mm}}(P(t))$. Using conditions~\ref{lem5.9-cond-1} and \ref{lem5.9-cond-2} with $j=1$, we obtain $\widetilde P \subset B_2(0) \backslash U_1 = B_2(0) \backslash \mathcalU$. Next, we will show that $\widetilde P$ satisfies \eqref{eq:lem5.9} with a constant $C$ that depends only on $K$. Using \eqref{eq:lem5.9-1-0} with $1 \leq j \leq \mm-1$ and noting that $U_1 = \mathcalU$, we obtain
\begin{equation*}
    \widetilde P \subset \widetilde P^{(1)} \subset \cup_{1 \leq j \leq \mm-1} \cup_{z \in \mathcalW_j} I_z^{(j)} \cup P|_{B_2(0) \backslash \mathcalU}\,.
\end{equation*}
Removing the detours that do not intersect $\widetilde P$ and taking the $D_n$-lengths of both sides yield:
\begin{equation}
\label{eq:lem5.9-4}
    {\rm len}(\widetilde P; D_n) \leq \sum_{j=1}^{\mm-1} \sum_{z \in \mathcalW_j} \mathbbm{1}_{ \{ I_z^{(j)} \cap \widetilde P \neq \emptyset \} } {\rm len}(I_z^{(j)}; D_n) + {\rm len}(P|_{B_2(0) \backslash \mathcalU}; D_n)\,.
\end{equation}
Applying \eqref{eq:lem5.9-1}, we obtain that for each $1 \leq j \leq \mm-1$
\begin{equation*}
\begin{aligned}
   &\quad \sum_{z \in \mathcalW_j} \mathbbm{1}_{ \{ I_z^{(j)} \cap \widetilde P \neq \emptyset \} } {\rm len}(I_z^{(j)}; D_n) \\
   &\leq \sum_{z \in \mathcalW_j} \mathbbm{1}_{ \{ I_z^{(j)} \cap \widetilde P \neq \emptyset \} } 2\exp(- K \sqrt{\mathsf{a}_j}) D_n\big(\mbox{across }B_{4 \cdot 2^{-n+\mathsf{a}_{j+1}}}(x_z) \backslash B_{2 \cdot 2^{-n+\mathsf{a}_{j+1}}}(x_z)\big)\,,
\end{aligned}
\end{equation*}
where $x_z$ is the point associated with $z$ defined in~\eqref{eq:lem5.9-1}. Recall from the condition above \eqref{eq:lem5.9-1} that $|x_z-z|_\infty \leq 2^{-n + \mathsf{a}_{j+1}}$, and hence $I_z^{(j)} \subset B_{4\RC \cdot 2^{-n+\mathsf{a}_j}}(z) \backslash B_{2\RC \cdot 2^{-n+\mathsf{a}_j}}(z) \subset B_{2 \cdot 2^{-n+\mathsf{a}_{j+1}}}(x_z)$. Consequently, for each $1 \leq j \leq \mm-1$, if $I_z^{(j)} \cap \widetilde P \neq \emptyset$, then the path $\widetilde P $ must enter the box $B_{2 \cdot 2^{-n+\mathsf{a}_{j+1}}}(x_z)$. Since $|P(0) - P(t)|_\infty > 10 \cdot 2^{-n+\mathsf{a}_\mm}$, the start and end points of $\widetilde P$ are at $|\cdot|_\infty$-distance at least $8 \cdot 2^{-n + \mathsf{a}_\mm}$ away from each other. Therefore, we can deduce that $\widetilde P $ must cross the hypercubic shell $B_{4 \cdot 2^{-n+\mathsf{a}_{j+1}}}(x_z) \backslash B_{2 \cdot 2^{-n+\mathsf{a}_{j+1}}}(x_z)$. According to the fact that $\mathcalW_j \subset \mathcalX_j$ and condition~\ref{cond-new-2} from Proposition~\ref{prop:cover}, each point on $\widetilde P$ is contained in at most $100^d$ such hypercubic shells. Therefore, for each $1 \leq j \leq \mm-1$,
\begin{equation*}
\begin{aligned}
   \sum_{z \in \mathcalW_j} \mathbbm{1}_{ \{ I_z^{(j)} \cap \widetilde P \neq \emptyset \} } {\rm len}(I_z^{(j)}; D_n) \leq 100^d \times 2\exp(- K \sqrt{\mathsf{a}_j}) {\rm len}(\widetilde P; D_n)\,.
\end{aligned}
\end{equation*}
Combining this with \eqref{eq:lem5.9-4} and the fact that $K>100^d$, we get inequality~\eqref{eq:lem5.9}. Therefore, the path $\widetilde P$ satisfies all the conditions in the lemma. \qedhere
\end{proof}

In the next lemma, we prove that for any continuous path, we can modify it within the region $B_2(0) \backslash \mathcalU$ such that the $D_n$-length and $D_{n+k}$-length satisfy the desired bound in Proposition~\ref{prop:compare}.

\begin{lemma}
\label{lem:modify-outU}
    There exists a constant $C$ that depends only on $L$ such that the following holds on the event $\mathcalJ_1 \cap \mathcalJ_2$.
    \begin{enumerate}
        \item For any continuous path $P:[0,t] \rightarrow B_2(0)$ with $|P(0) - P(t)|_\infty > 10\cdot 2^{-n}$, there exists a continuous path $\widetilde P$ in $B_2(0)$ that connects $P(0)$ and $P(t)$ and satisfies\footnote{Although not required in the latter proof, we remark that the first claim still holds if we upper-bound by ${\rm len}(P|_{B_2(0) \backslash \mathcalU} ; D_{n+k})$ instead of ${\rm len}(P; D_{n+k})$. Moreover, in the second claim, we can require $\widetilde P$ to be contained in $B_2(0) \backslash \mathcalU$.} 
        \begin{equation*}
            {\rm len}(\widetilde P|_{B_2(0) \backslash \mathcalU} ; D_n) \leq C 2^{(1 - \xi Q +\epsilon)k} {\rm len}(P; D_{n+k})\,.
        \end{equation*}
        Here, ${\rm len}(\widetilde P|_{B_2(0) \backslash \mathcalU} ; D_n)$ denotes the $D_n$-length of $\widetilde P$ restricted to the region $B_2(0) \backslash \mathcalU$, as defined in \eqref{eq:restrict-length}.
    \item For any continuous path $P:[0,t] \rightarrow B_2(0) \backslash \mathcalU$ with $|P(0) - P(t)|_\infty > 10\cdot 2^{-n}$, there exists a continuous path $\widetilde P$ in $B_2(0)$ that connects $B_{2^{-n}}(P(0))$ and $B_{2^{-n}}(P(t))$ and satisfies 
    \begin{equation*}
        {\rm len}(\widetilde P ; D_{n+k}) \leq C 2^{-(1 - \xi Q -\epsilon)k} {\rm len}(P; D_n)\,.
    \end{equation*}
    \end{enumerate}
\end{lemma}
\begin{proof}
Note that, in this proof, the constant $C$ depends only on $L$. We begin with the first claim. Let $P:[0,t] \rightarrow B_2(0)$ be a continuous path with $|P(0) - P(t)|_\infty > 10\cdot 2^{-n}$. We will use discrete paths on $2^{-n} \mathbb{Z}^d$ to approximate $P$ outside the domain $\mathcalU$, as illustrated in Figure~\ref{fig:lem5.10}. Conditions~\ref{cond-a1-new-1} and \ref{cond-a1-new-3} from Proposition~\ref{prop:cover} will be used to upper-bound the $D_n$-length of the new path. 

\begin{figure}[H]
\centering
\includegraphics[scale = 0.6]{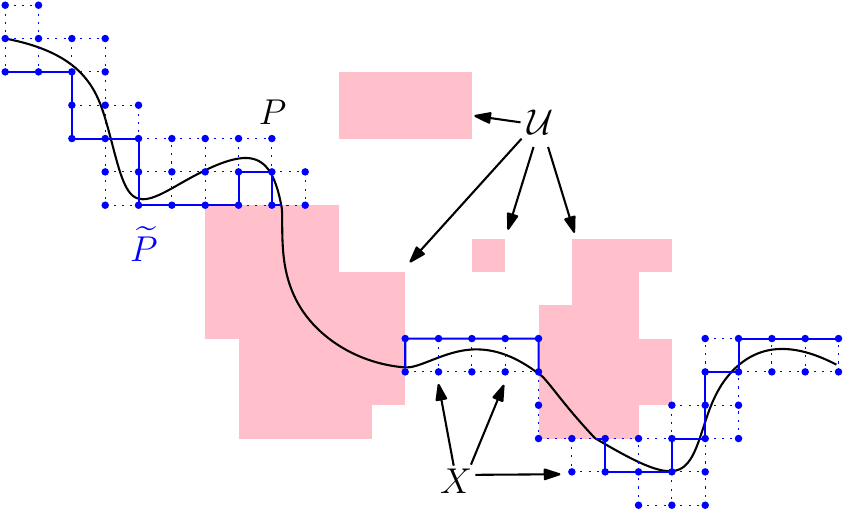}
\caption{Illustration of the domain $\mathcalU$, the path $P$, the modified path $\widetilde P$ and the set $X$. The path $P$ is represented by the black curve, while $\widetilde P$ combines the blue paths with segments of the black curve within $\mathcalU$. The dotted blue curves represent the paths between neighboring points in $X$. }
\label{fig:lem5.10}
\end{figure}

Define the set
\begin{equation*}
    X := \{ x \in \mathscrY_0 \cap B_2(0) \backslash \mathcalU : P \cap \overline{B_{2^{-n}}(x)} \neq \emptyset \}\,.
\end{equation*}
Note that for any $x \in X$, there exists an $\mathsf{a}_1$-box that covers $B_{2 \cdot 2^{-n}}(x) \cap B_2(0)$ and is contained in $B_2(0)$ but not in $\mathcalU$. Moreover, recall from Proposition~\ref{prop:cover} that all the $\mathsf{a}_1$-boxes contained in $B_2(0)$ but not in $\mathcalU$ satisfy the conditions~\ref{cond-a1-new-1} and \ref{cond-a1-new-3} therein. Therefore, for all $x \in X$, we have
\begin{equation}
    \label{eq:lem5.10-cond}
    \begin{aligned}
    &D_{n,n+k}(\mbox{across } B_{2 \cdot 2^{-n}}(x) \backslash B_{2^{-n}}(x)) > \frac{1}{L}2^{-n-(1-\xi Q +\epsilon) k} \quad \mbox{and} \\
    &\sup_{y \in B_{2 \cdot 2^{-n}}(x) \cap B_2(0)} 2^{-n}|\nabla h_n(y)|_\infty \leq L+1 \,.
    \end{aligned}
\end{equation}

By \eqref{eq:def-mathcal-U}, condition~\ref{cond-new-1} from Proposition~\ref{prop:cover}, and the fact that $\RC$ from~\eqref{eq:def-lambda} is an integer, we know that $\mathcalU$ is a union of $\mathsf{a}_0$-boxes. This implies that the boxes $\{B_{2^{-n}}(x) : x \in  \mathscrY_0 \cap B_2(0) \backslash \mathcalU  \}$ can cover $B_2(0) \backslash \mathcalU$. Therefore, for each segment $P[a,b]$ of $P$ that lies within $B_2(0) \backslash \mathcalU$, there exist $z,w \in X$ such that $|z-P(a)|_\infty \leq 2^{-n}$ and $|w-P(b)|_\infty \leq 2^{-n}$. Moreover, there exists a discrete path in $X$ connecting them. This is because for any $x \in X$ considering the first exit time of $P[a,b]$ from the box $\overline{B_{2^{-n}}(x)}$, we can find a vertex $y \in \mathscrY_0 \cap B_2(0) \backslash \mathcalU$ such that $|x-y|_1 = 2^{-n}$, and $P$ also enters the box $\overline{B_{2^{-n}}(y)}$. %\footnote{Some caution is needed when $P$ only hits one corner $w$ of the box $\overline{B_{2^{-n}}(x)}$. In this case, we can show that there exists $y \in \mathscrY_0 \cap B_2(0) \backslash \mathcalU$ such that $|x-y|_1 = |w-y|_1 = 2^{-n}$ and $P$ also enters the box $\overline{B_{2^{-n}}(y)}$. This follows from the observation that, if all such $y$s were contained in $\mathcalU$, then $P \cap \mathcalU \neq \emptyset$ since $\mathcalU$ is a union of $\mathsf{a}_0$-boxes. Since $\mathcalU$ is a union of $\mathsf{a}_0$-boxes, if $y \in \mathscrY_0 \cap \mathcalU$, then $B_{2^{-n}}(y) \subset \mathcalU$. So, we have $y \not \in \mathcalU$, hence $y \in X$.} 
By doing this procedure iteratively, we obtain a discrete path in $X$ that connects $P(a)$ and $P(b)$. In other words, there exists a path $(z_1, \ldots, z_J)$ on $X$ such that
\begin{align*}
    &|z_1 - P(a)|_\infty \leq 2^{-n}, \quad |z_J - P(b)|_\infty \leq 2^{-n}, \quad \mbox{and} \\
    &|z_i - z_{i+1}|_1 = 2^{-n} \quad \forall 1 \leq i \leq J-1\,.
\end{align*}
Therefore, we can replace each maximal segment of $P$ contained in $B_2(0) \backslash \mathcalU$ with the concatenation of the corresponding $(z_1, \ldots, z_J)$ path on $X$ and two straight lines connecting $P(a)$ to $z_1$ and $z_J$ to $P(b)$, respectively. Then, we apply a loop erasure procedure to the obtained path, resulting in a path $\widetilde P$ that hits each point in $X$ at most once, as illustrated in Figure~\ref{fig:lem5.10}. 

Let $x_1,x_2,\ldots,x_T$ denote the points in $X$ that are crossed by $\widetilde P$. Note that $\widetilde P|_{B_2(0) \backslash \mathcalU}$ can be covered by a collection of line segments with $l^\infty$-lengths at most $2^{-n}$ started from $\{x_i\}_{1 \leq i \leq T}$. Combining this with the second inequality in \eqref{eq:lem5.10-cond}, we obtain
\begin{equation}
\label{eq:lem5.10-1}
    {\rm len}(\widetilde P|_{B_2(0) \backslash \mathcalU}; D_n) \leq \sum_{i=1}^T C 2^{-n} e^{\xi h_n(x_i)},
\end{equation}
where the constant $C$ depends only on $L$. Furthermore, by the definition of the set $X$, for each $1 \leq i \leq T$, the path $P$ must cross the hypercubic shell $B_{2 \cdot 2^{-n}}(x_i) \backslash B_{2^{-n}}(x_i)$. In addition, each point on $P$ is contained in at most $5^d$ such hypercubic shells. Therefore,
\begin{equation*}
\begin{aligned}
    {\rm len}(P; D_{n+k}) &\geq 5^{-d} \sum_{i=1}^T D_{n+k}(\mbox{across }B_{2 \cdot 2^{-n}}(x_i) \backslash B_{2^{-n}}(x_i))\,.
\end{aligned}
\end{equation*}
Combining this with \eqref{eq:lem5.10-cond}, we obtain
\begin{equation}
\label{eq:lem5.10-2}
\begin{aligned}
    {\rm len}(P; D_{n+k}) &\geq C \sum_{i=1}^T e^{\xi h_n(x_i)} D_{n,n+k}(\mbox{across }B_{2 \cdot 2^{-n}}(x_i) \backslash B_{2^{-n}}(x_i))\\
    &\geq C 2^{-n - (1-\xi Q+\epsilon)k} \sum_{i=1}^T e^{\xi h_n(x_i)}. 
    \end{aligned}
\end{equation}
Combining \eqref{eq:lem5.10-1} and \eqref{eq:lem5.10-2} yields the desired inequality in the first claim. Therefore, the first claim holds with the path $\widetilde P$.

The second claim can be proved in a similar way. Let $P':[0,t] \rightarrow B_2(0) \backslash \mathcalU$ be a continuous path with $|P'(0) - P'(t)|_\infty > 10\cdot 2^{-n}$. We will first use a discrete path on $2^{-n} \mathbb{Z}^d$ to approximate $P'$, and then use conditions~\ref{cond-a1-new-1} and \ref{cond-a1-new-2} from Proposition~\ref{prop:cover} to construct a path of typical $D_{n+k}$-length along the discrete path. Define the set
\begin{equation*}
    X' := \{ x \in \mathscrY_0 \cap B_2(0) \backslash \mathcalU : P' \cap \overline{B_{2^{-n}}(x)} \neq \emptyset \}\,.
\end{equation*}
Similar to \eqref{eq:lem5.10-cond}, using conditions~\ref{cond-a1-new-1} and \ref{cond-a1-new-2} from Proposition~\ref{prop:cover}, we can show that for all $x \in X'$:
\begin{equation}
    \label{eq:lem5.10-cond2}
    \begin{aligned}
    &{\rm Diam}_{n,n+k}(B_{2^{-n}}(x); B_{2 \cdot 2^{-n}}(x)) < L2^{-n-(1-\xi Q - \epsilon) k} \quad \mbox{and} \\
    &\sup_{y \in B_{2 \cdot 2^{-n}}(x) \cap B_2(0)} 2^{-n}|\nabla h_n(y)|_\infty \leq L+1 \,.
    \end{aligned}
\end{equation}

Similarly to before, according to the definition of $X'$ and the fact that $P' \subset B_2(0) \backslash \mathcalU$, there exists a self-avoiding path $(w_1, \ldots, w_{J'})$ on $X'$ such that 
\begin{align*}
    &|w_1 - P'(0)|_\infty \leq 2^{-n}, \quad |w_{J'} - P'(t)|_\infty \leq 2^{-n}, \quad \mbox{and} \\
    &|w_i - w_{i+1}|_1 = 2^{-n} \quad \forall 1 \leq i \leq J'-1\,.
\end{align*}
Using \eqref{eq:lem5.10-cond2}, we can connect $w_i$ to $w_{i+1}$ with a path of $D_{n+k}$-length at most $C 2^{-n-(1-\xi Q - \epsilon) k} e^{\xi h_n(w_i)}$ within the box $B_{2 \cdot 2^{-n}}(w_i)$. Here, $C$ is a constant that depends only on $L$. By concatenating these paths, we obtain a continuous path $\widetilde P'$ from $B_{2^{-n}}(P'(0))$ to $B_{2^{-n}}(P'(t))$ in the box $B_2(0)$.\footnote{Careful readers may worry that the path $\widetilde P'$ could leave $B_2(0)$ when $\mathfrak{d}(w_i, \partial B_2(0)) = 2^{-n}$. To address this issue, we can introduce an additional condition: ${\rm Diam}_{n,n+k}(w, w+ \sigma e_i; B_{2 \cdot 2^{-n}}(w) \cap B_2(0)) < L 2^{-n - (1-\xi Q - \epsilon)k}$ for all such $w$ and each $\sigma \in \{-1,1\}$ and $1 \leq i \leq d$ such that $w+\sigma e_i \in B_2(0)$. This condition will be included in condition~\ref{cond-a1-new-2} of Proposition~\ref{prop:cover} and also in condition~\ref{condition-a1-nice-b} of Definition~\ref{def:nice-boxes}. Note that, with this additional condition, Lemma~\ref{lem:est-q1} still holds by using Lemma~\ref{lem:distance-any-point}, and the rest of the proof remains unchanged.} By the construction, we have
\begin{equation*}
   {\rm len}(\widetilde P'; D_{n+k}) \leq C \sum_{x \in X'} 2^{-n-(1-\xi Q - \epsilon) k} e^{\xi h_n(x)}. 
\end{equation*}
Furthermore, the path $P'$ must cross the hypercubic shell $B_{2 \cdot 2^{-n}}(x) \backslash B_{2^{-n}}(x)$ for all $x \in X'$. Moreover, each point on $P$ is contained in at most $5^d$ such hypercubic shells. Combining these with \eqref{eq:lem5.10-cond2}, we obtain
\begin{equation*}
\begin{aligned}
    {\rm len}(P'; D_n) &\geq 5^{-d} \sum_{x \in X'} D_n(\mbox{across }B_{2 \cdot 2^{-n}}(x) \backslash B_{2^{-n}}(x)) \geq  C \sum_{x \in X'} 2^{-n} e^{\xi h_n(x)}.
\end{aligned}
\end{equation*}
Combining the above two inequalities yields the desired inequality in the second claim. Therefore, the second claim holds with the path $\widetilde P'$.
\end{proof}

As a consequence of Lemmas~\ref{lem:bypassU} and \ref{lem:modify-outU}, we conclude that
\begin{proposition}
\label{prop:compare-1}
    There exists a constant $C$ that depends only on $L$ and $K$ such that, on the event $\mathcalJ_1 \cap \mathcalJ_2$, the following inequalities hold for all $z,w \in B_2(0)$:
    \begin{equation*}
    \begin{aligned}
        &D_n( B_{30 \cdot 2^{-n+\mathsf{a}_\mm}}(z), B_{30 \cdot 2^{-n+\mathsf{a}_\mm}}(w) ; B_2(0)) \leq C 2^{(1 - \xi Q +\epsilon)k} D_{n+k}(z,w; B_2(0))\,, \\
        &D_{n+k}( B_{30 \cdot 2^{-n+\mathsf{a}_\mm}}(z), B_{30 \cdot 2^{-n+\mathsf{a}_\mm}}(w) ; B_2(0)) \leq C 2^{-(1 - \xi Q -\epsilon)k} D_n(z,w; B_2(0))\,.
    \end{aligned}
    \end{equation*}
\end{proposition}
\begin{proof}
    Fix $z,w \in B_2(0)$. We assume that $|z-w|_\infty > 30 \cdot 2^{-n+\mathsf{a}_\mm}$; otherwise, the proposition becomes trivial. We start with the first inequality. By definition, there exists a continuous path $P$ that connects $z$ and $w$ in $B_2(0)$ with ${\rm len}(P; D_{n+k}) \leq 2 D_{n+k}(z,w; B_2(0))$. Applying the first claim in Lemma~\ref{lem:modify-outU}, we obtain a continuous path $P'$ that connects $z$ and $w$ in $B_2(0)$ with $${\rm len}(P'|_{B_2(0) \backslash \mathcalU}; D_n) \leq C 2^{(1-\xi Q+\epsilon)k} {\rm len}(P; D_{n+k}) \leq  C 2^{(1-\xi Q+\epsilon)k} D_{n+k}(z,w; B_2(0)).$$ Then, applying Lemma~\ref{lem:bypassU}, we obtain a continuous path $P''$ that connects $B_{10 \cdot 2^{-n+\mathsf{a}_\mm}}(z)$ and $B_{10 \cdot 2^{-n+\mathsf{a}_\mm}}(w)$ in $B_2(0)$ with $${\rm len}(P''; D_n) \leq C \cdot {\rm len}(P'|_{B_2(0) \backslash \mathcalU}; D_n) \leq  C 2^{(1-\xi Q+\epsilon)k} D_{n+k}(z,w; B_2(0)).$$ This confirms the first inequality.

    The second inequality can be proved in a similar way. By definition, there exists a continuous path $\widehat P$ that connects $z$ and $w$ in $B_2(0)$ with ${\rm len}(\widehat P; D_n) \leq 2 D_n(z,w; B_2(0))$. Applying Lemma~\ref{lem:bypassU}, we obtain a continuous path $\widehat P'$ that connects $B_{10 \cdot 2^{-n+\mathsf{a}_\mm}}(z)$ and $B_{10 \cdot 2^{-n+\mathsf{a}_\mm}}(w)$ in $B_2(0) \backslash \mathcalU$ with $${\rm len}(\widehat P'; D_n) \leq C \cdot {\rm len}(\widehat P; D_n) \leq C \cdot D_n(z,w; B_2(0)).$$ Then, applying the second claim in Lemma~\ref{lem:modify-outU} and using the fact that $|z-w|_\infty > 30 \cdot 2^{-n+\mathsf{a}_\mm}$, we obtain a continuous path $\widehat P''$ that connects $B_{11 \cdot 2^{-n+\mathsf{a}_\mm}}(z)$ and $B_{11 \cdot 2^{-n+\mathsf{a}_\mm}}(w)$ in $B_2(0)$ with $${\rm len}(\widehat P''; D_{n+k}) \leq C 2^{-(1-\xi Q - \epsilon)k} {\rm len}(\widehat P'; D_n) \leq C 2^{-(1-\xi Q - \epsilon)k} D_n(z,w; B_2(0)) .$$ This confirms the second inequality.
\end{proof}

In order to prove Proposition~\ref{prop:compare}, it remains to show that, for $l \in \{n, n+k\}$, $D_l(z,w ; B_2(0))$ can be upper-bounded by $D_l( B_{30 \cdot 2^{-n+\mathsf{a}_\mm}}(z), B_{30 \cdot 2^{-n+\mathsf{a}_\mm}}(w) ; B_2(0))$ when $z$ and $w$ are not too close. This claim is justified in the following lemma.

\begin{lemma}
    \label{lem:compare-j3}
    There exists a constant $C>0$ such that for all integers $n \geq 1$ and $0 \leq k \leq n/C$, with probability greater than $1 - C2^{-n/C}$, the following bound holds simultaneously for each pair of points $z,w \in B_1(0)$ with $|z-w|_\infty \geq 2^{-n/C}$:
    \begin{equation}
    \label{eq:lem-j3}
        D_{n+k}(z,w; B_2(0)) \leq 3 \cdot D_{n+k}( B_{30 \cdot 2^{-n+\mathsf{a}_\mm}}(z), B_{30 \cdot 2^{-n+\mathsf{a}_\mm}}(w) ; B_2(0))\,.
    \end{equation}
\end{lemma}

\begin{proof}
Using the triangle inequality, we obtain
\begin{equation}
\label{eq:j3-1}
    \begin{aligned}
        D_{n+k}(z,w; B_2(0)) 
        &\leq  {\rm Diam}_{n+k}(B_{30 \cdot 2^{-n + \mathsf{a}_\mm}}(z); B_2(0))\\
        &\qquad+ D_{n+k}(B_{30 \cdot 2^{-n + \mathsf{a}_\mm}}(z), B_{30 \cdot 2^{-n + \mathsf{a}_\mm}}(w); B_2(0))\\
        &\qquad\qquad + {\rm Diam}_{n+k}(B_{30 \cdot 2^{-n + \mathsf{a}_\mm}}(z); B_2(0))\,.
    \end{aligned}
\end{equation}
In order to prove~\eqref{eq:lem-j3}, it suffices to show that the diameter terms on the right-hand side of \eqref{eq:j3-1} are smaller than the distance term, with high probability.

We first show that there exists $u>0$ such that all integers $n \geq k \geq 0$:
    \begin{equation}
    \label{eq:lem5.12-1}
    \begin{aligned}
        &\mathbb{P}\Big[ {\rm Diam}_{n+k}(B_{30 \cdot 2^{-n + \mathsf{a}_\mm}}(x); B_2(0)) < 2^{-(1-\xi Q + u)(n+k)} \quad \forall x \in B_1(0) \Big]\\
        &\quad \geq 1-C2^{-n/C}.
    \end{aligned}
    \end{equation}
    This result is a consequence of Proposition~\ref{prop:chain-2} and the estimates from Lemma~\ref{lem:est-e-m}. Let $r$ be an integer such that $60 \cdot 2^{-n + \mathsf{a}_\mm} < 2^{-r} \leq 120 \cdot 2^{-n + \mathsf{a}_\mm}$. Applying Proposition~\ref{prop:chain-2} with $\eta = \xi (Q-\sqrt{2d})/2 $ and $n+k$ in place of $n$, we obtain that for any $q \in (0,1)$, on the event $(\cap_{r \leq p \leq n +k -1} \mathscrF_{p,n+k}) \cap \mathscrG_{n+k}$ (recall from Subsection~\ref{subsec:chaining}) \footnote{The box $B_{30 \cdot 2^{-n + \mathsf{a}_\mm}}(x)$ may not be entirely contained in $B_1(0)$. However, the argument in Proposition~\ref{prop:chain-2} holds for any $x,y \in B_{3/2}(0)$ with $|x-y|_\infty \leq 2^{-r}$.}
    \begin{equation}
    \label{eq:lem5.12-1-1}
        {\rm Diam}_{n+k}(B_{30 \cdot 2^{-n + \mathsf{a}_\mm}}(x); B_2(0)) \leq C_3 \sum_{p = r}^{n+k} p^{2d} 2^{-(1-\xi Q +\eta)p} a_{n+k-p}^{(q)} \quad \forall x \in B_1(0)\,.
    \end{equation}
    By Claim (1) of Lemma~\ref{lem:est-e-m}, we can choose some $q \in (0,1)$ such that \begin{equation}
    \label{eq:lem5.12-1-2}
        \mathbb{P}[(\cap_{r \leq p \leq n+k -1} \mathscrF_{p,n+k}) \cap \mathscrG_{n+k}] \geq 1-C2^{-n/C}.
    \end{equation} Fix such $q$. By Proposition~\ref{prop:exponent}, Lemma~\ref{lem:compare-median-p}, and the facts that $n \geq k$ and $r \geq cn$, we can choose some $u>0$ such that
    \begin{equation}
    \label{eq:lem5.12-1-3}
        C_3 \sum_{p = r}^{n+k} p^{2d} 2^{-(1-\xi Q +\eta)p} a_{n+k-p}^{(q)} < C 2^{-(1-\xi Q + u)(n+k)}.
    \end{equation}
    Combining \eqref{eq:lem5.12-1-1} and \eqref{eq:lem5.12-1-3} with the estimates \eqref{eq:lem5.12-1-2}, we conclude \eqref{eq:lem5.12-1}. 

    We now claim that there exists $\delta \in (0, \lambdaC)$ such that for all integers $n \geq 1$ and $0 \leq k \leq \delta n$:
    \begin{equation}
    \label{eq:lem5.12-2}
    \begin{aligned}
    &\mathbb{P}\Big[D_{n+k}(\mbox{across }B_{4\cdot 2^{-\delta n}}(x) \backslash B_{2 \cdot 2^{-\delta n}}(x)) \geq 2^{-(1-\xi Q + u)(n+k)},  \forall x \in B_1(0) \cap 2^{-\delta n} \mathbb{Z}^d \Big] \\
    &\qquad\qquad\qquad \geq 1-C2^{-n/C}.
    \end{aligned}
    \end{equation}

    Let us first prove the lemma assuming this claim. Fix any $z,w \in B_1(0)$ with $|z-w|_\infty \geq 60 \cdot 2^{-n + \mathsf{a}_\mm} + 8 \cdot 2^{-\delta n}$. There exists an $x \in B_1(0) \cap 2^{-\delta n} \mathbb{Z}^d$ such that $|x-z|_\infty \leq 2^{-\delta n}$. Recalling from \eqref{eq:def-am}, we have $-n + \mathsf{a}_\mm \leq - \lambdaC n$. Consequently, we have $30 \cdot 2^{-n + \mathsf{a}_\mm} < 2^{-\delta n}$ for all sufficiently large $n$. It then follows that, for all sufficiently large $n$, any path connecting $B_{30 \cdot 2^{-n + \mathsf{a}_\mm}}(z)$ and $B_{30 \cdot 2^{-n + \mathsf{a}_\mm}}(w)$ must enter the box $B_{2 \cdot 2^{-\delta n}}(x)$ and cross $B_{4 \cdot 2^{-\delta n}}(x) \backslash B_{2 \cdot 2^{-\delta n}}(x)$. Therefore,
    \begin{equation*}
    \begin{aligned}
        &D_{n+k}(B_{30 \cdot 2^{-n + \mathsf{a}_\mm}}(z), B_{30 \cdot 2^{-n + \mathsf{a}_\mm}}(w); B_2(0)) \\
        &\qquad \qquad \geq D_{n+k}(\mbox{across }B_{4 \cdot 2^{-\delta n}}(x) \backslash B_{2 \cdot 2^{-\delta n}}(x))\,.
    \end{aligned}
    \end{equation*}
    Combining this with the claim~\eqref{eq:lem5.12-2} yields that for all $n \geq 1$ and $0 \leq k \leq \delta n$, with probability greater than $1-C2^{-n/C}$, the following bound holds for all $z,w \in B_1(0)$ with $|z-w|_\infty \geq 60 \cdot 2^{-n + \mathsf{a}_\mm} + 8 \cdot 2^{-\delta n}$:
    \begin{equation*}
        D_{n+k}(B_{30 \cdot 2^{-n + \mathsf{a}_\mm}}(z), B_{30 \cdot 2^{-n + \mathsf{a}_\mm}}(w); B_2(0)) \geq 2^{-(1-\xi Q + u)(n+k)}.
    \end{equation*}
    This result, combined with \eqref{eq:lem5.12-1} and \eqref{eq:j3-1}, concludes the lemma by choosing a sufficiently large $C$.

    Next, we will prove the claim~\eqref{eq:lem5.12-2}, which is a consequence of the scaling property from Lemma~\ref{lem:basic-LFPP} and the concentration bound from Lemma~\ref{lem:super-exponential-2}. Let $\delta>0$ and $v>0$ be two small constants to be chosen. Let $n \geq 1$ and $0 \leq k \leq \delta n$ be integers, and define $l = \lfloor \delta n \rfloor$. Since $h_{n+k} = h_l + h_{l,n+k}$, we obtain 
    \begin{equation*}
    \begin{aligned}
&D_{n+k}(\mbox{across }B_{4 \cdot 2^{-\delta n}}(0) \backslash B_{2 \cdot 2^{-\delta n}}(0)) \\
&\qquad \geq \exp\big( \xi \inf_{y \in B_{4 \cdot 2^{-l}}(0)} h_l(y) \big)  D_{l, n+k}(\mbox{across }B_{4 \cdot 2^{-\delta n}}(0) \backslash B_{2 \cdot 2^{-\delta n}}(0))\,.
\end{aligned}
    \end{equation*}
    Therefore, 
    \begin{equation}
    \label{eq:lem5.14-relation}
    \begin{aligned}
        & \big{\{} D_{n+k}(\mbox{across }B_{4 \cdot 2^{-\delta n}}(0) \backslash B_{2 \cdot 2^{-\delta n}}(0)) < 2^{-(1-\xi Q + u)(n+k)} \big{\}}\\
        &\qquad \subset\big{\{}\inf_{y \in B_{4 \cdot 2^{-l}}(0)} h_l(y) < s \big{\}} \\
        &\qquad \qquad\cup \big{\{} D_{l, n+k}(\mbox{across }B_{4 \cdot 2^{-\delta n}}(0) \backslash B_{2 \cdot 2^{-\delta n}}(0)) < 2^{-l + (-1 +\xi Q - v)(n+k-l)} \big{\}}\,,
    \end{aligned}
    \end{equation}
    where $s = \frac{\log 2}{\xi} [v(n + k-l)+\xi Q l - u(n+k)] $. Applying the scaling property from Lemma~\ref{lem:basic-LFPP}, with $(l,n+k)$ in place of $(m,n)$, and the super-exponential concentration bound from Lemma~\ref{lem:super-exponential-2}, we have
    \begin{equation*}
    \begin{aligned}
        & \mathbb{P}\Big[D_{l, n+k}(\mbox{across }B_{4 \cdot 2^{-\delta n}}(0) \backslash B_{2 \cdot 2^{-\delta n}}(0)) < 2^{-l + (-1 +\xi Q - v)(n+k-l)} \Big] \\
        &\qquad= \mathbb{P}\Big[D_{n+k-l}(\mbox{across }B_{4 \cdot 2^{l-\delta n}}(0) \backslash B_{2 \cdot 2^{l-\delta n}}(0)) < 2^{(-1 +\xi Q - v)(n+k-l)} \Big] \\
        &\qquad \qquad \leq Ce^{-n \log n}.
    \end{aligned}
    \end{equation*}
    By the definitions of $s$ and $l$, we can choose sufficiently small $\delta$ and $v$ such that for all sufficiently large $n$ and $0 \leq k \leq \delta n$, the inequality $s < ( - 1- 2 \sqrt{d} \log 2)l $ holds. Combining this inequality with Claim (5) from Lemma~\ref{lem:field-estimate} for $n=l$, we obtain
    \begin{equation*}
        \mathbb{P}\Big[\inf_{y \in B_{4 \cdot 2^{-l}}(0)} h_l(y) < s \Big] \leq C 4^{-dl}.
    \end{equation*}
    Combining the relation~\eqref{eq:lem5.14-relation} with the above two estimates, we conclude that
    \begin{equation*}
        \mathbb{P}\Big[D_{n+k}(\mbox{across }B_{4\cdot 2^{-\delta n}}(0) \backslash B_{2 \cdot 2^{-\delta n}}(0)) < 2^{-(1-\xi Q + u)(n+k)} \Big] \leq C 4^{-dl}.
    \end{equation*}
    This, combined with the fact that $|B_1(0) \cap 2^{-\delta n} \mathbb{Z}^d| \leq C2^{dl}$, implies \eqref{eq:lem5.12-2}. 
\end{proof}

We now complete the proof of Proposition~\ref{prop:compare}.

\begin{proof}[Proof of Proposition~\ref{prop:compare}]
Proposition~\ref{prop:compare} follows by combining Proposition~\ref{prop:compare-1} with Lemma~\ref{lem:compare-j3} and the estimates of $\mathbb{P}[\mathcalJ_1 \cap \mathcalJ_2]$ from Lemma~\ref{lem:est-j12}.
\end{proof}

\section{Proof of Theorem~\ref{thm:tightness}}
\label{sec:final-proof}

In this section, we complete the proof of Theorem~\ref{thm:tightness}. We continue to assume that $Q(\xi)>\sqrt{2d}$. In Subsection~\ref{subsec:tightness}, we establish the tightness of $D_n$ when normalized by $a_n^{(q)}$ (recall from \eqref{def:quantile}) for $q$ close to one. Subsequently, in Subsection~\ref{subsec:non-degenerate}, we extend the tightness result to $\lambda_n$ and demonstrate that every subsequential limit is a metric on $\mathbb{R}^d$.

\subsection{Tightness of exponential metrics}
\label{subsec:tightness}
In this subsection, we establish the tightness of the exponential metric $D_n$ when normalized by $a_n^{(q)}$ for $q$ close to one. The main inputs are Proposition~\ref{prop:chain-2} and Corollary~\ref{cor:an-compare}. We first work with $q$ close to one because the bound in Proposition~\ref{prop:chain-2}, which is based on Lemma~\ref{lem:est-e-m}, holds with high probability only in this case. However, once we show that any subsequential limit is a metric, it follows that $a_n^{(q)}$ are up-to-constants equivalent for different $q \in (0,1)$.

\begin{lemma}
    \label{lem:internal-B_2(0)}
    For any $0< \beta < \xi (Q - \sqrt{2d})$, there exists a constant $q_0 = q_0(\beta) \in (0,1)$ such that 
    \begin{equation*}
        \Big{\{}\sup_{x,y \in B_1(0)} \frac{D_n(x,y; B_2(0))}{a_n^{(q_0)} |x-y|_\infty^\beta} \Big{\}}_{n \geq 0} \quad \mbox{is tight}\,.
    \end{equation*} 
\end{lemma}
\begin{proof}
    Fix $0< \beta < \xi(Q-\sqrt{2d})$. Let $q \in (0,1)$ be a constant to be chosen. Let $\eta$ and $\epsilon$ be two positive constants satisfying
    \begin{equation*}
        0 < \eta < \xi(Q-\sqrt{2d}) \quad \mbox{and} \quad \eta > \beta + \epsilon\,.
    \end{equation*}
    Applying Proposition~\ref{prop:chain-2} with the above choices of $\eta$ and $q$, we obtain that for all integers $n \geq r>100$, on the event $(\cap_{r \leq m \leq n -1} \mathscrF_{m,n}) \cap \mathscrG_n$ (recall from Subsection~\ref{subsec:chaining}),
    \begin{equation*}
    \begin{aligned}
        \sup_{\substack{x,y \in B_1(0)\\ |x-y|_\infty \leq 2^{-k}}} D_n(x,y;B_2(0)) \leq C_3 \sum_{m = k}^n m^{2d} 2^{-(1-\xi Q +\eta)m} a_{n-m}^{(q)} \quad \forall r \leq k \leq n\,.
    \end{aligned}
    \end{equation*}
    Combining this with Corollary~\ref{cor:an-compare} yields that on the event $(\cap_{r \leq m \leq n -1} \mathscrF_{m,n}) \cap \mathscrG_n$, it holds for all $r \leq k \leq n$ that
    \begin{equation}
    \label{eq:lem6.1-1}
    \sup_{\substack{x,y \in B_1(0)\\ |x-y|_\infty \leq 2^{-k}}} D_n(x,y;B_2(0))\leq C \sum_{m=k}^n m^{2d} 2^{(-\eta +\epsilon)m} a_{n}^{(q + 2^{-cn})} \leq C 2^{-\beta k }a_{n}^{(q + 2^{-cn})}  .
    \end{equation}
    In the last inequality, we used the fact that $\eta > \beta +\epsilon$ and increased the constant $C$. Moreover, on the event $\mathscrG_n$ from~\eqref{def:e-n-n} with the above choice of $\eta$, we have that for all $x,y \in B_1(0)$ with $|x-y|_\infty \leq 2^{-n}$:
    \begin{equation}
    \label{eq:lem6.1-2}
        D_n(x,y; B_2(0)) \leq d  2^{(\xi Q - \eta)n} |x-y|_\infty \leq d  2^{-n(1- \xi Q + \eta - \beta)} |x-y|_\infty^\beta \leq C  a_n^{(q)} |x-y|_\infty^\beta \,.
    \end{equation}
    In the second inequality, we used the fact that $|x-y|_\infty \leq 2^{-n}$ and $\beta < 1$, which follows from Lemma~\ref{lem:Q-lower} and the assumption that $\beta < \xi (Q -\sqrt{2d})$. In the last inequality, we used the fact that $\eta > \beta +\epsilon$ and $a_n^{(q)} = 2^{-(1- \xi Q )n + o(n)}$, which follows from Proposition~\ref{prop:exponent} and Lemma~\ref{lem:compare-median-p}.
    
    By Claim (1) in Lemma~\ref{lem:est-e-m}, we can select a $q \in (0,1)$ such that
    \begin{equation*}
        \lim_{r \rightarrow \infty} \sup_{n > r} \mathbb{P}[(\cap_{r \leq m \leq n -1} \mathscrF_{m,n}) \cap \mathscrG_n] = 1\,.
    \end{equation*}
    Combining this estimate with \eqref{eq:lem6.1-1} and \eqref{eq:lem6.1-2}, and increasing the value of $q$, yields the desired lemma.
\end{proof}

We now extend Lemma~\ref{lem:internal-B_2(0)} to the internal metric $D_{m,n}$ of an arbitrary open set for any integer $m \geq 0$.

\begin{lemma}
    \label{lem:internal-B_N(0)}
    Fix a connected open set $U$ and a bounded connected open set $V$ such that $\overline{V} \subset U$. Let $0<\beta < \xi (Q - \sqrt{2d})$, and let $q_1 = (q_0+1)/2$, where $q_0$ is the constant from Lemma~\ref{lem:internal-B_2(0)} with the same choice of $\beta$. Then, for any fixed integer $m \geq 0$, we have
    \begin{equation*}
        \Big{\{}\sup_{x,y \in V} \frac{D_{m,n}(x,y; U)}{a_n^{(q_1)} |x-y|_\infty^\beta} \Big{\}}_{n \geq 0} \quad \mbox{is tight}\,.
    \end{equation*} 
\end{lemma}
\begin{proof}
    We first prove the case where $m=0$. By Lemma~\ref{lem:internal-B_2(0)}, we have that $\sup_{x,y \in B_1(0)} \frac{D_n(x,y; B_2(0))}{a_n^{(q_0)} |x-y|_\infty^\beta}$ is tight for $n \geq 1$. By the scaling relation from Lemma~\ref{lem:basic-LFPP}, we obtain that for any integer $k \geq 1$
    \begin{equation*}
        (D_n(x,y; B_2(0)))_{x,y \in B_1(0)} \overset{d}{=} 2^k D_{k,n+k}(2^{-k}x,2^{-k}y; B_{2^{-k+1}}(0))_{x,y \in B_1(0)}\,.
    \end{equation*}
    Together with the fact that for all $x,y \in B_{2^{-k}}(0)$
    \begin{equation}
    \label{eq:lem6.2-1}
     \inf_{z \in B_{2^{-k+1}}(0)}e^{ \xi h_{0,k}(z)}\leq \frac{D_{n+k}(x,y; B_{2^{-k+1}}(0))}{D_{k,n+k}(x,y; B_{2^{-k+1}}(0))}\leq \sup_{z \in B_{2^{-k+1}}(0)}e^{ \xi h_{0,k}(z)},
    \end{equation}
    we obtain that for any fixed integer $k \geq 1$
    \begin{equation}
    \label{eq:lem6.2-2}
        \Big{\{}\sup_{x,y \in B_{2^{-k}}(0)} \frac{D_n(x,y; B_{2^{-k+1}}(0)))}{a_n^{(q_1)} |x-y|_\infty^\beta} \Big{\}}_{n \geq 0} \quad \mbox{is tight}\,.
    \end{equation}
    Here we also used the fact that for some $c>0$ (depending on $k$), the inequality $a_{n+k}^{(q_1)} \geq c a_{n}^{(q_0)}$ holds for all sufficiently large $n$, which follows from Corollary~\ref{cor:an-compare}.
    
    Given the sets $U$ and $V$ as stated in the lemma, we can choose a sufficiently large integer $k$ and cover $V$ with a finite number of boxes $\{ B_{2^{-k}}(x_i) \}_{i \geq 1}$. These boxes are connected, and for each $i$, the box $B_{2^{-k+1}}(x_i)$ is contained within $U$. This allows us to connect any two points in $V$ within $U$ via paths in these $B_{2^{-k+1}}(x_i)$ boxes. This observation, combined with \eqref{eq:lem6.2-2}, yields the lemma in the case where $m=0$. The case where $m \geq 1$ follows directly from the case where $m=0$ and the following fact:
    \begin{equation*}
     \inf_{z \in U}e^{-\xi h_m(z)} \leq \frac{D_{m,n}(x,y;U)}{D_n(x,y;U)} \leq \sup_{z \in U}e^{-\xi h_m(z)} \quad \forall x,y \in U\,. \qedhere
    \end{equation*}
\end{proof}

As a direct consequence of the above lemma, we deduce the tightness of the exponential metric when normalized by $a_n^{(q)}$ for $q$ close to 1. Recall that a \textit{metric} is a symmetric non-negative function $D: \mathbb{R}^d \times \mathbb{R}^d \to [0,\infty)$ that satisfies the triangle inequality, $D(x,x) = 0$, and $D(x,y)>0$ for all $x \neq y$. A \textit{pseudo-metric} is a symmetric non-negative function that satisfies the triangle inequality and $D(x,x) = 0$ but allows $D(x,y) = 0$ for some $x \neq y$.
\begin{proposition}
\label{prop:tightness}
    There exists $q_2 \in (0,1)$ such that for any open connected set $U \subset \mathbb{R}^d$ (including $U = \mathbb{R}^d$), $D_n(\cdot,\cdot;U)/a_n^{(q_2)}$ is tight with respect to the topology of uniform convergence on compact subsets of $U \times U$. Moreover, any subsequantial limit is a.s.\ a pseudo-metric on $U$ and is locally H\"older continuous.
\end{proposition}
\begin{proof}
    Taking $\beta = \xi (Q - \sqrt{2d}) / 2$ in Lemma~\ref{lem:internal-B_N(0)} and using the Arzel\`a-Ascoli theorem, we obtain the tightness part. According to Skorohod's representation theorem, each subsequential limit can be interpreted as the limit of almost sure convergence under some coupling. This implies that each subsequential limit is a pseudo-metric. Furthermore, combining Skorohod's representation theorem with Lemma~\ref{lem:internal-B_N(0)}, we obtain the local H\"older continuity property for the limiting pseudo-metric.
\end{proof}

\subsection{Non-degeneracy of the subsequential limit}
\label{subsec:non-degenerate}
In this subsection, our main result is Proposition~\ref{prop:non-degenerate}. It has two consequences: Firstly, for any fixed $q \in (0,1)$, the quantile ratio $  a_n^{(q)} / \lambda_n$ is bounded above and below by constants which depend on $q$ but not on $n$. This, combined with Proposition~\ref{prop:tightness}, directly implies the tightness part of Theorem~\ref{thm:tightness}. Secondly, the subsequential limit in Theorem~\ref{thm:tightness} is a metric on $\mathbb{R}^d$, thereby completing its proof.

Throughout this subsection, we let $q_2$ be as in Proposition~\ref{prop:tightness}. Recall from Subsection~\ref{subsec:cross} the definition of $D_n(\mbox{across }A)$.

\begin{proposition}
    \label{prop:non-degenerate}
    For any $q \in (q_2,1)$ and $r_1 > r_2 > 0$, the following inequality holds:
    \begin{equation*}
        \lim_{\epsilon \rightarrow 0} \liminf_{n \rightarrow \infty} \mathbb{P}[D_n(\mbox{across } B_{r_1}(0) \backslash B_{r_2}(0)) > \epsilon a_n^{(q)}] = 1\,.
    \end{equation*}
\end{proposition}

We first prove Theorem~\ref{thm:tightness} assuming this proposition.
\begin{proof}[Proof of Theorem~\ref{thm:tightness}]
For any fixed $q \in (1/2,1)$, using Proposition~\ref{prop:non-degenerate} and noting that $D_n(0,e_1;B_2(0)) \geq D_n(\mbox{across } B_{1/2}(0) \backslash B_{1/4}(0))$, we obtain
\begin{equation}
\label{eq:thm1.1-proof-0}
    c a_n^{(q)} \leq \lambda_n \leq a_n^{(q)}.
\end{equation}
Combining this with Lemma~\ref{lem:internal-B_N(0)}, we obtain that the sequence $\{\lambda_n^{-1} D_n(\cdot, \cdot)\}_{n \geq 1}$ is tight with respect to the local uniform topology. Moreover, each subsequential limit is a pseudo-metric on $\mathbb{R}^d$ and locally H\"older continuous, which can be proved similarly to Proposition~\ref{prop:tightness}. 

We now prove that each subsequential limit is a metric on $\mathbb{R}^d$ using Proposition~\ref{prop:non-degenerate}. Suppose that $\widetilde D$ is the weak limit of $\{\lambda_{n_k}^{-1} D_{n_k}(\cdot, \cdot)\}_{k \geq 1}$ for an increasing sequence $\{n_k\}_{k \geq 1}$ with respect to the local uniform topology. Then, for any $x \in \mathbb{R}^d$, $r_0>r_1>r_2>0$, and $\epsilon>0$:
\begin{equation*}
\begin{aligned}
&\mathbb{P}\Big[\inf_{z \in B_{r_2}(x), w \in B_{r_0}(x) \backslash B_{r_1}(x)} \widetilde D(z,w) \geq \epsilon\Big]  \\
&\qquad\geq \limsup_{k \rightarrow \infty} \mathbb{P}\Big[\inf_{z \in B_{r_2}(x), w \in B_{r_0}(x) \backslash B_{r_1}(x)} D_{n_k}(z,w) \geq \epsilon \lambda_{n_k}\Big]\,.
\end{aligned}
\end{equation*}
This is because the event in the bracket is closed with respect to the uniform topology on $B_{r_0}(x)$. Since $D_{n_k}$ is a continuous length metric for any $k \geq 1$, we have
\begin{equation*}
   \inf_{z \in B_{r_2}(x), w \in B_{r_0}(x) \backslash B_{r_1}(x)} D_{n_k}(z,w) =  D_{n_k}(\mbox{across }B_{r_1}(x) \backslash B_{r_2}(x))\,.
\end{equation*}
Combining the above two inequalities and then taking $r_0$ to infinity yields 
\begin{align*}
    &\mathbb{P}\Big[\inf_{z \in B_{r_2}(x), w \in \mathbb{R}^d \backslash B_{r_1}(x)} \widetilde D(z,w) \geq \epsilon\Big] \\
    &\qquad \qquad \geq \limsup_{k \rightarrow \infty} \mathbb{P}\Big[D_{n_k}(\mbox{across }B_{r_1}(x) \backslash B_{r_2}(x)) \geq \epsilon \lambda_{n_k}\Big].
\end{align*}
Combining this with Proposition~\ref{prop:non-degenerate} and \eqref{eq:thm1.1-proof-0}, we obtain that for any $x \in \mathbb{R}^d$ and $r_1>r_2>0$:
\begin{equation}
\label{eq:thm1.1-euclidean-1}
\begin{aligned}
    &  \mathbb{P}\Big[\inf_{z \in B_{r_2}(x), w \in \mathbb{R}^d \backslash B_{r_1}(x)} \widetilde D(z,w) >0 \Big] = \lim_{\epsilon \rightarrow 0} \mathbb{P}\Big[\inf_{z \in B_{r_2}(x), w \in \mathbb{R}^d \backslash B_{r_1}(x)} \widetilde D(z,w) \geq \epsilon \Big]\\
    &\qquad\geq \lim_{\epsilon \rightarrow 0} \limsup_{k \rightarrow \infty} \mathbb{P}[D_{n_k}(\mbox{across } B_{r_1}(x) \backslash B_{r_2}(x)) \geq \epsilon \lambda_{n_k}] = 1\,.
\end{aligned}
\end{equation}
Hence, a.s.\ for each $x\in\mathbb Q^d$ and each $r_1,r_2 \in \mathbb Q$ with $r_1 > r_2 > 0$, 
\begin{equation} \label{eq:rational-dist-pos}
 \inf_{z \in B_{r_2}(x), w \in \mathbb{R}^d \backslash B_{r_1}(x)} \widetilde D(z,w) >0 \,.
\end{equation}
This implies that $\widetilde D$ is a metric on $\mathbb{R}^d$ (not just a pseudo-metric). By the local H\"older continuity, we know that the identity mapping from $\mathbb{R}^d$, equipped with the Euclidean metric, to $\mathbb R^d$, equipped with the metric $\wt D$, is continuous. To see that the inverse of this map is also continuous, consider a sequence of points $z_n\in\mathbb R^d$ and a point $z\in\mathbb R^d$ such that $\wt D(z_n,z) \to 0$. 
Let $\ep > 0$ and choose $x \in \mathbb Q^d$ and $r_1,r_2 \in \mathbb Q$ with $r_1 > r_2 > 0$ such that $z \in B_{r_2}(x)$ and $B_\ep(z) \subset B_{r_1}(x)$. By~\eqref{eq:rational-dist-pos}, a.s.\ there exists a random $\delta >0$ (depending on $z$ and $\ep$) such that $\wt D(z,w) \geq \delta$ for each $w \in \mathbb R^d \setminus B_\ep(z)$. Hence $z_n \in B_\ep(z)$ for each sufficiently large $n$, i.e., $|z_n-z|_\infty \to 0$.
Therefore, $\widetilde D$ induces the Euclidean topology. This completes the proof of Theorem~\ref{thm:tightness}.
\end{proof}
In the remaining part of this subsection, we will finish the proof of Proposition~\ref{prop:non-degenerate}. In Lemmas~\ref{lem:non-degenerate-pt-pt} and \ref{lem:non-degenerate-cross}, we will show that the crossing distance of a fixed hypercubic shell, when normalized by $a_n^{(q)}$, remains bounded away from zero with positive probability. The proof relies on the definition of $a_n^{(q)}$ and the positive association property of $h_n$ (see e.g.\ \cite[Theorem 2.3]{df-lqg-metric}). Lemma~\ref{lem:zero-one} then gives a zero-one law, establishing that $\mathbb{P}[\widetilde D(\mbox{across})= 0] \in \{0,1\}$ where $\widetilde D(\mbox{across})$ is the distance across this hypercubic shell. Combining these two results, we can conclude Proposition~\ref{prop:non-degenerate}. 

\begin{lemma}
    \label{lem:non-degenerate-pt-pt}
    For any $q \in (0,1)$ and $r \in (0,1/10)$, there exists a constant $c = c(q, r)>0$ such that
    \begin{equation*}
        \liminf_{n \rightarrow \infty} \inf_{\substack{x \in \partial B_r(0)\\ y \in \partial B_{2r}(0)}}  \mathbb{P}[D_n(x,y; B_{3r}(0)) > c a_n^{(q)}] > 0\,.
    \end{equation*}
\end{lemma}
\begin{proof}
    Fix $0 < q < 1$ and $0<r<1/10$. By the definition of $a_n^{(q)}$ from \eqref{def:quantile}, we have
    \begin{equation}
    \label{eq:lem6.4-1}
       \mathbb{P}[D_n(0, e_1; B_2(0)) \leq a_n^{(q)}] = q < 1\,. 
    \end{equation}
    Fix $x \in \partial B_r(0)$ and $y \in \partial B_{2r}(0)$. Since $r<1/10$, we can connect the points $0$ and $e_1$ in $B_2(0)$ using a sequence of points $0 = z_1,z_2,\ldots, z_J = e_1$ such that $|z_i - z_{i+1}|_2 = |x-y|_2$ for each $1 \leq i \leq J-1$. For any $i$, we define the point $x_i$ such that the triple $(z_i,z_{i+1},B_{3r}(x_i))$ can be obtained from the triple $(x,y,B_{3r}(0))$ through translation and rotation (see Figure~\ref{fig:3} for an illustration of a similar case). Furthermore, we can require that $B_{3r}(x_i) \subset B_2(0)$ and $J \leq C$, where $C=C(r)$ is independent of $x,y$ but may depend on the fixed value of $r$. Therefore,
    \begin{equation*}
       D_n(0, e_1; B_2(0)) \leq  \sum_{i=1}^{J-1} D_n(z_i,z_{i+1}; B_2(0)) \leq \sum_{i=1}^{J-1} D_n(z_i,z_{i+1}; B_{3r}(x_i))\,.  
    \end{equation*}
    Combining this with the positive association property of $h_n$ (see e.g.\ \cite[Theorem 2.3]{df-lqg-metric}), we obtain\footnote{Note that the event $D_n(z_i,z_{i+1}; B_{3r}(x_i)) \leq a_n^{(q)}/J $ is a decreasing event of $h_n$. \cite[Theorem 2.3]{df-lqg-metric} is stated for a Gaussian field defined via finite-dimensional marginals; however, we can apply it to our case by approximating $h_n$ using step functions; see e.g. the proof of Lemma~\ref{lem:gauss-concentration}.}
    \begin{equation*}
    \begin{aligned}
        \mathbb{P}\Big[D_n(0, e_1; B_2(0)) \leq a_n^{(q)}\Big] &\geq 
        \mathbb{P}\Big[ \bigcap_{i=1}^{J-1} \big{\{} D_n(z_i,z_{i+1}; B_{3r}(x_i)) \leq a_n^{(q)}/J \big{\}} \Big] \\
        &\geq \prod_{i=1}^{J-1} \mathbb{P}\Big[ D_n(z_i,z_{i+1}; B_{3r}(x_i)) \leq a_n^{(q)}/J \Big] \,.
    \end{aligned}
    \end{equation*}
    Using the fact that $J \leq C$ and the translation and rotational invariance of $D_n$ as stated in Lemma~\ref{lem:basic-LFPP}, we can deduce
    \begin{equation*}
        \mathbb{P}\Big[D_n(0, e_1; B_2(0)) \leq a_n^{(q)}\Big] \geq \mathbb{P}\Big[D_n(x,y;B_{3r}(0)) \leq a_n^{(q)}/C\Big]^C.
    \end{equation*}
    Combining this with \eqref{eq:lem6.4-1} yields the desired result.
\end{proof}

\begin{lemma}
    \label{lem:non-degenerate-cross}
    Recall the constant $q_2$ from Proposition~\ref{prop:tightness}. For any $q \in (q_2,1)$ and $r \in (0, 1/10)$, there exists a constant $c = c(q, r)>0$ such that
    \begin{equation*}
        \liminf_{n \rightarrow \infty} \mathbb{P}[D_n(\partial B_r(0), \partial B_{2r}(0)) > c a_n^{(q)}] > 0\,.
    \end{equation*}
\end{lemma}
\begin{proof}
    Fix $q \in (q_2, 1)$ and $r \in (0, 1/10)$. Let $c'$ be the constant from Lemma~\ref{lem:non-degenerate-pt-pt}. Define the constant
    \begin{equation*}
        A := \liminf_{n \rightarrow \infty} \inf_{\substack{x \in \partial B_r(0)\\ y \in \partial B_{2r}(0)}} \mathbb{P}[D_n(x,y; B_{3r}(0)) > c' a_n^{(q)}] > 0
    \end{equation*}
    Applying Proposition~\ref{prop:tightness} with $U = B_{3r}(0)$, we can select a sufficiently large integer $m$ such that
    \begin{equation*}
        \mathbb{P}\Big[{\rm Diam}_n(B_{r2^{-m}}(z); B_{3r}(0))/a_n^{(q)} > \frac{c'}{4} \Big] < \frac{A}{4} \quad \forall z \in \partial B_r(0) \cup \partial B_{2r}(0)\,.
    \end{equation*}
    Combining this with the triangle inequality
    \begin{equation*}
    \begin{aligned}
        &\quad D_n(B_{r2^{-m}}(x), B_{r2^{-m}}(y); B_{3r}(0)) \\
        &\geq D_n(x, y; B_{3r}(0)) - {\rm Diam}_n(B_{r2^{-m}}(x); B_{3r}(0)) - {\rm Diam}_n(B_{r2^{-m}}(y) ; B_{3r}(0))\,,
    \end{aligned}
    \end{equation*}
    we obtain that for all sufficiently large $n$:
    \begin{equation}
    \label{eq:lem6.6-1}
        \mathbb{P}\Big[ D_n(B_{r2^{-m}}(x), B_{r2^{-m}}(y); B_{3r}(0)) / a_n^{(q)}  > \frac{c'}{2} \Big] > \frac{A}{2} \quad \forall x \in \partial B_r(0), y \in \partial B_{2r}(0)\,.
    \end{equation}
    Define the sets $X$ and $Y$ as $X = \partial B_r(0) \cap (r2^{-m}) \mathbb{Z}^d$ and $Y = \partial B_{2r}(0) \cap  (r2^{-m}) \mathbb{Z}^d$. Since the families of boxes $\{B_{r2^{-m}}(x)\}_{x \in X} $ and $\{B_{r2^{-m}}(y)\}_{y \in Y} $ cover $\partial B_r(0)$ and $\partial B_{2r}(0)$ respectively, it follows that
     \begin{equation*}
     \begin{aligned}
        &\mathbb{P}\Big[D_n(\partial B_r(0), \partial B_{2r}(0))/a_n^{(q)} > \frac{c'}{2}\Big] \\
        &\qquad \geq \mathbb{P}\Big[\bigcap_{x \in X, y \in Y} D_n(B_{r2^{-m}}(x), B_{r2^{-m}}(y); B_{3r}(0))/a_n^{(q)} > \frac{c'}{2}\Big]\,.
        \end{aligned}
    \end{equation*}
    Using the positive associativity property (see e.g.\ \cite[Theorem 2.3]{df-lqg-metric}), we obtain 
    \begin{equation*}
    \begin{aligned}
    &\mathbb{P}\Big[D_n(\partial B_r(0), \partial B_{2r}(0)) /a_n^{(q)} > \frac{c'}{2} \Big] \\
    &\qquad\geq \prod_{x \in X, y \in Y} \mathbb{P}\Big[ D_n(B_{r2^{-m}}(x), B_{r2^{-m}}(y); B_{3r}(0))/a_n^{(q)} > \frac{c'}{2} \Big]\,.
    \end{aligned}
    \end{equation*}
    Combining this with \eqref{eq:lem6.6-1} yields the desired result.
\end{proof}

In Lemma~\ref{lem:zero-one}, we present a zero-one argument showing that for each subsequential limit, the distance across a hypercubic shell is positive with probability either 0 or 1. First, we need two auxiliary results. The first result will eventually imply that for any subsequential limiting metric $\wt D$, the $\wt D$-distance from any compact set to $\infty$ is infinite.

\begin{lemma}
    \label{lem:dist-to-infinity}
    For any $q \in (q_2,1)$, $r>0$, and $T>0$, the following inequality holds:
    \begin{equation*}
        \lim_{R \rightarrow \infty} \liminf_{n \rightarrow \infty} \mathbb{P}\big[D_n(\mbox{across } B_R(0) \backslash B_r(0)) > T a_n^{(q)}\big] = 1\,.
    \end{equation*}
\end{lemma}
\begin{proof}
    Fix $q \in (q_2,1)$. We claim that there exists a constant $c = c(q)>0$ such that
    \begin{equation}
    \label{eq:lem-dist-inf-0}
        \mathbb{P}\big[ D_n( \mbox{across } B_{2^{k+1}}(0) \backslash B_{2^k}(0)) \geq c a_n^{(q)} \big] \geq c\quad \forall n, k \geq 1/c\,.
    \end{equation}
    According to Lemma~\ref{lem:dist-independent}, for positive integers $k' > k$, $D_n( \mbox{across } B_{2^{k+1}}(0) \backslash B_{2^k}(0))$ and $D_n(\mbox{across } B_{2^{k'+1}}(0) \backslash B_{2^{k'}}(0))$ are independent whenever $|2^{k'} - 2^{k+1}| \geq 2 \rr$, which happens whenever $k' \geq k + 2$ and $k \geq \lfloor \log \rr \rfloor + 10$. Combining this with \eqref{eq:lem-dist-inf-0} and the inequality 
    \begin{equation*}
    D_n(\mbox{across } B_R(0) \backslash B_r(0)) \geq \sum_{k=\lfloor \log r \rfloor +1}^{\lfloor \log R \rfloor - 1} D_n( \mbox{across } B_{2^{k+1}}(0) \backslash B_{2^k}(0)) , 
    \end{equation*}
    we can conclude the lemma. 

    In the remainder of the proof, we prove claim~\eqref{eq:lem-dist-inf-0}. Since $Q> \sqrt{2d}$, we can choose a small constant $\epsilon > 0$ such that $(\xi + \log 2)\epsilon< \xi \log 2 \cdot (Q - \sqrt{2d}) $. Applying the scaling property from Lemma~\ref{lem:basic-LFPP}, with $(k,n+k)$ in place of $(m,n)$, we obtain
    \begin{equation}
    \label{eq:lem-dist-inf-1}
        D_n( \mbox{across } B_{2^{k+1}}(0) \backslash B_{2^k}(0)) \overset{d}{=} 2^k D_{k,n+k}( \mbox{across } B_2(0) \backslash B_1(0))\,.
    \end{equation}
    
    We have
    \begin{equation*}
        D_{k,n+k}( \mbox{across } B_2(0) \backslash B_1(0)) \geq e^{-\xi \sup_{x \in B_2(0)} h_k(x)} D_{n+k}( \mbox{across } B_2(0) \backslash B_1(0))\,.
    \end{equation*}
    We now lower-bound the two terms on the right-hand side separately. Using Claims (3) and (4) from Lemma~\ref{lem:field-estimate}, we have 
    \begin{equation*}
        \mathbb{P}\Big[\sup_{x \in B_2(0)} h_k(x) > (\sqrt{2d} \log 2 +\epsilon) k\Big] \leq C e^{-k/C}.
    \end{equation*}
    Let $q' = (q+1)/2$. Applying Lemma~\ref{lem:non-degenerate-cross} with $q'$ in place of $q$, we know that there exists $A = A(q)>0$ such that for all sufficiently large $n+k$:
    \begin{equation}
        \mathbb{P}\big[D_{n+k}( \mbox{across } B_2(0) \backslash B_1(0)) \geq A a_{n+k}^{(q')} \big] \geq A\,.
    \end{equation}
    (Lemma~\ref{lem:non-degenerate-cross} is stated for $B_{2r}(0) \backslash B_r(0)$ with $r \in (0,1/10)$, but we can easily extend this result to $B_2(0) \backslash B_1(0)$ by covering $B_2(0) \backslash B_1(0)$ with hypercubic shells $B_{2r}(0) \backslash B_r(0)$ and then applying the positive association property.) Combining the above three inequalities, we obtain that for all sufficiently large $k$:
    \begin{equation}
    \label{eq:lem-dist-inf-2}
        \mathbb{P} \big[ D_{k,n+k}( \mbox{across } B_2(0) \backslash B_1(0)) \geq e^{-\xi(\sqrt{2d} \log 2 +\epsilon) k} \cdot Aa_{n+k}^{(q')} \big] \geq A/2 \,.
    \end{equation}
    
    Applying Corollary~\ref{cor:an-compare} with our chosen values of $(\epsilon,q)$, and with $(n+k,k)$ in place of $(n,k)$, we deduce that $a_n^{(q)} \leq \frac{1}{c} 2^{(1-\xi Q + \epsilon)k} a_{n+k}^{(q +2^{-c(n+k)})} \leq \frac{1}{c} 2^{(1-\xi Q + \epsilon)k} a_{n+k}^{(q')}$ for all sufficiently large $n$, where the constant $c$ depends only on $\epsilon$ and $q$. Combining this result with \eqref{eq:lem-dist-inf-1} and \eqref{eq:lem-dist-inf-2}, we obtain that for all sufficiently large $n$ and $k$, with probability at least $A/2$,
    \begin{equation*}
    \begin{aligned}
        2^{-k} D_n( \mbox{across } B_{2^{k+1}}(0) \backslash B_{2^k}(0)) 
        &\geq e^{-\xi(\sqrt{2d} \log 2 +\epsilon) k} \cdot A a_{n+k}^{(q')} \\
        &\geq c 2^{-k} e^{[\xi \log 2  \cdot (Q-\sqrt{2d}) - (\xi + \log 2)\epsilon]k } a_n^{(q)}.
        \end{aligned}
    \end{equation*}
    By the choice of $\epsilon$, we get $ e^{\xi \log 2  \cdot (Q-\sqrt{2d}) - (\xi + \log 2)\epsilon } \geq 1$. This implies claim~\eqref{eq:lem-dist-inf-0} and thus completes the proof of the lemma.
\end{proof}

\begin{lemma}
\label{lem:set-connect}
   Fix $q \in (q_2,1)$ and $r>0$. For each subsequential limit $\widetilde D$ of $\{ D_n(\cdot,\cdot)/ a_n^{(q)} \}_{n \geq 1}$, the set $\{x \in \mathbb R^d : \widetilde D(x, \partial B_r(0)) = 0 \}$ is a.s.\ a closed, bounded, and connected set.
\end{lemma}

Given what we have proven so far, the non-trivial part of Lemma~\ref{lem:set-connect} is the connectedness.
If we knew \textit{a priori} that $\wt D$ were a metric (not just a pseudo-metric), then by standard metric space arguments (see, e.g.,~\cite[Exercise 2.4.19]{bbi-metric-geometry}) and the fact that each $D_n$ is a length metric, we would get that also $\wt D$ is a length metric, which implies that $\wt D$-neighborhoods of connected sets are connected.  The connectedness part of Lemma~\ref{lem:set-connect} is a weaker version of this property.

\begin{proof}[Proof of Lemma~\ref{lem:set-connect}]
    Let 
    \[
    A = \{x  \in \mathbb R^d : \widetilde D(x, \partial B_r(0)) = 0 \} . 
    \]
    Since $\widetilde D$ is continuous, the set $A$ is closed. Boundedness of $A$ follows from Lemma~\ref{lem:dist-to-infinity}. This is because for any $\epsilon>0$, by Lemma~\ref{lem:dist-to-infinity}, there exists $R>0$ such that
    \begin{equation*}
    \begin{aligned}
    &\liminf_{n \rightarrow \infty} \mathbb{P}\big[\inf_{x \in \mathbb{R}^d \backslash B_R(0)} D_n(x, \partial B_r(0)) > a_n^{(q)}\big] \\
    &\qquad=  \liminf_{n \rightarrow \infty} \mathbb{P}\big[D_n(\mbox{across } B_R(0) \backslash B_r(0)) > a_n^{(q)}\big] \geq 1-\epsilon\,.
    \end{aligned}
    \end{equation*}
    The first equality follows from the length space property of $D_n$. Since $\widetilde D$ is the subsequential limit of $\{ D_n(\cdot,\cdot)/ a_n^{(q)} \}_{n \geq 1}$ with respect to local uniform topology, we have
    \begin{equation*}
        \mathbb{P}\big[\inf_{x \in \mathbb{R}^d \backslash B_R(0)}\widetilde D(x, \partial B_r(0)) \geq 1 \big] \geq 1-\epsilon\,.
    \end{equation*}
    (We first prove the case for $x \in B_{R'}(0) \backslash B_R(0)$ for all $R' > 0$, and then take $R'$ to infinity.) Therefore, $\mathbb{P}\big[ A \subset B_R(0)\big] \geq 1-\epsilon$. Since this holds for any $\epsilon>0$, we get that $A$ is bounded.
    
    Finally, we show that $A$ is connected. Suppose that $\widetilde D$ is the weak limit of $\{ D_{n_k}(\cdot, \cdot) / a_{n_k}^{(q)} \}_{k \geq 1}$ for an increasing sequence $\{n_k\}_{k \geq 1}$. Using Skorohod's representation theorem, we can assume that under some coupling, $\{ D_{n_k}(\cdot, \cdot) / a_{n_k}^{(q)} \}_{k \geq 1}$ converges to $\widetilde D$ almost surely. Let $x \in A$. Then we have $D_{n_k}(x,\partial B_r(0)) / a_{n_k}^{(q)} \to 0$ a.s. So, for any $\epsilon>0$, there exists an integer $k = k(\epsilon)\geq 1/\epsilon $ and a continuous path $P^{(\epsilon)}$ connecting $x$ to $\partial B_r(0)$ such that $\sup_{y \in P^{(\epsilon)}} D_{n_k}(y, \partial B_r(0)) < \epsilon a_{n_k}^{(q)}$. Let $R>0$, and consider the segment $P_R^{(\epsilon)}$ of $P^{(\epsilon)}$ stopped at its first exit time from the box $B_R(0)$. If $P^{(\epsilon)}$ does not exit the box $B_R(0)$, then $P_R^{(\epsilon)}$ is just $P^{(\epsilon)}$. By definition, $P_R^{(\epsilon)}$ contains a path connecting either $x$ or $\partial B_R(0)$ to $\partial B_r(0)$. After passing to a (random) subsequence $\{\epsilon_m \}_{m \geq 1}$, we can arrange that $P_R^{(\epsilon_m)}$ converges to a compact, connected set $P_R$ with respect to the Euclidean Hausdorff distance as $\epsilon_m$ tends to zero. Moreover, the set $P_R$ intersects $\partial B_r(0)$ and also either contains $x$ or intersects $\partial B_R(0)$. By the almost sure convergence of $\{D_{n_k}(\cdot, \cdot) / a_{n_k}^{(q)} \}_{k \geq 1}$ to $\widetilde D$, we see that $P_R \subset A$. By taking $R$ to infinity and using the boundedness of $A$, we obtain that there is a connected subset of $A$ which intersects $\partial B_r(0)$ and contains $x$. Since $x \in A$ is arbitrary and $\partial B_r(0)$ is connected, this confirms that $A$ is connected.
\end{proof}

We now present the zero-one argument.

\begin{lemma}
    \label{lem:zero-one}
    Fix $q \in (q_2,1)$ and $r>0$. For each subsequential limit $\widetilde D$ of $\{ D_n(\cdot,\cdot)/ a_n^{(q)} \}_{n \geq 1}$, we define the event $\mathcal{Z} := \{ \widetilde{D}(\mbox{across }B_{2r}(0) \backslash B_r(0)) > 0 \}$. Then,
    \begin{equation*}
        \mathbb{P}[\mathcal{Z}] \in \{0,1\}.
    \end{equation*}
\end{lemma}
\begin{proof}
    Suppose that $\widetilde D$ is the weak limit of $\{D_{n_k}(\cdot,\cdot)/ a_{n_k}^{(q)}\}_{k \geq 1}$ under the local uniform topology, where $\{n_k\}_{k \geq 1}$ is an increasing sequence. 
    
    Here is a heuristic argument for why the lemma is true, which we will make precise below. Define
    \begin{equation} \label{eq:zero-one-sets}
    E_1 = \{ x \in \mathbb{R}^d: \widetilde D(x, \partial B_r(0)) = 0 \} \mbox{ and } 
    E_2 = \{ x \in \mathbb{R}^d: \widetilde D(x, \partial B_{2r}(0)) = 0 \}
    \end{equation}
    Then $\mathcal Z = \{E_1\cap E_2 = \emptyset\}$. Let $U_1,U_2\subset \mathbb{R}^d$ be bounded open sets lying at positive Euclidean distance from each other such that $\partial B_r(0) \subset U_1$ and $\partial B_{2r}(0)\subset U_2$. We will argue that the events $\{E_1\subset U_1\}$ and $\{E_2\subset U_2\}$ are independent. To explain why this is true, let us make the simplifying assumption that $\wt D$ is a limit in probability (instead of just in law), so that $\wt D$ is a function of $h$ (in the actual proof, we will need to pass back and forth between $D_{n_k}$ and $\wt D$ to get around the lack of convergence in probability). The event $\{E_1\subset U_1\}$ depends only on the internal metric of $\wt D$ on $U_1$. Furthermore, adding a continuous function $f$ to $h$ has the effect of scaling internal distances in $U_1$ by a a factor of at most $\exp( \sup_{z\in U_1} \xi |f(z)|)$, so does not change which points in $U_1$ lie at zero $\wt D(\cdot,\cdot;U_1)$-distance from $\bdy B_r(0)$. From this, we get that the event $\{E_1\subset U_1\}$ depends only on the restriction to $U_1$ of the fields $\{h_{m,n}\}_{n\geq m}$, for any $m\in\mathbb N$.\footnote{If we actually knew that $\wt D$ was a measurable function of $h$, at this point we could use the Kolmogorov's zero-one law to say that $E_1$ is a.s.\ equal to some fixed deterministic set. We instead explain a less direct argument which is easier to adapt to our setting, where we do not know that $\wt D$ is a measurable function of $h$. Alternatively, one could try to directly show that the subsequential limit is a measurable function of $h$ following the two-dimensional arguments in~\cite{local-metrics, lqg-metric-estimates} without assuming it is a metric. We expect the arguments to be delicate, and there are also some technical difficulties---for instance the assumption in \cite[Theorem 1.6]{local-metrics} is hard to verify.} The analogous statement is also true for $\{E_2\subset U_2\}$. Since $h_{m,n}$ has range of dependence at most a constant times $2^{-m}$, we get that $\{E_1\subset U_1\}$ and $\{E_2\subset U_2\}$ are independent. Summing over a suitable countable collection of possible choices of $U_1$ and $U_2$ then shows that
    \[
    \mathbb P[\mathcal Z] = \mathbb P[E_1\cap E_2 =\emptyset] \leq \mathbb P[ E_1 \cap B_{2r}(0) = \emptyset] \mathbb P[E_2\cap B_r(0) = \emptyset] = \mathbb P[\mathcal Z]^2 
    \]
    which implies that $\mathbb P[ \mathcal Z] \in \{0,1\}$.
    
    Let us now proceed with the details. For any integer $j \geq 1$, an open domain $U$ is called \textit{$(r,j)$-dyadic} if it is bounded, connected, and can be written as the union $\cup_{i \geq 1} B_{2^{-j} r}(x_i)$, where $\{ x_i \}_{i \geq 1} \subset 2^{-j} r \mathbb{Z}^d$. An open domain is called \textit{$r$-dyadic} if it is $(r,j)$-dyadic for some integer $j \geq 1$. The proof will consist of four steps.
    \medskip

    \noindent\textit{Step 1: Joint convergence of the internal metrics in $r$-dyadic domains.} By Lemma~\ref{lem:internal-B_N(0)} and Proposition~\ref{prop:tightness}, we know that for all integer $m \geq 0$ and $r$-dyadic domain $U$ (and $U = \mathbb{R}^d$), the internal metric $\{D_{m,n}(\cdot,\cdot;U)/ a_n^{(q)}\}_{n \geq 1}$ is tight with respect to the local uniform topology on $U$. Since the number of $r$-dyadic domains is countable, we can apply a diagonal argument to select a subsequence $\{n_k'\}_{k \geq 1}$ from $\{n_k\}_{k \geq 1}$ such that 
    \begin{equation*}
    \begin{aligned}
        &\mbox{the internal metrics } \big{\{} D_{m, n_k'}(\cdot,\cdot; U)/a_{n_k'}^{(q)} : U \mbox{ is }r\mbox{-dyadic or }\mathbb{R}^d, \mbox{ } m \geq 0 \big{\}}_{k \geq 1} \\
        &\mbox{ jointly converge in distribution}\,.
        \end{aligned}
    \end{equation*}
    According to Skorohod's representation theorem, we can assume that they jointly converge almost surely under some coupling. For each $U$ and $m$, the limit is a pseudo-metric on $U$, denoted as $\widetilde D_{m,U}$. When $m=0$, we will abbreviate $\widetilde D_{0,U}$ as $\widetilde D_U$\footnote{We do not claim here that $\widetilde D_U$ is consistent with the internal metric induced by $\widetilde D_{\mathbb{R}^d}$ on $U$. However, this will be a consequence of Proposition~\ref{prop:non-degenerate} and Lemma~\ref{lem:dist-to-infinity}, since according to these results, as $n$ becomes large, the geodesic between two closely located points under $D_n$ will not get far away from these two points with high probability.}. With an abuse of notation, we will still use $\widetilde D$ to represent $\widetilde D_{\mathbb{R}^d}$ (i.e., the almost sure limit of $\{ D_{m, n_k'}(\cdot,\cdot)/a_{n_k'}^{(q)} \}_{k \geq 1}$).
    \medskip

    \noindent\textit{Step 2: Definitions and basic properties of $E_1,E_2,E_1^j$, and $E_2^j$.} 
    We refer to Figure~\ref{fig:zero-one} for an illustration. 
    Define the sets $E_1$ and $E_2$ as in~\eqref{eq:zero-one-sets}. By definition, $\partial B_r(0) \subset E_1$ and $\partial B_{2r}(0) \subset E_2$. According to Lemma~\ref{lem:set-connect}, both $E_1$ and $E_2$ are closed, bounded, and connected sets. Moreover, we have  
    \begin{equation}
    \label{eq:lem6.7-equivalence}
    \begin{aligned}
        \mathcal{Z} = \{ \widetilde D(\partial B_r(0), \partial B_{2r}(0))>0 \} 
        &=  \{E_1 \cap \partial B_{2r}(0) = \emptyset  \}  \\
        &= \{E_2 \cap \partial B_r(0) = \emptyset  \} = \{ E_1 \cap E_2 = \emptyset \}\,.
        \end{aligned}
    \end{equation}
    The first three equivalences follow directly from the definitions of $\mathcal{Z}$, $E_1$, and $E_2$. The last equivalence follows from the triangle inequality $\widetilde D(\partial B_r(0), \partial B_{2r}(0)) \leq \widetilde D(\partial B_r(0), x) + \widetilde D(x, \partial B_{2r}(0))$ for any $x \in \mathbb{R}^d$.

    For each integer $j \geq 1$, we further define $E_1^j$ as the open domain which contains all the $B_{2^{-j} r}(x)$ boxes for any $x \in (r 2^{-j})  \mathbb{Z}^d$ with $\mathfrak d_\infty(x,E_1) < 2r \cdot 2^{-j}$. Since $E_1$ is bounded and connected, $E_1^j$ is also bounded and connected, and thus $(r,j)$-dyadic. Similarly, we define the $(r,j)$-dyadic domain $E_2^j$ associated with $E_2$. By definition, we know that $E_1 \subset E_1^j$ and $E_2 \subset E_2^j$.

    \begin{figure}[H]
\centering
\includegraphics[scale = 0.5]{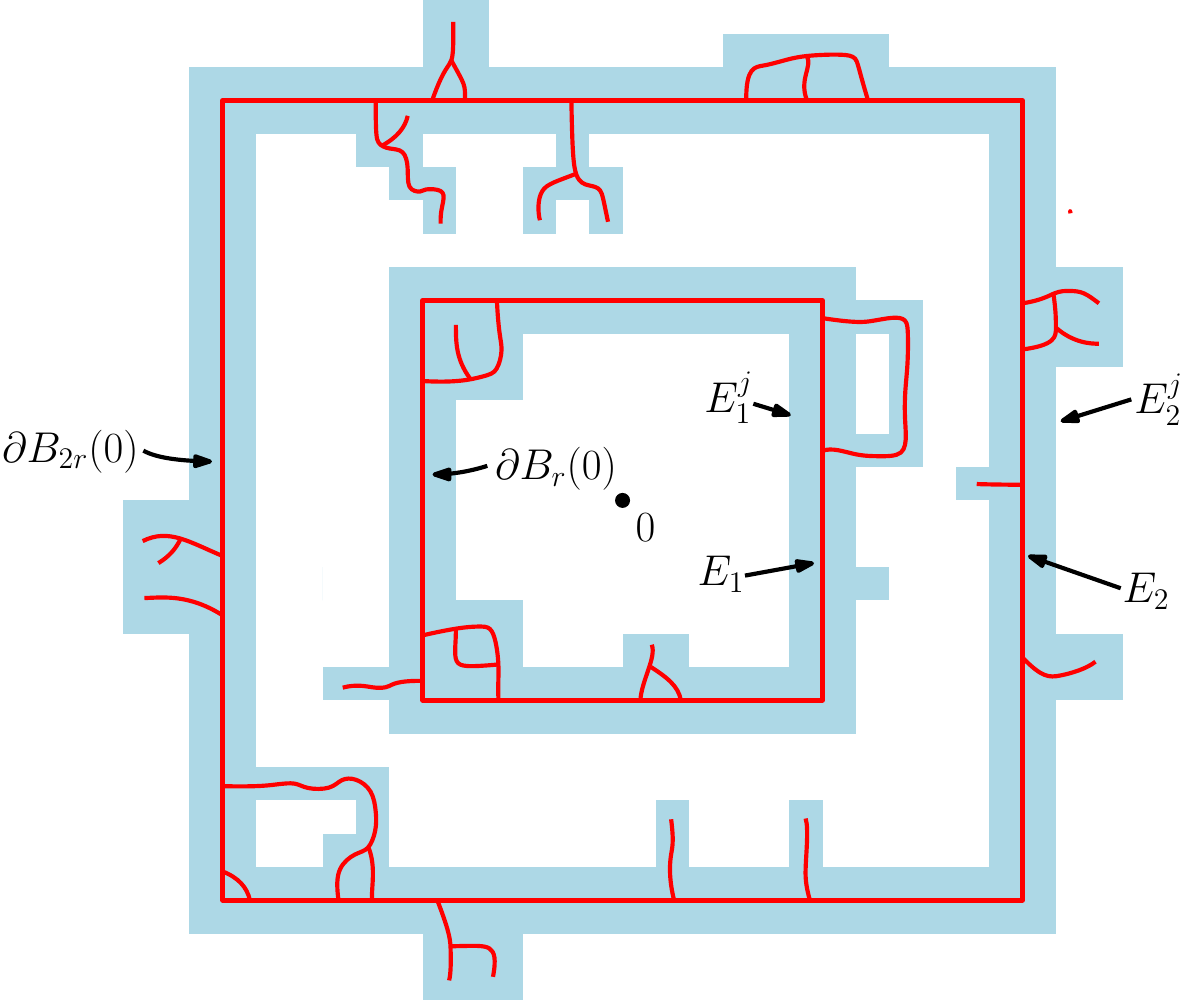}
\caption{Illustration of the sets $E_1,E_1^j,E_2$, and $E_2^j$ on the event $\mathcal{Z}$. The red sets represent $E_1$ and $E_2$, and the blue domains represent $E_1^j$ and $E_2^j$.}
\label{fig:zero-one}
\end{figure}

    \noindent\textit{Step 3: Independence of $\{E_1^j = U_1\}$ and $\{E_2^j = U_2\}$.} Given two disjoint $(r,j)$-dyadic domain $U_1$ and $U_2$ with $\mathfrak d_\infty (U_1,U_2) > r2^{-j}$, we now show that 
    \begin{equation}
    \label{eq:lem6.7-ind}
        \mathbb{P}[E_1^j = U_1, E_2^j = U_2] =  \mathbb{P}[E_1^j = U_1] \cdot \mathbb{P}[E_2^j = U_2]\,.
    \end{equation}
    
    We assume that $\partial B_r(0) \subset U_1$ and $\partial B_{2r}(0) \subset U_2$; otherwise, both sides of \eqref{eq:lem6.7-ind} would be 0. By definition, we also know that both $U_1$ and $U_2$ are open, bounded, and connected sets. For integer $m \geq 0$, define the sets 
    \begin{equation*}
    \begin{aligned}
        \widehat E_1^{m} &:= \mbox{closure of the connected component of }\{ x \in U_1: \widetilde D_{m, U_1}(x, \partial B_r(0)) = 0 \} \\
        &\qquad \mbox{ containing } \partial B_r(0)\,,\\
         \widehat E_2^{m} &:= \mbox{closure of the connected component of }\{ x \in U_2: \widetilde D_{m, U_2}(x, \partial B_{2r}(0)) = 0 \}\\
         &\qquad \mbox{ containing } \partial B_{2r}(0)  \,.
    \end{aligned}
    \end{equation*}
    Then, both $\widehat E_1^m$ and $\widehat E_2^m$ are closed, bounded, and connected sets. Similar to the definitions of $E_1^j$ and $E_2^j$, we also define $\widehat E_1^{m,j}$ and $\widehat E_2^{m,j}$ as the $(r,j)$-dyadic domains associated with $\widehat E_1^m$ and $\widehat E_2^m$, respectively. 
    
    Note that $D_{0,n_k'}(\cdot, \cdot; U_1)$ can be bounded both from above and below by $D_{m,n_k'}(\cdot,\cdot; U_1)$ up to a (random) constant. By the joint almost sure convergence of $\{D_{0,n_k'}(\cdot, \cdot; U_1) / a_{n_k'}^{(q)} \}_{k \geq 1}$ to $\widetilde D_{0,U_1}$ and $\{D_{m,n_k'}(\cdot, \cdot; U_1) / a_{n_k'}^{(q)} \}_{k \geq 1}$ to $\widetilde D_{m,U_1}$, we obtain that $\widetilde D_{0,U_1}$ can also be bounded both from above and below by $\widetilde D_{m, U_1}$ up to a (random) constant. This implies that $\widehat E_1^{m} = \widehat E_1^{0}$ for any integer $m \geq 0$. Similarly, we have $\widehat E_2^{m} = \widehat E_2^{0}$ for any integer $m \geq 0$. Therefore,
    \begin{equation}
    \label{eq:lem6.7-ind-11}
        \widehat E_1^{m,j} = \widehat E_1^{0,j} \quad \mbox{and} \quad \widehat E_2^{m,j} = \widehat E_2^{0,j} \quad \forall m \geq 0\,.
    \end{equation}

    According to Lemma~\ref{lem:dist-independent} and the fact that $\mathfrak d_\infty (U_1,U_2) > r2^{-j}$, for sufficiently large $m$, the internal metrics $D_{m,n_k'}(\cdot, \cdot; U_1)$ and $D_{m,n_k'}(\cdot, \cdot; U_2)$ are independent. By joint almost sure convergence, the pseudo-metrics $\widetilde D_{m,U_1}$ and $\widetilde D_{m,U_2}$ are also independent. This implies the independence of the sets $\widehat E_1^m$ and $\widehat E_2^m$. Therefore, for sufficiently large $m$, we have
    \begin{equation*}
        \mathbb{P}[\widehat E_1^{m, j} = U_1, \widehat E_2^{m, j} = U_2] =  \mathbb{P}[\widehat E_1^{m, j} = U_1] \cdot \mathbb{P}[\widehat E_2^{m, j} = U_2]\,.
    \end{equation*}
    Combining this with \eqref{eq:lem6.7-ind-11} yields
    \begin{equation}
    \label{eq:lem6.7-ind-12}
        \mathbb{P}[\widehat E_1^{0, j} = U_1, \widehat E_2^{0, j} = U_2] =  \mathbb{P}[\widehat E_1^{0, j} = U_1] \cdot \mathbb{P}[\widehat E_2^{0, j} = U_2]\,.
    \end{equation}

    In order to prove \eqref{eq:lem6.7-ind}, it suffices to show that the corresponding probabilities in \eqref{eq:lem6.7-ind} and \eqref{eq:lem6.7-ind-12} are equal. We now prove that even the corresponding events in \eqref{eq:lem6.7-ind} and \eqref{eq:lem6.7-ind-12} are identical. Since $D_{n_k'}(\cdot,\cdot; U_1) \geq D_{n_k'}(\cdot,\cdot)$, by the joint almost sure convergence, we have $\widetilde D_{U_1} \geq \widetilde D$. Hence, $\widehat E_1^0 \subset E_1$ and $\widehat E_1^{0,j} \subset E_1^j$. Recall that $U_1$ is an open, bounded, and connected set containing $\partial B_r(0)$. We first show that
    \begin{equation}
    \label{eq:lem6.7-ind-subset}
    \begin{aligned}
        \{ \widehat E_1^0 \subset U_1\} \subset \{ \widetilde D(\partial U_1, \partial B_r(0))>0\} \subset \{ \widehat E_1^0 = E_1 \}\,.
    \end{aligned}
    \end{equation}
    The first relation is proven similarly to Lemma~\ref{lem:set-connect}. Indeed, if $\widetilde D(\partial U_1, \partial B_r(0))=0$, then for any $\epsilon>0$, there exists an integer $k = k(\epsilon) \geq 1/\epsilon$ and a continuous path $P^{(\epsilon)}$ within the domain $U_1$ connecting $\partial U_1$ to $\partial B_r(0)$ such that $\sup_{y \in P^{(\epsilon)}} D_{n_k'}(y, \partial B_r(0); U_1) <\epsilon a_{n_k'}^{(q)}$. After passing to a (random) subsequence $\{\epsilon_m \}_{m \geq 1}$, we can arrange that $P_k$ converges to a compact, connected set $P$ with respect to the Euclidean Hausdorff distance as $\epsilon_m$ tends to zero. Moreover, the set $P$ connects $\partial U_1$ and $\partial B_r(0)$ within the domain $U_1$. By the almost sure convergence of the sequence $\{D_{n_k'}(\cdot, \cdot; U_1) / a_{n_k'}^{(q)} \}_{k \geq 1}$ to $\widetilde D_{U_1}$, we see that $\widetilde D(y, \partial B_r(0); U_1) = 0$ for every $y \in P$. Consequently, $P \subset \widehat E_1^0$ and $\widehat E_1^0 \not \subset U_1$. This yields the first relation in \eqref{eq:lem6.7-ind-subset}. 
    
    The second relation can be derived from the length space property of $D_{n_k'}$. For any $k \geq 1$ and $x \in U_1$, we have 
    \[
    D_{n_k'}(x, \partial B_r(0)) \geq \min \{ D_{n_k'}(x, \partial B_r(0); U_1), D_{n_k'}(\partial U_1, \partial B_r(0))\}.
    \]
    This is because for each $\epsilon>0$, there exists a path $P$ connecting $x$ to $\partial B_r(0)$ with $D_{n_k'}$-length at most $D_{n_k'}(x, \partial B_r(0)) + \epsilon$. If $P$ is entirely contained in $U_1$, then we have ${\rm len}(P; D_{n_k'}) \geq D_{n_k'}(x, \partial B_r(0); U_1)$; otherwise, we know that $P$ must connect $\partial U_1$ and $\partial B_r(0)$, hence ${\rm len}(P; D_{n_k'}) \geq D_{n_k'}(\partial U_1, \partial B_r(0))$. Sending $\epsilon$ to $0$ yields the above inequality.
    Sending $k$ to infinity and using the almost sure convergence, we obtain 
    \begin{equation*}
        \widetilde D(x, \partial B_r(0)) \geq \min \big{\{} \widetilde D_{U_1}(x, \partial B_r(0)), \widetilde D(\partial U_1, \partial B_r(0)) \big{\}} \quad \forall x \in U_1\,.
    \end{equation*}
    Similarly, for all $x \not \in U_1$, we have $D_{n_k'}(x, \partial B_r(0)) \geq D_{n_k'}(\partial U_1, \partial B_r(0))$, which implies that 
    \begin{equation*}
    \widetilde D(x, \partial B_r(0)) \geq \widetilde D(\partial U_1, \partial B_r(0)) \quad \forall x \not \in U_1 \,.
    \end{equation*}
   By the above two inequalities and the fact that $\widetilde D_{U_1} \geq \widetilde D$, we know that, on the event $\{\widetilde D(\partial U_1, \partial B_r(0)) > 0\}$, the equality $\widetilde D(x, \partial B_r(0)) = 0$ holds if and only if $\widetilde D_{U_1}(x, \partial B_r(0)) = 0$. Furthermore, this is possible only when $x \in U_1$. Therefore, on the event $\{\widetilde D(\partial U_1, \partial B_r(0)) > 0\}$, we have $\{ x \in \mathbb{R}^d: \widetilde D(x, \partial B_r(0)) = 0 \} = \{ x \in U_1: \widetilde D_{U_1}(x, \partial B_r(0)) = 0 \}$, hence $E_1 = \widehat E_1^0$ (recall that $E_1$ is a closed, bounded, and connected set).
   This establishes the second relation in \eqref{eq:lem6.7-ind-subset}. 
    
    We now prove $\{\widehat E_1^{0,j} = U_1\} = \{E_1^j = U_1\}$ using the claim~\eqref{eq:lem6.7-ind-subset}. This is because whenever $\widehat E_1^{0,j} = U_1$ or $E_1^j = U_1$, we have $\widehat E_1^0 \subset U_1$ (recall that $\widehat E_1^0 \subset \widehat E_1^{0,j}$ and $\wh E_1^0 \subset E_1 \subset E_1^j$). According to \eqref{eq:lem6.7-ind-subset}, we further have $\widehat E_1^0 = E_1$, and thus $\widehat E_1^{0,j} = E_1^j = U_1$. Similarly, we can show that $\{\widehat E_2^{0,j} = U_2\} = \{E_2^j = U_2\}$. By taking their intersections, we obtain $\{\widehat E_1^{0,j} = U_1, \widehat E_2^{0,j} = U_2\} = \{E_1^j = U_1, E_2^j = U_2\}$. Combining these relations with \eqref{eq:lem6.7-ind-12}, we conclude \eqref{eq:lem6.7-ind}.
    \medskip

    \noindent\textit{Step 4: Zero-one argument.} We are now ready to prove that $\mathbb{P}[\mathcal{Z}] \in \{0,1\}$ using \eqref{eq:lem6.7-equivalence} and \eqref{eq:lem6.7-ind}. By \eqref{eq:lem6.7-equivalence}, on the event $\mathcal{Z}$, we have $E_1 \cap \partial B_{2r}(0) = E_2 \cap \partial B_r(0) = \emptyset$. For a sufficiently large integer $j$ (which may depend on the realization of $\widetilde D$), we further have $E_1^j \cap \partial B_{2r}(0) = E_2^j \cap \partial B_r(0) = \emptyset$.
    Therefore, summing over all possible realizations of $E_1^j$ and $E_2^j$, which are $(r,j)$-dyadic domains and are countable, we obtain
    \begin{equation}
    \label{eq:lem6.7-ind-2} 
        \mathbb{P}[\mathcal{Z}] = \lim_{j \rightarrow \infty} \sum_{\substack{U \mbox{ }{\rm is}\mbox{ }(r,j)\mbox{-}{\rm dyadic}\\ U \cap \partial B_{2r}(0) = \emptyset}} \mathbb{P}[E_1^j = U] =\lim_{j \rightarrow \infty} \sum_{\substack{U \mbox{ }{\rm is}\mbox{ }(r,j)\mbox{-}{\rm dyadic}\\ U \cap \partial B_r(0) = \emptyset}} \mathbb{P}[E_2^j = U]\,. 
    \end{equation}
    Similarly, by \eqref{eq:lem6.7-equivalence}, on the event $\mathcal{Z}$, we have $\mathfrak d_\infty(E_1^j, E_2^j) > r2^{-j}$ for a sufficiently large integer $j$ (which may depend on $\widetilde D$). Therefore, 
    \begin{equation*}
    \begin{aligned}
        \mathbb{P}[\mathcal{Z}] 
        &= \lim_{j \rightarrow \infty} \sum_{\substack{U_1,U_2 \mbox{ }{\rm are}\mbox{ }(r,j)\mbox{-}{\rm dyadic}\\ \mathfrak d_\infty(U_1,U_2) > r2^{-j} }} \mathbb{P}[E_1^j = U_1, E_2^j = U_2]\\
        & = \lim_{j \rightarrow \infty} \sum_{\substack{U_1,U_2 \mbox{ }{\rm are}\mbox{ }(r,j)\mbox{-}{\rm dyadic}\\ \mathfrak d_\infty(U_1,U_2) > r2^{-j} }} \mathbb{P}[E_1^j = U_1] \cdot \mathbb{P}[E_2^j = U_2]\,.
        \end{aligned}
    \end{equation*}
    The second equality is due to \eqref{eq:lem6.7-ind}. By definition, we have $\partial B_r(0) \subset E_1 \subset E_1^j$ and $\partial B_{2r}(0) \subset E_2 \subset E_2^j$. Therefore, for admissible pairs of $(U_1, U_2)$ in the above sum (that is, $\mathbb{P}[E_1^j = U_1] \cdot \mathbb{P}[E_2^j = U_2] \neq 0$), we have $\partial B_r(0) \subset U_1$ and $\partial B_{2r}(0) \subset U_2 $, and thus $U_1 \cap \partial B_{2r}(0) = \emptyset$ and $U_2 \cap \partial B_r(0) = \emptyset$. Combining this fact with \eqref{eq:lem6.7-ind-2} yields 
    \begin{equation*}
        \mathbb{P}[\mathcal{Z}] \leq \lim_{j \rightarrow \infty} \sum_{\substack{U_1,U_2 \mbox{ }{\rm are}\mbox{ }(r,j)\mbox{-}{\rm dyadic}\\ U_1 \cap \partial B_{2r}(0) = \emptyset, U_2 \cap \partial B_r(0) = \emptyset  }} \mathbb{P}[E_1^j = U_1] \cdot \mathbb{P}[E_2^j = U_2] = \mathbb{P}[\mathcal{Z}]^2.
    \end{equation*}
    This implies that $\mathbb{P}[\mathcal{Z}] = 0$ or $1$. \qedhere
\end{proof}

We now complete the proof of Proposition~\ref{prop:non-degenerate} by combining Lemmas~\ref{lem:non-degenerate-cross} and \ref{lem:zero-one}.

\begin{proof}[Proof of Proposition~\ref{prop:non-degenerate}]
Fix $q \in (q_2,1)$ and any $r \in (0, 1/10)$. We will show that 
\begin{equation}
\label{eq:prop6.4-0}
    \lim_{\epsilon \rightarrow 0} \liminf_{n \rightarrow \infty} \mathbb{P}[D_n(\mbox{across } B_{2r}(0) \backslash B_{r}(0)) > \epsilon a_n^{(q)}] = 1\,.
\end{equation}
Suppose that the above inequality does not hold. Then, there exists a constant $\delta>0$ such that for any $\epsilon>0$, there exists an increasing sequence $\{n_k^{(\epsilon)}\}_{k \geq 1}$ satisfying
\begin{equation}
\label{eq:prop6.4-1}
    \mathbb{P}\Big[D_{n_k^{(\epsilon)}}(\mbox{across } B_{2r}(0) \backslash B_{r}(0))/ a_{n_k^{(\epsilon)}}^{(q)} > \epsilon \Big] < 1 - \delta \quad \forall k \geq 1\,.
\end{equation}
Furthermore, we can require that for any $0 < \epsilon' < \epsilon$, the sequence $\{n_k^{(\epsilon')}\}_{k \geq 1}$ is a subsequence of $\{n_k^{(\epsilon)}\}_{k \geq 1}$.

We apply Proposition~\ref{prop:tightness} with $U = \mathbb{R}^d$ and use the dyadic argument to the family of sequence $\{n_k^{(\epsilon)}\}_{k \geq 1}$ for $\epsilon$ being the inverse of an integer. This allows us to select an increasing sequence $\{n_k \}_{k \geq 1}$ such that $\{D_{n_k}(\cdot,\cdot)/a_{n_k}^{(q)}\}_{n_k \geq 1}$ converges in distribution to a pseudo-metric $\widetilde D$ on $\mathbb{R}^d$, and this sequence satisfies \eqref{eq:prop6.4-1} for any $\epsilon>0$. Recall from Lemma~\ref{lem:zero-one} that $\mathcal{Z} = \{\widetilde{D}(\mbox{across }B_{2r}(0) \backslash B_r(0))>0\}$ and $\mathbb{P}[\mathcal{Z}] = 0$ or $1$. 

Using \eqref{eq:prop6.4-1}, we obtain that for any $\epsilon>0$
\begin{align*}
    &\mathbb{P}[\widetilde{D}(\mbox{across }B_{2r}(0) \backslash B_r(0)) >\epsilon ] \\
    &\qquad \qquad\leq \liminf_{k \rightarrow \infty} \mathbb{P}[D_{n_k}(\mbox{across }B_{2r}(0) \backslash B_r(0)) / a_{n_k}^{(q)} >\epsilon ] \leq  1-\delta\,.
\end{align*}
Therefore, $\mathbb{P}[\mathcal{Z}] \leq 1-\delta$. Furthermore, applying Lemma~\ref{lem:non-degenerate-cross} with the constant $c$ therein, we obtain
\begin{equation*}
\begin{aligned}
&\mathbb{P}[\mathcal{Z}] \geq \mathbb{P}[\widetilde{D}(\mbox{across }B_{2r}(0) \backslash B_r(0)) \geq c ]\\
&\qquad\qquad \geq \limsup_{k \rightarrow \infty} \mathbb{P}[D_{n_k}(\mbox{across }B_{2r}(0) \backslash B_r(0))/a_{n_k}^{(q)} \geq c] > 0\,.
\end{aligned}
\end{equation*}
This contradicts our zero-one law for $\mathbb{P}[\mathcal{Z}]$ from Lemma~\ref{lem:zero-one}. Hence, Equation~\eqref{eq:prop6.4-0} holds for any $r \in (0,1/10)$. Proposition~\ref{prop:non-degenerate} follows directly from this fact.
\end{proof}

\section{Open problems}
\label{sec:open-problem} 

Here, we list some open problems and potential future directions concerning exponential metrics associated with log-correlated Gaussian fields. For potential relations with other models, we refer to Section~\ref{subsec:potential-relation}.
\subsubsection*{Uniqueness of limiting metrics}

\begin{prob} \label{prob:uniqueness}
    Prove the uniqueness of the subsequential limiting metric in Theorem~\ref{thm:tightness}.
\end{prob}

A natural approach for Problem~\ref{prob:uniqueness} is to adapt the arguments used to prove uniqueness of the LQG metric in dimension two in~\cite{gm-uniqueness, dg-uniqueness}. Probably, one would want to follow the argument in~\cite{dg-uniqueness}, which does not use confluence of geodesics, since it is unclear whether this property holds in dimension $d\geq 3$ (see Problem~\ref{prob:confluence} below). The uniqueness proof in~\cite{dg-uniqueness} does not appear to use two-dimensionality in a way that is as fundamental as the proof of tightness from~\cite{dddf-lfpp}, but we expect that nevertheless non-trivial ideas would be required to adapt the argument.  
We note that in the recent paper \cite{dfh-long-range-perc}, the authors established the uniqueness of the metric associated with long-range percolation in arbitrary dimensions using techniques inspired by \cite{gm-uniqueness, dg-uniqueness}.

\subsubsection*{Properties of $Q(\xi)$}

We know relatively little about the properties of $Q(\xi)$. Recall from Lemma~\ref{lem:Q-lower} that $\xi \mapsto Q(\xi)$ is a continuous and non-increasing function. However, it remains uncertain whether $Q$ is always positive. In the two-dimensional case, the positivity of $Q$ for every $\xi>0$ is established in \cite{dg-supercritical-lfpp}, building on the result from \cite{lfpp-pos}, when $h$ is a two-dimensional Gaussian free field or a minor variant thereof. The method in \cite{lfpp-pos} crucially relies on the Markov property of the Gaussian free field and some properties of its level sets. It remains unclear whether the methods of that paper can be extended to higher dimensions.

\begin{prob}
\label{prob:positive-Q}
    Prove or disprove that $Q(\xi)>0$ for every $\xi>0$.
\end{prob}

\medskip
\noindent \textit{Update:} It is shown in~\cite[Theorem 1.7]{dgz-thick-pts} that for each $\xi > 0$, the set-to-set distance exponent for the exponential metric is negative for each sufficiently large $d$. That is, for any fixed $A  , \xi > 0$, for any fixed disjoint compact sets $K_1,K_2 \subset \mathbb R^d$ with non-empty interiors, it holds for each sufficiently large $d \in \mathbb N$ that $\BB P[ D_n(K_1,K_2) \leq 2^{-\xi A n} ] \to 1$ as $n \to \infty$. With some technical work, it should be possible to deduce from this that for each fixed $\xi > 0$, we have $Q(\xi) = 0$ for all large enough $d \in \mathbb N$. 
\medskip

One of the most fundamental open problems in the theory of the LQG metric is to determine the relation between $\xi$ and $Q$. Indeed, in the subcritical case when $Q(\xi) >2$, this is equivalent to computing the Hausdorff dimension of the LQG metric space. We refer to Problem 5.1 of \cite{ddg-metric-survey} for the state of the art. The explicit value of $Q$ is not known except in the case $Q(1/\sqrt{6}) = 5/\sqrt{6}$, which corresponds to the fact that the Hausdorff dimension of the $\sqrt{8/3}$-LQG metric space is 4. It is natural to ask about the relation between $\xi$ and $Q$ in higher dimensions and whether we can determine the value of $Q$ for any specific choices of $\xi$ (and log-correlated Gaussian field).

\begin{prob} \label{prob:special}
Is there any special value of $\xi$ (and a log-correlated Gaussian field) for which the value of $Q(\xi)$ can be explicitly calculated?
\end{prob}

In a different flavour, it is also natural to ask about the asymptotic behavior of $Q(\xi)$ when $\xi$ is small. In \cite{ding-goswami-watabiki}, a lower bound for the Hausdorff dimension of the $\gamma$-LQG metric is derived when $\gamma$ is small.  Equivalently, they established that $1-\xi Q \geq c \xi^{4/3}/\log(\xi^{-1})$ for small $\xi$, where the exponent $4/3$ is expected to be sharp. It is natural to ask about the analog of this asymptotic in higher dimensions.

\begin{prob}
    Derive the asymptotics of $1- \xi Q$ as $\xi$ tends to zero.
\end{prob}

\subsubsection*{Metric and geodesic properties}
 
We now list some open problems about the properties of the limiting metric. 

Let $\wt D$ be a subsequential limit of the metrics $\lambda_n^{-1} D_n$, as in Theorem~\ref{thm:tightness}.
As a consequence of Theorem~\ref{thm:tightness}, Proposition~\ref{prop:non-degenerate}, and Lemma~\ref{lem:dist-to-infinity}, together with a straightforward Arz\'ela-Ascoli argument, we can establish that $\wt D$ is a geodesic metric space, i.e., for any two points $z,w \in \mathbb R^d$, there exists a path from $z$ to $w$ of $\wt D$-length equal to $\wt D(z,w)$.  In the two-dimensional case, when the underlying field is a two-dimensional Gaussian free field or a variant thereof, it has been established in \cite{gm-confluence, dg-confluence} that the geodesics of the LQG metric satisfy the \textit{confluence property}. Namely, for any fixed $x\in \mathbb{R}^2$ and any arbitrary $y,z\in\mathbb R^2$, the geodesics from $x$ to $y$ and from $x$ to $z$ coincide for a non-trivial initial interval of time.

\begin{prob}
\label{prob:confluence} 
Prove or disprove that the geodesics for $\wt D$ satisfy the confluence property. Does the answer depend on $d$?
\end{prob}

It is also natural to ask about possible \textit{geodesic networks} which can arise for the subsequential limiting metric $\wt D$, i.e., the possible topologies of the set of geodesics from $x$ to $y$ for distinct points $x,y\in\mathbb R^d$. See \cite{akm-geodesics,mq-strong-confluence} for results concerning geodesic networks for the $\sqrt{8/3}$-LQG metric (i.e., the Brownian map) and~\cite{gwynne-geodesic-network} for results concerning geodesic networks for the $\gamma$-LQG metric with general $\gamma\in(0,2)$.

\begin{prob}
    What can be said about geodesic networks for the subsequential limiting metric $\wt D$? Are there differences depending on the dimension? Are there any interesting behaviors when the dimension is sufficiently high?
\end{prob}

As was already alluded to in the introduction, the Hausdorff dimension of $\mathbb{R}^2$ with respect to the LQG metric in the subcritical case $Q(\xi) >2$ is equal to $\gamma/\xi$, where $\gamma \in (0,2)$ satisfies $Q = 2/\gamma + \gamma/2$~\cite[Corollary 1.7]{gp-kpz}. 

\begin{prob} \label{prob:dimension}
Give a formula for the Hausdorff dimension of $\mathbb R^d$ with respect to $\wt D$ in terms of $\xi$ and $Q$.
\end{prob}
 
In the two-dimensional case, there are also a number of additional results concerning Hausdorff dimensions of random fractals associated with the LQG metric.  
For example, the Hausdorff dimension of the boundary of an LQG metric ball is computed (in terms of $\xi$ and $Q$) in~\cite{gwynne-ball-bdy}. Moreover, from~\cite{gp-kpz}, one has a version of the KPZ formula~\cite{kpz-scaling} which relates the Hausdorff dimensions of a deterministic set $X \subset\mathbb R^2$ (or a random set sampled independently from the field) with respect to the LQG metric and with respect to the Euclidean metric. 

\begin{prob} \label{prob:more-dimension}
Compute the Hausdorff dimensions of interesting fractals associated with $\wt D$ (with respect to both $\wt D$ itself and the Euclidean metric), e.g., metric ball boundaries, geodesics, and sets of the form $\{z \in\mathbb R^d : \wt D(z,x) = \wt D(z,y)\}$ for fixed $x,y\in\mathbb R^d$.
\end{prob}

\begin{prob} \label{prob:kpz}
Is there a version of the KPZ formula for the metric $\wt D$, i.e., a formula relating the Hausdorff dimensions of a deterministic set $X\subset\mathbb R^d$ with respect to $\wt D$ and with respect to the Euclidean metric?
\end{prob}

\subsubsection*{Hypersurfaces}

Continue to let $\wt D$ be a subsequential limiting metric as in Theorem~\ref{thm:tightness}. 
A novel feature of random metrics on $\mathbb R^d$ for $d\geq 3$, which one does not see for $d\leq 2$, is the presence of hypersurfaces. We expect that if $M\subset \mathbb R^d$ is a deterministic topological submanifold of dimension $k \leq d-1$, then the internal metric of $\wt D$ on $M$ is a.s.\ infinite (see~\cite[Proposition 4.1]{lqg-metric-estimates} for a result along these lines in dimension 2). However, one could ask about random fractal submanifolds. 

\begin{prob} \label{prob:submanifold}
Let $k\leq d-1$. Are there any natural random sets $M  \subset \mathbb R^d$ with the topology of a $k$-dimensional manifold which are $\wt D$-rectifiable, in the sense that the internal metric $\wt D(\cdot,\cdot;M)$ is finite? 
\end{prob}

A possibly related problem is the following. 
Recall that if $g$ is a smooth Riemannian metric on $\mathbb R^d$ and $k\leq d-1$, a \textit{minimal hypersurface} for $g$ is a $k$-dimensional submanifold of $\mathbb R^d$ which (at least locally) has minimal $k$-dimensional volume among all other $k$-dimensional submanifolds with the same boundary. Note that a 1-dimensional minimal hypersurface is a geodesic.

\begin{prob} \label{prob:minimal}
Is there a notion of minimal surfaces with respect to the metric $\wt D$? If so, what can be said about their geometry (e.g., Hausdorff dimension, properties of the internal metric, interaction between different minimal hypersurfaces)?
\end{prob}

\subsubsection*{Connection to discrete models}

It would be of substantial interest to find a discrete model which is related to exponential metrics for log-correlated Gaussian fields in higher dimensions in a similar manner to how random planar maps are (at least conjecturally) related to the LQG metric. See Section~\ref{subsec:potential-relation} for additional discussion and references.

\begin{prob}
Find a natural discrete random geometry whose scaling limit is described by the exponential metrics as in Theorem~\ref{thm:tightness}.
\end{prob}

\subsubsection*{Supercritical case}

In this paper, we consider the case where $Q(\xi) > \sqrt{2d}$ and prove that each subsequential limit of the metric is indeed a metric. For $Q(\xi) \in (0, \sqrt{2d})$, we expect that if a limiting metric exists, it will have similar behavior to the supercritical LQG metric considered in \cite{dg-supercritical-lfpp, dg-uniqueness} (see also Remark~\ref{remark-critical}).

\begin{prob} \label{prob:supercritical}
   When $\xi$ satisfies $Q(\xi) \in (0,\sqrt{2d}]$, do the renormalized exponential metrics $\lambda_n^{-1} D_n$ converge to a limiting metric with respect to some topology? What properties does this limiting metric have?
\end{prob}
 
In the two-dimensional case, it is proven in~\cite{dg-critical-lqg} that in the critical case $Q(\xi) = 2$, the limiting metric induces the Euclidean topology on $\mathbb R^2$. It is unclear whether the same is true in higher dimensions in the case when $Q(\xi) =\sqrt{2d}$. 
As mentioned in Problem~\ref{prob:positive-Q}, it remains uncertain whether $Q$ is always positive. The case where $Q(\xi) \leq 0$ seems more mysterious. Our best guess is that if there exist values of $\xi$ for which $Q(\xi) < 0$, then it is not possible to extract any limiting metric for these values of $\xi$.

\appendix

\section{Index of notation}
\label{appendix:index}
 
Here we record some commonly used symbols in the paper, along with their meaning and the location where they are first defined. Local notations will not be included.

\begin{multicols}{2}
\begin{itemize}
\item $d$: dimension; Subsection~\ref{subsec:intro-1}.
\item $\KK(x)$: convolution kernel; Subsection~\ref{subsec:intro-1}.
\item $\rr$: support radius of the convolution kernel; Subsection~\ref{subsec:intro-1}.
\item $\xi$: parameter in the exponential metric; Subsection~\ref{subsec:intro-1}.
\item $\lambda_n$: median of point-to-point distance; \eqref{eq:intro-lambda}.
\item $\mathscrL_n$: subset of rescaled lattice; \eqref{eq:def-rescaled-lattice}.
\item $h_n$, $h_{m,n}$: approximations of log-correlated Gaussian field; \eqref{eq:field-def} and Definition~\ref{def:log-field}.
\item $D_n$, $D_{m,n}$: exponential metric; Definition~\ref{def:LFPP}.
\item $e_i$: $i$-th standard basis vector; \eqref{eq:def-ei}.
\item $a_n^{(p)}$: quantile of point-to-point distance; \eqref{def:quantile}.
\item $Q(\xi)$: decay exponent of $\lambda_n$; Proposition~\ref{prop:exponent}.
\item ${\rm Diam}_n$, ${\rm Diam}_{m,n}$: diameter; \eqref{def:diameter-1}.
\item $\mathscrE_{m,n}$, $\mathscrG_n$, and $\mathscrF_{m,n}$: events that are used to bound diameter; \eqref{def:e-m-n}, \eqref{def:e-n-n}, and \eqref{eq:def-f-m-n}.
\item $D_{m,n}(\mbox{across})$: distance across a hypercubic shell; \eqref{eq:def-box-cross}.
%\item $f_n^{(p)}$: quantile for distance across box; \eqref{def:quantile}.
\item $D_{m,n}(\mbox{around})$: distance around a hypercubic shell; \eqref{eq:def-around}.
\item $\alphaC$, $\lambdaC$, and $\RC$: fixed constants in coarse-graining argument; \eqref{eq:def-lambda}.
\item $\mathsf{a}_i$: coarse-graining scale; \eqref{eq:def-ai}.
\item $\mathscrY_i$, $\mathscrY_0$: sets in $\mathbb{Z}^d$; \eqref{eq:def-am}
\item $\qq_j$: probability of bad boxes; \eqref{eq:def-qj}.
\item $\mathcalJ_1$, $\mathcalJ_2$: events used to construct a covering; \eqref{eq:def-j12}.
\item $\mathcalX_j$: centers of boxes in the covering; Proposition~\ref{prop:cover}.
\item $\mathcalU$: domain covered by boxes centered at points of $\mathcalX_j$; \eqref{eq:def-mathcal-U}.
\item $\widetilde D$: subsequential limit of the exponential metrics; Subsection~\ref{subsec:non-degenerate}.

\end{itemize}
\end{multicols}

\bibliographystyle{alpha}
\bibliography{theta,cibib}

\end{document}